\documentclass[final]{amsart} 
\usepackage[pdftex]{graphicx}
\usepackage{mathrsfs}
\usepackage{verbatim}
\usepackage{lmodern}
\usepackage{dsfont}
\usepackage{bbm}

\usepackage{enumitem}
\setlist{nosep}

\usepackage[hidelinks]{hyperref}

\usepackage{amsmath,amssymb,amsthm,amscd,yhmath,mathtools}
\usepackage[all,cmtip]{xy}
\usepackage[pdftex]{graphicx}
\usepackage{tikz}
\usetikzlibrary{decorations.markings}
\usetikzlibrary{arrows}

\theoremstyle{plain}
\newtheorem{thm}{Theorem}[section]
\newtheorem{prop}[thm]{Proposition}
\newtheorem{cor}[thm]{Corollary}
\newtheorem{lemma}[thm]{Lemma}

\theoremstyle{definition}
\newtheorem{defn}[thm]{Definition}
\newtheorem*{rem}{Remark}

\newcommand{\e}{\varepsilon}
\newcommand{\sm}{\smallsetminus}
\newcommand{\es}{\ensuremath{\varnothing}}

\newcommand{\A}{\mathbb{A}}
\newcommand{\C}{\mathbb{C}}
\newcommand{\F}{\mathbb{F}}
\newcommand{\G}{\mathbb{G}}

\newcommand{\Q}{\mathbb{Q}}
\newcommand{\R}{\mathbb{R}}
\newcommand{\T}{\mathbb{T}}
\newcommand{\Z}{\mathbb{Z}}

\renewcommand{\O}{\mathcal{O}}

\newcommand{\Ac}{\mathcal{A}}
\newcommand{\Bc}{\mathcal{B}}
\newcommand{\Cc}{\mathcal{C}}
\newcommand{\Dc}{\mathcal{D}}
\newcommand{\Ec}{\mathcal{E}}
\newcommand{\Fc}{\mathcal{F}}
\newcommand{\Gc}{\mathcal{G}}
\newcommand{\Hc}{\mathcal{H}}
\newcommand{\Ic}{\mathcal{I}}
\newcommand{\Jc}{\mathcal{J}}
\newcommand{\Kc}{\mathcal{K}}
\newcommand{\Lc}{\mathcal{L}}
\renewcommand{\Mc}{\mathcal{M}}
\newcommand{\Nc}{\mathcal{N}}
\newcommand{\Oc}{\mathcal{O}}
\newcommand{\Pc}{\mathcal{P}}

\newcommand{\Rc}{\mathcal{R}}
\newcommand{\Sc}{\mathcal{S}}

\newcommand{\Uc}{\mathcal{U}}
\newcommand{\Vc}{\mathcal{V}}
\newcommand{\Wc}{\mathcal{W}}
\newcommand{\Xc}{\mathcal{X}}
\newcommand{\Yc}{\mathcal{Y}}
\newcommand{\Zc}{\mathcal{Z}}

\newcommand{\Gs}{\mathscr{G}}
\newcommand{\Is}{\mathscr{I}}

\newcommand{\Ns}{\mathscr{N}}

\newcommand{\mf}{\mathfrak{m}}

\newcommand{\pf}{\mathfrak{p}}
\newcommand{\qf}{\mathfrak{q}}

\newcommand{\Mf}{\mathfrak{M}}

\newcommand{\Pf}{\mathfrak{P}}

\newcommand{\Rhat}{\widehat{R}}

\newcommand{\Pt}{\widetilde{P}}
\newcommand{\Stilde}{\widetilde{S}}

\newcommand{\omegat}{\widetilde{\omega}}
\newcommand{\vt}{\widetilde{v}}
\newcommand{\Sigmat}{\widetilde{\Sigma}}
\newcommand{\Deltat}{\widetilde{\Delta}}
\newcommand{\Sct}{\widetilde{\Sc}}

\DeclareMathOperator{\init}{in}

\newcommand{\Fbar}{{\overline{\F}}}
\newcommand{\Ebar}{{\overline{E}}}
\newcommand{\Lbar}{{\overline{L}}}

\newcommand{\Rbar}{\overline{R}}

\newcommand{\Mbar}{\overline{M}}
\newcommand{\Nbar}{\overline{N}}
\newcommand{\Tbar}{{\overline{T}}}
\newcommand{\Ubar}{{\overline{U}}}
\newcommand{\ebar}{{\overline{e}}}

\newcommand{\Qbar}{{\overline{\Q}}}

\newcommand{\kbar}{{\overline{k}}}
\newcommand{\Bbbkbar}{{\overline{\Bbbk}}}

\newcommand{\pibar}{{\overline{\pi}}}

\newcommand{\lambdabar}{\overline{\lambda}}
\newcommand{\omegabar}{\overline{\omega}}

\newcommand{\Hct}{\widetilde{\Hc}}
\newcommand{\Ict}{\widetilde{\Ic}}

\newcommand{\Nct}{\widetilde{\Nc}}
\newcommand{\Rct}{\widetilde{\Rc}}

\newcommand{\rbar}{\overline{r}}
\newcommand{\rhobar}{\overline{\rho}}

\newcommand{\ul}{\underline}

\newcommand{\aun}{\mathbf{a}}
\newcommand{\bun}{\mathbf{b}}
\newcommand{\nun}{\underline{n}}
\newcommand{\mun}{\underline{m}}

\newcommand{\ds}{\displaystyle}

\DeclareMathOperator{\Hom}{Hom}
\DeclareMathOperator{\RHom}{RHom}
\DeclareMathOperator{\Aut}{Aut}
\DeclareMathOperator{\End}{End}
\DeclareMathOperator{\Gal}{Gal}

\DeclareMathOperator{\Art}{Art}
\DeclareMathOperator{\Frob}{Frob}
\DeclareMathOperator{\nInd}{n-Ind}
\DeclareMathOperator{\Irr}{Irr}
\DeclareMathOperator{\WD}{WD}

\DeclareMathOperator{\length}{length}
\DeclareMathOperator{\height}{height}
\DeclareMathOperator{\diag}{diag}

\DeclareMathOperator{\LT}{LT}
\DeclareMathOperator{\LC}{LC}
\DeclareMathOperator{\LM}{LM}

\DeclareMathOperator{\GL}{GL}

\DeclareMathOperator{\Spec}{Spec}

\DeclareMathOperator{\Proj}{Proj}

\DeclareMathOperator{\Cl}{Cl}
\DeclareMathOperator{\supp}{supp}

\DeclareMathOperator{\stab}{stab}

\DeclareMathOperator{\im}{im}
\DeclareMathOperator{\id}{id}
\DeclareMathOperator{\ab}{ab}
\DeclareMathOperator{\ad}{ad}
\DeclareMathOperator{\tr}{tr}

\DeclareMathOperator{\Nm}{Nm}
\DeclareMathOperator{\op}{op}

\DeclareMathOperator{\st}{st}
\DeclareMathOperator{\St}{St}
\renewcommand{\sp}{\operatorname{sp}}
\DeclareMathOperator{\Sp}{Sp}
\newcommand{\triv}{\mathbbm{1}}

\DeclareMathOperator{\loc}{loc}
\DeclareMathOperator{\rec}{rec}
\DeclareMathOperator{\univ}{univ}
\DeclareMathOperator{\unip}{unip}
\renewcommand{\min}{\operatorname{min}}
\renewcommand{\mod}{\operatorname{mod}}
\DeclareMathOperator{\HT}{HT}

\DeclareMathOperator{\crys}{crys}

\DeclareMathOperator{\Div}{Div}
\renewcommand{\div}{\operatorname{div}}
\DeclareMathOperator{\ord}{ord}

\DeclareMathOperator{\depth}{depth}

\newcommand{\Char}{\operatorname{char}}

\DeclareMathOperator{\Span}{span}

\newcommand{\del}{\partial}

\newcommand{\Set}{{\bf Set}}

\DeclareMathOperator{\rank}{rank}

\newcommand{\Ssf}{\mathsf{S}}

\newcommand{\LJ}{\mathbf{LJ}}
\newcommand{\JL}{\mathbf{JL}}

\newcommand{\BC}{\mathrm{BC}}

\newcommand{\into}{\hookrightarrow}
\newcommand{\onto}{\twoheadrightarrow}
\newcommand{\isomto}{\xrightarrow{\sim}}

\newcommand{\Grobner}{Gr\"obner\ }

\title{Mod $\ell$ multiplicities in certain $U(4)$ Shimura varieties}
\author{Jeffrey Manning}

\begin{document}

\maketitle

\begin{abstract}
We use the Taylor--Wiles--Kisin patching method to investigate the multiplicities with which Hecke eigensystems appear in the mod-$\ell$ cohomology of unitary Shimura sets, associated to central simple algebras of the form $B=M_2(D)$, for $D$ a nonsplit quaternion algebra over a $CM$ field. We follow a similar strategy to the one used in our prior work for Shimura curves, exploiting the natural self-duality in this setting. Our method requires a careful analysis of certain irreducible components of local deformation rings. We introduce and analyze a new local model for local deformation rings specific to the case of a banal prime, which is significantly better behaved than the standard local models for local deformation rings.

Our main result is a ``multiplicity $2^a$ result'', in the case where quaternion algebra $D$ ramifies only at banal primes, where $a$ is the number of places in the discriminant of $D$ which satisfy a certain Galois theoretic condition. We also prove an additional statement about the endomorphism ring of the cohomology of the Shimura set, which has an application to the study of congruence modules.
\end{abstract}

\renewcommand{\baselinestretch}{0}\normalsize
\tableofcontents
\renewcommand{\baselinestretch}{1.0}\normalsize

\section{Introduction}\label{sec:intro}

In their proof of Fermat's Last Theorem, Taylor and Wiles \cite{Wiles,TaylorWiles} introduced what is now known as the \emph{Taylor--Wiles--Kisin patching method}. This method has since become one of the most powerful tools in the Langlands program.

In the modern formulation of this method (due to Kisin \cite{Kisin} and others) one considers a ring $R_\infty = \widehat{\bigotimes}_v R^\square_v[[x_1,\ldots,x_g]]$,
where $v$ runs over a finite set of primes and each $R_v^\square$ is a local (framed) Galois deformation ring at $v$, and constructs a nicely behaved module $M_\infty$ over $R_\infty$, called a \emph{patched module}.

The module $M_\infty$ is constructed by gluing together various cohomology groups on certain Shimura varieties, and hence its properties are closely related to those of the cohomology groups used in its construction. Determining the structure of $M_\infty$ as an $R_\infty$-module thus has many applications to the study of automorphic forms and Shimura varieties, beyond simply proving automorphy lifting results.

Of particular interest to this paper, the dimension of the special fiber $M_\infty/\mf_{R_\infty}M_\infty$ equals the multiplicity of a certain Hecke eigensystem appearing in the mod $\ell$ cohomology of a Shimura variety. Determining this multiplicity has applications towards the mod $\ell$ analogue of the classical Langlands correspondence (see for example \cite[Section 4]{BDJ}).

While the ring $R_\infty$ depends only on local Galois theoretic information at a finite set of primes, $M_\infty$ is explicitly constructed as global object. From the construction, it is not at all obvious that the $R_\infty$-module structure of $M_\infty$ does not depend global information, or even that it is independent of the (infinitely many) non-canonical choices made in its construction. 

The surprising property of $M_\infty$, which is the key to many applications of the patching method is that, despite its construction $M_\infty$ often behaves as if it is a purely local object, depending only on local Galois theoretic information at the same finite set of primes defining $R_\infty$. The recently formulated \emph{categorical Langlands program} \cite{EGH} turns this into a precise conjecture, predicting that $M_\infty$ can be explicitly described in terms of certain local objects appearing in the categorical local Langlands correspondence.

In the special case when $R_\infty$ is  regular, Diamond \cite{DiamondMult1} and Fujiwara \cite{Fujiwara} proved this conjecture (prior to the formulation of the general conjecture) by using the  Auslander--Buchsbaum formula to prove that $M_\infty$ is free over $R_\infty$. However in many more general situations the categorical Langlands conjectures predict that $M_\infty$ is \emph{not} free.

In the case when $M_\infty$ is associated to the cohomology of a Shimura curve or Shimura set over a totally real number field, prior work of the author in \cite{Manning} presented a new method for determining the structure of $M_\infty$ which can work in situations when $M_\infty$ is not free. The method exploits the natural self duality of $M_\infty$, which comes from the existence of a Hecke equivariant perfect pairing on the cohomology groups, as well as a computation of the Weil class group $\Cl(R_\infty)$ of $R_\infty$.

The patching method has been extended to the case of Shimura sets associated to  definite unitary groups of rank $n>2$ in \cite{CHT} and \cite{TaylorIharaAvoidance}. The method of \cite{Manning} does not immediately generalize to this setting, as the ``natural'' perfect pairing on the relevant cohomology groups is no longer Hecke equivariant. Roughly, the reason for this is that any matrix $A\in\GL_2(\Bbbk)$ is conjugate to its inverse transpose $^tA^{-1}$, up to a scalar, but the same is not true for matrices in $\GL_n(\Bbbk)$ for $n>2$.

However the work of \cite[Section 5.2]{CHT} shows that in certain cases, one can produce a Hecke equivariant perfect pairing in this context. Specifically if $G$ is a definite unitary group associated to a central simple algebra $B$ of rank $n$ over a CM field $F$, and $X_K$ (for $K\subseteq G(\A_{F^+,f})$ a sufficiently small compact open subgroup) is an associated Shimura set, then there will exist a Hecke equivariant perfect pairing on $H^0(X_k,\Z_\ell)$ provided that $B^{\op}\cong B$, or equivalently that either $B\cong M_n(F)$ or $B\cong M_{n/2}(D)$ for $D$ a quaternion algebra over $F$. In such cases it is thus possible to generalize the results of \cite{Manning} to such settings, provided that one can compute $\Cl(R_\infty)$.

The primary purpose of this paper is to compute $\Cl(R_\infty)$, and hence generalize the main result of \cite{Manning}, in what is arguably the simplest nontrivial case with $n>2$: the case where $B = M_2(D)$ for a quaternion algebra over $F$, such that all primes at which $D$ ramifies are \emph{banal}. We believe that our techniques (particularly the methods of computing the class groups used in Section \ref{sec:R22} and \ref{sec:R121}) can be used to prove similar results much greater generally. We will pursue this in future work.

Our main result (stated using notation and terminology defined later) is:

\begin{thm}\label{main theorem}
Let $F^+$ be a totally real number field with $F/F^+$ a quadratic CM extension. Let $\ell \ge 5$ be a prime which does not ramify in $F$ and such that all primes of $F^+$ lying above $\ell$ split in $F$. Let $E/\Q_\ell$ be a finite extension with ring of integers $\Oc$ and residue field $\F$.

Let $D$ be a nonsplit quaternion algebra over $F$ such that:
\begin{itemize}
	\item If $w$ is a finite prime of $F$ at which $D$ ramifies, then $w$ is split over $F^+$ (that is, there is a prime $v$ of $F^+$ which splits in $F$ as $v = ww^c$) and $D$ also ramifies at $w^c$;
	\item If $w$ is a finite prime of $F$ lying over $\ell$, then $D$ splits at $\ell$.
\end{itemize}
Let $\Delta$ be the set of finite places of $F^+$ lying below a prime of $F$ at which $D$ does not split.

Let $G/F^+$ be a unitary group associated to an involution of the second kind on $B:=M_2(D)$ (see Section \ref{sec:hecke} for a precise definition) which is compact at all infinite places of $F^+$ and quasisplit at all finite places $v\not\in \Delta$ of $F^+$ which split in $F$.

Let $K = \prod_v K_v\subseteq G(\A_{F^+,f})$ be a compact open subgroup such that $K_v$ is a maximal compact subgroup of $G(F^+_v)$ for \emph{all} places $v$, and is hyperspecial whenever $v\not\in\Delta$.

Let $X_K$ be the Shimura set associated to $K$, and let $M(K) = H^0(X_K,\Oc)$. Let $\T(K)$ be the Hecke algebra acting on $M(K)$ and let $\mf\subseteq \T(K)$ be a non-Eisenstein maximal ideal with $\T(K)/\mf = \F$. Let $\rhobar_{\mf}:G_{F^+}\to \Gc_4(\F)$ be the associated (absolutely irreducible) Galois representation, and assume that
\begin{enumerate}
\item $\breve{\rhobar}_{\mf}(G_{F(\zeta_\ell)})\subseteq \GL_4(\F)$ is adequate;
\item The extension $F/F^+$ is unramified at all finite places;
\item For each $v\in\Delta$, $v$ is \emph{banal}, in the sense that $v\nmid \ell$ and $q_v^i \not\equiv 1 \pmod{\ell}$ for all $i=1,2,3,4$, where $q_v := \Nm_{F^+/\Q}(v)$
\end{enumerate}
Then $\dim_{\F} M(K)/\mf M(K) = 2^a$, where $a$ is the number of places $v=ww^c\in\Delta$ such that $\breve{\rhobar}_{\mf}|_{G_{F_w}}$ is unramified with $\breve{\rhobar}_{\mf}(\Frob_w) \sim \lambdabar \diag(q_v^2,q_v,q_v,1)$ for some $\lambdabar\in \F^\times$. Moreover the map $R\to \T(K)_{\mf}$ is an isomorphism, where $R(=R^B_{\Sigma,\lambda,\mu})$ is the global Galois deformation ring defined in Section \ref{sec:patch}.
\end{thm}

As explained above, the strategy for proving Theorem \ref{main theorem} is similar to that of \cite{Manning}. Namely, we consider the patched module $M_\infty$ associated to $M(K)_{\mf}$ and show that this is a self dual $R_\infty^B:= R_\Sigma^{\loc,B}[[x_1,\ldots,x_g]]$-module (where $R_\Sigma^{\loc,B}$ is a completed tensor product of certain local Galois deformation rings, as defined in Section \ref{sec:patch}) of generic rank $1$, and hence must correspond to an element of $\Cl(R_\infty^B)$ satisfying $2M_\infty = \omega_{R_\infty^B}$. 

At this point, if $\Cl(R_\infty^B)$ is $2$-torsion free, this is enough to determine $M_\infty$ up to isomorphism, and in particular shows that $M_\infty$ depends only on the ring $R_\infty^B$, and hence only on the local Galois theoretic information determining $R_\infty^B$, confirming the predictions of the categorical Langlands program. Moreover if one can explicitly compute the unique $M\in \Cl(R_\infty^B)$ satisfying $2M = \omega_{R_\infty^B}$, then one can explicitly determine the structure of $M_\infty$, at which point one can easily deduce results like Theorem \ref{main theorem} by simply computing $M_\infty/\mf_{R_\infty^B}M_\infty$.

For convenience in this paper (similarly to as in \cite{Manning}) we instead work with $\Rbar_\infty^B := R_\infty^B/(\varpi)$ and thus only determine $\Mbar_\infty := M_\infty/\varpi M_\infty$ up to isomorphism (see Theorem \ref{thm:M_infty structure}), which is still sufficient for our main results. This is mainly to avoid having to prove a mixed characteristic version of Proposition \ref{prop:class group completion}, and so it is likely that a slightly more careful analysis would be enough to determine $M_\infty$ itself.

The conditions in Theorem \ref{main theorem} were mainly imposed in order facilitate this method, and to allow for the computation of $\Cl(R_\infty^B)$. Condition (1) is the standard Taylor--Wiles condition from \cite{ThorneAdequate}, and is needed to apply the patching method. Condition (2) is needed in order to apply the automorphic multiplicity one results of \cite{Labesse}, in order to show that $M_\infty$ has generic rank $1$.

The conditions that $K_v$ is a maximal compact subgroup of $G(F_v^+)$ for all $v\not\in\Delta$, as well as the conditions on $\ell$, ensure that the relevant local deformation rings are formally smooth at all primes $v\not\in\Delta$ and so
\[
R_\infty^B = \left(\widehat{\bigotimes_{v\in \Delta}} R_v\right)[[t_1,\ldots,t_s]]
\]
for some $s\ge 0$ and some rings $R_v$ depending only on $D_w$, $K_v$ and $\breve{\rhobar}_{\mf}|_{G_{F_w}}$ (for $v = ww^c$ in $F$). Moreover for all $v\in \Delta$, the condition that $K_v$ is a maximal compact subgroup of $G(F^+_v)$ implies that $R_v$ is the fixed type local deformation ring $R_v^\square(\breve{\rhobar}_{\mf}|_{G_{F_w}},\st_{2,2})$ from Section \ref{ssec:fixed type}. Thus it suffices to study the rings $R_v^\square(\breve{\rhobar}_{\mf}|_{G_{F_w}},\st_{2,2})$, and in particular to compute their class groups.

As shown in \cite{Shotton,Shotton2}, the scheme $\Mc(4,q_v)$ parameterizing pairs of matrices $\Phi,\Sigma\subset\GL_4$ satisfying $\Phi\Sigma\Phi^{-1} = \Sigma^{q_v}$ forms a local model for the unrestricted framed deformation ring $R^\square(\breve{\rhobar}_{\mf}|_{G_{F_w}})$, and a certain irreducible component of $\Mc(4,q_v)$ forms a local model for $R_v^\square(\breve{\rhobar}_{\mf}|_{G_{F_w}},\st_{2,2})$.

Unfortunately this local model can often be quite difficult to work with explicitly, as the equation $\Phi\Sigma\Phi^{-1} = \Sigma^{q_v}$ involves $32$ variables (given by the entries of the two matrices) with $16$ relations, and involves an arbitrary power of a $4\times 4$ matrix. Moreover it can be difficult to precisely determine the additional equations needed to cut out the component modeling $R_v^\square(\breve{\rhobar}_{\mf}|_{G_{F_w}},\st_{2,2})$.

The purpose of condition (3), namely that we only consider such rings at \emph{banal} primes, is to restrict our attention to a special case where $R_v^\square(\breve{\rhobar}_{\mf}|_{G_{F_w}},\st_{2,2})$ is far more manageable. In Section \ref{sec:banal local models}, we present a new local model, $\Nc^{\nun}$ for the ring $R^\square(\breve{\rhobar}_{\mf}|_{G_{F_w}})$, which is specific to the banal (and unipotent) case, which is much simpler to work with. In particular, this local model involves significantly fewer than 32 variables, and involves only quadratic relations between these variables. 

This model turns out to be simple enough to allow us to completely analyze all of the local deformation rings arising in the context of Theorem \ref{main theorem}. Specifically we show that $R_v^\square(\breve{\rhobar}_{\mf}|_{G_{F_w}},\st_{2,2})$ is formally smooth in all but two cases: when  $\breve{\rhobar}_{\mf}|_{G_{F_w}}$ is unramified and either $\breve{\rhobar}_{\mf}(\Frob_w) \sim \lambdabar \diag(q_v,q_v,1,1)$ or $\breve{\rhobar}_{\mf}(\Frob_w) \sim \lambdabar \diag(q_v^2,q_v,q_v,1)$ for some $\lambdabar\in \F^\times$. Hence if $R_v^\square(\breve{\rhobar}_{\mf}|_{G_{F_w}},\st_{2,2})$ is not formally smooth, then it is a power series ring over one of two specific rings corresponding to these two cases, called $\Rhat^{2,2}(\st_{2,2})$ and $\Rhat^{1,2,1}(\st_{2,2})$ respectively in the main body of the paper. 

Hence the proof of Theorem \ref{main theorem} reduces to a careful analysis of these two rings. We show that both of these rings are Cohen--Macaulay (which implies the final statement of Theorem \ref{main theorem}) and compute the class groups of the mod $\ell$ reductions of these rings, $\Rhat_{\F}^{2,2}(\st_{2,2})$ and $\Rhat_{\F}^{1,2,1}(\st_{2,2})$. We compute that $\Rhat^{2,2}(\st_{2,2})$ is Gorenstein with $\Cl(\Rhat_{\F}^{2,2}(\st_{2,2}))=0$ and $\Rhat^{1,2,1}(\st_{2,2})$ is non-Gorenstein, with $\Cl(\Rhat_{\F}^{1,2,1}(\st_{2,2}))\cong\Z$. From this we deduce that $\Cl(\Rbar_\infty^B)\cong \Z^a$ (where $a$ is as in the statement of Theorem \ref{main theorem}), which is particular is $2$-torsion free. This allows us to explicitly determine the structure of $\Mbar_\infty$, and hence to prove Theorem \ref{main theorem}.

Additionally, our computations imply that the natural duality map $\Mbar_{\infty}\otimes_{\Rbar_\infty^B}\Mbar_\infty\to \omega_{\Rbar_\infty^B}$ is surjective. As in \cite[Theorem 1.2]{Manning}, this implies the following additional result:

\begin{thm}\label{thm:balanced}
In the setting of Theorem \ref{main theorem}, the map $\T(K)_{\mf}\to \End_{\T(K)_{\mf}}(M(K)_{\mf})$ is an isomorphism.
\end{thm}

As shown in \cite[Theorem 3.12]{BKM1}, this implies an important statement about the \emph{congruence module} of $M(K)_{\mf}$, namely that the congruence modules of $M(K)_{\mf}$ and $\T(K)_{\mf}$ are equal.

While the main results of this paper are specific to the $\GL_4$ case, and to the case where $B = M_2(D)$ for $D$ a quaternion algebra, for the purposes of future applications much of the paper is written in far greater generality. In particular the local models in the banal case are constructed for $\GL_n$ for general $n$. We suspect that many parts of our argument can be generalized significantly. In particular, in the banal case, it should be possible to use our local model to compute $\Cl(\Rbar_\infty^B)$ in much greater generality. We shall return to this idea in future work.

Comparing Theorem \ref{main theorem} with other standard results in the literature (and with \cite[Theorem 1.1]{Manning}) one might expect the same result to hold with the condition that $K_v$ is maximal for all $v\not\in\Delta$ split in $F$ replaced by a \emph{minimal ramification} condition at all primes $v\not\in \Delta$ which split in $F$. Indeed by \cite[Lemma 2.4.19]{CHT} this would still imply that the relevant local deformation rings are formally smooth. We have primarily avoided doing this to avoid having to give a detailed treatment of types, and in particular the way in which they interact with self-duality (or with the generic rank $1$ result from Proposition \ref{prop:gen rank 1}), as this would take us too far afield for this current paper.

As a final note, the categorical Langlands conjectures predict that the main result of Theorem \ref{main theorem} should hold more generally. Specifically in Theorem \ref{main theorem}, rather than taking $B = M_2(D)$, instead let $B$ be a division algebra over $F$ of rank $4$ (and still assume all of the other conditions, in particular that $K_v$ is maximal compact for all $v$). In this case the patched module $M_\infty$ will be a module over the ring:
\[
R_\infty^B = \left(\widehat{\bigotimes_{v\in \Delta}} R_v\right)[[t_1,\ldots,t_s]]
\]
where $\Delta$ is the set of all places where $B$ does not split, and we have $R_v = R_v^\square(\breve{\rhobar}_{\mf}|_{G_{F_w}},\st_{4})$ if $B_w$ is a division algebra and $R_v^\square(\breve{\rhobar}_{\mf}|_{G_{F_w}},\st_{2,2})$ if $B_w = M_2(D_w)$ for a quaternion algebra $D_w$. In this situation, \cite[Section 5.2]{CHT} no longer implies that $M_\infty$ is self-dual. However the categorical Langlands conjectures now predict that $M_\infty$ can be decomposed as
\[
M_\infty = \left(\boxtimes_{v\in \Delta} M_v\right)[[t_1,\ldots,t_s]]
\]
where each $M_v$ is a module over $R_v$ depending only on $\breve{\rhobar}_{\mf}|_{G_{F_w}}$ and on $B_w$. The proof of Theorem \ref{main theorem} determines $M_v$ (or rather $\Mbar_v := M_v/\varpi M_v$) for all $v$ with $B_w = M_2(D_w)$, and so to determine $M_\infty$ one only needs to determines $M_v$ for the $v\in\Delta$ for which $B_w$ is a division algebra.

In particular, if one restricts attention to the case when $R_v = R_v^\square(\breve{\rhobar}_{\mf}|_{G_{F_w}},\st_{4})$ is formally smooth whenever $B_w$ is a division algebra (which by Proposition \ref{prop:Nc^1,...,1} happens, in particular, whenever $v$ is banal) one has $M_v = R_v$ for all such $v$, and so $M_\infty$ has the same form as in the proof of Theorem \ref{main theorem}. Hence the categorical Langlands conjectures imply that the results of Theorems \ref{main theorem} and \ref{thm:balanced} should hold in this more general context, and moreover they suggest that $M_\infty$ should indeed be self-dual in this case, even though this is not directly implied by the work of \cite[Section 5.2]{CHT}.

However in order to prove these results by the method of this paper,  we would need an unconditional proof that $M_\infty$ is still self-dual in this case. We intend to return to this issue in a future joint paper with David Helm. 
\subsection*{Outline of the paper}

This paper is divided into three parts. Part \ref{part:local} is concerned with the local aspects of the argument. Section \ref{sec:local Langlands} summarizes the relevant results from the local Langlands correspondence, while Section \ref{sec:deformation} summarizes the theory of local (framed) Galois deformation rings, and reviews the main properties of the local model $\Mc(n,q)$. Most of the results in these sections are fairly standard. Most of the new theoretical results are contained in Section \ref{sec:banal local models}, which introduces and studies the new local models in the banal case.

Part \ref{part:global} is concerned with the global aspects of the argument, and in particular the results needed for the patching method. Sections \ref{sec:global Galois}, \ref{sec:hecke} and \ref{sec:patch} cover, respectively, Global deformation rings, automorphic forms on unitary groups, and the main patching argument. All of the results in these sections are again standard. Section \ref{sec:M_infty} then uses the method of \cite{Manning} to compute $\Mbar_\infty$ and uses this to prove Theorems \ref{main theorem} and \ref{thm:balanced}.

Lastly, Part \ref{part:computations} is concerned with the explicit computations involving the rings $\Rhat^{2,2}(\st_{2,2})$ and $\Rhat^{1,2,1}(\st_{2,2})$. Section \ref{sec:comm alg} recalls the main commutative algebra results needed for these computations, and the rings $\Rhat^{2,2}(\st_{2,2})$ and $\Rhat^{1,2,1}(\st_{2,2})$ are studied in Sections \ref{sec:R22} and \ref{sec:R121}, respectively. These computations are somewhat involved, so for the convenience of the reader all of the results from these sections that are needed for our main results are summarized in Theorems \ref{thm:R22 main results} and \ref{thm:R121 main results}, and so the entirety of Part \ref{part:computations} may be skipped on a first reading.

\subsection*{Acknowledgments} 
This project has received funding from the European Research Council (ERC) under the European Union's Horizon 2020 research and innovation programme (grant agreement No. 884596). We would like to thank George Boxer, Matt Emerton, Toby Gee, David Helm, Jack Sempliner and Sug Woo Shin for helpful discussions.
\part{Local Theory}\label{part:local}

\section{Types and the local Langlands Correspondence}\label{sec:local Langlands}

In this section we will the properties of the local Langlands correspondence that we will need for the remainder of the paper. For the sake of simplicity we will not work in full generality, instead we will focus mainly on the case of unipotent representations, as this is all that will be needed for our results.

For this section, we will fix a prime $p$ and let $L$ be a finite extension of $\Q_p$ with residue field of order $q$. We will let $G_L = \Gal(\Lbar/L)$ be the absolute Galois group of $L$, $I_L\unlhd G_L$ be the inertia subgroup and $P_L\unlhd I_L$ be the wild inertia subgroup. Let $W_L$ be the Weil group of $L$.

Let $\varphi\in W_L$ be a lift of arithmetic Frobenius and let $\sigma\in I_L/P_L$ be a topological generator, so that $\varphi\sigma \varphi^{-1} = \sigma^q$ in $W_L/P_L$, and moreover $W_L/P_L$ is identified with the pro-finite completion of
\[
\left\langle \varphi,\sigma\middle|\varphi\sigma\varphi^{-1} =\sigma^q\right\rangle = \Z\ltimes \Z\left[\frac{1}{p} \right].
\]

\subsection{Inertial Types}

\begin{defn}
We say that a \emph{Weil--Deligne} representation of $L$ (of dimension $n$) is a pair $(r,N)$ where:
\begin{itemize}
\item $r:W_L\to \GL_n(\C)$ is a continuous representation;
\item $N\in M_n(\C)$ is a nilpotent matrix;
\item $Nr(g) = r(g)N$ for all $g\in I_L$;
\item $r(\varphi)N = qNr(\varphi)$.
\end{itemize}
and we say that $(r,N)$ is \emph{Frobenius semisimple} if $r$ is semisimple.

We say that an \emph{inertial type} is an equivalence class (up to conjugation by elements of $\GL_n(\C)$) of pairs $\tau = (r_\tau,N_\tau)$ where $r:I_L\to \GL_n(\C)$ is a continuous representation and $N_\tau\in M_n(\C)$ is nilpotent which extends to a Weil--Deligne representation $(r,N_\tau)$.

We say a Weil--Deligne representation $(r,N)$ has type $\tau = (r_\tau,N_\tau)$ if $(r|_{I_L},N)$ is $\GL_n(\C)$-conjugate to $\tau$.

If $\rho:G_L\to \GL_n(\Qbar_\ell)$ is a continuous Galois representation, we will use $\WD(\rho)$ to denote the associated Weil--Deligne representation, and we will say that $\rho$ has type $\tau$ if $\WD(\rho)$ does.
\end{defn}

We will now construct the specific Weil--Deligne representations and inertial types we will work with in this paper. First define a character $\|\ \|:W_L\to\C^\times$ by $\|g\| = 1$ for $g\in I_L$ and $\|\varphi\| = q$. Let $n$ be any positive integer and let $\chi:W_L\to \C^\times$ be a character, define the $n$-dimensional Weil--Deligne representation $\Sp_n(\chi) = (\sp_{n,\chi},N_n)$ by
\begin{align*}
\sp_{n,\chi}(g) &=
\begin{pmatrix}
\|g\|^{n-1}\chi(g)&&&\\
&\|g\|^{n-2}\chi(g)&&\\
&&\ddots&\\
&&&\chi(g)
\end{pmatrix}, &
N_n &=
\begin{pmatrix}
0&1&&&\\
&0&1&&\\
&&\ddots&\ddots&\\
&&&0&1\\
&&&&0
\end{pmatrix}
\end{align*}
which is clearly a Frobenius semisimple Weil--Deligne representation.

For any $k$-tuple of positive integers $\nun = (n_1,\ldots,n_k)$ and any characters $\chi_1,\ldots,\chi_k:W_L\to \C^\times$ we define
\[\Sp_{\nun}(\chi_1,\ldots,\chi_k) = \Sp_{n_1,\ldots,n_k}(\chi_1,\ldots,\chi_k) = \Sp_{n_1}(\chi_1)\oplus\cdots\oplus \Sp_{n_k}(\chi_k).\]
We say that the representation $\Sp_{\nun}(\chi_1,\ldots,\chi_k)$ is \emph{generic} if there is no pair $(i,j)$ with $\chi_j = \|\ \|^m\chi_i$ for some $m$ with $0\le m\le n_i$ and $m+n_j>n_i$. One can also define generic representations more generally, but we will omit this as we will not need to consider any Weil--Deligne representations not in the form $\Sp_{\nun}(\chi_1,\ldots,\chi_k)$.

We will say that a Weil--Deligne representation $(r,N)$ is \emph{unipotent} if it is Frobenius semisimple and $r|_{I_L}$ is trivial (i.e. $r$ is unramified), and note that the unipotent Weil--Deligne representations are exactly $\Sp_{\nun}(\chi_1,\ldots,\chi_k)$ for $\chi_1,\ldots,\chi_k$ unramified.

For any $\nun=(n_1,\ldots,n_k)$, consider the nilpotent matrix $N_{\nun} = N_{n_1,\ldots,n_k} = \begin{pmatrix}N_{n_1}&&\\&\ddots&\\&&N_{n_k}\end{pmatrix}$ where $N_{n_i}$ is as in the definition of $\Sp_{n}(\chi)$ above. Let $\tau_{\nun}=\tau_{n_1,\ldots,n_k}$ denote the inertial type $(\triv,N_{\nun})$. Note that up to conjugation, $\tau_{n_1,\ldots,n_k}$ is independent of the ordering of the $k$-tuple $(n_1,\ldots,n_k)$, and so we will often treat $\nun$ as an \emph{unordered} partition of $n$.

By definition, if $\Sp_{\nun}(\chi_1,\ldots,\chi_k)$ is a unipotent Weil--Deligne representation it is of type $\tau_{\nun}$, and moreover it is easy to see that any Frobenius semisimple Weil--Deligne representation of type $\tau_{\nun}$ is unipotent. We thus call the types $\tau_{\nun}$ \emph{unipotent} types.

\subsection{The local Langlands correspondence}

For any positive integer $n$, let $B_n\subseteq \GL_n(L)$ denote the Borel subgroup of upper triangular matrices. For any $k$-tuple $\nun = (n_1,\ldots,n_k)$ with $n_1+\cdots+n_k$ let $P_{\nun} = P_{n_1,\ldots,n_k}$ denote the parabolic subgroup
\[
P_{n_1,\ldots,n_k} = 
\begin{pmatrix}
\GL_{n_1}(L)&*&\cdots&*\\
0&\GL_{n_2}(L)&\cdots&*\\
\vdots&\vdots&\ddots&\vdots\\
0&0&\cdots&\GL_{n_k}(L)
\end{pmatrix}
\subseteq \GL_n(L)
\]
Let $M_{n_1,\ldots,n_k} = \GL_{n_1}\times\cdots\times\GL_{n_k}\subseteq P_{n_1,\ldots,n_k}$ be the corresponding Levi subgroup, which we will treat as a quotient of $P_{n_1,\ldots,n_k}$.

For any admissible representations $\pi_1,\ldots,\pi_k$ of $\GL_{n_1}(L),\ldots,\GL_{n_k}(L)$, define an admissible representation 
\[
\pi_1\times\pi_2\times\cdots\times\pi_k = \nInd_{P_{n_1,\ldots,n_k}}^{\GL_n(L)}(\pi_1\otimes\pi_2\otimes\cdots\otimes\pi_k)\]
of $\GL_n(L)$, where $\pi_1\otimes\pi_2\otimes\cdots\otimes\pi_k$ denotes the representation of $P_{n_1,\ldots,n_k}$ inflated from $M_{n_1,\ldots,n_k}$, and $\nInd_{P_{n_1,\ldots,n_k}}^{\GL_n(L)}$ denotes the normalized induction (as defined in \cite{HarrisTaylor}).

For any character $\chi:L^\times\to \C^\times$ and any integer $m$, we will let $\chi(m)$ be the character $\|\ \|_L^m\otimes\chi$. Now we recall the following standard results (cf \cite{Rodier}):

\begin{prop}
For any character $\chi$ and and integer $n$, the representation
\[\chi\times\chi(1)\times\cdots\times\chi(n-1)\]
of $\GL_n(L)$ has a unique irreducible quotient, which we will denote $\St_n(\chi)$.
\end{prop}

For any character $\chi$ and any integer $n$, let $\Delta(\chi,n)$ denote the set $\{\chi,\chi(1),\ldots,\chi(n-1)\}$. We say that $\Delta(\chi_1,n_1)$ and $\Delta(\chi_2,n_2)$ are \emph{linked} if $\Delta(\chi_1,n_1)\not\subseteq \Delta(\chi_2,n_2)$, $\Delta(\chi_2,n_2)\not\subseteq \Delta(\chi_1,n_1)$ and $\Delta(\chi_1,n_1)\cup \Delta(\chi_2,n_2) = \Delta(\chi',n')$ for some $\chi'$ and $n'$. We say that $\Delta(\chi_1,n_1)$ \emph{precedes} $\Delta(\chi_2,n_2)$ if $\Delta(\chi_1,n_1)$ and $\Delta(\chi_2,n_2)$ are linked and $\Delta(\chi_1,n_1)\cup \Delta(\chi_2,n_2) = \Delta(\chi',n')$ (i.e. if $\chi_2 = \chi_1(m)$ for some $0\le m\le n_1$).

Then we have

\begin{prop}
For any $k$-tuple of integers $(n_1,\ldots,n_k)$ with $n_1+\cdots+n_k=n$ and any characters $\chi_1,\ldots,\chi_k$. If the $n_i$'s and $\chi_i$'s are ordered so that $\Delta(\chi_i,n_i)$ does not precede $\Delta(\chi_j,n_j)$ for any $i<j$ then the representation
\[\St_{n_1}(\chi_1)\times\St_{n_2}(\chi_2)\times\cdots\times\St_{n_k}(\chi_k)\]
of $\GL_n(L)$ has a unique irreducible quotient, which we will denote $Q(\St_{n_1}(\chi_1)\times\St_{n_2}(\chi_2)\times\cdots\times\St_{n_k}(\chi_k))$. Moreover, if no two $\Delta(\chi_i,n_i)$'s are linked then the representation $\St_{n_1}(\chi_1)\times\St_{n_2}(\chi_2)\times\cdots\times\St_{n_k}(\chi_k)$ is already irreducible, and so $Q(\St_{n_1}(\chi_1)\times\St_{n_2}(\chi_2)\times\cdots\times\St_{n_k}(\chi_k)) = \St_{n_1}(\chi_1)\times\St_{n_2}(\chi_2)\times\cdots\times\St_{n_k}(\chi_k)$.
\end{prop}

Let $\rec_L$ denote the local Langlands correspondence, constructed in \cite[Theorem A]{HarrisTaylor}, normalized so that $\rec_L(\chi) = \chi\circ \Art_L^{-1}:W_L^{\ab}\to L^\times\to \C^\times$ for any character $\chi:L^\times\to\C^\times$. Then we have:

\begin{prop}\label{prop:local langlands}
For any $k$-tuple of integers $(n_1,\ldots,n_k)$ with $n_1+\cdots+n_k=n$ and any characters $\chi_1,\ldots,\chi_k$. If the $n_i$'s and $\chi_i$'s are ordered so that $\Delta(\chi_i,n_i)$ does not precede $\Delta(\chi_j,n_j)$ for any $i<j$ then
\[\rec_L(Q(\St_{n_1}(\chi_1)\times\cdots\times\St_{n_k}(\chi_k))) = \Sp_{n_1,\ldots,n_k}(\rec_L(\chi_1),\ldots,\rec_L(\chi_k)).\]
\end{prop}

We will say that $Q(\St_{n_1}(\chi_1)\times\St_{n_2}(\chi_2)\times\cdots\times\St_{n_k}(\chi_k))$ is \emph{generic} if two $\Delta(\chi_i,n_i)$'s are linked (in which case $Q(\St_{n_1}(\chi_1)\times\cdots\times\St_{n_k}(\chi_k)) = \St_{n_1}(\chi_1)\times\cdots\times\St_{n_k}(\chi_k)$), and note that Proposition \ref{prop:local langlands} implies that $\rec_L(Q(\St_{n_1}(\chi_1)\times\cdots\times\St_{n_k}(\chi_k)))$ is generic if and only if $Q(\St_{n_1}(\chi_1)\times\cdots\times\St_{n_k}(\chi_k))$ is.

To avoid giving a detailed discussion of types we will simply say that an irreducible admissible representation $\pi$ of $\GL_n(L)$ has type $\st_{n_1,\ldots,n_k}$ if $\rec_L(\pi)$ has inertial type $\tau_{n_1,\ldots,n_k}$.
\subsection{The local Jacquet--Langlands correspondence}\label{ssec:LJ}

Now fix an integer $d|n$ and let $D/L$ be a division algebra of dimension $d^2$ with ring of integers $\Oc_D$. Let $r=n/d$ and consider the algebraic group $\GL_{r}(D)$ over $L$, and note that this is an inner twist of $\GL_n(L)$.

Let $\Irr_L$ and $\Irr_D$ denote the set of isomorphism classes of irreducible unitary representations of $\GL_n(L)$ and $\GL_{r}(D)$, respectively.

Badulescu \cite{Badu1} constructs a map $\LJ:\Irr_L\to \Irr_D\cup\{0\}$ (which is in general neither invective nor surjective). We will not need to use the general construction of this map. Rather we will simply recall the following special case:

\begin{prop}\label{prop:LJ(st) = triv}
Take any $\Pi\in \Irr_L$ and let $\pi = \LJ(\Pi)$, and assume that $\pi\in \Irr_D$ (so that $\pi\ne 0$). Assume that $\Pi$ is generic and that $\pi^{\GL_r(\Oc_D)}\ne 0$. Then $\Pi$ has type $\st_{d,d,\ldots,d}$.
\end{prop}
\begin{proof}
As $\Pi$ is generic, by the Bernstein–Zelevinsky classification (cf. \cite{Rodier}) there exists an ordered partition $\nun = (n_1,\ldots,n_k)$ of $n$ corresponding to a standard parabolic $P_{n_1,\ldots,n_k}\subseteq \GL_n(L)$ with Levi $M_{n_1,\ldots,n_k} = \GL_{n_1}(L)\times\cdots\times \GL_{n_k}(L)$, and supercuspidal representations $\sigma_1,\ldots,\sigma_k$ of $\GL_{n_1}(L),\ldots,\GL_{n_k}(L)$ such that $\Pi = \nInd_{P_{n_1,\ldots,n_k}}^{GL_n(L)}(\sigma_1\otimes\cdots\otimes\sigma_k)$ (and in particular, this representation is irreducible).

Now by the definition of $\LJ$ (cf. \cite[Section 2.7]{Badu2}) we have $\LJ(\Pi) = 0$ unless $d|n_i$ for all $i$. If $d|n_i$ for all $i$, then $P_{n_1,\ldots,n_k}$ corresponds to a standard parabolic $P'_{n_1/d,\ldots,n_k/d}\subseteq \GL_r(D)$ (defined analogously to $P_{n_1,\ldots,n_k}$) with Levi $M'_{n_1/d,\ldots,n_k/d} = \GL_{n_1/d}(D)\times\cdots\times \GL_{n_k/d}(D)$ and one has
\[
\pi = \LJ(\Pi) = \nInd^{\GL_r(D)}_{P'_{n_1/d,\ldots,n_k/d}}(\sigma'_1\otimes\cdots\otimes \sigma'_k)
\]
where each $\sigma_i'$ is a square integrable representation of $\GL_{n_i/d}(D)$ for which $\JL(\sigma_i') = \sigma_i$, where $\JL$ is the usual Jacquet--Langlands map for square integrable representations.

But now the condition that $\pi^{\GL_r(\Oc_D)} \ne 0$ implies that $P'_{n_1/d,\ldots,n_k/d} = P'_{1,\ldots,1}$ is the group of upper triangular matrices in $\GL_r(D)$ (so $k=r$ and $n_i = d$ for all $i$) and that each for each $i$, $\sigma_i' = \chi_i\circ\nu:D^\times\to\C^\times$ for some unramified character $\chi_i:L^\times\to \C^\times$, where $\nu:D^\times\to L^\times$ is the reduced norm.

But now one has
\[
\sigma_i = \JL(\sigma_i') = \JL(\chi_i\circ\nu) = \St_{n_i}(\chi_i) = \St_d(\chi_i)
\]
for each $i$, and so (as $\Pi$ is generic)
\[
\Pi = \nInd_{P_{d,\ldots,d}}^{GL_n(L)}(\sigma_1\otimes\cdots\otimes\sigma_k) = \St_d(\chi_1)\times\cdots\times\St_d(\chi_r) = Q(\St_d(\chi_1)\times\cdots\times\St_d(\chi_r)),
\]
which implies that $\Pi$ indeed has type $\st_{d,\ldots,d}$ by Proposition \ref{prop:local langlands}.
\end{proof}

\section{Local Galois Deformation Rings}\label{sec:deformation}

In this section we will define the various local Galois deformation rings which we will consider in the rest of the paper, and review their relevant properties.

For the rest of the paper, we will fix a finite extension $E$ of $\Q_\ell$ with ring of integers $\Oc$, uniformizer $\varpi\in \Oc$ and residue field $\Oc/\varpi = \F$.

Let $\Cc_{\Oc}$ (resp. $\Cc_{\Oc}^\wedge$) be the category of Artinian (resp. complete Noetherian) local $\O$-algebras with residue field $\F$. Consider the (framed) deformation functor $\Dc^\square(\rbar):\Cc_\O\to \Set$ defined by
\begin{align*}
\Dc^\square(\rbar)(A) 
&= \{r:G_L\to \GL_n(A), \text{ continouos lift of }\rbar \}\\
&= \left\{\big(M,r,(e_1,\ldots,e_n)\big)\middle|\parbox{4.2in}{$M$ is a free rank $n$ $A$-module with a basis $(e_1,\ldots,e_n)$ and $r:G_L\to \Aut_A(M)$ such that the induced map $G\to \Aut_A(M)= \GL_n(A)\to \GL_n(\F)$ is $\rbar$}
\right\}_{/\sim}
\end{align*}
It is well-known that this functor is \emph{pro-representable} by some $R^\square(\rbar)\in \Cc^\wedge_\O$, in the sense that $\Dc^\square(\rbar) \equiv \Hom_\O(R^\square(\rbar),-)$. Furthermore, $\rbar$ admits a universal lift $r^\square:G_L\to \GL_n(R^\square(\rbar))$.

For any continuous homomorphism, $x:R^{\square}(\rbar)\to \Ebar$, we obtain a Galois representation $r_x:G_L\to \GL_n(\Ebar)$ lifting $\rbar$, from the composition $G_L\xrightarrow{r^\square} \GL_L(R^\square(\rbar))\xrightarrow{x} \GL_L(\Ebar)$.

We call $R^{\square}(\rbar)$ the \emph{framed deformation ring} of $\rbar$.

Most of our results will specifically be concerned with the local deformation rings occurring when $v\nmid\ell$ (the ``$\ell\ne p$'' case).

In the case when $v|\ell$, we will only need to consider local deformation rings in the \emph{Fontaine--Laffaille} range. Specifically, for each embedding $\tau:L\to \Qbar_\ell$, fix an $n$-tuple of integers $\lambda_\tau = (\lambda_{\tau,1},\ldots,\lambda_{\tau,n})$ for which $\lambda_{\tau,1}\ge \lambda_{\tau,2}\ge\cdots\ge\lambda_{\tau,n}$. Let $R^{\square,\crys,\{\lambda_\tau\}}(\rbar)$ be the Zariski closure of the set of points $x\in (\Spec R^\square(\rbar))(\Ebar)$ for which $r_x$ is crystalline with Hodge--Tate weights $\HT_\tau(r_x) = \{\lambda_{\tau,1},\lambda_{\tau,2},\ldots,\lambda_{\tau,n}\}$ for each $\tau$. We will refrain from giving a precise definition of the terms involved here, as it is not relevant to our discussion. We will refer the reader to \cite{CHT} or \cite{BLGGT14} for more details, and use only the following result from  \cite[Section 2.4]{CHT}:

\begin{prop}\label{prop:Fontaine--Laffaille}
	If $L/\Q_\ell$ is unramified and \[\lambda_{\tau,1}>\lambda_{\tau,2}>\cdots>\lambda_{\tau,n}\ge \lambda_{\tau,1}-(\ell-2)\]
	for all $\tau$, then $R^{\square,\crys,\{\lambda_\tau\}}(\rbar)\cong \O[[X_1,\ldots,X_{n^2+[F_v:\Q_\ell]n(n-1)/2}]]$.
\end{prop}
We will say that $\lambda_\tau = (\lambda_{\tau,1},\ldots,\lambda_{\tau,n})$ is \emph{in the Fontaine--Laffaille range} if it satisfies the inequality from this proposition.

\subsection{Local Deformation Rings for $v\nmid \ell$}\label{ssec:local def}

From now on assume that $v\nmid \ell$. By enlarging $\Oc$ if necessary, we will assume that $\F$ contains all $q^n-1$st roots of unity.

Let $I_L\to I_L/\Pt_L$ be the maximal pro-$\ell$ quotient of $I_L$. We will restrict to the case when $\rbar(\Pt_L) = 1$ (it is well known that the general case of computations of local deformation rings in the $\ell\ne p$ case may be reduced to this case, see for instance \cite[Section 2]{Shotton2}).

Now for any $A\in \Cc_\O^\wedge$, $\ker (\GL_n(A)\onto \GL_n(\F))$ is pro-$\ell$, and so any lift $r:G_L\to \GL_n(A)$ also satisfies $r(\Pt_L) = 1$. As $G_L/\Pt_L$ is topologically generated by $\varphi$ and $\sigma$, it follows that any lift $r:G_L\to \GL_n(A)$ is determined by $r(\varphi),r(\sigma)\in \GL_n(A)$.

Now as in \cite[Section 2]{Shotton2} for any ring $R$, let $\Mc(n,q)_R$ denote the moduli space of pairs $(\Phi,\Sigma) \in \GL_{n,R}\times_{\Spec R} \GL_{n,R}$ satisfying $\Phi\Sigma\Phi^{-1} = \Sigma^q$. Write $\Mc(n,q) = \Mc(n,q)_\Oc$

We will let $x_{\rbar} = (\rbar(\varphi),\rbar(\sigma))$ denote the $\F$-point of $\Mc(n,q)_\Oc$ corresponding to $\rbar$. Then $R^\square(\rbar)$ is canonically isomorphic to the complete local ring of $\Mc(n,q)_{\Oc}$ at the point $x_{\rbar}$ and the universal lift $r^\square:G_L\to \GL_n(R^\square(\rbar))$ is given by $r^\square(\varphi) = \Phi$ and $r^\square(\sigma) = \Sigma$.

Now \cite[Theorem 2.5]{Shotton2} gives:

\begin{prop}\label{prop: M lci}
$\Mc(n,q)_\Oc$ is a complete intersection, reduced and flat over $\Oc$ of relative dimension $n^2$, and hence $R^\square(\rbar)$ is as well.
\end{prop}

Moreover the proof of \cite[Theorem 2.5]{Shotton2} gives:

\begin{lemma}\label{lem:conn comps}
The characteristic polynomial $\chi_\Sigma(t)$ of $\Sigma$ is constant on each connected component of $\Mc(n,q)_E$.
\end{lemma}
\begin{proof}
As $\Mc(n,q)$, and hence $\Mc(N,q)_E$, is reduced it is enough to check this on $\Ebar$ points. As explained in \cite[Theorem 2.5]{Shotton2}, for any $x = (\Phi_x,\Sigma_x)\in \Mc(n,q)(\Ebar)$ the relation $\Phi_x\Sigma_x\Phi_x^{-1} = \Sigma_x^q$ implies that all of the eigenvalues of $\Sigma_x$ are $q^n-1$st roots of unity. Hence there are only finitely many possible polynomials $\chi_{\Sigma_x}(t)\in \Ebar[t]$ for $x\in \Mc(n,q)(\Ebar)$. If these polynomials are denoted $f_1,\ldots,f_N$, then $\{V(\chi_\Sigma(t) = f_i)\}_{i=1,\ldots,N}$ is a finite cover of $\Mc(n,q)(\Ebar)$, and hence of $\Mc(n,q)_E$, by \emph{disjoint} closed subsets, and so each $V(\chi_\Sigma(t) = f_i)\subseteq \Mc(n,q)_E$ must be a union of connected components of $\Mc(n,q)_E$.
\end{proof}

\subsection{Fixed type deformation rings}\label{ssec:fixed type}

In order to prove our main results, we will also need to consider various deformation rings with \emph{fixed type}. We will again restrict to the case when $v\nmid \ell$.

Let $\tau$ be any inertial type defined over $E$. For any $x\in(\Spec R^\square(\rbar))(\Ebar)$, we will say that $x$ has type $\tau$ if $r_x:G_L\to \GL_n(\Ebar)$ does. We will let $\Spec R^\square(\rbar,\tau)\subseteq \Spec R^\square(\rbar)$ be the Zariski closure of the set of $\Ebar$ points with type $\tau$. Explicitly this means that $R^\square(\rbar,\tau)$ is the minimal reduced $\ell$-torsion free quotient of $R^\square(\rbar)$ with the property that any $\Ebar$ point $x:R^\square(\rbar)\to \Ebar$ with type $\tau$ factors through $R^\square(\rbar,\tau)$.

For any unipotent inertial type $\tau_{n_1,\ldots,n_k}$, we will also denote $R^\square(\rbar,\st_{n_1,\ldots,n_k}) = R^\square(\rbar,\tau_{n_1,\ldots,n_k})$.

We say that an $\Ebar$ point $x\in (\Spec R^\square(\rbar))(\Ebar)$ is \emph{non-degenerate} if $\WD(r_x)$ is generic. Now \cite[Proposition 4.6]{Shotton2} gives

\begin{prop}\label{prop:R(tau)}
For any any inertial type $\tau$ defined over $E$:
\begin{enumerate}
	\item $\Spec R^\square(\rbar,\tau)$ is a union of irreducible components of $\Spec R^\square(\rbar)$, and hence is flat and equidimensional of dimension $n^2$ over $\Oc$;
	\item If $x$ is a non-degenerate $\Ebar$ point of $\Spec R^\square(\rbar)$, then $x$ lies on a unique irreducible component of $\Spec R^\square(\rbar)$ and $\Spec R^\square(\rbar)[1/\ell]$ is formally smooth at $x$;
	\item The non degenerate $\Ebar$ points are Zariski dense in $\Spec R^\square(\rbar)$;
	\item If $x$ is a non-degenerate $\Ebar$ point of $\Spec R^\square(\rbar,\tau)$ then $x$ has type $\tau$.
\end{enumerate}
\end{prop}

\subsection{Unipotent deformation rings}

The exponent of $q$ in the relation $\Phi\Sigma\Phi^{-1} = \Sigma^q$ can make explicit computations somewhat difficult, since it involves arbitrary powers of an $n\times n$ matrix, so we will need to find a simpler description of $\Mc^{\unip}(n,q)$ in order to carry out the computations in this paper.

Given any $\Oc$-algebra $A$ and any $n\times n$ matrix $X\in M_n(A)$, we will let $\chi_X(t) = \det(tI-X)$ denote the characteristic polynomial of $X$. We say that a matrix $X\in \GL_n(A)$ is \emph{strongly unipotent} if the characteristic polynomial of $X$ is $(t-1)^n$, and we say that $X$ is \emph{unipotent} if $(X-I)^n=0$. Similarly we will say that a matrix $N\in M_n(A)$ is \emph{strongly nilpotent} if its characteristic polynomial is $t^n$, and \emph{nilpotent} if $N^n = 0$. Note that strongly unipotent matrices are unipotent and strongly nilpotent matrices are nilpotent, and in both cases the definitions are equivalent if $A$ is a field.

Now for any unipotent $X\in \GL_n(A)$ and any nilpotent $N\in M_n(A)$ we will define
\begin{align*}
	\log X &= \sum_{k=1}^n(-1)^{k-1}\frac{(X-I)^k}{k} = (X-I)-\frac{(X-I)^2}{2}+\cdots\pm \frac{(X-I)^n}{n} \in M_n(A)\\
	\exp N &= \sum_{k=0}^n\frac{N^k}{k!} = I+N+\frac{N^2}{2}+\cdots+ \frac{N^n}{n!} \in \GL_n(A)
\end{align*}
and note that these are well defined, as we have assumed that $\ell>n$, and so $k$ and $k!$ are units in any $\Oc$-algebra $A$ for all $k=1,\ldots,n$. One can easily check that $\log X$ is strongly nilpotent, that $\exp N$ is unipotent, and that $\log$ and $\exp$ are inverses, and moreover that Moreover, that $\log X^m = m\log X$ for all $m\ge 0$.

We say that a Galois representation $r:G_L\to \GL_n(A)$ is \emph{unipotent} if $r(g)$ is unipotent for all $g\in I_L$. If $r(\Pt_L) = 1$, this is equivalent to simply requiring that $r(\sigma)$ is unipotent. From now on assume that $\rbar$ is unipotent.

Let $R^{\unip}(\rbar)$ be the minimal, reduced $\ell$-torsion free quotient of $R^\square(\rbar)$ with the property that for any $x\in (\Spec R^\square(\rbar))(\Ebar)$, if $r_x$ is unipotent, then $x$ factors through $R^{\unip}(\rbar)$. Notice that for any $x$, that $r_x$ is unipotent if and only if $\WD(r_x)$ has a unipotent inertial type. It follows from Proposition \ref{prop:R(tau)} that $\Spec R^{\unip}(\rbar)$ is the scheme theoretic union of all of the $\Spec R^\square(\rbar,\tau)$, as $\tau$ ranges over all unipotent inertial types. In particular, it is flat and equidimensional of relative dimension $n^2$ over $\Oc$.

Let $\Mc^{\unip}(n,q) = \Mc^{\unip}(n,q)_\Oc\subseteq \Mc(n,q)_\Oc$ be the Zariski closure of the set of points $(\Phi,\Sigma)\in \Mc(n,q)(\Ebar)$ for which $\Sigma$ is strongly unipotent, and let $\Mc^{\unip}(n,q)_A = \Mc^{\unip}(n,q)\times_\Oc\Spec A$ for any $A$. Then Lemma \ref{lem:conn comps} implies that $\Mc^{\unip}(n,q)_E$ is a union of connected components of $\Mc(n,q)_E$, and so (as $\Mc^{\unip}(n,q)$ is the closure of $\Mc^{\unip}(n,q)_E$ in $\Mc(n,q)$) $\Mc^{\unip}(n,q)$ is a union of irreducible components of $\Mc(n,q)$. In particular, $\Mc^{\unip}(n,q)$ is flat and equidimensional over $\Oc$ of relative dimension $n^2$. 

From our definitions, it follows that $R^{\unip}(\rbar)$ is the complete local ring of $\Mc^{\unip}(n,q)$ at the $\F$-point $(\rbar(\varphi),\rbar(\sigma))$.

Our goal is to find a simpler description of the space $\Mc^{\unip}(n,q)$. Note that if $\Sigma$ is unipotent, then the relation $\Phi\Sigma\Phi^{-1} = \Sigma^q$ is equivalent to $\Phi(\log \Sigma)\Phi^{-1} = q(\log \Sigma)$. So now for any ring $R$, define $\Nct(n,q)_R$ to be the moduli space of all pairs of matrices $(\Phi,N)\in \GL_{n,R}\times_{\Spec R} M_{n,R}$ satisfying $\Phi N \Phi^{-1} = qN$, and write $\Nct(n,q) = \Nct(n,q)_\Oc$.

Notice that in general, $\Nct(n,q)$ is not flat over $\Oc$. Indeed consider the case where $q\equiv 1\pmod{\ell}$ (but $q>1$). Then $\Nct(n,q)_E$ has dimension $n^2$ by the work below, but $\Nct(n,q)_\F$ is the variety parameterizing pairs $(\Phi,N)\in \GL_{n,\F}\times_\F M_{n,\F}$ satisfying $\Phi N = N\Phi$, which is known to have dimension $n^2+n$.

For any $\Oc$-algebra $A$, define closed subschemes $\Mc^{\unip}(n,q)_A^\circ = V(\chi_\Sigma(t) = (t-1)^n)\subseteq \Mc(n,q)_A$ and $\Nc(n,q)_A^\circ = V(\chi_N(t) = t^n) \subseteq \Nct(n,q)_A$. That is, $\Mc^{\unip}(n,q)_A^\circ$ (resp. $\Nc(n,q)_A^\circ$) is the maximal subscheme on which $\Sigma$ is strongly unipotent (resp. $N$ is strongly nilpotent). Then we have an isomorphism $\Lc_A:\Mc^{\unip}(n,q)_A^\circ\isomto \Nc(n,q)_A^\circ$ defined by $\Lc_A(\Phi,\Sigma) = (\Phi,\log \Sigma)$, with inverse given by $(\Phi,N)\mapsto (\Phi,\exp N)$.

Define $\Nc(n,q)=\Nc(n,q)_\Oc\subseteq \Nct(n,q)_\Oc$ to be the maximal reduced $\Oc$-flat subscheme and let $\Nc(n,q)_A = \Nc(n,q)_\Oc\times_\Oc \Spec A$ for any $\Oc$-algebra $A$. Note that $\Nc(n,q)$ is the Zariski closure of the set of $\Ebar$-points of $\Nct(n,q)_\Oc$, and hence $\Nc(n,q)_E$ is the underlying reduced subscheme of $\Nct(n,q)_E$. 

We claim the following:

\begin{prop}\label{prop:N(n,q) local model}
There is an isomorphism $\Lc:\Mc^{\unip}(n,q)\isomto \Nc(n,q)$ given by $(\Phi,\Sigma)\mapsto (\Phi,\log \Sigma)$. In particular, $R^{\unip}(\rbar)$ is isomorphic to the complete local ring of $\Nc(n,q)$ at the $\F$-point $(\rbar(\varphi),\log\rbar(\sigma))$, and the universal lift $r^{\unip}:G_L\to \GL_n(R^{\unip}(\rbar))$ is given by $r^{\unip}(\varphi) = \Phi$ and $r^{\unip}(\sigma) = \exp N$.
\end{prop}

By definition, $\Mc^{\unip}(n,q)$ is the Zariski closure of the set of $\Ebar$ points of $\Mc^{\unip}(n,q)_\Oc^\circ$, and so is the maximal reduced $\Oc$-flat subscheme of $\Mc^{\unip}(n,q)_\Oc^\circ$. The proof of Proposition \ref{prop:N(n,q) local model} will thus hinge on the observation that $\Nc(n,q)^\circ_E=\Nct(n,q)_E$.

In fact, we shall prove something slightly stronger than this, which will be useful for our discussion of the banal case below.

Let $\Bbbk$ be a field. We will say that $q$ is \emph{banal} in $\Bbbk$ if $q\ne 0$ and $q^i\ne 1$ in $\Bbbk$ for all $i=1,\ldots,n$. Since $q\in\Z$, this depends only on $\Char\Bbbk$, and as $q>1$, $q$ is automatically banal in $\Bbbk$ if $\Char\Bbbk = 0$ (so in particular, $q$ is banal in $E$).

\begin{lemma}\label{lem:N str nilp}
If $\Bbbk$ is a field which is an $\Oc$-algebra, and $q$ is banal in $\Bbbk$, then $\Nc(n,q)^\circ_\Bbbk=\Nct(n,q)_\Bbbk$.
\end{lemma}
\begin{proof}
By definition, we must prove that $\chi_N(t) = t^n$, where we view $N$ as an $n\times n$ matrix in $M_n(\Gamma(\Nct(n,q)_\Bbbk))$. Let $\chi_N(t) = t^n+\gamma_{n-1}t^{n-1}+\cdots +\gamma_1t +\gamma_0$ for $\gamma_0,\gamma_1,\ldots,\gamma_{n-1}\in \Gamma(\Mc(n,q)_\Bbbk)$. Since $\Phi N \Phi^{-1} = qN$ on $\Mc(n,q)_E$ we get
\begin{align*}
\chi_N(t) &= \det(tI-N) = \det(tI-\Phi N\Phi^{-1}) = \det(tI-qN)\\
&= t^n+q\gamma_{n-1}t^{n-1}+\cdots+ q^{n-1}\gamma_1+q^n\gamma_0
\end{align*}
and so $\gamma_k = q^{n-k}\gamma_k$ for all $k=0,\ldots,n-1$. Now as $\Gamma(\Mc(n,q)_\Bbbk)$ is an $\Bbbk$-algebra, and $q$ is banal in $\Bbbk$, it follows $\gamma_k=0$ for all $k=0,\ldots,n-1$, and so $\chi_N(t) = t^n$.
\end{proof}

\begin{proof}[Proof of Proposition \ref{prop:N(n,q) local model}]
We have $\Nc(n,q)^\circ_E=\Nct(n,q)_E$. Now as $\Nc(n,q)^\circ_E = \Nc(n,q)^\circ_\Oc\times_\Oc\Spec E$, it follows that $\Nc(n,q)^\circ_\Oc(\Ebar) = \Nct(n,q)_\Oc(\Ebar)$, and so the Zariski closure of $\Nc(n,q)^\circ_\Oc(\Ebar)$ in $\Nc(n,q)^\circ_\Oc$ is equal to the Zariski closure of $\Nct(n,q)_\Oc(\Ebar)$ in $\Nct(n,q)_\Oc$, which is $\Nc(n,q)$ by definition. Hence $\Nc(n,q)$ is the maximal reduced $\Oc$-flat subscheme of $\Nc(n,q)^\circ_\Oc$. The isomorphism $\Lc_\Oc:\Mc^{\unip}(n,q)_\Oc^\circ\isomto \Nc(n,q)_\Oc^\circ$ thus restricts to the desired isomorphism $\Mc^{\unip}(n,q) \isomto \Nc(n,q)$. The rest of Proposition \ref{prop:N(n,q) local model} follows automatically from this.
\end{proof}

We now claim the following:

\begin{prop}\label{prop:Nc(n,q)_E is CI}
If $\Bbbk$ is a field which is an $\Oc$-algebra, and $q$ is banal in $\Bbbk$, then $\Nct(n,q)_\Bbbk$ is a reduced complete intersection of dimension $n^2$ over $\Bbbk$.

In particular, $\Nct(n,q)_E = \Nc(n,q)_E$.
\end{prop}
\begin{proof}
Note that as $\Nc(n,q)_E$ is the underlying reduced subscheme of $\Nct(n,q)_E$, the last claim follows from the fact that $\Nct(n,q)_E$ is reduced (as $q$ is banal in $E$).

Since $\Nct(n,q)_\Bbbk$ is the open subscheme $\GL_{n,\Bbbk}\times_\Bbbk M_{n,\Bbbk}\subseteq \A^{2n^2}_\Bbbk$ by $n^2$ relations (given by the matrix equation $\Phi N = q N\Phi$), to prove that $\Nct(n,q)_\Bbbk$ is a complete intersection of dimension $n^2$, it is enough to prove that $\dim \Nct(n,q)_{\Bbbk}\le n^2$.

Now define a map $\pi:\Nct(n,q)_\Bbbk\to M_{n,\Bbbk}$ by $(\Phi,N)\mapsto N$. By Lemma \ref{lem:N str nilp}, for any $(\Phi,N)\in \Nct(n,q)_\Bbbk(\Bbbkbar)$, $N\in M_{n}(\Bbbkbar)$ is nilpotent. Now note that $\pi$ is clearly equivariant for the natural conjugation action of $\GL_{n,\Bbbk}$ on $\Nct(n,q)_\Bbbk$ and $M_{n,\Bbbk}$.

Now the set of conjugacy classes of nilpotent matrices in $M_{n,\Bbbk}$ is in bijection is the set of unordered partitions of $n$ (by sending a nilpotent matrix to the collection of sizes of its Jordan blocks). For any unordered partition $\nun = (n_1,\ldots,n_k)$, let $C_{\nun}\subseteq M_n(\Bbbkbar)$ be the conjugacy class of nilpotent matrices corresponding to $\nun$, and let $\Uc_{\nun} = \pi^{-1}(C_{\nun})\subseteq \Nct(n,q)_{\Bbbk}(\Bbbkbar)$. Note that $\Nct(n,q)_{\Bbbk}(\Bbbkbar) = \bigsqcup_{\nun}\Uc_{\nun}$.

Now for any integer $d\ge 1$, consider the matrices 
\begin{align*}
	N_d &=
	\begin{pmatrix}
		0&1&&&\\
		&0&1&&\\
		&&\ddots&\ddots&\\
		&&&0&1\\
		&&&&0
	\end{pmatrix}&
	&\text{and}&
	\Phi_d&
	=
	\begin{pmatrix}
		q^{d-1}&&&&\\
		&q^{d-2}&&&\\
		&&\ddots&&\\
		&&&q&\\
		&&&&1
	\end{pmatrix}
\end{align*}
in $M_d(\Bbbk)$, and note that $N_d$ is nilpotent and $\Phi_dN_d = qN_d\Phi_d$. For any partition $\nun=(n_1,\ldots,n_k)$ of $n$, define the matrices
\begin{align*}
	N_{\nun} &=\begin{pmatrix}N_{n_1}&&\\&\ddots&\\&&N_{n_k}\end{pmatrix}&
	&\text{and}&
	\Phi_{\nun}&=\begin{pmatrix}\Phi_{n_1}&&\\&\ddots&\\&&\Phi_{n_k}\end{pmatrix}
\end{align*}
in $M_n(\Bbbk)$. Then we again have $\Phi_{\nun}N_{\nun} = qN_{\nun}\Phi_{\nun}$ so $(\Phi_{\nun},N_{\nun})\in \Nct(n,q)_{\Bbbk}(\Bbbk)$. Moreover by definition we have $N_{\nun}\in C_{\nun}$ and $(\Phi_{\nun},N_{\nun})\in \Uc_{\nun}$ for all $\nun$. In particular, each $\Uc_{\nun}$ is nonempty.

In fact, we have the following:

\begin{lemma}\label{lem:U_n irreducible}
For each partition $\nun$ of $n$, $\Uc_{\nun}\subseteq \Nct(n,q)_{\Bbbk}(\Bbbkbar)$ is irreducible.
\end{lemma}
\begin{proof}
Consider the fiber $\pi^{-1}(N_{\nun})\subseteq \Uc_{\nun}$. By the above, $(\Phi_{\nun},N_{\nun})\in \pi^{-1}(N_{\nun})$, so $\pi^{-1}(N_{\nun})$ is nonempty. Now by the definition of $\Nct(n,q)_{\Bbbk}$, $\pi^{-1}(N_{\nun})$ is the variety of $\GL_n(\Bbbkbar)$ cut out by the equation $\Phi N_{\nun} = qN_{\nun} \Phi$. As this equation is linear in the entries of $\Phi$, and $\GL_n(\Bbbkbar)$ is a Zariski open subset of $\A^{n^2}_{\Bbbk}(\Bbbkbar)$, it follows that $\pi^{-1}(N_{\nun})$ is isomorphic to a (nonempty) Zariski open subset of $\A^{d_{\nun}}_{\Bbbk}(\Bbbkbar)$ for some integer $d_{\nun}\le n^2$. In particular, $\pi^{-1}(N_{\nun})$ is irreducible.

Now under the conjugation actions of $\GL_n(\Bbbkbar)$ on $\Nct(n,q)_{\Bbbk}(\Bbbkbar)$ and $M_n(\Bbbkbar)$, $C_{\nun}$ is generated by $N_{\nun}$, and hence $\Uc_{\nun} = \pi^{-1}(C_n)$ is generated by $\pi^{-1}(N_{\nun})$, so there is a Zariski continuous surjection $\GL_n(\Bbbkbar)\times \pi^{-1}(N_{\nun})\onto \Uc_{\nun}$. Since $\GL_n(\Bbbkbar)$ and $\pi^{-1}(N_{\nun})$ are irreducible $\Bbbkbar$-varieties, so is $\GL_n(\Bbbkbar)\times \pi^{-1}(N_{\nun})$, and hence $\Uc_{\nun}$ is indeed irreducible.
\end{proof}

It follows that $\Nct(n,q)_{\Bbbk}(\Bbbkbar) = \bigcup_{\nun}\overline{\Uc}_{\nun}$, and each $\overline{\Uc}_{\nun}$ is an irreducible closed subvariety of $\Nct(n,q)_{\Bbbk}(\Bbbkbar)$. Hence for any irreducible component $\Zc\subseteq \Nct(n,q)_{\Bbbk}$ we have $\Zc(\Bbbkbar) = \overline{\Uc}_{\nun}$ for some $\nun$. In particular, each irreducible component of $\Nct(n,q)_{\Bbbk}$ contains the $\Bbbkbar$-point $(\Phi_{\nun},N_{\nun})$ for some $\nun$.

Let $\pf_{\nun}\subseteq \Gamma(\Nct(n,q)_\Bbbk)$ be the maximal ideal of $\Gamma(\Nct(n,q)_\Bbbk)$ corresponding to the point $\Bbbk$-point $(\Phi_{\nun},N_{\nun})$. We claim the following:

\begin{lemma}
For each $\nun$, $\dim_{\Bbbk} \pf_{\nun}/\pf_{\nun}^2 = n^2$.
\end{lemma}

\begin{proof}
Write $\nun = (n_1,\ldots,n_k)$. Let $X = (x_{ij}) := \Phi-\Phi_{\nun}$ and $A = (a_{ij}) := N-N_{\nun}$. Then by definition the local ring $\Gamma(\Nct(n,q)_\Bbbk)_{\pf_{\nun}}$ is isomorphic to the polynomial ring
\begin{align*}
\frac{\Bbbk[\{x_{ij},a_{ij}\}]_{(x_{ij},a_{ij})}}{(\Phi N-qN\Phi)}
&=
\frac{\Bbbk[\{x_{ij},a_{ij}\}]_{(x_{ij},a_{ij})}}
{\left((\Phi_{\nun}+X)(N_{\nun}+A)-q(N_{\nun}+A)(\Phi_{\nun}+X)\right)}
\\
&=
\frac{\Bbbk[\{x_{ij},a_{ij}\}]_{(x_{ij},a_{ij})}}
{\left( (\Phi_{\nun}A-qA\Phi_{\nun})+(XN_{\nun}-qN_{\nun}X) + (XA-qAX)\right)}.
\end{align*}
As $XA-qAX\equiv 0 \pmod{\pf_{\nun}^2}$ it follows that 
\begin{align*}
\pf_{\nun}/\pf^2_{\nun} &= 
\frac{(x_{ij},a_{ij})_{ij}/(x_{ij},a_{ij})_{ij}^2}{\left( (\Phi_{\nun}A-qA\Phi_{\nun})+(XN_{\nun}-qN_{\nun}X)\right)}
=
\frac{\bigoplus_{1\le i,j\le n}\Bbbk x_{ij}\oplus \bigoplus_{1\le i,j\le n}\Bbbk a_{ij}}{\left( (\Phi_{\nun}A-qA\Phi_{\nun})+(XN_{\nun}-qN_{\nun}X)\right)}.
\end{align*}

Now note the following:

\begin{lemma}
Let $\Bbbk$ be any field and assume that $q$ is banal in $\Bbbk$. Take integers $a,b\in \{1,\ldots,n\}$ and consider a $\Bbbk$-vector space $W$ of dimension $2ab$ with basis denoted by $\{y_{ij},z_{ij}\}_{1\le i \le a, 1\le j\le b}$. Consider the two $a\times b$ matrices $Y = (y_{ij})$ and $Z = (z_{ij})$ with entries in $W$, and define
\[
C = (c_{ij}) := (\Phi_aY-qY\Phi_b)+(ZN_b-qN_aZ),
\]
interpreted as an $a\times b$ matrix with entries in $W$. Define $\Span(C) := \Span\{c_{ij}\}_{ij}\subseteq W$ to be the $\Bbbk$-span of the entries of $C$. Let $\Wc_{a,b} := W/\Span(C)$. 

Then $\dim_\Bbbk \Wc_{a,b} = ab$.
\end{lemma}
\begin{proof}
Since $C$ has $ab$ entries, we clearly have $\dim_\Bbbk\Wc_{a,b} \ge \dim_\Bbbk W - ab = ab$. It will thus suffice to show that $\Wc_{a,b}$ is spanned by a set of $ab$ vectors. If $a\le b$, define
\[
\Bc = \{y_{i,b-a+1+i}|i=1,\ldots,a-1\}\cup \{z_{ab}\} \cup \{z_{ij}|1\le i \le a, 1\le j\le b, i-j\ne a-b\}\subseteq \Wc_{a,b}
\]
and if $a>b$ define
\[
\Bc =  \{y_{j+a-b-1,j}|j=1,\ldots,b\}\cup \{z_{ij}|1\le i \le a, 1\le j\le b, i-j\ne a-b\}\subseteq \Wc_{a,b}
\]
and note that in either case $|\Bc| = ab$. We claim that $\Span(\Bc) = \Wc_{a,b}$.

For ease of notation, set $y_{ij} = z_{ij} = 0$ whenever $(i,j)\not\in \{1,\ldots,a\}\times \{1,\ldots,b\}$. Now,
\begin{align*}
\Phi_aY-qY\Phi_b 
&= \Phi_a(y_{ij})_{ij} - q \Phi_b(y_{ij})_{ij}
 = (q^{a+1-i}y_{ij})_{ij} - q (q^{b+1-j}y_{ij})_{ij}
 = ((q^{a+1-i}-q^{b+2-j})y_{ij})_{ij}
\end{align*}
and
\begin{align*}
ZN_b-qN_aZ
&= (z_{ij})_{ij}N_b - qN_a(z_{ij})_{ij}
 = (z_{i,j-1})_{ij} - qN_b(z_{i+1,j})_{ij}
 = (z_{i,j-1}-qz_{i+1,j})_{ij}
\end{align*}
and so
\[
c_{ij} = (q^{a+1-i}-q^{b+2-j})y_{ij} + (z_{i,j-1}-qz_{i+1,j})
\]
for all $(i,j)\in \{1,\ldots,a\}\times \{1,\ldots,b\}$. In particular, if $(i,j)\in \{1,\ldots,a\}\times \{1,\ldots,b\}$ and $i-j = a-b-1$ then $q^{a+1-i} = q^{b+2-j}$ and so $c_{ij} = z_{i,j-1}-qz_{i+1,j}$ giving $z_{i,j-1} = qz_{i+1,j}$ in $\Wc_{a,b}$.

If $a\le b$ the relations $c_{i,b-a+1+i}$ for $i=1,\ldots,a-1$ give $z_{i,b-a+i} = qz_{i+1,b-a+i+1}$ for $i=1,\ldots,a-1$, and so
\[
z_{i,b-a+i} = q^{b-i}z_{ab}\in \Span(\Bc)
\]
for $i = 1,\ldots,a$ by induction.

If $a>b$ the relations $c_{j+a-b-1,j}$ for $j=1,\ldots,b$ give $z_{j+a-b-1,j-1} = qz_{j+a-b,j}$ for  $j=1,\ldots,b$ and so
\[
z_{j+a-b,j} = q^{-j} z_{a-b,0} = 0 \in \Span(\Bc)
\]
for all $j=1,\ldots,b$.

So in either case, $z_{ij}\in \Span(\Bc)$ whenever $i-j = a-b$, and so $z_{ij}\in \Span(\Bc)$ for all $i,j\in\Z$.

So now take any $(i,j)\in  \{1,\ldots,a\}\times \{1,\ldots,b\}$. We claim that $y_{ij}\in \Span(\Bc)$. If $i-j = a-b-1$, this is by definition. Otherwise, as $a,b\le n$, $a+1-i,b+1-j\in \{1,\ldots,n\}$ and so $|(a+1-i)-(b+2-j)| \le n$, giving that $q^{a+1-i} \ne q^{b+2-j}$ (as $q$ is banal in $\Bbbk$) and so the relation $c_{ij}$ gives
\[
(q^{a+1-i}-q^{b+2-j})y_{ij} + (z_{i,j-1}-qz_{i+1,j}) = 0
\]
and so
\[y_{ij} = -\frac{z_{i,j-1}-qz_{i+1,j}}{q^{a+1-i}-q^{b+2-j}} \in \Span(\{z_{ij}\}_{ij}) \subseteq \Span(\Bc).
\]
So indeed, $y_{ij},z_{ij}\in \Span(\Bc)$ for all $i$ and $j$, giving $\Span(\Bc) = \Wc_{a,b}$ and so $\dim_E \Wc_{a.b} = ab$.
\end{proof}

So now in the vector space $\pf_{\nun}/\pf_{\nun}^2$, write $A$ and $X$ as block matrices
\begin{align*}
A&=\begin{pmatrix}A_{11}&\cdots&A_{1k}\\
	\vdots&&\vdots\\
	A_{k1}&\cdots&A_{kk}
\end{pmatrix}&
&\text{and}&
X&=\begin{pmatrix}X_{11}&\cdots&X_{1k}\\
	\vdots&&\vdots\\
	X_{k1}&\cdots&X_{kk}
\end{pmatrix}
\end{align*}
where for each $i$ and $j$ the $A_{ij}$ and $N_{ij}$ are $n_i\times n_j$ matrices. Now in block matrix form we have
\begin{align*}
\Phi_{\nun}A-qA\Phi_{\nun}&= (\delta_{ij}\Phi_{n_i})(A_{ij}) - q(A_{ij})(\delta_{ij}\Phi_{n_j}) = (\Phi_{n_i}A_{ij}-qA_{ij}\Phi_{n_j})\\
XN_{\nun}-qN_{\nun}X &= (X_{ij})(\delta_{ij}N_{n_j})-q(\delta_{ij}N_{n_i})(X_{ij}) = (X_{ij}N_{n_j}-qN_{n_i}X_{ij})
\end{align*}
and so 
\[
(\Phi_{\nun}A-qA\Phi_{\nun})+(XN_{\nun}-qN_{\nun}X) = \left((\Phi_{n_i}A_{ij}-qA_{ij}\Phi_{n_j})+(X_{ij}N_{n_j}-qN_{n_i}X_{ij})\right).
\]
From this it follows that
\[
\pf_{\nun}/\pf_{\nun}^2 \cong \bigoplus_{1\le i,j\le k}\Wc_{n_i,n_j}
\]
and so
\[
\dim_\Bbbk \pf_{\nun}/\pf_{\nun}^2 = \sum_{1\le i,j\le k}\dim_E\Wc_{n_i,n_j} = \sum_{1\le i,j\le k}n_in_j = \left(\sum_{i=1}^kn_i \right)^2 = n^2.
\]
\end{proof}

Now take any irreducible component $\Zc$ of $\Nct(n,q)_{\Bbbk}$, and take a partition $\nun$ of $n$ with $(\Phi_{\nun},N_{\nun})\in\Zc(\Bbbkbar)$. As $\Zc$ is irreducible, it follows that 
\[
\dim \Zc = \dim \Gamma(\Zc)_{\pf_{\nun}} \le \dim \Gamma(\Nct(n,q))_{\pf_{\nun}} \le \dim_{\Bbbk}\pf_{\nun}/\pf_{\nun}^2 = n^2.
\]
Thus each irreducible component of $\Nct(n,q)_{\Bbbk}$ has dimension at most $n^2$, so $\dim \Nct(n,q)_{\Bbbk}$. As observed above this implies that $\Nct(n,q)_{\Bbbk}$ is a complete intersection of dimension $n^2$. In particular, $\Nct(n,q)_{\Bbbk}$ is equidimensional of dimension $n^2$. Thus as $\dim_{\Bbbk} \pf_{\nun}/\pf_{\nun}^2 = n^2 = \dim \dim \Gamma(\Nct(n,q))_{\pf_{\nun}}$ for all $\nun$, $\Nct(n,q)_{\Bbbk}$ is regular at each of the points $(\Phi_{\nun},N_{\nun})$.

Hence each irreducible component $\Zc$ of $\Nct(n,q)_{\Bbbk}$ is regular at at least one point, and so $\Nct(n,q)_{\Bbbk}$ is generically reduced. As $\Nct(n,q)_{\Bbbk}$ is a complete intersection, and thus Cohen--Macaulay, it is in fact reduced.
\end{proof}

We shall also isolate the following simple consequence of this proof:

\begin{prop}\label{prop:N(n,q) components}
If $\Bbbk$ is a field which is an $\Oc$-algebra, and $q$ is banal in $\Bbbk$, then the irreducible components of $\Nct(n,q)_{\Bbbk}$ are in bijection with the set of unordered partitions of $n$.

For any unordered partition $\nun = (n_1,\ldots,n_k)$ of $n$, the irreducible component of $\Nct(n,q)_{\Bbbk}$ corresponding to $\nun$ is the Zariski closure of the set of points $(\Phi,N)\in \Nct(n,q)_{\Bbbk}(\Bbbkbar)$ such that $N\in M_n(\Bbbkbar)$ is nilpotent has Jordan blocks of sizes $n_1,\ldots,n_k$.
\end{prop}
\begin{proof}
In the notation of the proof of Proposition \ref{prop:Nc(n,q)_E is CI}, it suffices to show that $\dim \Uc_{\nun} = n^2$ for each $n$. Indeed, as is a finite type reduced $\Bbbk$-scheme, each irreducible component of $\Nct(n,q)_{\Bbbk}$ is the Zariski closure of its $\Bbbkbar$-points, and so is in the form $\overline{\Uc}_{\nun}$ for some $\nun$. As each $\overline{\Uc}_{\nun}$ is irreducible and $\Nct(n,q)_{\Bbbk} = \bigcup_{\nun}\overline{\Uc}_{\nun}$, if we have $\dim \Uc_{\nun} = n^2 = \dim \Nct(n,q)_{\Bbbk}$ for all $\nun$, then it will follow that the $\overline{\Uc}_{\nun}$'s are exactly the irreducible components of $\Nct(n,q)_{\Bbbk}$.

Now by the proof of Lemma \ref{lem:U_n irreducible} and the orbit-stabilizer theorem:
\[
\dim \Uc_{\nun} = \dim \pi^{-1}(N_{\nun}) + \dim C_{\nun} = n^2 + \dim_{\Bbbkbar} \pi^{-1}(N_{\nun}) - \dim_{\Bbbkbar}\stab_{\GL_n(\Bbbkbar)}(N_{\nun}).
\]
Now by definition, we have
\begin{align*}
\stab_{\GL_n(\Bbbkbar)}(N_{\nun}) &= \{X\in \GL_n(\Bbbkbar)|XN_{\nun} = N_{\nun}X\}\\
\pi^{-1}(N_{\nun}) &= \{Y\in \GL_n(\Bbbkbar)|YN_{\nun} = qN_{\nun}Y\}.
\end{align*}
But now as $\Phi_{\nun}\in \GL_n(\Bbbkbar)$ satisfies $\Phi_{\nun}N_{\nun} = qN_{\nun}\Phi_{\nun}$ (and hence also $\Phi_{\nun}^{-1}N_{\nun} = q^{-1}N_{\nun}\Phi_{\nun}^{-1}$) we get that for all $X\in \stab_{\GL_n(\Bbbkbar)}(N_{\nun})$ and $Y\in \pi^{-1}(N_{\nun})$ that
\begin{align*}
(X\Phi_{\nun})N &= X\Phi_{\nun}N = qXN\Phi_{\nun}=qNX\Phi_{\nun} = qN(X\Phi_{\nun})\\
(Y\Phi_{\nun}^{-1})N &= Y\Phi_{\nun}^{-1}N = q^{-1}YN\Phi_{\nun}^{-1}=q^{-1}qNY\Phi_{\nun}^{-1} = N(Y\Phi_{\nun}^{-1}),
\end{align*}
and so the map $X\mapsto X\Phi_{\nun}$ defines an isomorphism $\stab_{\GL_n(\Bbbkbar)}(N_{\nun})\isomto \pi^{-1}(N_{\nun})$ of $\Bbbkbar$-varieties, giving $\dim_{\Bbbkbar} \pi^{-1}(N_{\nun}) = \dim_{\Bbbkbar}\stab_{\GL_n(\Bbbkbar)}(N_{\nun})$, and so indeed $\dim \Uc_{\nun} = n^2$.
\end{proof}

We shall also record the following simple fact:

\begin{prop}\label{prop:N(n,q) connected}
$\Nc(n,q)_E$, $\Nc(n,q)$, $\Mc^{\unip}(n,q)_E$ and $\Mc^{\unip}(n,q)$ are all connected schemes.
\end{prop}
\begin{proof}
Since $\Nc(n,q)$ is $\Oc$ flat, if $\Nc(n,q)_E$ is connected then $\Nc(n,q)$ will be as well. Also Proposition \ref{prop:N(n,q) local model} implies that $\Nc(n,q)\cong \Mc^{\unip}(n,q)$ and $\Nc(n,q)_E\cong \Mc^{\unip}(n,q)_E$. Thus it will suffice to prove that $\Nc(n,q)_E$ is connected.

For any point $(\Sigma,N)\in \Nc(n,q)(\Ebar)$, note that $t\mapsto (\Sigma,Nt)$ defines a map of schemes $\A^1_{\Ebar}\to \Nc(n,q)_{\Ebar}$ with connected image. Thus $(\Sigma,N)$ and $(\Sigma,0)$ are in the same connected component of $\Nc(n,q)_E$. But now the subscheme $V(N = 0)\subseteq \Nc(n,q)_E$ is isomorphic to $\GL_{n,E}$, and hence is connected. Thus all points of $\Nc(n,q)(\Ebar)$ in the form $(\Sigma,0)$ lie in the same connected component of $\Nc(n,q)_E$. Thus all $\Ebar$-points of $\Nc(n,q)_E$ lie in the same connected component, from whence it follows that $\Nc(n,q)_E$ is connected.
\end{proof}

\subsection{The banal case}\label{ssec:local banal}

While the matrix equation $\Phi N \Phi^{-1} = qN$ defining $\Nct(n,q)$ is certainly simpler than the matrix equation $\Phi \Sigma\Phi^{-1} = \Sigma^q$ defining $\Mc(n,q)$, the scheme $\Nc(n,q)$ appearing in Proposition \ref{prop:N(n,q) local model} is defined as the maximal reduced $\Oc$-flat subscheme of $\Nct(n,q)$, and thus might be defined by a by a set of additional equations beyond the equation $\Phi N \Phi^{-1} = qN$. There is no apparent easy way to find these additional equations, so directly working with $\Nc(n,q)$ is still likely to be difficult in general.

To get around this issue, we will restrict our attention to a particular values of $q$ which make the situation much nicer. We will say that $q$ is \emph{banal} if $q\not\equiv 0\pmod{\ell}$ and $q^i\not\equiv 1\pmod{\ell}$ for each $i=1,\ldots,n$. This is equivalent to saying that $q$ is banal in a field of characteristic $\ell$ under our above definition.

In this case, we have the following:

\begin{prop}\label{prop:Nc(n,q) is CI}
If $q$ is banal then:
\begin{enumerate}
	\item $\Nc(n,q) = \Nct(n,q)$.
	\item $\Nc(n,q)$ is a complete intersection, reduced, flat and equidimensional over $\Oc$ of relative dimension $n^2$.
	\item $\Nc(n,q)_{\F}$ is a reduced complete intersection over $\F$ of dimension $n^2$.	
	\item $\Nc(n,q)\cong \Mc^{\unip}(n,q)$ is a connected component of $\Mc(n,q)$.
\end{enumerate}
\end{prop}
\begin{proof}
Since $q$ is banal, $\Nct(n,q)_E$ and $\Nct(n,q)_{\F}$ are both reduced complete intersections of dimension $n^2$ by Proposition \ref{prop:Nc(n,q)_E is CI}.

In particular it follows that $\Nct(n,q)$ has relative dimension at most $n^2$ over $\Oc$. But $\Gamma(\Nct(n,q))$ is generated over $\Oc$ by $2n^2$ generators (the entries of the matrices $\Phi,N\in \Gamma(\Nct(n,q))$) with $n^2$ relations (given by the matrix equation $\Phi N = q N\Phi$), and so it follows that $\Nct(n,q)$ is in fact a complete intersection over $\Oc$, equidimensional of relative dimension $n^2$.

In particular, it is flat over $\Oc$, and reduced as $\Nct(n,q)_E$ is reduced. By definition, this implies that $\Nc(n,q) = \Nct(n,q)$, proving (1) and (2).

Now $\Nc(n,q)_{\F} = \Nct(n,q)_{\F}$ by definition, so (3) follows.

For (4), the isomorphism $\Nc(n,q)\cong \Mc^{\unip}(n,q)$ comes from Proposition \ref{prop:N(n,q) local model}. It remains to show that $\Mc^{\unip}(n,q)$ is a connected component of $\Mc(n,q)$. Recall that it is indeed connected by Proposition \ref{prop:N(n,q) connected}.

Let $P = (\Phi,\Sigma)\in \Mc(n,q)$ be any point, and let $\kappa(P)$ be the residue field of $P$, which is a field of either characteristic $0$ or $\ell$, since $\Mc(n,q)$ is a $\Oc$-scheme. Let $\lambda_1,\ldots,\lambda_n\in \overline{\kappa(P)}$ be the eigenvalues of $\Sigma$. Then as $\Phi\Sigma\Phi^{-1} = \Sigma^q$ in $M_n(\kappa(P))$, the eigenvalues of $\Sigma$ are the same as those of $\Phi\Sigma\Phi^{-1} = \Sigma^q$, and hence $(\lambda_1^q,\lambda_2^q,\ldots,\lambda_n^q)$ is a permutation of $(\lambda_1,\lambda_2,\ldots,\lambda_n)$. In particular, for each $j=1,\ldots,n$ there is some $i_j\in\{1,\ldots,n\}$ with $\lambda_j = \lambda_j^{q^{i_j}-1}$.

Thus letting $M = \prod_{i=1}^n(q^i-1)$, it follows that each eigenvalue of $\Sigma$ is a $M^{th}$ root of unity in $\overline{\kappa(P)}$. Note that the condition that $q$ is banal implies that $\ell\nmid M$. Now let $\zeta_M\in \Qbar$ be a primitive $M^{th}$ root of unity, and consider the finite set of polynomials
\[
\Pc_M = \{(t-\zeta_M^{e_1})\cdots(t-\zeta_M^{e_n})|e_1,\ldots,e_n\in \Z\}\subseteq \Z[\zeta_M][t].
\]
By the above observation, there is some $p(t)\in \Pc_M$ with $\chi_\Sigma(t) = p(t)$ in $\kappa(P)$. On the other hand, since $\ell\nmid M$, the $M^{th}$ roots of unity $1,\zeta_M,\zeta_M^2,\ldots,\zeta_M^{M-1}$ are distinct in $\Fbar_\ell$. Since $\kappa(P)$ has characteristic either $0$ or $\ell$, it follows that the polynomials in $\Pc_M$ are all distinct in $\kappa(P)[t]$, and so in fact there is a unique $p(t)\in \Pc_M$ with $\chi_\Sigma(t) = p(t)$ in $\kappa(P)[t]$.

It follows from this, and the fact that $\Mc(n,q)$ is reduced, that $\ds\Mc(n,q) = \bigsqcup_{p(t)\in \Pc_M}V(\chi_\Sigma(t) = p(t))$. Since each $V(\chi_\Sigma(t) = p(t))$ is a closed subscheme of $\Mc(n,q)$, it follows that each $V(\chi_\Sigma(t) = p(t))$ is a union of connected components of $\Mc(n,q)$.

In particular, the scheme $M^{\unip}(n,q)^\circ_{\Oc} := V(\chi_\Sigma(t) = (t-1)^n)$ considered above is a union of connected components of $\Mc(n,q)$. Since $\Mc(n,q)_{\Oc}$ is reduced and $\Oc$-flat, this implies that $\Mc^{\unip}(n,q)^\circ_{\Oc}$ is reduced and $\Oc$-flat as well. In particular, $\Mc^{\unip}(n,q) = \Mc^{\unip}(n,q)^\circ_{\Oc}$, so $\Mc^{\unip}(n,q)$ is a union of connected components of $\Mc(n,q)$, and hence is a connected component as it is already connected.
\end{proof}

\section{Models for local deformation rings in the banal case}\label{sec:banal local models}

For this section we will assume that $q$ is banal (i.e. that $q\not\equiv 0\pmod{\ell}$ and $q^i\not\equiv 1\pmod{\ell}$ for all $i=1,\ldots,n$). Again let $\rbar:G_L\to \GL_n(\F)$ be a continuous Galois representation with $\rbar(\Pt_L) = 1$, and assume that $\rbar$ is unipotent. We shall also assume, by enlarging $\F$ if necessary, that $\F$ contains all eigenvalues of $\rbar(\varphi)$.

By Propositions \ref{prop:N(n,q) local model} and \ref{prop:Nc(n,q) is CI} we have $R^{\square}(\rbar) = R^{\unip}(\rbar)$ and $R^\square(\rbar)$ is isomorphic to the complete local ring of $\Nc(n,q)$ at the $\F$-point $(\rbar(\varphi),\log\rbar(\sigma))$.

While the ring $\Gamma(\Nc(n,q))$ is a fairly nicely behaved when $q$ is banal by Proposition \ref{prop:Nc(n,q) is CI}, it is still in general defined by $2n^2$ generators with $n^2$ relations. Even for small values of $n$, like $n=3$ or $n=4$, this can make it quite difficult to work with explicitly.

The purpose of this section is to construct a much simpler scheme, specific to the banal case, which can also be used to compute $R^\square(\rbar)$. This ring will be simple enough to allow us to carry out explicit computations in Section \ref{sec:R22} and \ref{sec:R121}.

\begin{defn}\label{def:Nc^nun}
Fix an integer $n\ge 1$ and an \emph{ordered} partition $\nun = (n_1,n_2,\ldots,n_k)$. For any ring $R$, let $\Nc^{\nun}(R)$ be the set $2k-1$ tuples $(X_1,N_1,X_2,N_2,\ldots,N_{k-1},X_k)$ where for each $i$, $X_i \in M_{n_i\times n_i}(R)$ and $N_{i}\in M_{n_i\times n_{i+1}}(R)$ satisfying the equations $X_iN_i = N_iX_{i+1}$ for each $i = 1,\ldots,k-1$.

This defines an affine scheme $\Nc^{\nun}_{\Z}$ over $\Spec \Z$. For any ring $R$, let $\Nc^{\nun}_R = \Nc^{\nun}_{\Z}\times_{\Spec \Z}\Spec R$. We will write $\Nc^{\nun} = \Nc^{\nun}_{\Oc}$.
\end{defn}

\subsection{Modeling $R^\square(\rbar)$ via $\Nc^{\nun}$}
We shall first show that the spaces $\Nc^{\nun}$ can be used to model the local deformation ring $R^\square(\rbar)$.

For this section, fix an $A\in \Cc_{\Oc}$, that is, an Artinian local $\Oc$-algebra $A$ with residue field $\F$. Let $\mf_A$ denote the maximal ideal of $A$, so that $A/\mf_A = \F$.

As noted above, $R^\square(\rbar)$ is isomorphic to the complete local ring of $\Nc(n,q)$ at the $\F$-point $x_{\rbar}:=(\rbar(\varphi),\log \rbar(\sigma))$. It follows that there are natural bijections
\begin{align*}
\Hom_{\Oc}(R^\square(\rbar),A) 
&\equiv \Nc(n,q)_{x_{\rbar}}(A)
\equiv 
\left\{(\Phi,N)\in \GL_n(A)\times M_{n\times n}(A)\middle|\parbox{1.55in}{$\Phi N = qN\Phi$,\\ $\Phi\equiv \rbar(\varphi)\pmod{\mf_A}$,\\ $N\equiv \log\rbar(\sigma)\pmod{\mf_A}$}\right\}.
\end{align*}
We aim to find a simpler description of the set on the right.

First, we note the following lemmas:

\begin{lemma}\label{lem:A^n unique decomposition}
Let $M$ be a free $A$-module of finite rank, and let $T:M\to M$ be an $A$-linear operator. Let $\Mbar = M/\mf_AM$ and let $\Tbar:\Mbar\to\Mbar$ be the induced $\F$-linear operator.

Let the distinct eigenvalues of $\Tbar$ be $\lambdabar_1,\lambdabar_2,\ldots,\lambdabar_k$, and assume all of these are in $\F$. For each $i$, let $\lambda_i\in A$ be a lift of $\lambdabar_i$. Then there is a unique decomposition
\[
M = M_{1}\oplus M_{2}\oplus\cdots\oplus M_{k}
\]
where for each $i$, $T(M_i)\subseteq M_i$ and $T-\lambda_i$ acts nilpotently on $M_i$.

Moreover, each $M_i$ is a free $A$-module and $\Mbar_i := M_i/\mf_A M_i\subseteq \Mbar$ is the generalized eigenspace of the eigenvalue $\lambdabar_i$ of $\Tbar$.
\end{lemma}
\begin{proof}
We shall prove the existence of this decomposition by induction on $k$. When $k=0$, we have $M = 0$, and so the claim is automatically true.

Now assume that the claim holds whenever $\Tbar$ has $k$ distinct eigenvalues. Let $M$ and $T$ satisfy the conditions of the lemma, and assume that $\Tbar$ has $k+1$ distinct eigenvalues $\lambdabar_1,\ldots,\lambdabar_k,\lambdabar_{k+1}$.

Since $A$ is Artinian, $M$ has finite length as an $A$-module. It follows that the chains of submodules $\{\ker (T-\lambda_{k+1})^m\}$ and $\{(T-\lambda_{k+1})^mM\}$ both eventually stabilize. That is, there is some $N\ge 0$ such that $\ker (T-\lambda_{k+1})^m = \ker (T-\lambda_{k+1})^N$ and $(T-\lambda_{k+1})^mM = (T-\lambda_{k+1})^NM$ for all $m\ge N$.

Let $M_{k+1} = \ker (T-\lambda_{k+1})^N$ and $M' = (T-\lambda_{k+1})^NM$. Clearly $T(M_{k+1})\subseteq M_{k+1}$ and $T(M')\subseteq M'$. By definition $(T-\lambda_{k+1})^NM_{k+1} = 0$, so $T-\lambda_{k+1}$ acts nilpotently on $M_{k+1}$. Also 
\[(T-\lambda_{k+1})M' = (T-\lambda_{k+1})^{N+1}M = (T-\lambda_{k+1})^NM = M'\]
so (as $M'$ has finite length), $T-\lambda_{k+1}$ is invertible on $M'$.

In particular this implies that $M_{k+1}\cap M' = 0$, and so there is an injection $M'\oplus M_{k+1}\into M$. But also by considering the $A$-linear operator $(T-\lambda_{k+1})^N:M\to M$ we see that
\begin{align*}
\length_A(M) &= \length_A \ker (T-\lambda_{k+1})^N + \length_A (T-\lambda_{k+1})^NM = \length_A M_{k+1}+\length_A M'\\
& = \length_A(M'\oplus M_{k+1}).
\end{align*}
As $M$ has finite length, this gives $M = M'\oplus M_{k+1}$. Since $M$ is a free $A$-module, $M'$ and $M_{k+1}$ are thus projective, and hence free, $A$-modules.

If $\Mbar' = M'/\mf_AM'$ and $\Mbar_{k+1} = M_{k+1}/\mf_A$, then $\Mbar = \Mbar'\oplus \Mbar_{k+1}$, and $\Tbar-\lambdabar_{k+1}$ acts invertibly on $\Mbar'$ and nilpotently on $\Mbar_{k+1}$. It follows that $\Mbar_{k+1}$ is the generalized $\Tbar$-eigenspace of $\lambdabar_{k+1}$ and $\lambdabar_{k+1}$ is not an eigenvalue of the restriction $\Tbar|_{\Mbar'}$. Moreover it follows that for each $i\le k$ the generalized $\Tbar$-eigenspace of $\lambdabar_i$ is contained in $\Mbar'$, and hence coincides with the generalized $\Tbar|_{\Mbar'}$-eigenspace of $\lambdabar_k$.

The claim now follows by applying the inductive hypothesis to $M'$ and $T|_{M'}$, as the eigenvalues of $\Tbar|_{\Mbar'}$ are precisely $\lambdabar_1,\ldots,\lambdabar_k$.

To prove uniqueness, note that if 
\[
M = M_1'\oplus M_2'\oplus \cdots\oplus M_k'
\]
is any decomposition for which $T-\lambda_i$ acts nilpotently on $M_i'$, then for all $j\ne i$, $\lambda_j\not\equiv \lambda_i\pmod{\mf_A}$, and so $T-\lambda_i = (T-\lambda_j)+(\lambda_j-\lambda_i)$ acts invertibly $M_j'$. It follows that $\ker (T-\lambda_i)^N = M_i'$ for all $N\gg 0$, and so the decomposition is indeed unique. 
\end{proof}

\begin{lemma}\label{lem:unique change of basis}
Let $\lambdabar_1,\ldots,\lambdabar_k\in \F$ be distinct and nonzero and let $n_1,\ldots,n_k\ge 1$ be integers with $n_1+\cdots+n_k = n$. For each $i=1,\ldots,k$ let $\Ubar_i\in \GL_{n_i}(\F)$ be a unipotent matrix. Consider the block diagonal matrix
\[
\Tbar = 
\begin{pmatrix}
\lambdabar_1\Ubar_1 & 0 & \cdots & 0\\
0 & \lambdabar_2\Ubar_2 & \cdots & 0\\
\vdots & \vdots & \ddots & \vdots\\
0 & 0 & \cdots & \lambdabar_k\Ubar_k
\end{pmatrix}
\in M_{n\times n}(\F)
\]
For each $i=1,\ldots,k$, fix a lift $\lambda_i\in A$ of $\lambdabar_i$. Then any lift $T\in M_{n\times n}(A)$ of $\Tbar$ may be written \emph{uniquely} in the form:
\[
T = 
\begin{pmatrix}
I_{n_1} & B_{12} & \cdots & B_{1n}\\
B_{21} & I_{n_2} & \cdots & B_{2n}\\
	\vdots & \vdots & \ddots & \vdots\\
	B_{k1} & B_{k2} & \cdots & I_{n_k}
\end{pmatrix}
\begin{pmatrix}
	\lambda_1Y_1 & 0 & \cdots & 0\\
	0 & \lambda_2Y_2 & \cdots & 0\\
	\vdots & \vdots & \ddots & \vdots\\
	0 & 0 & \cdots &  \lambda_kY_k
\end{pmatrix}
\begin{pmatrix}
	I_{n_1} & B_{12} & \cdots & B_{1n}\\
	B_{21} & I_{n_2} & \cdots & B_{2n}\\
	\vdots & \vdots & \ddots & \vdots\\
	B_{k1} & B_{k2} & \cdots & I_{n_k}
\end{pmatrix}^{-1}
\]
where
\begin{itemize}
	\item $B_{ij}\in \mf_A M_{n_i\times n_j}(A)$ for all $i\ne j$; and
	\item $Y_i \equiv \Ubar_i\pmod{\mf_A}$ for all $i=1,\ldots,k$.
\end{itemize}
\end{lemma}
\begin{proof}
Label the standard basis for $A^n$ as 
\[e^{(1)}_{1},\ldots,e^{(1)}_{n_1},e^{(2)}_{1},\ldots,e^{(2)}_{n_2},\ldots,e^{(k)}_{1},\ldots,e^{(k)}_{n_k}\]
in order. This means that for each $i$, $\Tbar$ stabilizes the subspace $\Span\{\ebar^{(i)}_{1},\ldots,\ebar^{(i)}_{n_i}\}\subseteq \F^n$, where $\ebar^{(i)}_{j} = e^{(i)}_{j}\pmod{\mf_A}$. Moreover as each $\Ubar_i$ is unipotent, we in fact have that $\Span\{\ebar^{(i)}_{1},\ldots,\ebar^{(i)}_{n_i}\}$ is the generalized $\Tbar$-eigenspace of $\lambdabar_i$.

Fix a lift $T$ of $\Tbar$. By Lemma \ref{lem:A^n unique decomposition} we may write
\[
A^n = M_1\oplus M_2 \oplus \cdots \oplus M_k
\]
where each $M_i$ is a $T$-stable free $A$ module with $\Mbar_i:=M_i/\mf_A M_i = \Span\{\ebar^{(i)}_{1},\ldots,\ebar^{(i)}_{n_i}\}$.

For each $i$, consider the projection map $\pi^{(i)}: A^n\to \Span\{e^{(i)}_1,\ldots,e^{(i)}_{n_i}\}$ given by $\pi^{(i)}(e^{(i)}_j) = e^{(i)}_j$ and $\pi^{(i)}(e^{(\alpha)}_j) = 0$ if $\alpha\ne i$. Let $\pibar^{(i)}:\F^n\to \Span\{\ebar^{(i)}_1,\ldots,\ebar^{(i)}_{n_i}\}$ be the induced map.

Then as $\pibar^{(i)}(\Mbar_i) = \Span\{\ebar^{(i)}_1,\ldots,\ebar^{(i)}_{n_i}\}$, Nakayama's Lemma gives $\pi^{(i)}(M_i) = \Span\{e^{(i)}_1,\ldots,e^{(i)}_{n_i}\}$. Since $M_i$ and $\Span\{e^{(i)}_1,\ldots,e^{(i)}_{n_i}\}$ are both free modules of rank $n_i$, this implies that $\pi^{(i)}|_{M_i}:M_i\to \Span\{e^{(i)}_1,\ldots,e^{(i)}_{n_i}\}$ is an isomorphism. Hence there is a unique $A$-basis $\{v^{(i)}_1,\ldots,v^{(i)}_{n_j}\}$ for $M_i$ with the property that $\pi^{(i)}(v^{(i)}_j) = e^{(i)}_j$ for $j=1,\ldots,n_i$.

Since $A^n = M_1\oplus M_2 \oplus \cdots \oplus M_k$, it follows that $\{v^{(i)}_j\}_{ij}$ is an $A$-basis for $A^n$.

Writing
\[v^{(i)}_j = \sum_{\alpha = 1}^k\sum_{\beta = 1}^{n_a} b^{(i\alpha)}_{j\beta}e^{(\alpha)}_\beta\]
we see that $b^{(ii)}_{j\beta} = \delta_{j\beta}$ for all $i,j$ and $\beta$ from the assumption that $\pi^{(i)}(v^{(i)}_j) = e^{(i)}_j$. Also as $M_i \subseteq \Span\{e^{(i)}_1,\ldots,e^{(i)}_{n_i}\}+\mf_A A^n$ we get $b^{(i\alpha)}_{j\beta} \in \mf_A$ for $\alpha\ne i$.

Now let $B$ be the change of basis matrix between the bases $\{e^{(i)}_j\}_{ij}$ and $\{v^{(i)}_j\}_{ij}$. Then $B$ has the form 
\[
B = 
\begin{pmatrix}
	I_{n_1} & B_{12} & \cdots & B_{1n}\\
	B_{21} & I_{n_2} & \cdots & B_{2n}\\
	\vdots & \vdots & \ddots & \vdots\\
	B_{k1} & B_{k2} & \cdots & I_{n_k}
\end{pmatrix}
\]
for $B_{ij}\in \mf_A M_{n_i\times n_j}(A)$, and we may write $T = BCB^{-1}$, where $C$ is the matrix representation of $T$ with respect to the basis $\{v^{(i)}_j\}$.

Since each of the submodules $M_i = \Span\{v^{(i)}_1,\ldots,v^{(i)}_{n_i}\}$ is stabilized by $T$, $C$ is a block diagonal matrix:
\[
C = 
\begin{pmatrix}
	C_1 & 0 & \cdots & 0\\
	0 & C_2 & \cdots & 0\\
	\vdots & \vdots & \ddots & \vdots\\
	0 & 0 & \cdots & C_k
\end{pmatrix}.
\]
As $\lambdabar_i\ne 0$ by assumption, $\lambda_i\in A^\times$, and so we may write $C_i = \lambda_iY_i$ for some $Y_i\in M_{n_i\times n_i}(A)$.

But now $B\equiv I_n\pmod{\mf_A}$, so $\Tbar\equiv BCB^{-1}\equiv C \pmod{\mf_A}$, and so 
\[
\lambdabar_iY_i \equiv \lambda_iY_i\equiv C_i \equiv \lambdabar_i\Ubar_i \pmod{\mf_A}
\]
for all $i$. And so $Y_i\equiv \Ubar_i\pmod{\mf_A}$ as $\lambdabar_i\in\F^\times$. So indeed, $T$ can be written in the desired form.

To prove uniqueness, say that $T = B'C'(B')^{-1}$ where
\begin{align*}
B' &= 
\begin{pmatrix}
	I_{n_1} & B'_{12} & \cdots & B'_{1n}\\
	B'_{21} & I_{n_2} & \cdots & B'_{2n}\\
	\vdots & \vdots & \ddots & \vdots\\
	B'_{k1} & B'_{k2} & \cdots & I_{n_k}
\end{pmatrix}&
&\text{and}&
C' &= 
\begin{pmatrix}
	\lambda_1Y'_1 & 0 & \cdots & 0\\
	0 & \lambda_2Y'_2 & \cdots & 0\\
	\vdots & \vdots & \ddots & \vdots\\
	0 & 0 & \cdots &  \lambda_kY'_k
\end{pmatrix}
\end{align*}
for $B'_{ij}\in \mf_AM_{n_i\times n_j}(A)$ and $Y'_i\in M_{n_i\times n_i}(A)$ with $Y'_i\equiv \Ubar_i\pmod{\mf_A}$. 

Since $B'\equiv I_n\pmod{\mf_A}$, $B'\in \GL_n(A)$, and so $\{w^{(i)}_j := B'e_j^{(i)}\}_{ij}$ is an $A$-basis for $A^n$. From the definition of $B'$ we see that $\pi^{(i)}(w^{(i)}_j) = e^{(i)}_j$ for all $i$ and $j$. Since $C'$, the matrix of $T$ in the basis $\{w^{(i)}_j\}_{ij}$, is block diagonal, we get that $T$ stabilizes $M_i':= \Span\{w^{(i)}_1,\ldots, w^{(i)}_{n_i}\}$ for each $i$ and $A^n = M_1'\oplus \cdots \oplus M_k'$.

Now for each $i$ consider the matrix $\lambda_iY_i'-\lambda_iI_{n_i}$. We have that 
\[
\lambda_iY_i'-\lambda_iI_{n_i} \equiv \lambdabar_i\Ubar_i-\lambdabar_iI_{n_i} = \lambdabar_i(\Ubar_i-I_{n_i})\pmod{\mf_A}.
\]
Since $\Ubar_i$ is unipotent, $\Ubar_i-I_{n_i}$ is nilpotent. Hence $(\lambda_iY_i'-\lambda_iI_{n_i})^{n_i} \in \mf_AM_{n_i\times n_i}(A)$. Since $A$ is Artinian and local, $\mf_A^d = 0$ for some $d>0$, and so $(\lambda_iY_i'-\lambda_iI_{n_i})^{n_id} = 0$, and so $\lambda_iY_i'-\lambda_iI_{n_i}$ acts nilpotently on $M_{n_i\times n_i}(A)$. It follows that $T-\lambda_i I_n$ acts nilpotently on $M_i'$ for each $i$.

By the uniqueness part of Lemma \ref{lem:A^n unique decomposition}, it now follows that $M_i' = M_i$. Now for each $i$ and $j$ we get that $w^{(i)}_j \in M_i$ and $\pi^{(i)}(w^{(i)}_j)$. By the uniqueness of $v^{(i)}_j$ noted above, this gives $w^{(i)}_j = v^{(i)}_j$ for all $i$ and $j$. In particular, this implies that $B' = B$, and so $B_{ij}' = B_{ij}$ for all $i\ne j$.

Lastly, as $B = B'$ we get $BCB^{-1} = T = BC'B^{-1}$ so $C = C'$ and so $\lambda_iY_i = \lambda_iY_i'$ for all $i$. Since $\lambda_i\in A^\times$ this indeed implies $Y_i = Y_i'$, and so the decomposition is indeed unique.
\end{proof}

\begin{lemma}\label{lem:alpha U N = beta N V}
Let $a,b\ge 1$ be integers. Pick elements $\alpha,\beta\in A$ with $\alpha\not\equiv \beta \pmod{\mf_A}$ and matricies $U\in M_{a\times a}(A)$ and $V \in M_{b\times b}(A)$ with $U \pmod{\mf_A}$ and $V\pmod{\mf_B}$ unipotent in $M_{a\times a}(\F)$ and $M_{b\times b}(\F)$ respectively.

If $N\in M_{a\times b}(A)$ is any matrix satisfying $\alpha U N = \beta N V$, then $N = 0$.
\end{lemma}
\begin{proof}
As $(\alpha U)N = N(\beta V)$ we get that $(\alpha U)^i N = N (\beta V)^i$ for all $i\ge 0$, and so $p(\alpha U) N = N p(\beta V)$ for all $p(t)\in A[t]$.

As $V$ is unipotent, all eigenvalues of $V\pmod{\mf_A}$ in $M_{b\times b}(\F)$ are $1$, and hence all eigenvalues of $(\beta V - \beta I_n) \pmod{\mf_A}$ are $0$. As in the proof of Lemma \ref{lem:unique change of basis} this implies that $(\beta V - \beta I_b)^D = 0$ in $M_{b\times b}(A)$ for some $D\gg 0$. Hence, letting $p(t) = (t-\beta)^D \in A[t]$ we get
\[
(\alpha U-\beta I_a)^D N = N(\beta V-\beta I_b)^D = N0 = 0.
\]
Now as $U$ is unipotent, all eigenvalues of $(\alpha U - \beta I_a)\pmod{\mf_A}$ are equal to $\alpha-\beta \not\equiv 0 \pmod{\mf_A}$. Thus $(\alpha U - \beta I_a)\pmod{\mf_A}$ is invertible in $M_{a\times a}(\F)$, from whence it follows that $(\alpha U - \beta I_a)\in \GL_a(A)$. So as $(\alpha U-\beta I_a)^D N = 0$, we indeed get $N = 0$.
\end{proof}

Now consider the Galois representation $\rbar:G_L\to \GL_n(\F)$. Recall that we have assumed that all eigenvalues of $\rbar(\varphi)$ lie in $\F$. Now note that we may write the set of eigenvalues of $\rbar(\varphi)$ uniquely in the form
\[
\bigsqcup_{\alpha = 1}^d \{\lambdabar_\alpha,\lambdabar^{(\alpha)} q,\lambdabar^{(\alpha)} q^2,\ldots,\lambdabar^{(\alpha)} q^{k_\alpha-1}\}
\]
for $\lambdabar^{(1)},\ldots,\lambdabar^{(d)}\in \F^\times$ with $\lambdabar^{(\alpha)} \neq \lambdabar^{(\beta)} q^{i}$ for all $\alpha\ne \beta$ and $i = 0,1,\ldots,k_\beta$.

For $i=1,\ldots,k_\alpha$, let $n^{(\alpha)}_i$ be the (algebraic) multiplicity of $\lambdabar^{(\alpha)} q^{k_\alpha-i}$ as an eigenvalue of $\rbar(\varphi)$. Let $\nun^{(\alpha)} = (n^{(\alpha)}_,\ldots,n^{(\alpha)}_{k_\alpha})$, and note that $n = \sum_\alpha\sum_i n^{(\alpha)}_i$.

To avoid having to repeatedly write out large matrices, we will write $C = \left(C^{(\alpha\beta)}_{ij}\right)^{\alpha\beta}_{ij}$ to mean that $C$ is a block diagonal matrix
\[
C = 
\begin{pmatrix}
	C^{(11)}&C^{(12)}&\cdots&C^{(1d)}\\
	C^{(21)}&C^{(22)}&\cdots&C^{(2d)}\\
	\vdots & \vdots & \ddots & \vdots\\
	C^{(d1)}&C^{(d2)}&\cdots&C^{(dd)}
\end{pmatrix}
\]
where each $C^{(\alpha\beta)}$ is in turn a block diagonal matrix:
\[
C^{(\alpha\beta)}=
\begin{pmatrix}
	C^{(\alpha\beta)}_{11}&C^{(\alpha\beta)}_{12}&\cdots&C^{(\alpha\beta)}_{1k_\beta}\\
	C^{(\alpha\beta)}_{21}&C^{(\alpha\beta)}_{22}&\cdots&C^{(\alpha\beta)}_{2k_\beta}\\
	\vdots & \vdots & \ddots & \vdots\\
	C^{(\alpha\beta)}_{k_\alpha 1}&C^{(\alpha\beta)}_{k_\alpha 2}&\cdots&C^{(\alpha\beta)}_{k_\alpha k_\beta}
\end{pmatrix}
\]
with each $C^{(\alpha\beta)}_{ij}$ a $n^{(\alpha)}_i\times n^{(\beta)}_j$ matrix. We will also write $C = \diag\left(C^{(\alpha)}_i\right)^\alpha_i$ to mean that $C = \left(C^{(\alpha\beta)}_{ij}\right)^{\alpha\beta}_{ij}$ with $C^{(\alpha\alpha)}_{ii} = C^{(\alpha)}_i$ and $C^{(\alpha\beta)}_{ij} = 0$ whenever $(\alpha,i)\ne (\beta,j)$.
 
By changing the basis for $\F^n$, we may assume without loss of generality that $\rbar(\varphi)$ is in Jordan normal form. That is, $\rbar(\varphi)$ is a block diagonal matrix in the form
\[
\rbar(\varphi) = \diag\left(\lambdabar^{(\alpha)}_iq^{k_\alpha-i}\Ubar_i^{(\alpha)}\right)^\alpha_i
\]
where each $\Ubar^{(\alpha)}_i\in M_{n^{(\alpha)}_i\times n^{(\alpha)}_i}(\F)$ is unipotent.

Lemmas \ref{lem:unique change of basis} and \ref{lem:alpha U N = beta N V} give the following:
\begin{cor}\label{cor:(Phi,N) decomposition}
Assume that $\rbar(\varphi) = \diag\left(\lambdabar^{(\alpha)}_iq^{k_\alpha-i}\Ubar_i^{(\alpha)}\right)^\alpha_i$ is in the above form. Fix a lift $\lambda^{(\alpha)}\in A$ of $\lambdabar^{(\alpha)}$ for each $\alpha$.
	
Then:
\begin{enumerate}
	\item If $\log\rbar(\sigma) = \left(\Nbar^{(\alpha\beta)}_{ij}\right)^{\alpha\beta}_{ij}$ then $\Nbar^{(\alpha\beta)}_{ij} = 0$ unless $\alpha=\beta$ and $j=i+1$. Moreover, for all $\alpha$ and all $i=1,\ldots,k_\alpha-1$ we have $\Ubar_i^{(\alpha)}\Nbar_{i,i+1}^{(\alpha\alpha)} = \Nbar_{i,i+1}^{(\alpha\alpha)}\Ubar_{i+1}^{(\alpha)}$ 
	\item If $\Phi\in \GL_n(A)$ and $N\in M_{n\times n}(A)$ are any matrices satisfying $\Phi N = qN\Phi$, $\Phi\equiv \rbar(\varphi)\pmod{\mf_A}$ and $N\equiv \log\rbar(\sigma)\pmod{\mf_A}$, then there are \emph{unique} matrices $B = \left(B^{(\alpha\beta)}_{ij}\right)^{\alpha\beta}_{ij}$, $\Phi' = \diag\left(\lambda^{(\alpha)}q^{k_\alpha-i}Y_i^{(\alpha)}\right)^{\alpha}_{i}$ and $N' = \left(N^{(\alpha\beta)}_{ij}\right)^{\alpha\beta}_{ij}$ in $M_{n\times n}(A)$ with the properties:
	\begin{enumerate}
		\item $\Phi = B\Phi'B^{-1}$ and $N = BN'B^{-1}$;
		\item $B^{(\alpha\alpha)}_{ii} = I_{n^{(\alpha)}_i}$ for all $\alpha$ and $i$, and $B^{(\alpha\beta)}_{ij} \in \mf_A M_{n^{(\alpha)}_i\times n^{(\beta)}_j}(A)$ for all $(\alpha,i)\ne (\beta,j)$;
		\item $Y_i^{(\alpha)}\equiv \Ubar_i^{(\alpha)}\pmod{\mf_A}$ for all $i$ and $\alpha$;
		\item $N^{(\alpha\beta)}_{ij} = 0$ unless $\alpha=\beta$ and $j=i+1$, and $N^{(\alpha\alpha)}_{i,i+1} \equiv \Nbar^{(\alpha\alpha)}_{i,i+1}\pmod{\mf_A}$ for all $\alpha$ and all $i=1,\ldots,k_\alpha-1$;
		\item $Y_i^{(\alpha)} N_{i,i+1}^{(\alpha\alpha)} = N_{i,i+1}^{(\alpha\alpha)}Y_{i+1}^{(\alpha)}$ for all $\alpha$ and all $i=1,\ldots,k_\alpha-1$.
	\end{enumerate}
	Moreover if $B,\Phi'$ and $N'$ are any matrices in $M_{n\times n}(A)$ satisfying the above properties, then $\Phi:= B\Phi'B^{-1}$ and $N := BN'B^{-1}$ satisfy $\Phi N = qN\Phi$, $\Phi\equiv \rbar(\varphi)\pmod{\mf_A}$ and $N \equiv \log\rbar(\sigma)\pmod{\mf_A}$.
\end{enumerate} 
\end{cor}
\begin{proof}
Note that (1) is simply the special case of (2) with $A = \F$, so it will be enough to prove (2) (in the case of $A = \F$, the claim that $N^{(\alpha\alpha)}_{i,i+1} \equiv \Nbar^{(\alpha\alpha)}_{i,i+1}\pmod{\mf_A}$ is simply the definition of $\Nbar^{(\alpha\alpha)}_{i,i+1}$).

Fix any $\Phi$ and $N$ satisfying the given properties. By Lemma \ref{lem:unique change of basis}, there exist unique matrices $B = \left(B^{(\alpha\beta)}_{ij}\right)^{\alpha\beta}_{ij}$ and $\Phi' = \diag\left(\lambda^{(\alpha)}q^{k_\alpha-i}Y_i^{(\alpha)}\right)^{\alpha}_{i}$ satisfying $\Phi = B\Phi'B^{-1}$ and properties (b) and (c). Now let $N' = B^{-1}NB = \left(N^{(\alpha\beta)}_{ij}\right)^{\alpha\beta}_{ij}$.

It remains to show that properties (d) and (e) are satisfied. First
\[
N' \equiv B^{-1}NB \equiv I_n^{-1}(\log \rbar(\sigma))I_n \equiv \log \rbar(\sigma) \equiv \left(\Nbar^{(\alpha\beta)}_{ij}\right)^{\alpha\beta}_{ij}\pmod{\mf_A}
\]
so indeed $N^{(\alpha\alpha)}_{i,i+1} \equiv \Nbar^{(\alpha\alpha)}_{i,i+1}\pmod{\mf_A}$.

Since $\Phi' = B^{-1}\Phi B$ and $N' = B^{-1}NB$ we have $\Phi'N' = q N'\Phi'$. Since $\Phi'$ is block diagonal, this implies that
\[
\lambda^{(\alpha)}q^{k_\alpha-i}Y_i^{(\alpha)}N^{(\alpha\beta)}_{ij} = \lambda^{(\beta)}q^{k_\beta-j+1}N^{(\alpha\beta)}_{ij}Y_j^{(\beta)}.
\]
Since $Y_i^{(\alpha)} \equiv \Ubar^{(\alpha)}_i\pmod{\mf_A}$ and $Y_j^{(\beta)}\equiv \Ubar^{(\beta)}_j\pmod{\mf_A}$ are both unipotent, Lemma \ref{lem:alpha U N = beta N V} implies that $N^{(\alpha\beta)}_{ij} = 0$ whenever $\lambdabar^{(\alpha)}q^{k_\alpha-i}\ne \lambdabar^{(\beta)}q^{k_\beta-j+1}$.

If $\alpha\ne \beta$ then $\lambdabar^{(\alpha)}q^{k_\alpha-i}\ne \lambdabar^{(\beta)}q^{k_\beta-j+1}$ for all $1\le i\le k_\alpha$ and $1\le j\le k_\beta$ by assumption, as
\[
(k_\alpha-i)-(k_\beta-j+1) \in [-k_\beta,k_\alpha-2].
\]
while $\lambdabar^{(\beta)}\ne \lambdabar^{(\alpha)}q^{x}$ for $x\in [-k_\beta,k_\alpha]$.

If $\alpha=\beta$, and $\lambdabar^{(\alpha)}q^{k_\alpha-i}= \lambdabar^{(\alpha)}q^{k_\alpha-j+1}$, then $q^{i-j+1} = 1$ (as $\lambdabar^{(\alpha)}\ne 0$). Since $1\le i,j\le n$, $-n+2\le i-j+1 \le n$, so as $q$ is banal, this implies $i-j+1 = 0$. So indeed $N^{(\alpha\beta)}_{ij} = 0$ unless $\alpha=\beta$ and $j=i+1$, proving (d).

Finally for all $\alpha$ and $i = 1,\ldots,k_\alpha-1$ we now have
\[
\lambda^{(\alpha)}q^{k_\alpha-i}Y_i^{(\alpha)}N^{(\alpha\alpha)}_{i,i+1} = \lambda^{(\beta)}q^{k_\alpha-i}N^{(\alpha\alpha)}_{i,i+1}Y_{i+1}^{(\alpha)},
\]
so (e) follows as $\lambda^{(\alpha)},q\in A^\times$.

It is now easy to verify that conditions (a)-(e) indeed imply that $\Phi N = qN\Phi$, $\Phi\equiv \rbar(\varphi)\pmod{\mf_A}$ and $N \equiv \log\rbar(\sigma)\pmod{\mf_A}$. 
\end{proof}

Thus we have natural bijections
\begin{align*}
\Hom_{\Oc}(R^\square(\rbar),A) 
	&\equiv 
	\left\{(\Phi,N)\in \GL_n(A)\times M_{n\times n}(A)\middle|\parbox{1.8in}{$\Phi N = qN\Phi$,\\ $\Phi\equiv \rbar(\varphi)\pmod{\mf_A}$,\\ $N\equiv \log\rbar(\sigma)\pmod{\mf_A}$}\right\}.\\
	&\equiv \prod_{(\alpha,i)\ne (\beta,j)}\mf_AM_{n_i^{(\alpha)}\times n_j^{(\beta)}}(A) \times 
	\prod_{\alpha=1}^d \left\{(Y_i^{(\alpha)})_{i=1}^{k_\alpha}, (N^{(\alpha\alpha)}_{i,i+1})\middle|\parbox{1.8in}{$Y_i^{(\alpha)}\equiv \Ubar_i^{(\alpha)}\pmod{\mf_A}$,\\$N^{(\alpha\alpha)}_{i,i+1}\equiv \Nbar^{(\alpha\alpha)}_{i,i+1}\pmod{\mf_A}$,\\ $Y_i^{(\alpha)} N_{i,i+1}^{(\alpha\alpha)} =  N_{i,i+1}^{(\alpha\alpha)}Y_{i+1}^{(\alpha)}$}\right\}\\
	&\equiv \mf_A^{n^2-\sum_\alpha\sum_i \left(n_i^{(\alpha)}\right)^2} \times 
\prod_{\alpha=1}^d \left\{(X_i^{(\alpha)})_{i=1}^{k_\alpha}, (N^{(\alpha\alpha)}_{i,i+1})\middle|\parbox{2in}{$X_i^{(\alpha)}\equiv \Ubar_i^{(\alpha)}-I_{n_i^{(\alpha)}}\pmod{\mf_A}$,\\$N^{(\alpha\alpha)}_{i,i+1}\equiv \Nbar^{(\alpha\alpha)}_{i,i+1}\pmod{\mf_A}$,\\ $X_i^{(\alpha)} N_{i,i+1}^{(\alpha\alpha)} =  N_{i,i+1}^{(\alpha\alpha)}X_{i+1}^{(\alpha)}$}\right\},
\end{align*}
where we made a variable substitution $X_i^{(\alpha)} = Y_i^{(\alpha)}-I_{n_i^{(\alpha)}}$.

Now noting that $\Hom_{\Oc}(\Oc[[t]],A) \cong \mf_A$, and that $\Oc$-algebras in $\Cc_{\Oc}^{\wedge}$ pro-representing the same functor $\Cc_{\Oc}\to \Set$ are isomorphic, we get the following:

\begin{thm}\label{thm:banal local model}
Let $q$ be banal and let $\rbar:G_L\to \GL_n(\F)$ be a Galois representation with $\rbar(\Pt_L) = 1$ and $\rbar(\sigma)$ unipotent. Write the set of eigenvalues of $\rbar(\varphi)$ as
\[
\bigsqcup_{\alpha = 1}^d \{\lambdabar_\alpha,\lambdabar^{(\alpha)} q,\lambdabar^{(\alpha)} q^2,\ldots,\lambdabar^{(\alpha)} q^{k_\alpha-1}\}
\]
for $\lambdabar^{(1)},\ldots,\lambdabar^{(d)}\in \F^\times$ with $\lambdabar^{(\alpha)} \neq \lambdabar^{(\beta)} q^{i}$ for all $\alpha\ne \beta$ and $i = 0,1,\ldots,k_\beta$.

For $i=1,\ldots,k_\alpha$, let $n^{(\alpha)}_i$ be the algebraic multiplicity of $\lambdabar^{(\alpha)} q^{k_\alpha-i}$ as an eigenvalue of $\rbar(\varphi)$. Let $\nun^{(\alpha)} = (n^{(\alpha)}_,\ldots,n^{(\alpha)}_{k_\alpha})$.

Let $\Ubar_i^{(\alpha)}$ and $\Nbar_{i,i+1}^{(\alpha\alpha)}$ be as in Corollary \ref{cor:(Phi,N) decomposition}(1). For each $\alpha$, let 
\[
x^{(\alpha)} = \left(\Ubar_1^{(\alpha)}-I_{n_1^{(\alpha)}},\Nbar^{(\alpha\alpha)}_{12},\Ubar_2^{(\alpha)}-I_{n_2^{(\alpha)}},\ldots,\Nbar^{(\alpha\alpha)}_{k_\alpha-1,k_\alpha},\Ubar_{k_\alpha}^{(\alpha)}-I_{n_{k_\alpha}}^{(\alpha)}\right) \in \Nc^{\nun^{(\alpha)}}(\F).
\]
Let $R^{\nun^{(\alpha)}}_{x^{(\alpha)}}$ denote the complete local ring of $\Nc^{\nun^{(\alpha)}}$ at the $\F$-point $x^{(\alpha)}$. Then there is an isomorphism
\[
R^\square(\rbar) \cong \widehat{\bigotimes}_{\alpha=1}^dR^{\nun^{(\alpha)}}_{x^{(\alpha)}}\left[\left[t_1,\ldots,t_{n^2-\sum_\alpha\sum_i \left(n_i^{(\alpha)}\right)^2}\right]\right].
\]
\end{thm}

Thus to understand the local deformation rings $R^\square(\rbar)$ in the banal, unipotent case, it will suffice to understand the spaces $\Nc^{\nun}$.

Before moving on, we will note the following simple consequence of Theorem \ref{thm:banal local model}:

\begin{cor}\label{cor:Nc^nun is CI}
For any $\nun = (n_1,\ldots,n_k)$, the $\Oc$-scheme $\Nc^{\nun}$ is reduced and flat over $\Oc$ of relative dimension $\sum_{i=1}^k n_i^2$. Moreover it is a complete intersection.

Also $\Nc^{\nun}_{\F}$ is a reduced complete intersection of dimension $\sum_{i=1}^k n_i^2$ over $\F$.
\end{cor}
\begin{proof}
Let $n = \sum_{i=1}^k n_i$, and consider the local Galois representation $\rbar:G_L\to \GL_n(\F)$ defined by $\rbar(I_L) = 1$ and
\[
\rbar(\varphi) =
\begin{pmatrix}
q^{k-1}I_{n_1} & 0 & \cdots & 0\\
0 & q^{k-2}I_{n_2} & \cdots & 0\\
\vdots&\vdots&\cdots&\vdots\\
0 & 0 & \cdots & I_{n_k}
\end{pmatrix}\in \GL_n(\F).
\]
By Theorem \ref{thm:banal local model} we have an isomorphism
\[
R^\square(\rbar) \cong R^{\nun}_x\left[\left[t_1,\ldots,t_{n^2-\sum_i n_i^2}\right]\right]
\]
where $R^{\nun}_x$ is the complete local ring of $\Nc^{\nun}$ at the $\F$-point
\[
x = (I_{n_1}-I_{n_1},0,I_{n_2}-I_{n_2},0,\ldots,0,I_{n_k}-I_{n_k}) = (0,0,\ldots,0).
\]
Now as $R^\square(\rbar)$ is reduced, flat and equidimensional over $\Oc$ of relative dimension $n^2$ (by Propositions \ref{prop:N(n,q) local model} and \ref{prop:Nc(n,q) is CI}) it follows that $R^{\nun}_x$ is reduced, flat and equidimensional over $\Oc$ of relative dimension $\sum_{i=1}^k n_i^2$. Similarly as $R^\square(\rbar)/(\varpi)$ is reduced and equidimensional over $\F$ of dimension $n^2$ (by Propositions \ref{prop:N(n,q) local model} and \ref{prop:Nc(n,q) is CI}), $R^{\nun}_x/(\varpi)$ is reduced and equidimensional over $\F$ of dimension $\sum_{i=1}^k n_i^2$

Now consider the $\Oc$-algebra $\Rc^{\nun} := \Gamma(\Nc^{\nun})$ and interpret $x = \left(\varpi,X_1,N_1,\ldots,N_{i-1},X_i\right)$ as a maximal ideal of $\Rc^{\nun}$, so that $R^{\nun}_x$ is the completion of $\Rc^{\nun}$ at $x$. By definition, $\Rc^{\nun}$ is cut out by the equations $X_iN_i = N_iX_{i+1}$, all of which are homogeneous of degree $2$ in the entries of $X_i$ and $N_i$. It follows that $\Rc^{\nun}$ is a graded $\Oc$-algebra with $x$ a homogeneous ideal containing the irrelevant ideal of $\Rc^{\nun}$. This means that $R^{\nun}_x$ is in fact the completion of $\Rc^{\nun}$ as a graded $\Oc$-algebra. Hence $\Rc^{\nun}$ is also reduced, flat and equidimensional over $\Oc$ of relative dimension $\sum_{i=1}^kn_i^2$, as these properties are preserved under completions of graded $\Oc$-algebras, and $\Rc^{\nun}/(\varpi)$ is reduced and equidimensional over $\F$ of dimension $\sum_{i=1}^kn_i^2$.

Lastly, as $\Rc^{\nun}$ (respectively $\Rc^{\nun}/(\varpi)$) is generated over $\Oc$ (respectively $\F$) by $\sum_{i=1}^k n_i^2 + \sum_{i=1}^{k-1}n_in_{i+1}$ generators (the entries of $X_i$ and $N_i$) with $\sum_{i=1}^{k-1}n_in_{i+1}$ relations (the equations $X_iN_i = N_iX_{i+1}$ for $i=1,\ldots,k-1$) the fact that $\dim_{\Oc}\Rc^{\nun} = \dim_{\F}\Rc^{\nun}/(\varpi) = \sum_{i=1}^k n_i^2$ implies that $\Rc^{\nun}$, and thus $\Nc^{\nun}$, (respectively $\Rc^{\nun}/(\varpi)$ and thus $\Nc^{\nun}_{\F}$) is a complete intersection.
\end{proof}

\begin{rem}
It is likely that one could also prove this result by a similar argument to the one used in Proposition \ref{prop:Nc(n,q)_E is CI} using the analysis of the irreducible components of $\Nc^{\nun}$ given below. We have omitted doing this here, as Theorem \ref{thm:banal local model} is necessary for our main results anyway, and easily implies all of the needed properties of $\Nc^{\nun}$.
\end{rem}

\subsection{The irreducible components of $\Nc^{\nun}$}

The goal of this section is to give a simple description of the irreducible components of $\Nc^{\nun}$, and of $\Nc^{\nun}_{\Bbbk}$ for a field $\Bbbk$, analogous to Proposition \ref{prop:N(n,q) components}.

First as $\Nc^{\nun}$ is flat and finite type over $\Oc$ by Corollary \ref{cor:Nc^nun is CI}, $\Nc^{\nun} = \overline{\Nc^{\nun}(\Ebar)}$, and so the irreducible components of $\Nc^{\nun}$ are precisely the Zariski closures of the irreducible components of $\Nc^{\nun}_{\Ebar}$. Thus it will suffice to determine the irreducible components of $\Nc^{\nun}_{\Bbbk}$ for all fields $\Bbbk$.

First, for any ring $R$, write 
\[\Wc^{\nun}_R = M_{n_1\times n_2,R}\times M_{n_2\times n_3,R}\times\cdots \times M_{n_{k-1}\times n_k,R},\]
and note that this is isomorphic to $\A_R^{n_1n_2+n_2n_3+\cdots+n_{k-1}n_k}$.

We have a natural map $\Pi^{\nun}_R:\Nc^{\nun}_R\to \Wc^{\nun}_R$ given by $\Pi^{\nun}_R(X_1,N_1,X_2,N_2,\ldots,N_{k-1},X_k) = (N_1,N_2,\ldots,N_{k-1})$.

Now, consider the algebraic group $\Hc^{\nun}_R := \GL_{n_1,R}\times \GL_{n_2,R}\times\cdots \times \GL_{n_k,R}$ and define actions of $\Hc^{\nun}_R$ on the $R$-schemes $\Nc^{\nun}_R$ and $\Wc^{\nun}_R$ via:
\begin{align*}
	(g_1,\ldots,g_k)\cdot (X_1,N_1,X_2,N_2,\ldots,N_{k-1},X_k) 
	&=  (g_1X_1g_1^{-1},g_1N_1g_2^{-1},g_2X_2g_2^{-1},g_2N_2g_3^{-1},\ldots,g_{k-1}N_{k-1}g_k^{-1},g_kX_kg_k^{-1})\\
	(g_1,\ldots,g_k)\cdot (N_1,N_2,\ldots,N_{k-1}) 
	&=  (g_1N_1g_2^{-1},g_2N_2g_3^{-1},\ldots,g_{k-1}N_{k-1}g_k^{-1}).
\end{align*}
Clearly these actions make $\Pi^{\nun}_R$ $\Hc^{\nun}_R$-equivariant.

Write $\Wc^{\nun} = \Wc^{\nun}_{\Z}$, $\Pi^{\nun} = \Pi^{\nun}_{\Z}$ and $\Hc^{\nun} = \Hc^{\nun}_{\Z}$.

As it turns out, we can give a precise combinatorial description of the $\Hc^{\nun}(\Bbbk)$-orbits of in $\Wc^{\nun}(\Bbbk)$ for any field $\Bbbk$.

\begin{defn}\label{def:Is^nun}
For any $\nun = (n_1,\ldots,n_k)$ define the set $\Is^{\nun}\subseteq \Z^{k(k-1)/2}_{\ge 0}$ of $k(k-1)/2$-tuples of nonnegative integers
\[
(d_{ij})_{1\le j\le k-i \le k-1} = (d_{11},d_{12},\ldots,d_{1,k-1};d_{21},\ldots,d_{2,k-2};\ldots;d_{k-2,1},d_{k-2,2};d_{k-1,1})
\]
as follows.

First, set $d_{0j} = n_j$ for $j=1,\ldots,k$ and $d_{ij} = 0$ whenever $i\not\in [0,k-1]$ or $j\not\in [1,k-i]$. Then we will say that $(d_{ij})_{1\le j\le k-i \le k-1}\in \Is^{\nun}$ if 
\begin{itemize}
	\item $0\le d_{ij}\le \min\{d_{i-1,j},d_{i-1,j+1}\}$ for all $i\ge 1$ and all $j$.
	\item For all $i\ge 1$ and all $1\le j < k-i$,
	\[
	d_{ij}+d_{i,j+1} \le d_{i+1,j}+d_{i-1,j+1}.
	\]
\end{itemize}
\end{defn}

Note that the condition $0\le d_{ij}\le \min\{d_{i-1,j},d_{i-1,j+1}\}$ implies that $0\le d_{ij}\le \min\{n_j,\ldots,n_{i+j}\}$ for all $i$ and $j$, and so $\Is^{\nun}$ is a finite set.

\begin{prop}\label{prop:H orbits}
Let $\Bbbk$ be any field. Define a map $\Delta^{\nun}:\Wc^{\nun}(\Bbbk) \to \Z^{k(k-1)/2}$ by

\[
\Delta^{\nun}(N_1,\ldots,N_{k-1}) := \left(\rank(N_j N_{j+1}\cdots N_{i+j-1})\right)_{1\le j\le k-i \le k-1}.
\]
Then
\begin{enumerate}
	\item The image of the map $\Delta^{\nun}$ is the set $\Is^{\nun}\subseteq \Z^{k(k-1)/2}$;
	\item Any two points $(N_1,\ldots,N_{k-1}), (N'_1,\ldots,N'_{k-1})\in \Wc^{\nun}(\Bbbk) = \prod_{i=1}^{k-1} M_{n_i\times n_{i+1}}(\Bbbk)$ are in the same $\Hc^{\nun}(\Bbbk)$-orbit if and only if
	\[
	\Delta^{\nun}(N_1,\ldots,N_{k-1}) = \Delta^{\nun}(N'_1,\ldots,N'_{k-1}).
	\]
\end{enumerate}
Hence for any field $\Bbbk$, the set of $\Hc^{\nun}(\Bbbk)$-orbits in $\Wc^{\nun}(\Bbbk)$ is in bijection with the set $\Is^{\nun}$.
\end{prop}
\begin{proof}
Let $V_i = \Bbbk^{n_i}$ for all $i$, so that for any $(N_1,\ldots,N_{k-1})\in \Wc^{\nun}(\Bbbk)$, each $N_i$ may be thought of as a $\Bbbk$-linear map $V_{i+1}\to V_i$.

First, for any $(g_1,\ldots,g_k)\in \Hc^{\nun}(\Bbbk)$ and any $(N_1,\ldots,N_{k-1}) \in \Wc^{\nun}(\Bbbk)$ we have
\[
(g_jN_jg_{j+1}^{-1})(g_{j+1}N_{j+1}g_{j+2}^{-1})\cdots (g_{i+j-1}N_{i+j-1}g_{i+j}^{-1}) = g_i(N_jN_{j+1}\cdots N_{i+j-1})g_{i+j}^{-1}
\]
for all $i$ and $j$. Thus, as $g_j$ and $g_{i+j}$ are invertible, if $(N_1',\ldots,N_{k-1}') = (g_1,\ldots,g_k)\cdot (N_1,\ldots,N_{k-1})$ we indeed have
\[
\rank(N'_jN'_{j+1}\cdots N'_{i+j-1}) = \rank(g_j(N'_jN'_{j+1}\cdots N_{i+j-1}')g_{i+j}^{-1}) = 
\rank(N_jN_{j+1}\cdots N_{i+j-1})
\]
for all $i$ and $j$, so $\Delta^{\nun}(N_1,\ldots,N_{k-1}) = \Delta^{\nun}(N'_1,\ldots,N'_{k-1})$.

Now take any $(N_1,\ldots,N_{k-1})\in \Wc^{\nun}(\Bbbk)$ and set $d_{ij} = \rank(N_jN_{j+1}\cdots N_{i+j-1})$ for any $i$ and $j$. Since $\rank(AB) \le \rank(A),\rank(B)$ for any linear maps $B:U\to V$ and $A:V\to W$ between finite dimensional vector spaces we have
\[
d_{ij} = \rank(N_jN_{j+1}\cdots N_{i+j-2}N_{i+j-1}) \le \rank(N_jN_{j+1}\cdots N_{i+j-2}) = d_{i-1,j}
\] and
\[
d_{ij} = \rank(N_jN_{j+1}\cdots N_{i+j-2}N_{i+j-1}) \le \rank(N_{j+1}\cdots N_{i+j-1}) = d_{i-1,j+1}.
\]
for all $i$ and $j$.

Also the following lemma (with $A=N_j$, $B=N_{j+1}\cdots N_{i+j-1}$ and $C = N_{i+j}$ ) implies that $d_{ij}+d_{i,j+1} \le d_{i+1,j}+d_{i-1,j+1}$ for all $i\ge 1$ and $1\le j < k-i$:
\begin{lemma}
For any linear maps $C:U\to V$, $B:V\to W$ and $A:W\to X$ between finite dimensional vector spaces we have
\[
\rank(AB)+\rank(BC)\le \rank(B)+\rank(ABC)
\]
\end{lemma}
\begin{proof}
By rank-nullity
\begin{align*}
\rank(B) - \rank(AB) &= (\dim V - \dim \ker B) - (\dim V - \dim \ker AB) = \dim\ker (AB) - \dim \ker B\\
& = \dim (\ker AB)/(\ker B)
\end{align*}
and similarly 
\[
\rank(BC)-\rank(ABC) = \dim(\ker ABC)/(\ker BC).
\]
But now as $C^{-1}(\ker B) = \ker(BC)$ and $C^{-1}(\ker AB) = \ker ABC$, $C$ induces an injection $(\ker ABC)/(\ker BC)\into (\ker AB)/(\ker B)$ and so
\[
\rank(BC)-\rank(ABC) = \dim(\ker ABC)/(\ker BC) \le \dim (\ker AB)/(\ker B) = \rank(B) - \rank(AB) 
\]
giving the desired inequality.
\end{proof}
So indeed $\Delta^{\nun}(N_1,\ldots,N_{k-1})\in \Is^{\nun}$.

It remains to prove that for any $(d_{ij})_{ij}\in\Is^{\nun}$ there is some $(N_1,\ldots,N_{k-1})\in \Wc^{\nun}(\Bbbk)$ with $\Delta^{\nun}(N_1,\ldots,N_{k-1}) = (d_{ij})_{ij}$ and that if $(N_1',\ldots,N_{k-1}')\in \Wc^{\nun}(\Bbbk)$ is any other element with $\Delta^{\nun}(N'_1,\ldots,N'_{k-1}) = (d_{ij})_{ij}$ then $(N'_1,\ldots,N'_{k-1})$ is conjugate to $(N_1,\ldots,N_{k-1})$. We will prove both of these claims simultaneously by induction on $k$

In the case $k=1$ the claim is vacuous. In the case $k=2$, it is simply the standard fact that a matrix $N\in M_{a\times b}(\Bbbk)$ has rank $r$ if and only if it is conjugate to the matrix 
\[
\begin{pmatrix}
I_{r\times r}& 0_{r\times (b-r)}\\
0_{(a-r)\times r} & 0_{(a-r)\times(b-r)}
\end{pmatrix}
\]
under the action of $\GL_a(\Bbbk)\times \GL_b(\Bbbk)$.

Now assume that the claim holds for some value of $k\ge 2$. Fix any $\nun = (n_1,\ldots,n_k,n_{k+1})$ and take any $(d_{ij})\in \Is^{\nun}$. By the inductive hypothesis there exist matrices $N_i\in M_{n_i\times n_{i+1}}(\Bbbk)$ for $i=1,\ldots,k-1$ with $\rank(N_jN_{j+1}\cdots N_{i+j-1}) = d_{ij}$ whenever $i+j\le k$.

There is a filtration
\[
V_k = U_0 \supseteq U_1 \supseteq U_2 \supseteq \cdots \supseteq U_{k-1} \supseteq U_k = 0
\]
given by $U_j = \ker N_jN_{j+1}\cdots N_{k-1}$ for each $j=1,\ldots,k-1$. Note that 
\[
\dim_{\Bbbk} U_j = \dim_{\Bbbk} V_k - \rank(N_jN_{j+1}\cdots N_{k-1}) = n_k - d_{k-j,j}
\]
and note that this also trivially holds for $j=0$ and $j=k$ by taking $d_{k0} = 0$ and $d_{0k} = n_k$. Thus $\dim_{\Bbbk} U_j - \dim_{\Bbbk} U_{j+1} = d_{k-j-1,j+1}-d_{k-j,j}$ for all $j = 0,\ldots,k-1$.

Now consider any $N:V_{k+1}\to V_k$. For each $j=1,\ldots, k-1$ we have
\[
\rank(N_j\cdots N_{k-1}N) = \rank(N_j\cdots N_{k-1}|_{\im N}) 
= \rank(N) - \dim_{\Bbbk}\ker(N_j\cdots N_{k-1}|_{\im N})
= \rank(N) - \dim_{\Bbbk}(U_j\cap \im N).
\]
Thus we will have $\Delta^{\nun}(N_1,\ldots,N_{k-1},N) = (d_{ij})_{ij}$ if and only if 
\[
\dim_{\Bbbk}(U_j\cap \im N) = d_{1k} - d_{k+1-j,j}
\]
for all $j = 1,\ldots,k-1$ and $\dim_{\Bbbk} (U_0\cap \im N) = \rank(N) = d_{1k} = d_{1k} - d_{k+1,0}$. Equivalently (as $U_k\cap \im N = 0$), $\Delta^{\nun}(N_1,\ldots,N_{k-1},N) = (d_{ij})_{ij}$ if and only if
\[
\dim_{\Bbbk}(U_j\cap \im N)-\dim_{\Bbbk}(U_{j+1}\cap \im N) = d_{k-j,j+1} - d_{k+1-j,j}
\]
for all $j=0,\ldots,k-1$.

But now the definition of $\Is^{\nun}$ we have $d_{k-j,j+1} - d_{k+1-j,j}\le d_{k-j-1,j+1}-d_{k-j,j}$ for all $j=0,\ldots,k-1$. It is now easy to see that we can pick a $N_{k}:V_{k+1}\to V_k$ satisfying the desired conditions. Indeed, for each $j=0,\ldots,k-1$, pick a $W_j\subseteq U_j$ with $W_j\cap U_{j+1} = 0$ and $\dim_{\Bbbk}W_j = d_{k-j,j+1} - d_{k+1-j,j}$ (which is possible as $\dim_{\Bbbk}W_j+\dim_{\Bbbk}U_{j+1} \le \dim_{\Bbbk}U_j$). Then 
\[\dim_{\Bbbk}\sum_{a=0}^{k-1} W_a = \dim_{\Bbbk}\bigoplus_{a=0}^{k-1} W_a =  \sum_{a=0}^{k-1}\dim_{\Bbbk}W_a = \sum_{a=0}^{k-1}(d_{k-a,a+1} - d_{k+1-a,a}) = d_{1k} - d_{k+1,0} = d_{1k} \le n_{k+1} = \dim_{\Bbbk}V_{k+1}\]
and so we may define a linear map $N_{k}:V_{k+1}\to V_k$ so that $\im N_{k} = \bigoplus_{a=0}^{k-1} W_a$. We then have $U_j \cap \im N_{k} = \bigoplus_{a=0}^{j} W_a$ and so we get 
\[
\dim_{\Bbbk}(U_j\cap \im N_k)-\dim_{\Bbbk}(U_{j+1}\cap \im N_k) = \dim_{\Bbbk} W_j = d_{k-j,j+1} - d_{k+1-j,j}
\]
as desired. Hence there exists a $N_{k}$ with $\Delta^{\nun}(N_1,\ldots,N_{k-1},N_k) = (d_{ij})_{ij}$.

Now assume that $(N_1',\ldots,N_k')\in \Wc^{\nun}(\Bbbk)$ is any other element with $\Delta^{\nun}(N'_1,\ldots,N'_{k-1},N'_k) = (d_{ij})_{ij}$. We claim that $(N_1,\ldots,N_{k})$ and $(N_1',\ldots,N_{k}')$ are $\Hc(\Bbbk)$-conjugate. By the inductive hypothesis, there is some $(g_1,\ldots,g_k,1)\in \Hc^{\nun}(\Bbbk)$ with 
\[
(g_1,\ldots,g_k,1)(N_1,\ldots,N_{k-1},N_{k}) = (N'_1,\ldots,N'_{k-1},g_kN_{k})
\]
so from now on we may assume without loss of generality that $N_i = N_i'$ for all $i=1,\ldots,k-1$.

Letting $U_j$ be as before, the above work implies that
\[
\dim_{\Bbbk}(U_j\cap \im N'_k)-\dim_{\Bbbk}(U_{j+1}\cap \im N'_k) = d_{k-j,j+1} - d_{k+1-j,j} = \dim_{\Bbbk}(U_j\cap \im N_k)-\dim_{\Bbbk}(U_{j+1}\cap \im N_k)
\]
for all $j=0,\ldots,k-1$.

Now we claim by induction on $j$, that for each $j=0,\ldots,k$ $(N_1,\ldots,N_{k-1},N'_{k})$ is $\Hc^{\nun}(\Bbbk)$-conjugate to some $(N_1,\ldots,N_{k-1},N'_{k,j})$ with $U_j\cap \im N'_{k,j} = U_j\cap \im N_{k}$. Since $U_k = 0$, the claim is automatic for $j=k$.

Assume the claim holds for some $j>0$. By replacing $N'_k$ by $N'_{k,j}$ we may assume that $U_j\cap \im N_k' = U_j\cap \im N_{k}$. As noted above, we also have
\[
\dim_{\Bbbk}(U_{j-1}\cap \im N'_k)-\dim_{\Bbbk}(U_{j}\cap \im N'_k) = \dim_{\Bbbk}(U_{j-1}\cap \im N_k)-\dim_{\Bbbk}(U_{j}\cap \im N_k)
\] 
and so there exists an automorphism $g\in \GL_{U_{j-1}}$ with $g|_{U_j} = I_{U_j}$ and $U_{j-1}\cap \im gN'_k = g(U_{j-1}\cap \im N'_k) = U_{j-1} \cap \im N_k$. Extend $g$ to an endomorphism $g_k\in \GL_{n_k}(\Bbbk)$ with $g_k(U_a) = U_a$ for all $a$.

By repeatedly applying the following lemma (to the filtration 
\[0\subseteq \ker N_{a-1} \subseteq \ker (N_{a-2}N_{a-1})\subseteq \cdots \subseteq \ker (N_{1}\cdots N_{a-1}) \subseteq V_a\]
of $V_a$) we can construct an element $(g_1,\ldots,g_k,1)\in \Hc^{\nun}(\Bbbk)$ with the property that $g_a N_a = N_a g_{a+1}$ for all $a = 1,\ldots, k-1$. Hence $(N_1,N_2,\ldots,N_{k-1},N_k')$ is conjugate to
\[
(g_1N_1g_2^{-1},g_2N_2g_3^{-1},\ldots,g_{k-1}N_{k-1}g_k^{-1},g_kN_k') = (N_1,N_2,\ldots,N_{k-1},g_kN_k')
\]
and so we may set $N'_{k,j-1} = g_kN_k'$, as $U_{j-1} \cap \im g_kN_k' = U_{j-1} \cap \im N_k$.

\begin{lemma}
Let $T:V\to W$ be a linear transformation between finite dimensional $\Bbbk$-vector spaces. Take filtrations
\[
0=X_0\subseteq X_1\subseteq \cdots \subseteq X_d = V
\] 
and
\[
0=Y_0\subseteq Y_1\subseteq \cdots \subseteq Y_{d-1}= W
\]
with the property that $X_j = T^{-1}(Y_{j-1})$ for all $j>0$. Assume that $G\in \GL(V)$ satisfies $G(X_j) = X_j$ for all $j$. Then there is some $H\in \GL(W)$ with $H(Y_j) = Y_j$ for all $j$ and $HT=TG$.
\end{lemma}
\begin{proof}
First note that by assumption $X_1 = T^{-1}(Y_0) = T^{-1}(0) = \ker T$. Also the condition $X_j = T^{-1}(Y_{j-1})$ implies that $T(X_j) = \im T \cap Y_{j-1}$ for all $j>0$.

Define $H_0:\im T\to \im T$ by $H_0(T(v)) = T(G(v))$. Note that this is well defined. Indeed if $T(v_1) = T(v_2)$ then $v_1-v_2\in \ker T = X_1$ and so $G(v_1)-G(v_2)\in G(X_1) = X_1 = \ker T$ giving $H_0(T(v_1)) = T(G(v_1)) = T(G(v_2)) = H_0(T(v_2))$. Since $T$ and $G$ are linear, it follows that $H_0:\im T\to \im T$ is also linear. Similarly the function $T(v)\mapsto T(G^{-1}(v))$ is also a well defined linear map $\im T\to \im T$, and so it follows that $H_0:\im T\to \im T$ is invertible.

For any $j=0,\ldots, d$, as $G(X_j) = X_j$ we have $H_0(T(X_j)) = T(G(X_j)) = T(X_j)$. And so $H_0(\im T\cap Y_{j}) = \im T\cap Y_{j}$ for all $j=0,\ldots,d-1$. It follows that we may extend $H_0$ to an automorphism $H\in \GL(W)$ with $H(Y_j) = Y_j$ for all $j=0,\ldots,d-1$. But as $H|_{\im T} = H_0$ we have $H(T(v)) = H_0(T(v)) = T(G(v))$ for all $v\in V$ by definition, and so indeed $HT = TG$. 
\end{proof}

Thus we have shown (as $U_0 = V_k$) that $(N_1,\ldots,N_{k-1},N'_{k})$ is $\Hc^{\nun}(\Bbbk)$-conjugate to some $(N_1,\ldots,N_{k-1},N'_{k,0})$ with $\im N'_{k,0} = U_0\cap \im N'_{k,0} = U_0\cap \im N_{k} = \im N_k$. Replacing $N_k'$ by $N'_{k,0}$ we may assume without loss of generality that $\im N_k' = \im N_k$.

We now claim that there is some $g_{k+1}\in \GL_{n_{k+1}}(\Bbbk)$ with $N_k' = N_kg_{k+1}$, which will imply that $(N_1,\ldots,N_{k-1},N'_{k})$ is $\Hc^{\nun}(\Bbbk)$-conjugate to
\[
(1,\ldots,1,g_{k+1})\cdot (N_1,\ldots,N_{k-1},N'_{k}g_{k+1}^{-1}) = (N_1,\ldots,N_{k-1},N_{k})
\]
finishing the proof.

Let $r = r_{kk} = \dim_{\Bbbk} \im N_k$. Fix a basis $\{e_1,\ldots,e_r\}$ for $\im N_k = \im N'_k$. For each $i=1,\ldots,r$ pick elements $v_i,w_i\in V_{k+1}$ with $N_k(v_i) = e_i$ and $N'_k(w_i) = e_i$. Pick bases $\{v_{r+1},\ldots,v_{n_{k+1}}\}$ and $\{w_{r+1},\ldots,w_{n_{k+1}}\}$ for $\ker N_k$ and $\ker N_k'$ respectively (both of which indeed have dimension $n_{k+1}-r$). Then $\{v_1,\ldots,v_{n_{k+1}}\}$ and $\{w_1,\ldots,w_{n_{k+1}}\}$ are two bases for $V_{k+1}$ satisfying $N_k(v_i) = N'_k(w_i) = e_i$ for $i\le r$ and $N_k(v_i) = N'_k(w_i) = 0$ for $i>r$. Thus if we define $g_{k+1}\in \GL_{n_{k+1}}(\Bbbk) = \GL(V_{k+1})$ by $g(w_i) = v_i$ for all $i$ we will indeed have $N_k' = N_kg_{k+1}$.
\end{proof}

For any $P\in \Is^{\nun}$ we will use $\Wc^{\nun}_{\Bbbk}(P)\subseteq \Wc^{\nun}(\Bbbk)$ to denote the $\Hc^{\nun}(\Bbbk)$-orbit in $\Wc^{\nun}(\Bbbk)$ corresponding to $P$.

We can now use this to give a complete description of the irreducible components of $\Nc^{\nun}_{\Bbbk}$.

\begin{prop}\label{prop:Nc^nun irreducible components}
	Let $\Bbbk$ be a field. Let $C\subseteq \Wc^{\nun}_{\Bbbk}(\Bbbkbar)$ be an $\Hc^{\nun}_{\Bbbk}(\Bbbkbar)$-orbit under the action of $\Hc^{\nun}_{\Bbbk}(\Bbbkbar)$ on $\Wc^{\nun}_{\Bbbk}(\Bbbkbar)$. Then the Zariski closure of $(\Pi^{\nun}_{\Bbbk})^{-1}(C)\subseteq \Nc^{\nun}_{\Bbbk}(\Bbbkbar)$ in $\Nc^{\nun}_{\Bbbk}$ is an irreducible component of $\Nc^{\nun}_{\Bbbk}$. Moreover every irreducible component of $\Nc^{\nun}_{\Bbbk}$ is in this form.
\end{prop}
\begin{proof}
	For any $P = (N_1,N_2,\ldots,N_{k-1})\subseteq \Wc^{\nun}_{\Bbbk}(\Bbbkbar)$, consider the subvariety 
	\[\Vc_P := \left\{(X_1,\ldots,X_k)\in \GL_{n_1}(\Bbbkbar)\times \cdots \times \GL_{n_k}(\Bbbkbar)\middle|X_iN_i = N_iX_{i+1} \text{ for }i=1,\ldots,k-1 \right\}\]
	of $M_{n_1\times n_1}(\Bbbkbar)\times\cdots\times M_{n_k\times n_k}(\Bbbkbar) \cong \A_{\Bbbkbar}^{n_1^2+\cdots+n_k^2}$.
	
	Clearly $\Vc_P$ is isomorphic to the fiber $(\Pi^{\nun}_{\Bbbkbar})^{-1}(P)\subseteq\Nc^{\nun}(\Bbbkbar)$. Also as the equations $X_iN_i = N_iX_{i+1}$ are all linear in the entries of the $X_i$'s, it follows that $(\Pi^{\nun}_{\Bbbk})^{-1}(P)\cong \Vc_P \cong \A_{\Bbbkbar}^{d(P)}$ for some integer $d(P)$ depending on $P$. In other words, the fibers of the map $\Pi^{\nun}_{\Bbbk}$ on $\Bbbkbar$-points are all affine spaces of varying degrees.
	
	Now let $C\subseteq \Wc^{\nun}_{\Bbbk}(\Bbbkbar)$ be an $\Hc^{\nun}_{\Bbbk}(\Bbbkbar)$-orbit. As $\Pi^{\nun}_{\Bbbk}$ is $\Hc^{\nun}_{\Bbbk}$-equivariant, the fibers $(\Pi^{\nun}_{\Bbbk})^{-1}(P)$ are isomorphic for all $P\in C$. It follows that $(\Pi^{\nun}_{\Bbbk})^{-1}(C)$ is the total space of a vector bundle of rank $d(P)$ over $C$, for any $P\in C$. In particular
	\[
	\dim\overline{(\Pi^{\nun}_{\Bbbk})^{-1}(C)} = \dim {(\Pi^{\nun}_{\Bbbk})^{-1}(C)} = \dim_{\Bbbkbar}C + d(P)
	\]
	
	Now as $\Hc^{\nun}_{\Bbbk}(\Bbbkbar)$ is irreducible, $C$ is also irreducible. Hence $(\Pi^{\nun}_{\Bbbk})^{-1}(C)$ is irreducible, so its Zariski closure $\overline{(\Pi^{\nun}_{\Bbbk})^{-1}(C)} \subseteq \Nc^{\nun}_{\Bbbk}$ is irreducible as well.
	
	By the orbit-stabilizer theorem, $\dim_{\Bbbkbar} C = \dim_{\Bbbkbar}\Hc^{\nun}(\Bbbkbar) - \stab_{\Hc^{\nun}(\Bbbkbar)}(P)$ for any $P\in C$. Now consider the natural embedding $\Hc^{\nun}(\Bbbkbar) = \prod_{i=1}^k \GL_{n_i}(\Bbbkbar)\into \prod_{i=1}^k M_{n_i\times n_i}(\Bbbkbar)$. From the definition of the action of $\Hc^{\nun}(\Bbbkbar)$ on $\Wc^{\nun}(\Bbbkbar)$ we see that
	\[
	\stab_{\Hc^{\nun}(\Bbbkbar)}(P) = \Hc^{\nun}(\Bbbkbar)\cap \Vc_P\subseteq \prod_{i=1}^k M_{n_i\times n_i}(\Bbbkbar).
	\]
	Since $\Hc^{\nun}(\Bbbkbar)$ is Zariski open in $\prod_{i=1}^k M_{n_i\times n_i}(\Bbbkbar)$ and $\Hc^{\nun}(\Bbbkbar)\cap \Vc_P\ne \es$ (since $(I_{n_1},\ldots,I_{n_k})\in \Hc^{\nun}(\Bbbkbar)\cap \Vc_P$), $\Hc^{\nun}(\Bbbkbar)\cap \Vc_P$ is a Zariski open subset of $\Vc_P\cong \A_{\Bbbkbar}^{d(P)}$, and so $\dim_{\Bbbkbar}\stab_{\Hc^{\nun}(\Bbbkbar)}(P) = d(P)$.
	
	Combining this, we see that
	\[
	\dim_{\Oc}\overline{(\Pi^{\nun})^{-1}(C)} = \dim_{\Bbbkbar}C + \dim_{\Bbbkbar}\stab_{\Hc^{\nun}(\Bbbkbar)}(P) = \dim_{\Bbbkbar}\Hc^{\nun}(\Bbbkbar) = \sum_{i=1}^kn_i^2.
	\]
	Since $\overline{(\Pi^{\nun})^{-1}(C)}$ is closed and irreducible, and $\Nc^{\nun}$ is equidimensional of relative dimension $\sum_{i=1}^kn_i^2$ over $\Oc$ (by Corollary \ref{cor:Nc^nun is CI}), this implies that $\overline{(\Pi^{\nun})^{-1}(C)}$ is indeed an irreducible component of $\Nc^{\nun}$.
	
	Also note that this implies that $\Hc^{\nun}(\Bbbkbar)$ has finitely many orbits in $\Wc^{\nun}(\Bbbkbar)$, since $\Nc^{\nun}$ is finite type over $\Oc$, and hence has only finitely many irreducible components.
	
	Now as $\Nc^{\nun}$ is flat and finite type over $\Oc$, $\Nc^{\nun} = \overline{\Nc^{\nun}(\Bbbkbar)}$. Now if $S = \Hc^{\nun}(\Bbbkbar)\backslash \Wc^{\nun}(\Bbbkbar)$ is the set of $\Hc^{\nun}(\Bbbkbar)$-orbits in $\Wc^{\nun}(\Bbbkbar)$ we have
	\begin{align*}
		\Nc^{\nun} &= \overline{\Nc^{\nun}(\Bbbkbar)} 
		= \overline{(\Pi^{\nun})^{-1}(\Wc^{\nun}(\Bbbkbar))}
		= \overline{(\Pi^{\nun})^{-1}\left(\bigsqcup_{C\in S}C\right)}
		= \overline{\bigsqcup_{C\in S}(\Pi^{\nun})^{-1}\left(C\right)}
		= \bigsqcup_{C\in S}\overline{(\Pi^{\nun})^{-1}\left(C\right)}.
	\end{align*}
	So every irreducible component of $\Nc^{\nun}$ is indeed in the form $\overline{(\Pi^{\nun})^{-1}\left(C\right)}$.
\end{proof}

Propositions \ref{prop:Nc^nun irreducible components} and \ref{prop:H orbits} now give a natural bijection between the irreducible components of $\Nc^{\nun}$ and the set $\Is^{\nun}$. For any $P\in \Is^{\nun}$, let $\Wc^{\nun}(P)\subseteq \Wc^{\nun}(\Ebar)$ be the corresponding $\Hc^{\nun}(\Ebar)$-orbit and let
\[\Nc^{\nun}(P) = \overline{(\Pi^{\nun})^{-1}\left(\Wc^{\nun}(P)\right)}\subseteq \Nc^{\nun}\]
be the corresponding irreducible component.

\subsection{Irreducible components with unipotent types}

Theorem \ref{thm:banal local model} allows us to compute the framed local deformation ring $R^{\square}(\rbar)$ for a unipotent Galois representation $\rbar$ at a banal prime $q$ in terms of the spaces $\Nc^{\nun}$ for various $\nun$. The goal of this section will be to give a corresponding description of the fixed type deformation ring $R^\square(\rbar,\st_{\mun})$ for a unipotent type $\st_{\mun}$. Note that by Proposition \ref{prop:Nc(n,q) is CI}(3), the fact that $\rbar$ is unipotent and $q$ is banal implies that all lifts of $\rbar$ are also unipotent, and so $R^\square(\rbar,\tau) = 0$ whenever $\tau$ is a non-unipotent type.

Let $\mun = (m_1,\ldots,m_c)$ be any partition of $n$. Recall that $R^\square(\rbar,\st_{\mun})$ is by definition the Zariski closure of the set of all $x\in R^\square(\rbar,\st_{\mun})(\Ebar)$ such that $\WD(r_x)$ has inertial type $\tau_{\mun}$. Since each $r_x$ is unipotent, this is equivalent to saying that $\log r_x(\sigma)$ is conjugate to the matrix $N_{\mun}$ in $M_{n\times n}(\Ebar)$. Now recall that two nilpotent matrices $A,B\in M_{n\times n}(\Ebar)$ are conjugate in $M_{n\times n}(\Ebar)$ if and only if $\rank(A^i) = \rank(B^i)$ for all $i\ge 0$. Thus $\WD(r_x)$ will have inertial type $\tau_{\mun}$ if and only if
\begin{align*}
\rank (\log r_x(\sigma))^i
&= \rank N_{\mun}^i 
 = \sum_{a=1}^c \rank N_{m_a}^i
 = \sum_{a=1}^c \max\{m_a-i,0\}
\end{align*}
for all $i\ge 0$. In light of this we make the following definition:

\begin{defn}
Take any $\nun = (n_1,\ldots,n_k)$ and let $n = \sum_in_i$. Let $\mun = (m_1,\ldots,m_c)$ be any (unordered) partition of $n$. We say that an element $P = (d_{ij})_{ij} \in \Is^{\nun}$ has \emph{type $\st_{\mun}$} if 
\[
\sum_{j=1}^{k-i} d_{ij} = \sum_{a=1}^c \max\{m_a-i,0\}
\]
for all $i\ge 0$, where we take $d_{0j} = n_j$ for $j=1,\ldots,k$, and we interpret $\sum_{j=1}^{k-i} d_{ij} = 0$ for $i\ge k$.

Let $\Is^{\nun}(\st_{\mun})\subseteq \Is^{\nun}$ be the set of all $P \in \Is^{\nun}$ with type $\st_{\mun}$. Define
\[
\Wc^{\nun}(\st_{\mun})= \bigcup_{P\in \Is^{\nun}(\st_{\mun})}\Wc^{\nun}(P)\subseteq 
\Wc^{\nun}(\Ebar)
\]
and
\[
\Nc^{\nun}(\st_{\mun}) := \bigcup_{P\in \Is^{\nun}(\st_{\mun})}\Nc^{\nun}(P) = \overline{\left(\Pi^{\nun})^{-1}\left(\Wc^{\nun}(\Ebar)\right)\right)}\subseteq 
\Nc^{\nun}
\]
treated as a subscheme of $\Nc^{\nun}$.
\end{defn}

This definition is justified by the following lemma:

\begin{lemma}
Take any $\nun = (n_1,\ldots,n_k)$ and let $n = \sum_in_i$. Let $\mun = (m_1,\ldots,m_c)$ be any (unordered) partition of $n$. Take any $(N_1,\ldots,N_{k-1})\in \Wc^{\nun}(\Ebar)$. Then the nilpotent matrix
\[
\eta(N_1,\ldots,N_{k-1}):=
\begin{pmatrix}
0&N_1&0&\cdots &0\\
0&0&N_2&\cdots &0\\
\vdots&\vdots&\vdots&\ddots&\vdots\\
0&0&0&\cdots& N_{k-1}\\
0&0&0&\cdots& 0
\end{pmatrix}
\in M_{n\times n}(\Ebar)
\]
is conjugate to $N_{\mun}$ in $M_{n\times n}(\Ebar)$ if any only if $(N_1,\ldots,N_{k-1})\in \Wc^{\nun}(\st_{\mun})$.

Moreover we have $\Is^{\nun} = \bigsqcup_{\mun} \Is^{\nun}(\st_{\mun})$ and $\Nc^{\nun} = \bigcup_{\mun}\Nc^{\nun}(\st_{\mun})$, where the unions run over all unordered partitions $\mun$ of $n$.
\end{lemma}
\begin{proof}
This follows from the above discussion and definitions after noting that
\[
\eta(N_1,\ldots,N_{k-1})^i =
\begin{pmatrix}
	0&\cdots &N_1\cdots N_{i}&\cdots &0\\
	\vdots&\vdots&\vdots&\ddots&\vdots\\
	0&\cdots&0&\cdots & N_{k-i}\cdots N_{k-1}\\
	0&\cdots&0&\cdots&0\\
	\vdots&\vdots&\vdots&\vdots&\vdots\\	
	0&\cdots&0&\cdots&0
\end{pmatrix}
\]
and so
\[\rank \eta(N_1,\ldots,N_{k-1})^i = \sum_{j=1}^{k-i} \rank(N_j\cdots N_{i+j-1})\]
for all $i=1,\ldots,k-1$ (and $\rank\eta(N_1,\ldots,N_{k-1})^i = 0$ for $i\ge k$).

The final statement follows by noting that each $\eta(N_1,\ldots,N_{k-1})$ is nilpotent, and thus is conjugate to $N_{\mun}$ for \emph{exactly} one unordered partition $\mun$ of $n$, and so the sets $\Wc^{\nun}(\st_{\mun})$ form a partition of $\Wc^{\nun}(\Ebar)$.
\end{proof}

Now fix $\nun^{(\alpha)} = (n^{(\alpha)}_1,\ldots,n^{(\alpha)}_{k_\alpha})$ for $\alpha=1,\ldots,d$. Let $n^{(\alpha)} = \sum_i n^{(\alpha)}_i$ and let $n = \sum_\alpha n^{(\alpha)}$. Define the $\Oc$-scheme
\[
\Ns^{\nun^{(1)};\ldots;\nun^{(d)}} := \Nc^{\nun^{(1)}}\times_{\Oc}\cdots \times_{\Oc} \Nc^{\nun^{(d)}}.
\]
Note that this is reduced, flat and equidimensional over $\Oc$ of relative dimension $\sum_\alpha\sum_i(n^{(\alpha)}_i)^2$. Moreover each irreducible component of $\Ns^{\nun^{(1)};\ldots;\nun^{(d)}}$ is in the form
\[
\Nc^{\nun^{(1)}}(P^{(1)})\times_{\Oc}\cdots \times_{\Oc} \Nc^{\nun^{(d)}}(P^{(d)})
\]
for some $(P^{(1)},\ldots,P^{(d)})\in \Is^{\nun^{(1)}}\times \cdots \times \Is^{\nun^{(d)}}$.

For any collection of unordered partitions $\mun^{(\alpha)} = (m_1^{(\alpha)},\ldots,m_{c_\alpha}^{(\alpha)})$ of $n^{(\alpha)}$, define
\[
\bigcup_\alpha \mun^{(\alpha)} = (m_1^{(1)},\ldots,m_{c_1}^{(1)},\ldots,m_1^{(d)},\ldots,m_{c_d}^{(d)})
\]
treated as an unordered partition of $n$.

For any unordered partition $\mun$ of $n$, define
\[
\Ns^{\nun^{(1)};\ldots;\nun^{(d)}}(\st_{\mun}) = \bigcup_{\bigcup_\alpha \mun^{(\alpha)} = \mun}\Nc^{\nun^{(1)}}(\st_{\mun^{(1)}})\times_{\Oc}\cdots \times_{\Oc} \Nc^{\nun^{(d)}}(\st_{\mun^{(d)}})
\]
and note that this is a union of irreducible components of $\Ns^{\nun^{(1)},\ldots,\nun^{(d)}}$.

We can now prove the main result of this section:

\begin{thm}\label{thm:fixed type local model}
Let $q$ be banal and let $\rbar:G_L\to \GL_n(\F)$ be a Galois representation with $\rbar(\Pt_L) = 1$ and $\rbar(\sigma)$ unipotent. Let $\nun^{(1)},\ldots,\nun^{(d)}$ and
\[
x^{(\alpha)} = \left(\Ubar_1^{(\alpha)}-I_{n_1^{(\alpha)}},\Nbar^{(\alpha\alpha)}_{12},\Ubar_2^{(\alpha)}-I_{n_2^{(\alpha)}},\ldots,\Nbar^{(\alpha\alpha)}_{k_\alpha-1,k_\alpha},\Ubar_{k_\alpha}^{(\alpha)}-I_{n_{k_\alpha}}^{(\alpha)}\right) \in \Nc^{\nun^{(\alpha)}}(\F)
\]
be as in the statement of Theorem \ref{thm:banal local model}. Let 
\[
x = (x^{(1)},\ldots,x^{(d)}) \in \Ns^{\nun^{(1)};\ldots;\nun^{(d)}}(\F).
\]
For any unordered partition $\mun$ of $n$ we have an isomorphism
\[
R^{\square}(\rbar,\st_{\mun})\cong R^{\nun^{(1)};\ldots;\nun^{(d)}}_{x}(\st_{\mun})\left[\left[t_1,\ldots,t_{n^2-\sum_\alpha\sum_i \left(n_i^{(\alpha)}\right)^2}\right]\right].
\]
where $R^{\nun^{(1)};\ldots;\nun^{(d)}}_{x}(\st_{\mun})$ is the complete local ring of the scheme $\Ns^{\nun^{(1)};\ldots;\nun^{(d)}}(\st_{\mun})$ at the $\F$-point $x$ (taken to be $0$ if $x\not\in \Ns^{\nun^{(1)};\ldots;\nun^{(d)}}(\st_{\mun})(\F)$).
\end{thm}
\begin{proof}
Let $R^{\nun^{(1)};\ldots;\nun^{(d)}}_x$ be the complete local ring of $\Ns^{\nun^{(1)};\ldots;\nun^{(d)}}$ at $x$, so that $R^{\nun^{(1)};\ldots;\nun^{(d)}}_x(\st_{\mun})$ is a quotient of $R^{\nun^{(1)};\ldots;\nun^{(d)}}_x$. Let $\Ns^{\nun^{(1)};\ldots;\nun^{(d)}}_x =\Spec R^{\nun^{(1)};\ldots;\nun^{(d)}}_x$ and $\Ns^{\nun^{(1)};\ldots;\nun^{(d)}}_x(\st_{\mun}) = \Spec R^{\nun^{(1)};\ldots;\nun^{(d)}}_x(\st_{\mun})$. By Theorem \ref{thm:banal local model}, $R^{\square}(\rbar)\cong R^{\nun^{(1)};\ldots;\nun^{(d)}}_x\left[\left[t_1,\ldots,t_{n^2-\sum_\alpha\sum_i \left(n_i^{(\alpha)}\right)^2}\right]\right]$ and so there is an isomorphism
\[
\psi:(\Spec R^\square(\rbar))(\Ebar) \isomto \Ns^{\nun^{(1)};\ldots;\nun^{(d)}}_x(\Ebar)\times (\Spec \Oc[[t]])(\Ebar)^{n^2-\sum_\alpha\sum_i \left(n_i^{(\alpha)}\right)^2}.
\]
For each $\alpha$ and each $y^{(\alpha)} = (X_1^{(\alpha)},N_1^{(\alpha)},\ldots,N_{k_\alpha-1}^{(\alpha)},X_{k_\alpha}^{(\alpha)}) \in \Nc^{\nun^{(\alpha)}}(\Ebar)$ set
\[\eta^{(\alpha)}(y^{(\alpha)})
=
\begin{pmatrix}
	0&N^{(\alpha)}_1&0&\cdots &0\\
	0&0&N^{(\alpha)}_2&\cdots &0\\
	\vdots&\vdots&\vdots&\ddots&\vdots\\
	0&0&0&\cdots& N^{(\alpha)}_{k_\alpha-1}\\
	0&0&0&\cdots& 0
\end{pmatrix}
\in M_{n^{(\alpha)}\times n^{(\alpha)}}(\Ebar)\]
and set
\[
\eta(y^{(1)},y^{(2)},\ldots,y^{(d)}) = 
\begin{pmatrix}
	\eta^{(1)}(y^{(1)})&0&\cdots &0\\
	0&\eta^{(2)}(y^{(2)})&\cdots &0\\
	\vdots&\vdots&\ddots&\vdots\\
	0&0&\cdots& \eta^{(d)}(y^{(d)})\\
\end{pmatrix}
\in M_{n\times n}(\Ebar)
\]
for all $(y^{(1)},y^{(2)},\ldots,y^{(d)})\in \Ns^{\nun^{(1)};\ldots;\nun^{(d)}}(\Ebar)$.

Then note that $\Ns^{\nun^{(1)};\ldots;\nun^{(d)}}(\st_{\mun})$ is the Zariski closure of the set of points $y\in \Ns^{\nun^{(1)};\ldots;\nun^{(d)}}(\Ebar)$ with $\eta(y)$ conjugate to $N_{\mun}$ in $M_{n\times n}(\Ebar)$.

As $\Ns^{\nun^{(1)};\ldots;\nun^{(d)}}$ and $\Ns^{\nun^{(1)};\ldots;\nun^{(d)}}(\st_{\mun})$ are finite type $\Oc$-schemes it follows that $\Ns^{\nun^{(1)};\ldots;\nun^{(d)}}_x(\st_{\mun})$ is the Zariski closure of the set of points $y\in \Ns^{\nun^{(1)};\ldots;\nun^{(d)}}_x(\Ebar)$ with $\eta(y)$ conjugate to $N_{\mun}$.

Now take any $z\in R^\square(\rbar)(\Ebar)$ and note that the Weil--Deligne representation $\WD(r_z)$ will have inertial type $\tau_{\mun}$ if and only if $\log r_z(\sigma)$ is conjugate to $N_{\mun}$ in $M_{n\times n}(\Ebar)$.

Now take $\psi(z)=(y,b)\in  \Ns^{\nun^{(1)};\ldots;\nun^{(d)}}_x(\Ebar)\times (\Spec \Oc[[t]])(\Ebar)^{n^2-\sum_\alpha\sum_i \left(n_i^{(\alpha)}\right)^2}$. Note that the isomorphism constructed in Theorem \ref{thm:banal local model} implies that $\log r_z(\sigma)$ is conjugate to $\eta(y)$ in $M_{n\times n}(\Ebar)$, and thus $\log r_z(\sigma)$ will be conjugate to $N_{\mun}$ if and only if $\eta(y)$ is.

Thus the isomorphism $\psi$ sends the set of points $z$ for which $\WD(r_z)$ has inertial type $\tau_{\mun}$ to the set of points $(y,b)$ where $\eta(y)$ is conjugate to $N_{\mun}$. Taking the Zariski closures of these respective sets of points thus gives the desired isomorphism
\[
R^{\square}(\rbar,\st_{\mun})\cong R^{\nun^{(1)};\ldots;\nun^{(d)}}_{x}(\st_{\mun})\left[\left[t_1,\ldots,t_{n^2-\sum_\alpha\sum_i \left(n_i^{(\alpha)}\right)^2}\right]\right].
\]
\end{proof}

\begin{rem}
In general $\Is^{\nun}(\st_{\mun})$ may have more than one element (see for example Proposition \ref{prop:Nc^1,...,1}) which means that $\Nc^{\nun}(\st_{\mun})$, and hence $\Ns^{\nun^{(1)};\ldots;\nun^{(d)}}(\st_{\mun})$, will not always be irreducible. On the other hand, by Proposition \ref{prop:N(n,q) components} (and Proposition \ref{prop:Nc(n,q) is CI}(1)) the space $\Nc(n,q)$ has a single irreducible component, which we can denote $\Nc(n,q)(\st_{\mun})$, corresponding to Weil--Deligne representations of inertial type $\tau_{\mun}$. This does not contradict Theorem \ref{thm:fixed type local model} and Proposition \ref{prop:N(n,q) local model}. Rather it simply shows that if $\rbar$ is chosen so that the space $\Ns^{\nun^{(1)};\ldots;\nun^{(d)}}(\st_{\mun})$ from Theorem \ref{thm:fixed type local model} has multiple components, then the formal completion of $\Nc(n,q)(\st_{\mun})$ at the $\F$-point corresponding to $\rbar$ must also have multiple components.
\end{rem}

\subsection{Special cases}
In this section we will study the spaces $\Nc^{\nun}$ and $\Nc^{\nun}(\st_{\mun})$ in the special cases that will be needed for our main results.

First, note the following simple consequences of our definitions:

\begin{lemma}\label{lem:Nc^nun(triv)}
For any $\nun = (n_1,\ldots,n_k)$,
$\Nc^{\nun}(\st_{1,1,\ldots,1}) \cong \A_{\Oc}^{n_1^2+\cdots+n_k^2}$.
\end{lemma}
\begin{proof}
For any $(d_{ij})_{ij}\in \Is^{\nun}$, by definition we will have $(d_{ij})_{ij}\in \Is^{\nun}(\st_{1,1,\ldots,1})$ if and only if $\sum_jd_{ij} = 0$ for all $i\ge 1$, which is equivalent to saying that $d_{ij} = 0$ for all $i\ge 1$ and all $j$. Thus $\Is^{\nun}(\st_{1,1,\ldots,1})$ is the singleton set $\left\{(0)_{ij}\right\}$. It follows that
\[
\Nc^{\nun}(\st_{1,1,\ldots,1}) = \overline{(\Pi^{\nun})^{-1}(0)} = (\Pi^{\nun})^{-1}(0).
\]
But $(\Pi^{\nun})^{-1}(0)$ is simply the subscheme of $\Nc^{\nun}$ cut out by the equations $N_1 = N_2 = \cdots = N_{k-1} = 0$, which is obviously isomorphic to $\A^{n_1^2+\cdots n_k^2}_{\Oc}$ (since the equations $N_1 = N_2 = \cdots = N_{k-1} = 0$ already imply that $X_i N_i = N_i X_{i+1}$ for all $i$).
\end{proof}

\begin{lemma}\label{lem:Nc^n}
For any $n\ge 1$, $\Nc^n = \Nc^n(\st_{1,1,\ldots,1}) \cong \A^{n^2}_{\Oc}$.
\end{lemma}
\begin{proof}
By definition, $\Nc^n$ is simply the moduli space of matrices $X_1\in M_{n\times n}$, subject to no equations, so the claim follows trivially.
\end{proof}

Next, it is possible to completely describe the space $\Nc^{1,1,\ldots,1}$.

\begin{prop}\label{prop:Nc^1,...,1}
Take any $n\ge 1$ and consider the space $\Nc^{1,1,\ldots,1}$ corresponding to the ordered partition $(1,1,\ldots,1)$ of $n$. Then we have an isomorphism
\[
\Nc^{1,1,\ldots,1} \cong \Spec\frac{\Oc[y_1,z_1,y_2,z_2,\ldots,y_{n-1},z_{n-1},t]}{(y_1z_1,y_2z_2,\ldots,y_{n-1}z_{n-1})}.
\]
Hence $\Nc^{1,1,\ldots,1}$ has $2^{n-1}$ irreducible components, each isomorphic to $\A^n_{\Oc}$.

For any partition $\mun = (m_1,\ldots,m_c)$ and any $i=1,\ldots,n$ let $c_i = \#\{a|m_a = i\}$. Then $\Nc^{1,1,\ldots,1}(\st_{\mun})$ is a union of
\[
|\Is^{\nun}(\st_{\mun})| = \frac{c!}{c_1!c_2!\cdots c_n!}
\]
irreducible components of $\Nc^{1,1,\ldots,1}$.

In particular, for any $d|n$ we have $\Nc^{1,1,\ldots,1}(\st_{d,d,\ldots,d})\cong \A^n_{\Oc}$.
\end{prop}
\begin{proof}
Write $\Nc^{1,1,\ldots,1} = \Spec \Rc^{1,1,\ldots,1}$. Then by definition (writing $X_i = (x_i)$ and $N_i = (z_i)$),
\[
\Rc^{1,1,\ldots,1} \cong \frac{\Oc[x_1,z_1,x_2,z_2,\ldots,z_{n-1},x_{n}]}{(x_1z_1-x_2z_1,x_2z_2-x_3z_2,\ldots,x_{n-1}z_{n-1}-x_nz_{n-1})}.
\]
The variable change $y_i = x_i-x_{i+1}$ for $i=1,\ldots,n-1$ and $t = x_n$ then gives the desired isomorphism. It is clear from this isomorphism that $\Nc^{1,1,\ldots,1}$ indeed has $2^{n-1}$ irreducible components each isomorphic to $\A^n_{\Oc}$.

Now the action of $\Hc^{1,1,\ldots,1}(\Ebar)\cong (\Ebar^\times)^n$ on $\Wc^{1,1,\ldots,1}(\Ebar)\cong \Ebar^{n-1}$ is given by
\[
(g_1,\ldots,g_n)\cdot (z_1,\ldots,z_{n-1}) = (g_1g_2^{-1}z_1,\ldots,g_{n-1}g_n^{-1}z_{n-1}).
\]
It is easy to see that $(z_1,\ldots,z_{n-1}),(z_1',\ldots,z_{n-1}')\in \Wc^{1,1,\ldots,1}(\Ebar)$ are in the same $\Hc^{1,1,\ldots,1}(\Ebar)$-orbit if and only if $\{i|z_i = 0\} = \{i|z'_i = 0\}$. Thus there are indeed $2^{n-1}$ $\Hc^{1,1,\ldots,1}(\Ebar)$-orbits in $\Wc^{1,1,\ldots,1}(\Ebar)$, generated by the elements of the set $\{0,1\}^{n-1}\subseteq \Wc^{1,1,\ldots,1}(\Ebar)$. For any $Z = (z_1,\ldots,z_{n-1})\in \{0,1\}^{n-1}$, let 
\[
\eta(Z) :=
\begin{pmatrix}
	0&z_1&0&\cdots &0\\
	0&0&z_2&\cdots &0\\
	\vdots&\vdots&\vdots&\ddots&\vdots\\
	0&0&0&\cdots& z_{k-1}\\
	0&0&0&\cdots& 0
\end{pmatrix}\in M_{n\times n}(\Ebar).
\]
Then as $z_i\in \{0,1\}$ for each $i$, $\eta(Z)$ is already in Jordan canonical form. Thus $\eta(Z)$ will be conjugate to $N_{\mun}$ if and only if $\eta(Z)$ is a permutation of the Jordan blocks $N_{m_1},\ldots,N_{m_c}$ of $N_{\mun}$.

Then $|\Is^{1,1,\ldots,1}(\st_{\mun})|$ is equal to the number of nontrivial permutations of the $c$-tuple $(m_1,\ldots,m_c)$, which is indeed equal to $\frac{c!}{c_1!c_2!\cdots c_n!}$, and so $\Nc^{1,1,\ldots,1}(\st_{\mun})$ is indeed equal to a union of $\frac{c!}{c_1!c_2!\cdots c_n!}$ irreducible components of $\Nc^{1,1,\ldots,1}$.

For the final claim, if $d|n$ and $m = (d,d,\ldots,d)$, then we have $c = c_d = n/d$ and $c_i = 0$ for $i\ne n/d$. Hence 
\[
|\Is^{1,1,\ldots,1}(\st_{d,d,\ldots,d})| = \frac{(n/d)!}{0!\cdots 0!(n/d)!0!\cdots 0!} = 1.
\]
So $\Nc^{1,1,\ldots,1}(\st_{d,d,\ldots,d})$ is an irreducible component of $\Nc^{1,1,\ldots,1}$, which is thus isomorphic to $\A^n_{\Oc}$ by the above.
\end{proof}

\begin{cor}\label{cor:Nc(st_n)}
For any ordered partition $\nun = (n_1,\ldots,n_k)$ of $n\ge 1$,
$
\Nc^{\nun}(\st_n) \cong
\begin{cases}
\A^n_{\Oc} & \text{ if }\nun = (1,1,\ldots,1)\\
\es & \text{otherwise}
\end{cases}
$.
\end{cor}
\begin{proof}
The case $\nun = (1,1,\ldots,1)$ follows from Proposition \ref{prop:Nc^1,...,1}. So now assume that $\nun \ne (1,1,\ldots,1)$. This in particular implies that $k<n$. Now if $(d_{ij})_{ij} \in \Is^{\nun}(\st_n)$. Then by definition
\[
\sum_{j=1}^{k-(n-1)}d_{ij} = \max\{n-(n-1),0\} = 1,
\]
which is impossible if $k\le n-1$. Thus $\Is^{\nun}(\st_n) = \es$ and so $\Nc^{\nun}(\st_n) = \es$.
\end{proof}

We have thus completely determined $\Nc^{\nun}(\st_{1,1,\ldots,1})$ and $\Nc^{\nun}(\st_n)$ for all $n$ and all $\nun$. To study banal deformation rings at ``spherical level'' for $n=4$ it remains to study the spaces $\Nc^{\nun}(\st_{2,2})$.

\begin{prop}\label{prop:st(2,2) cases}
Let $\nun = (n_1,\ldots,n_k)$ be an ordered partition of $n=4$. Then $\Nc^{\nun}(\st_{2,2}) = \es$ unless $\nun = (2,2),(1,2,1)$ or $(1,1,1,1)$. In each of these cases, $\Nc^{\nun}(\st_{2,2})$ is a single irreducible component of $\Nc^{\nun}$. If $\nun=(2,2)$ then $\Nc^{2,2}(\st_{2,2}) = \Nc^{2,2}(2)$, which is irreducible of relative dimension $8$ over $\Oc$. If $\nun=(1,2,1)$ then $\Nc^{1,2,1}(\st_{2,2}) = \Nc^{1,2,1}(1,1;0)$, which is irreducible of relative dimension $6$ over $\Oc$. If $\nun = (1,1,1,1)$, then $\Nc^{1,1,1,1}(\st_{2,2}) \cong \A^4_{\Oc}$.
\end{prop}
\begin{proof}
By definition for $(d_{ij})_{ij}\in \Is^{\nun}$ we have $(d_{ij})_{ij}\in \Is^{\nun}(\st_{2,2})$ if and only if $d_{11}+\cdots+d_{1,k-1} = 2$ and $d_{ij} = 0$ for $i\ge 2$.

If $k=1$ then $\nun = (4)$ and $\Nc^4(\st_{2,2}) = \es$ by Lemma \ref{lem:Nc^n}.

If $k=4$ then $\nun = (1,1,1,1)$, and $\Nc^{1,1,1,1}(\st_{2,2}) = \A^4_{\Oc}$ by Proposition \ref{prop:Nc^1,...,1}, which is indeed an irreducible component of $\Nc^{1,1,\ldots,1}$.

Now let $k=2$ so $\nun = (n_1,n_2) = (1,3),(2,2)$ or $(3,1)$. Then $\Ic^{n_1,n_2} \subseteq \Z^{2(2-1)/2}_{\ge 0} = \Z_{\ge 0}$ and we have $d_{11}\in \Ic^{n_1,n_2}$ if and only if $0\le d_{11}\le \min\{n_1,n_2\}$. But by the above, $d_{11}\in \Ic^{n_1,n_2}(\st_{2,2})$ if and only if $d_{11} = 2$. Hence $\Ic^{n_1,n_2}(\st_{2,2}) = \es$, and thus $\Nc^{n_1,n_2}(\st_{2,2}) = \es$, unless $\min\{n_1,n_2\}\ge 2$, which happens if and only if $(n_1,n_2) = (2,2)$. If $(n_1,n_2)=(2,2)$ then we have $\Ic^{n_1,n_2}(\st_{2,2}) = \{2\}$, so $\Nc^{2,2}(\st_{2,2}) = \Nc^{2,2}(2)$ is an irreducible component of $\Nc^{2,2}$.

Now finally let $k=3$, so $\nun = (n_1,n_2,n_3) = (1,1,2),(1,2,1)$ or $(2,1,1)$. Then $\Ic^{n_1,n_2,n_3}\subseteq \Z^{3(3-1)/2}_{\ge 0} = \Z_{\ge 0}^3$ is by definition the set of $3$-tuples of integers $(d_{11},d_{12};d_{21})$ satisfying the inequalities
\begin{align*}
0&\le d_{11} \le \min\{n_1,n_2\},&
0&\le d_{12} \le \min\{n_2,n_3\},&
\max\{d_{11}+d_{12}-n_2,0\} &\le d_{21} \le \min\{d_{11},d_{12}\}.
\end{align*}
Note that in all three cases for $\nun$ we have $\min\{n_1,n_2\} = \min\{n_2,n_3\} = 1$.

Now we will have $(d_{11},d_{12};d_{21})\in \Is^{n_1,n_2,n_3}(\st_{2,2})$ if and only if $d_{11}+d_{12} = 2$ and $d_{21} = 0$. Since $d_{11},d_{12}\in\{0,1\}$ this happens if and only if $(d_{11},d_{12};d_{21}) = (1,1;0)$.

But from the above inequalities we see that $(1,1;0)\in \Is^{n_1,n_2,n_3}$ if and only if $0 = d_{21}\ge d_{11}+d_{12}-n_2 = 2-n_2$, which happens if and only if $n_2 = 2$. Thus $\Is^{1,1,2}(\st_{2,2}) = \Is^{2,1,1}(\st_{2,2}) = \es$ and $\Is^{1,2,1}(\st_{2,2}) = \{(1,1;0)\}$. So $\Nc^{1,1,2}(\st_{2,2}) = \Nc^{2,1,1}(\st_{2,2}) = \es$ and $\Nc^{1,2,1}(\st_{2,2}) = \Nc^{1,2,1}(1,1;0)$ is a single irreducible component of $\Nc^{1,2,1}$.

The claims about dimensions follow from Corollary \ref{cor:Nc^nun is CI}, as complete intersections are equidimensional.
\end{proof}

It remains to study the two irreducible schemes $\Nc^{2,2}(\st_{2,2})$ and $\Nc^{1,2,1}(\st_{2,2})$. This will be the subject of Sections \ref{sec:R22} and \ref{sec:R121}. For the convenience of the reader, we will summarize all of the relevant results from these two sections in the following two theorems:

\begin{thm}\label{thm:R22 main results}
There is an isomorphism $\Nc^{2,2}(\st_{2,2})\cong\Spec \Rct^{2,2}(\st_{2,2})\times_{\Oc} \A^1_\Oc$, for some \emph{graded} $\Oc$-algebra $\Rct^{2,2}(\st_{2,2})$. Moreover letting $\Rct^{2,2}_k(\st_{2,2}) = \Rct^{2,2}(\st_{2,2})\otimes_{\Oc} k$ for $k = \F,E$, we have
\begin{enumerate}
	\item\label{R22:domain} $\Rct^{2,2}_\F(\st_{2,2})$ and $\Rct^{2,2}_E(\st_{2,2})$ are domains of dimension $7$. $\Rct^{2,2}(\st_{2,2})$ is a domain, which is flat over $\Oc$ of relative dimension $7$;
	\item\label{R22:smooth}The singular locus of $\Nc^{2,2}(\st_{2,2})$ is precisely the subscheme paramaterized by 
	\[\left\{\left(\begin{pmatrix}t&0\\0&t\end{pmatrix}, \begin{pmatrix}0&0\\0&0\end{pmatrix},\begin{pmatrix}t&0\\0&t\end{pmatrix} \right)\right\}.\]
	Under the isomorphism $\Nc^{2,2}(\st_{2,2})\cong \Spec \Rct^{2,2}(\st_{2,2})\times \A^1_{\Oc}$ this is identified with the subscheme $\Spec (\Rct^{2,2}(\st_{2,2})/\Ic_+)\times \A^1_{\Oc}$, where $\Ic_+\subseteq  \Rct^{2,2}(\st_{2,2})$ is the ideal generated by the positive degree elements of the graded $\Oc$-algebra $\Rct^{2,2}(\st_{2,2})$;
	\item\label{R22:Rational} For $k=\F,E$, there is a closed subscheme $\Zc^{2,2}_k\subseteq \Spec \Rct^{2,2}_k(\st_{2,2})$ of codimension $1$ such that each component of $\Zc^{2,2}_k$ is geometrically integral and $(\Spec \Rct^{2,2}_k(\st_{2,2}))\sm \Zc^{2,2}_k$ is isomorphic to a Zariski open subset of $\A^7_k$ containing at least one point of $\A^7_k(k)$.
	\item\label{R22:Gorenstein} $\Rct^{2,2}(\st_{2,2})$ is normal and Gorenstein but is not a complete intersection;
	\item\label{R22:class group} $\Cl(\Rct^{2,2}(\st_{2,2})) = \Cl(\Rct^{2,2}_\F(\st_{2,2})) = \Cl(\Rct^{2,2}_E(\st_{2,2})) = 0$.
\end{enumerate}
\end{thm}

\begin{thm}\label{thm:R121 main results}
There is an isomorphism $\Nc^{1,2,1}(\st_{2,2})\cong\Spec \Rct^{1,2,1}(\st_{2,2})\times_{\Oc} \A^1_\Oc$, where $\Rct^{1,2,1}(\st_{2,2})$ is \emph{graded} $\Oc$-algebra. Moreover letting $\Rct^{1,2,1}_k(\st_{2,2}) = \Rct^{1,2,1}(\st_{2,2})\otimes_{\Oc} k$ for $k = \F,E$, we have
\begin{enumerate}
	\item\label{R121:domain} $\Rct^{1,2,1}_\F(\st_{2,2})$ and $\Rct^{1,2,1}_E(\st_{2,2})$ are domains of dimension $5$. $\Rct^{1,2,1}(\st_{2,2})$ is a domain, which is flat over $\Oc$ of relative dimension $5$;
	\item\label{R121:smooth} The singular locus of $\Nc^{1,2,1}(\st_{2,2})$ is precisely the subscheme paramaterized by 
	\[\left\{\left(t,\begin{pmatrix}0&0\end{pmatrix}, \begin{pmatrix}t&0\\0&t\end{pmatrix},\begin{pmatrix}0\\0\end{pmatrix},t \right)\right\}.\]
	Under the isomorphism $\Nc^{1,2,1}(\st_{2,2})\cong \Spec \Rct^{1,2,1}(\st_{2,2})\times \A^1_{\Oc}$ this is identified with the subscheme $\Spec (\Rct^{1,2,1}(\st_{2,2})/\Ic_+)\times \A^1_{\Oc}$, where $\Ic_+\subseteq  \Rct^{1,2,1}(\st_{2,2})$ is the ideal generated by the positive degree elements of the graded $\Oc$-algebra $\Rct^{1,2,1}(\st_{2,2})$;
	\item\label{R121:CM} $\Rct^{1,2,1}(\st_{2,2})$ is normal and Cohen--Macaulay but not Gorenstein;
	\item\label{R121:Rational} For $k=\F,E$, there is a closed subscheme $\Zc^{1,2,1}_k\subseteq \Spec \Rct^{1,2,1}_k(\st_{2,2})$ of codimension $1$ such that each component of $\Zc^{1,2,1}_k$ is geometrically integral and $(\Spec \Rct^{1,2,1}_k(\st_{2,2}))\sm \Zc^{1,2,1}$ is isomorphic to a Zariski open subset of $\A^5_k$ containing at least one point of $\A^5_k(k)$.	
	\item\label{R121:class group} There are isomorphisms $\Cl(\Rct^{1,2,1}(\st_{2,2})) = \Cl(\Rct^{1,2,1}_\F(\st_{2,2})) = \Cl(\Rct^{1,2,1}_E(\st_{2,2})) \cong \Z$. If $\omega = \omega_{\Rct^{1,2,1}(\st_{2,2})}$ is the dualizing module of $\Rct^{1,2,1}(\st_{2,2})$ then these isomorphisms send $\omega$ (respectively $\omega\otimes_\Oc \F$ and $\omega\otimes_\Oc E$) to $2\in \Z$;
	\item\label{R121:M} There is a unique self dual $\Rct^{1,2,1}(\st_{2,2})$-module $\Mc$ of generic rank $1$. Moreover $\Mc\otimes_{\Oc}\F$ and $\Mc\otimes_{\Oc}E$ are, respectively, the unique self dual $\Rct^{1,2,1}_\F(\st_{2,2})$ and $\Rct^{1,2,1}_E(\st_{2,2})$ modules of generic rank $1$;
	\item\label{R121:mult 2} We have $\Mc/\Ic_+\Mc\cong \Oc^2$;
	\item\label{R121:surj} The duality map $\tau_{\Mc}:\Mc\otimes_{\Rct^{1,2,1}(\st_{2,2})}\Mc\to \omega$ is surjective.
\end{enumerate}
\end{thm}

Using these theorems, and the previous work we can now give a complete description of the local deformation rings $R^\square(\rbar,\st_{1,1,1,1})$, $R^\square(\rbar,\st_{2,2})$ and $R^{\square}(\rbar,\st_4)$.

\begin{prop}\label{prop:n=4 banal cases}
Let $\Rhat^{2,2}(\st_{2,2})$ and $\Rhat^{1,2,1}(\st_{2,2})$ denote the completions (as graded $\Oc$-algebras) of the rings $\Rct^{2,2}(\st_{2,2})$ and $\Rct^{1,2,1}(\st_{2,2})$ from Theorems \ref{thm:R22 main results} and \ref{thm:R121 main results}, respectively.	
	
Let $n=4$ and let $q$ be banal. Let $\rbar:G_L\to \GL_4(\F)$ be a Galois representation with $\rbar(\Pt_L) = 1$ and $\rbar(\sigma)$ unipotent. Then
\begin{enumerate}
	\item Either $R^\square(\rbar,\st_{1,1,1,1}) = 0$ or $R^\square(\rbar,\st_{1,1,1,1}) \cong \Oc[[t_1,\ldots,t_{16}]]$;
	\item Either $R^\square(\rbar,\st_{4}) = 0$ or $R^\square(\rbar,\st_{4}) \cong \Oc[[t_1,\ldots,t_{16}]]$;
	\item If $\rbar(\sigma) = 1$ and $\rbar(\varphi)$ is conjugate to $\lambdabar\diag(q,q,1,1)$ in $\GL_4(\Fbar)$ for some $\lambdabar \in \Fbar^\times$, then
	\[
	R^\square(\rbar,\st_{2,2}) \cong \Rhat^{2,2}(\st_{2,2})[[t_1,\ldots,t_{9}]];
	\]
	\item If $\rbar(\sigma) = 1$ and $\rbar(\varphi)$ is conjugate to $\lambdabar\diag(q^2,q,q,1)$ in $\GL_4(\Fbar)$ for some $\lambdabar \in \Fbar^\times$, then
	\[
	R^\square(\rbar,\st_{2,2}) \cong \Rhat^{1,2,1}(\st_{2,2})[[t_1,\ldots,t_{11}]];
	\]
	\item If $\rbar$ does not satisfy either of the two previous conditions, then either $R^\square(\rbar,\st_{2,2}) = 0$ or $R^\square(\rbar,\st_{2,2}) \cong \Oc[[t_1,\ldots,t_{16}]]$
\end{enumerate}
\end{prop}
\begin{proof}
Let $\nun^{(1)},\ldots,\nun^{(d)}$ and $x\in \Ns^{\nun^{(1)},\ldots,\nun^{(d)}}(\F)$ be as in the statement of Theorem \ref{thm:fixed type local model}, so that $\sum_\alpha\sum_i n^{(\alpha)}_i = 4$.
	
By Theorem \ref{thm:fixed type local model}, for any $\mun$ we will have $R^\square(\rbar,\st_{\mun}) = 0$ if $\Ns^{\nun^{(1)},\ldots,\nun^{(d)}}(\st_{\mun}) = \es$. Moreover if $\Ns^{\nun^{(1)},\ldots,\nun^{(d)}}(\st_{\mun}) \cong \A^{\sum_\alpha\sum_i\left(n_i^{(\alpha)}\right)^2}_{\Oc}$ then either $R^\square(\rbar,\st_{\mun}) = 0$ or $R^\square(\rbar,\st_{\mun}) \cong \Oc[[t_1,\ldots,t_{16}]]$, depending on whether $x\in \Ns^{\nun^{(1)},\ldots,\nun^{(d)}}(\st_{\mun})(\F)$.

Now if $\mun^{(\alpha)}$ is a partition of $n^{(\alpha)} = \sum_in_i^{(\alpha)}$ for each $i$, then $\bigcup_\alpha \mun^{(\alpha)} = (1,1,1,1)$ if and only if $\mun^{(\alpha)} = (1,\ldots,1)$ for each $\alpha$. Hence
\[
\Ns^{\nun^{(1)},\ldots,\nun^{(d)}}(\st_{1,1,1,1}) = \prod_{\alpha}\Nc^{\nun^{(\alpha)}}(\st_{1,\ldots,1}) \cong \A^{\sum_\alpha\sum_i\left(n_i^{(\alpha)}\right)^2}_{\Oc}.
\]
Part (1) follows.

Now if $\bigcup_\alpha \mun^{(\alpha)} = (4)$ then we must have $d=1$ and $\mun^{(1)} = (4)$. Thus from Corollary \ref{cor:Nc(st_n)} we have either $\Ns^{\nun^{(1)},\ldots,\nun^{(d)}}(\st_{4}) = \es$ or $\Ns^{\nun^{(1)},\ldots,\nun^{(d)}}(\st_{4}) = \Nc^{1,1,1,1}(\st_{4}) \cong \A^4_{\Oc}$. Part (2) follows.

Now consider $\Ns^{\nun^{(1)},\ldots,\nun^{(d)}}(\st_{2,2})$. If $\bigcup_\alpha \mun^{(\alpha)} = (2,2)$ then either $d=1$ and $\mun^{(1)} = (2,2)$ or $d=2$ and $\mun^{(1)} = \mun^{(2)} = (2)$. In the second case we have
\[
\Ns^{\nun^{(1)},\ldots,\nun^{(d)}}(\st_{2,2}) \cong \Nc^{\nun^{(1)}}(\st_2)\times \Nc^{\nun^{(2)}}(\st_2)
\]
so Corollary \ref{cor:Nc(st_n)} gives that this is either $\es$ or $\A^{\sum_\alpha\sum_i\left(n_i^{(\alpha)}\right)^2}_{\Oc}$.

So now assume that $d=1$ and $\mun^{(1)} = (2,2)$. Then $\Ns^{\nun^{(1)},\ldots,\nun^{(d)}}(\st_{2,2}) = \Nc^{\nun^{(1)}}(\st_{2,2})$. By Proposition \ref{prop:st(2,2) cases}, this is either $\es$ or $\A^4_{\Oc}$ unless $\nun^{(1)} = (2,2)$ or $(1,2,1)$.

First assume that $\nun^{(1)} = (2,2)$. By assumption this means that up to conjugation 
\begin{align*}
\rbar(\varphi) &= \lambdabar\begin{pmatrix}q\Ubar_1&0\\0&\Ubar_2\end{pmatrix}&
&\text{and}&
\rbar(\sigma) &= \begin{pmatrix}0&N_1\\0&0\end{pmatrix}
\end{align*}
for some $\lambdabar\in \F^\times$, $N\in M_{2\times 2}(\F)$ and some unipotent $\Ubar_1,\Ubar_2\in \GL_2(\F)$. Then by Theorem \ref{thm:fixed type local model}, $R^\square(\rbar,\st_{2,2})\cong R^{2,2}_x(\st_{2,2})[[t_1,\ldots,t_8]]$, where $R^{2,2}_x(\st_{2,2})$ is the complete local ring of $\Nc^{2,2}(\st_{2,2})$ at the point $x = (\Ubar_1 - I_2,N_1,\Ubar_2 - I_2)\in \Nc^{2,2}(\st_{2,2})(\F)$. By Theorem \ref{thm:R22 main results}, $R^{2,2}_x(\st_{2,2})\cong \Oc[[z_1,\ldots,z_8]]$ unless $N_1 = 0$ and $\Ubar_1 = \Ubar_2$ are scalar (which implies $\Ubar_1 = \Ubar_2 = I_2$, since $\Ubar_1$ and $\Ubar_2$ are unipotent). In this latter case we have $R^{2,2}_x(\st_{2,2}) \cong \Rhat^{2,2}(\st_{2,2})[[t]]$. Part (3) follows.

Now assume that $\nun^{(1)} = (1,2,1)$. By assumption this means that up to conjugation 
\begin{align*}
	\rbar(\varphi) &= \lambdabar\begin{pmatrix}q^2&0&0\\0&q\Ubar&0\\0&0&1\end{pmatrix}&
	&\text{and}&
	\rbar(\sigma) &= \begin{pmatrix}0&N_1&0\\0&0&N_2\\0&0&0\end{pmatrix}
\end{align*}
for some $\lambdabar\in \F^\times$, $N_1\in M_{1\times 2}(\F), N_2\in M_{2\times 1}(\F)$ and some unipotent $\Ubar\in \GL_2(\F)$. Again by Theorem \ref{thm:fixed type local model}, $R^\square(\rbar,\st_{2,2})\cong R^{1,2,1}_x(\st_{2,2})[[t_1,\ldots,t_10]]$, where $R^{1,2,1}_x(\st_{2,2})$ is the complete local ring of $\Nc^{1,2,1}(\st_{2,2})$ at the point $x = (0,N_1,\Ubar - I_2,N_2,0)\in \Nc^{1,2,1}(\st_{2,2})(\F)$. By Theorem \ref{thm:R121 main results}, $R^{1,2,1}_x(\st_{2,2})\cong \Oc[[z_1,\ldots,z_6]]$ unless $N_1=N_2= 0$ and $\Ubar = I_2$. In this latter case $R^{1,2,1}_x(\st_{2,2}) \cong \Rhat^{1,2,1}(\st_{2,2})[[t]]$. Parts (4) and (5) follow.
\end{proof}

\part{Global Theory}\label{part:global}

Let $F$ be a CM field and let $F^+\subseteq F$ be the maximal totally real subfield. Write $\Gal(F/F^+) = \{1,c\}$. Let $\Sc$ be the set of all finite places of $F^+$ which split in $F$, and let $\Sct$ be the set of all finite places of $F$ lying over places in $\Sc$.

Let $\ell>n$ be a prime number, and assume that for each prime $v$ of $F^+$ with $v|\ell$ that $v$ is split in the quadratic extension $F/F^+$, that is that $v\in \Sc$. Write $\Sigma_\ell\subseteq \Sc$ (respectively $\Sigmat_\ell$) for the set of places of $F^+$ (respectively $F$) lying over $\ell$.

We will fix an isomorphism $\iota:\Qbar_\ell\to \C$, and use this to identify embeddings $\tau:F\to \Qbar_\ell$ with embeddings $\tau:F\to \C$.

As before, let $E/\Q_\ell$ be a finite extension with ring of integers $\Oc$, uniformizer $\varpi\in\Oc$ and residue field $\Oc/\varpi = \F$.

Let $\Z^n_+$ denote the set
\[\Z^n_+ := \{(\lambda_1,\ldots,\lambda_n)\in \Z^n|\lambda_1\ge \cdots \ge\lambda_n\}.\]
For any integer $w\in\Z$, let $(\Z^n_+)^{\Hom(F,\C)}_w$ denote the set of tuples $(\lambda_{\tau,i})_{\tau,i}$ indexed by integers $i\in\{1,\ldots,n\}$ and embeddings $\tau:F\to \C$ with the property that:
\begin{itemize}
	\item $\lambda_{\tau,1}\ge \lambda_{\tau,2}\ge \cdots \ge \lambda_{\tau,n}$ for each $\tau:F\to\C$ (that is, $(\lambda_{1,\tau},\ldots,\lambda_{n,\tau})\in\Z^n_+$);
	\item $\lambda_{i,\tau\circ c} = w - \lambda_{n+1-i,\tau}$ for all $\tau:F\to \C$ and all $i=1,\ldots,n$.
\end{itemize}

\section{Global deformation rings}\label{sec:global Galois}
In this section, we introduce the global Galois deformation rings that will be needed in our future arguments. Our presentation will largely follow that in \cite[Section 1.5]{BLGGT14} and \cite[Section 2.3]{CHT}.

Consider the algebraic groups $\Gc_n^0 := \GL_n\times \GL_1$ and $\Gc_n := \Gc_n^0\rtimes \{1,j\}$ where the group $\{1,j\}\cong \Z/2\Z$ acts on $\Gc_n^0$ via
\[
j(g,a)j^{-1} = (a{^tg^{-1}},a)
\]
where for $g\in \GL_n$, $^tg$ denotes the transpose of $g$. Let $\nu:\Gc_n\to \GL_1$ be the map defined by $\nu(g,a) = a$ for $(g,a)\in \Gc_n^0$ and $\nu(j) = -1$. 

For any ring $A$ and any homomorphism $r:G_{F^+}\to \Gc_n(A)$ with $r^{-1}(\Gc_n^0(A)) = G_F$, let $\breve{r}:G_F\to \GL_n(A)$ denote the composition $G_F\xrightarrow{r|_{G_F}}\Gc_n^0(A) = \GL_n(A)\times \GL_1(A) \onto \GL_n(A)$. Note that $\breve{r}$ is \emph{essentially conjugate self dual} in the sense that
\[
\breve{r}^\vee \cong \breve{r}^c \otimes \mu
\]
where $\breve{r}^\vee(g) = ^t\breve{r}(g)^{-1}$, $\breve{r}^c(g) = \breve{r}(cgc^{-1})$ and $\mu$ is the character $\mu = \nu\circ r|_{G_F}:G_F\to A^\times$.

Now fix a finite subset $S\subseteq \Sc$, and assume that $\Sigma_\ell\subseteq S$. Consider a continuous homomorphism
\[
\rhobar:G_{F^+}\onto G_{F^+,S^+}\to \Gc_n(\F)
\]
with $\rhobar^{-1}(\Gc_n^0(\F)) = G_F$, and assume that $\breve{\rhobar}:G_F\to \GL_n(\F)$ is absolutely irreducible. Let $\mu:G_{F^+}\to \Oc^\times$ be a continuous character lifting $\nu\circ \rhobar:G_{F^+}\to \F^\times$. Assume that $\mu$ is de Rham, that $\mu(c_v) = -1$ for all $v|\infty$ (where $c_v$ is complex conjugation at the infinite place $v$), and that there is some integer $w$ for which $\HT_\tau(\mu) = \{w\}$ for all $\tau:F\to \Qbar_\ell$.

Now fix a tuple $\lambda = (\lambda_{\tau,i})_{\tau,i}\in (\Z^n_+)^{\Hom(F,\C)}_w$. Then as in \cite[Section 2]{CHT}, there exists a \emph{universal deformation ring} $R^{\univ}_{S,\lambda,\mu}$ and a universal lift $\rho^{\univ}_{S,\lambda,\mu}:G_{F^+}\onto G_{F^+,S}\to \Gc_n(R^{\univ}_{S,\lambda,\mu})$. We will refrain from giving a full definition of of $R^{\univ}_{S,\lambda,\mu}$, and simply note the following properties (which are not enough to uniquely characterize $R^{\univ}_{S,\lambda,\mu}$): 
\begin{itemize}
	\item $(\rho^{\univ}_{S,\lambda,\mu})^{-1}(\Gc_n^0(R^{\univ}_{S,\lambda,\mu})) = G_{F,S}$;
	\item $\nu\circ \rho^{\univ}_{S,\lambda,\mu} = \mu$;
	\item For each $x\in \left(\Spec R^{\univ}_{S,\lambda,\mu}\right)(\Ebar)$, each place $\vt|\ell$ of $F$ and each embedding $\tau:F\to \Qbar_\ell$ inducing $\vt$ we have \[
	\HT_\tau\left(\breve{\rho}_x|_{G_{F_{\vt}}}\right) = \{\lambda_{\iota\circ \tau,1},\ldots,\lambda_{\iota\circ \tau,n}\},\]
	where for $x\in \left(\Spec R^{\univ}_{S,\lambda,\mu}\right)(\Ebar)$ we let $\rho_x$ denote the map $\rho_x:G_{F^+,S}\to \Gc_n(R^{\univ}_{S,\lambda,\mu}) \xrightarrow{x} \Gc_n(\Ebar)$ induced by $x$.
\end{itemize}
See in particular \cite[Section 2.4.1]{CHT} for more details on the precise statement that should replace the third condition above.

Now fix any finite set $\Sigma\subseteq S$ containing $\Sigma_\ell$. For each $v\in\Sigma$ fix a choice of place $\vt$ of $F$ with $\vt|v$, so that $v = \vt\vt^c$. For each $v\in \Sigma$, identify $G_{F^+_v}$ with $G_{F_{\vt}}$ and let $\rhobar_v$ be the representation 
\[\rhobar_v = \breve{\rhobar}|_{G_{F^+_v}}:G_{F^+_v} = G_{F_{\vt}}\to  \GL_n(\F).\]

For each $v\in\Sigma\sm\Sigma_\ell$, let $R_v^\square = R^\square(\rhobar_v)$ denote the universal framed deformation ring of $\rhobar_v$. For each $v\in \Sigma_\ell$, let $R^\lambda_v$ denote the ring $R^{\square,\crys,\{\lambda_{\tau}\}}(\rhobar_v)$, where $\lambda_\tau = (\lambda_{\iota\circ\tau,1},\ldots,\lambda_{\iota\circ\tau,n})$ and $\tau$ runs over all embeddings $\tau:F\to \Qbar_\ell$ inducing $\vt$. Let 
\[
R^{\loc}_\Sigma := \left(\widehat{\bigotimes_{v\in \Sigma\sm\Sigma_\ell}} R_v^\square\right) \widehat{\otimes}_{\Oc} \left(\widehat{\bigotimes_{v\in\Sigma_\ell}} R_v^\lambda\right),
\]
and note that we are suppressing the $\lambda$ in our notation for $R^{\loc}_\Sigma$.

Then as in \cite[Section 2.2]{CHT} we may also define the universal $\Sigma$-framed deformation ring $R^{\square_\Sigma}_{S,\lambda,\mu}$. Note that this is naturally an $R^{\loc}_\Sigma$-algebra and we have a natural embedding $R^{\univ}_{S,\lambda,\mu}\into R^{\square_\Sigma}_{S,\lambda,\mu}$. For any $x\in\left(\Spec R^{\square_\Sigma}_{S,\lambda,\mu}\right)(\Ebar)$, we will also use $\rho_x$ for the map $\rho_x:G_{F^+}\onto G_{F^+,S}\to \Gc_n(\Ebar)$ induced by the restriction $x|_{R^{\univ}_{S,\lambda,\mu}}\in \left(\Spec \rhobar^{\univ}_{S,\lambda,\mu}\right)(\Ebar)$. 

Note that for any $x\in\left(\Spec R^{\square_\Sigma}_{S,\lambda,\mu}\right)(\Ebar)$, if the pullback of $x$ under the map $R^{\loc}_\Sigma\to R^{\square_\Sigma}_{S,\lambda,\mu}$ is equal to
\[\left(r_{x,v}\right)_{v\in \Sigma} \in \left(\Spec R^{\loc}_\Sigma\right)(\Ebar) = \prod_{v\in\Sigma\sm\Sigma_\ell}\left(\Spec R_v^\square\right)(\Ebar)\times\prod_{v\in \Sigma_\ell}\left(\Spec R_v^\lambda\right)(\Ebar)\]
then for each $v\in \Sigma$, the lift $r_{x,v}$ of $\rhobar_v$ is conjugate to $\breve{\rho}_x|_{G_{F_{\vt}}}$.
\section{Unitary groups and Hecke algebras}\label{sec:hecke}

Let $B$ be a central simple $F$-algebra with $\dim_FB = n^2$ and assume that
\begin{itemize}
\item $B^{\op}\cong B\otimes_{F,c}F$
\item For any finite place $w$ of $F$ with $w\not\in\Sc$ (i.e. which is inert or ramified over $F^+$), $B\otimes_FF_w\cong M_n(F_w)$.
\end{itemize}

Say that $B$ is split at $w$ if $B_w\cong M_n(F_w)$. Let $\Deltat$ be the set of all places of $F$ at which $B$ does not split. By assumption, $\Deltat\subseteq \Sct$ and we have $w\in\Deltat$ if and only if $w^c\in\Deltat$. Let $\Delta$ be the set of all finite places of $F^+$ lying under an element of $\Deltat$.

Now let $\ddagger:B\to B$ be any involution satisfying $(xy)^\ddagger = y^\ddagger x^\ddagger$ and $\ddagger|_F = c$.
Define an algebraic group $G_{\ddagger}/F^+$ by
\[G_{\dagger}(R) = \{x\in (B\otimes_{F^+}R)^\times | xx^{\ddagger\otimes 1} = 1\}.\]
Assume that there exists a $\ddagger$ for which $G_\ddagger$ has the following properties:\footnote{Note that as in \cite{CHT}, the existence of such a $\ddagger$ requires certain restrictions on $[F:\Q]$ and on $B$.}
\begin{itemize}
	\item For all $v\in \Sc\sm \Delta$, $G_{\ddagger}(F^+_v)$ is quasisplit;
	\item For all $v|\infty$, $G_{\ddagger}(F^+_v)\cong U_n(\R)$, where $U_n(\R)$ is the compact real unitary group.
\end{itemize}
From now on, let $G = G_\ddagger$.

As in \cite[Section 3.3]{CHT} we can fix an order $\O_B\subseteq B$ for which $\O_B^\ddagger = \O_B$ and $\O_{B,w} := \O_{B}\otimes_{\O_F}F_w$ is a maximal compact in $B_w$ for all $w\in\Sc$.

Define an algebraic group $G/F^+$ by
\[G(R) = \{x\in (B\otimes_{F^+}R)^\times | xx^{\ddagger\otimes 1} = 1\},\]
for any $F^+$-algebra $R$, and use $\O_B$ to give this an integral structure\footnote{As noted in \cite{CHT}, this integral structure may be rather poorly behaved at places $v\not\in\Sc^+$, but this is not relevant for our purposes}. By our assumption on $\ddagger$, $G(F^+_\tau) \cong U_n(\R)$ for any infinite place $\tau$ of $F^+$. That is, $G$ is compact at all finite places.

Now take any $v\in\Sc$, and let $v$ split in $F$ as $v = ww^c$. The inclusion $F^+\into F$ induces isomorphisms $F_v^+\isomto F_w$ and $F_v^+\isomto F_{w^c}$ compatible with the isomorphism $c:F_w\to F_{w^c}$ induced by $c:F\to F$. Also note that we canonically have $F\otimes_{F^+}F^+_v\cong F_w\oplus F_{w^c}$. Identifying this with $F^+_v\oplus F^+_v$ via the maps above, we see that induced isomorphism $c\otimes 1:F\otimes_{F^+}F^+_v\isomto F\otimes_{F^+}F^+_v$ is just identified with the map $(x,y)\mapsto (y,x)$.

For any place $w$ of $F$, let $B_w$ denote $B\otimes_FF_w$, and note that $\ddagger$ induces an isomorphism $\ddagger_w:B_{w}\isomto B^{\op}_{w^c}$ of $F^+$-algebras, satisfying $\ddagger_{w^c}\circ \ddagger_w = \id$ and moreover $\O_{B,w}^\ddagger = \O_{B,w^c}$ for $w\in\Sc$. By the above, for any place $v\in \Sc$ with $v=ww^c$ we have $B\otimes_{F^+}F^+_v\cong B_w\oplus B_{w^c}$ and under this isomorphism, $\ddagger\otimes 1$ is identified with the map $(x,y)\mapsto (y^{\ddagger_{w^c}},x^{\ddagger_w})$.

This gives us an isomorphism
\begin{align*}
G(F_v^+) 
&= \{x\in (B\otimes_{F^+}F^+_v)^\times | xx^{\ddagger\otimes 1} = 1\}
 \isomto \{(x,y)\in B_w^\times\oplus B_{w^c}^\times | xy^{\ddagger_{w^c}} = yx^{\ddagger_w} = 1\}\\
& \isomto \{(x,(x^{-1})^{\ddagger_w})|x\in B_w^\times\}
 \isomto B_w^\times.
\end{align*}

Now for any $w\in\Sct$ we can identify $B_w = M_{r_v}(D_w)$ for some integer $r_v|n$ and some division algebra $D_w$ over $F_w$ of dimension $(n/r_v)^2$. Since $B_{w^c}\cong B_w^{\op}= M_{r_v}(D_w)^{\op}\cong M_{r_v}(D_w^{\op})$, we may identify $D_{w^c} = D_w^{\op}$ (and note that this implies that $r_v$ indeed depends only on $v$, and not on $w$)

By definition $w\in\Sct\sm\Deltat$ if and only if $D_w = F_w$ (or equivalently, $r_v = n$).

As in \cite[Section 3.3]{CHT}, we can and do pick the identifications $B_w = M_{r_v}(D_w)$ so that the isomorphisms $G(F_v^+)\isomto B_w^\times$ induce a collection of isomorphisms $\iota_w:G(F^+_v)\isomto B_w^\times = \GL_{r_v}(D_w)$, for all $v\in \Sc$ and all $w|v$, with the following properties:
\begin{itemize}
	\item If $w\not\in\Deltat$, then $\iota_w(\Oc_{B,w}) = \GL_n(\Oc_{F,w})$.
	\item If $w\in\Deltat$, then $\iota_w(\Oc_{B,w}) = \GL_{r_v}(\Oc_{D_w})$, where $\Oc_{D_w}\subseteq D_w$ is the ring of integers.
	\item If $w\not\in\Deltat$, the map 
	\[\GL_n(F_w)\xrightarrow{\iota_w^{-1}} G(F_v^+)\xrightarrow{\iota_{w^c}} \GL_n(F_{w^c}) \to \GL_n(F_{w})\]
	(where the last map is induced by the isomorphism $F_{w^c}\isomto F_w$ coming from complex conjugation) is given by $x\mapsto {^tx^{-1}}$.
\end{itemize}
We will also use $\iota_w^{-1}$ to denote the injection $\GL_{r_v}(D_w)\xrightarrow{\iota_w^{-1}} G(F^+_v)\into G(\A_{F^+,f})$.

\subsection{Hecke modules}

Now we will say that a \emph{level} is a compact open subgroup $K
\subseteq G(\A_{F^+,f})$ in the form $K=\prod_vK_v$ for some $K_v\subseteq G(F^+_v)$ satisfying the following:
\begin{itemize}
	\item There is a finite set $\Sigma(K)\subseteq \Sc$ such that $\Sigma(K)\cap (\Delta\cup\Sigma_\ell) = \es$ and for all $v\in \Sc\sm \Sigma(K)$, $K_v = G(\O_{F^+,v})$.
	\item For all places $v\not\in\Sc$, $K_v$ is a hyperspecial maximal compact subgroup of $G(F^+_v)$.
\end{itemize}
In more general contexts, one could also consider levels where $K_v$ is also nonmaximal finitely many $v\not\in \Sc$. However for our applications in this paper we will never need to consider such levels, and so for convenience we exclude such levels from our definition.

For any level $K\subseteq G(\A_{F^+,f})$, consider the Shimura set
\[X_K = G(F^+)\backslash G(\A_{F^+,f})/K \]
and let $M(K) = H^0(X_K,\O)$.

Let $\Sigmat(K)\subseteq \Sct$ be the set of places of $F$ lying over the places in $\Sigma(K)$. For any $w\in \Sct\sm (\Sigmat(K)\cup \Deltat)$ and any $j=0,1,\ldots,n$ let $T_w^{(j)}$ denote the double coset operator:
\[
T_w^{(j)} = \left[
K\iota_w^{-1}\left(
\begin{pmatrix}
\varpi I_j&0\\
0& I_{n-j}
\end{pmatrix}\right)
K
\right]
\]
acting on $M(K)$, and note that $T_w^{(0)} = 1$ and $T_{w^c}^{(j)} = (T_w^{(n)})^{-1}T_w^{(n-j)}$.

Let $\Stilde\subseteq \Sct$ be a finite set of finite places of $F$ containing $\Sigmat(K)\cup \Deltat\cup \Sigmat_\ell$ and define
\[\T(K) = \O\left[T_w^{(1)}, T_w^{(2)},\ldots,T_w^{(n)},(T_w^{(n)})^{-1}\middle|w\in \Sct\sm \Stilde\right] \subseteq \End_{\O}(M(K))\]
and note that this is a finite rank commutative $\Oc$-algebra, which is independent of the choice of $\Stilde$.

Let $\mf\subseteq \T(K)$ be a \emph{non-Eisenstein} maximal ideal of $\T(K)$, and note that the localization $\T(K)_{\mf}$ is reduced and finite flat over $\Oc$. Then we have the following well known result:

\begin{prop}\label{prop:local global compatibility}
Let $\Sigma = \Sigma(K)\cup \Delta\cup\Sigma_\ell$. There is a unique continuous homomorphism
\[
\rho^{\mod}_{\mf}:G_{F^+}\onto G_{F^+,\Sigma}\to \Gc_n(\T(K)_{\mf})
\]
satisfying the following properties:
\begin{enumerate}
	\item $(\rho^{\univ}_{\mf})^{-1}(\Gc_n^0(\T(K)_{\mf})) = G_F$;
	\item $\nu\circ \rho^{\mod}_{\mf} = \varepsilon^{1-n}\delta_{F/F^+}^u$, where $\varepsilon:G_F\to \Oc^\times$ is the cyclotomic character, $\delta_{F/F^+}$ is the unique quadratic character of  $G_{F^+}/G_F$ and $u\in\Z/2\Z$ is some integer depending on $B$.
	\item For $v\in \Sc\sm \Sigma$, if $v$ splits in $F$ as $v = ww^c$, then $\breve{\rho}^{\mod}_{\mf}(\Frob_w)\in \GL_n(\T(K)_{\mf})$ has characteristic polynomial:
	\[
	\sum_{j=0}^n(-1)^j(\Nm_{F/\Q}(w))^{j(j-1)/2}T^{(j)}_w t^{n-j}.
	\]
	\item For any homomorphism $x:\T(K)_{\mf}\to \Ebar$, if $\rho^{\mod}_x$ is the composition $G_{F^+}\xrightarrow{\rho^{\mod}_{\mf}} \Gc_n(\T(K)_{\mf}) \xrightarrow{x} \Gc_n(\Ebar)$, then for any place $w$ of $F$ lying over $\ell$, and any embedding $\tau:F\into \Qbar_\ell$ inducing $w$, we have
	\[
	\HT_{\tau}(\breve{\rho}^{\mod}_x|_{G_{F_w}}) = \{n-1,n-2,\ldots,1,0\}
	\]
	\item Take any $v\in\Delta$, and let $v = ww^c$ in $F$. Write $B_w \cong M_{n/d_v}(D_w)$ as above, with $\dim_{F_w}D_w = d_v^2$. Assume that $K_v = G(\Oc_{F^+,v}) \cong \GL_{n/d_v}(\Oc_{D_w})$ is a maximal compact subgroup. Then for any homomorphism $x:\T(K)_{\mf}\to \Ebar$, $\breve{\rho}^{\mod}_x|_{G_{F_w}}$ has inertial type $\tau_{d_v,\ldots,d_v}$ (where $\rho_x:G_{F^+}\to \Gc_n(\Ebar)$ is as above).
\end{enumerate}
\end{prop}
\begin{proof}
Parts (1), (2), (3) and (4) follow from \cite[Proposition 3.4.4]{CHT} (using \cite[Proposition 6.5.1]{HKV} for a more general form of the needed base change results). 

For (5) fix any $x:\T(K)_{\mf}\to \Ebar$. Then $x$ corresponds to an automorphic representation $\sigma_x$ of $G(\A_{F^+})$. By \cite[Proposition 6.5.1]{HKV} and the construction of $\rho^{\mod}_{\mf}$, there is an automorphic representation $\Pi_x = \BC(\sigma_x)$ of $\GL_n(\A_F)$ with the property that, for each place $v\in \Sc$ with $v = ww^c$ in $F$:
\begin{itemize}
	\item $\breve{\rho_x}|_{G_{F_{w}}} \cong \rec_{L_{w}}(\Pi_{x,w})$;
	\item If $v\not\in\Delta$, $\Pi_{x,w} = \sigma_{x,v}\circ \iota_{w}^{-1}$;
	\item If $v\in\Delta$, $\LJ(\Pi_{x,w}) = \sigma_{x,v}\circ \iota_w^{-1}$, where $\LJ$ is the map from Section \ref{ssec:LJ}.
\end{itemize}
As $\breve{\rhobar}_{\mf}$ is absolutely irreducible (as $\mf$ is assumed to be non-Eisenstein), $\breve{\rho}_x$ is absolutely irreducible as well, which implies that $\Pi$ is necessarily cuspidal. In particular, $\Pi_{x,w}$ is generic for all $w$, so Proposition \ref{prop:LJ(st) = triv} implies that $\Pi_{x,w}$ has type $\st_{d_v,\ldots,d_v}$ for all $w\in \Deltat$, so $\breve{\rho_x}|_{G_{F_{w}}} \cong \rec_{L_{w}}(\Pi_{x,w})$ has inertial type $\tau_{d_v,\ldots,d_v}$, proving (5).
\end{proof}

We will also let 
\[\rhobar_{\mf}:G_{F^+}\onto G_{F^+,\Sigma}\xrightarrow{\rho_{\mf}^{\mod}}\Gc_n(\T(K)_{\mf})\to \Gc_n(\T(K)_{\mf}/\mf)\into \Gc_n(\Fbar)\]
denote the mod $\mf$ reduction of $\rho^{\mod}_{\mf}$. Since we assumed that $\mf$ was non-Eisenstein, $\breve{\rhobar}_{\mf}$ is absolutely irreducible.

We will also need the following generic multiplicity one result:

\begin{prop}\label{prop:gen rank 1}
Assume that $K_v = G(\Oc_{F^+,v})$ for all $v\in\Sc$ and $K_v$ is a hyperspecial maximal compact subgroup of $G(F^+_v)$ for all $v\not\in \Sc$. Also assume that the extension $F/F^+$ is unramified at all finite places.

Then for any homomorphism $x:\T(K)_{\mf}\to \Ebar$ we have
\[
\dim_{\Ebar}(M(K)_{\mf}\otimes_{x}\Ebar) = 1.
\]
\end{prop}
\begin{proof}
As in the proof of Proposition \ref{prop:local global compatibility}, $x$ corresponds to an automorphic representation $\sigma_x = \bigotimes'_v\sigma_{x,v}$ of $G(\A_{F^+})$ over $\C$. Let $m(\sigma_x)$ be the multiplicity with which $\sigma_x$ occurs in the $G(\A_{F^+})$ representation $L_{\mathrm{disc}}^2(G(F^+)\backslash G(\A_{F^+}))$

By the assumptions on $K_v$ we get that $\dim_{\C}\sigma_{x,v}^{K_v} = 1$ for all $v$, so $\dim_{\C}\sigma_x^K = \dim_{\C}\bigotimes'_v\sigma_{x,v}^{K_v} = 1$. It follows from this that
\[
\dim_{\Ebar}(M(K)_{\mf}\otimes_{x}\Ebar) = m(\sigma_x).
\]
Now \cite[Theoreme 5.4]{Labesse} implies that $m(\sigma_x) = 1$, completing the proof. Note that the conditions of that theorem are indeed satisfied in our setup, as $F/F^+$ is unramified, $K_v$ is maximal hyperspecial at all $v\not\in \Sc$, and $G$ is compact at all infinite places.
\end{proof}
We note that the automorphic multiplicity one result needed in the above proof has been announced in far greater generality in \cite[Theorem 1.7.1]{KMSW} (although in the case relevant to us this fact is still conditional on a sequel to that paper, as we are considering unitary groups associated to nonsplit central simple algebras), so it is likely that the conditions of Proposition \ref{prop:gen rank 1} can be weakened significantly.

\subsection{Sufficiently small levels}

In order to apply the standard Taylor--Wiles--Kisin patching method, one must restrict attention to levels satisfying the following condition:

\begin{defn}
A level $K = \prod_vK_v$ is \emph{sufficiently small} if for all $t\in G(\A_{F^+,f})$ the order of the finite group $t^{-1}G(F^+)t\cap K$ is not divisible by $\ell$.
\end{defn}

Fortunately, the hypotheses of Theorem \ref{main theorem} already imply that all levels we consider will automatically be sufficiently small (as we assume that $D$ is nonsplit, and so $\Delta$ is nonempty, and any $v_0\in\Delta$ splits in $F$ and is banal):

\begin{prop}\label{prop:banal => sufficiently small}
If there exists a finite prime $v_0$ of $F^+$ which splits in $F$ and is banal, then any level $K = \prod_vK_v\subseteq G(\A_{F^+,f})$ is sufficiently small.
\end{prop}
\begin{proof}
Consider the Galois extension $F(\zeta_\ell)/F^+$. As $v_0$ splits in $F/F^+$, and does not ramify in $F^+(\zeta)/F^+$ (as $v_0\nmid \ell$), $v_0$ is unramified in $F(\zeta_\ell)/F^+$. By the Chebotarev density theorem, there are infinitely many primes $v$ of $F^+$ with $\Frob_v = \Frob_{v_0}$ in $\Gal(F(\zeta_\ell)/F^+)$. But for any such $v$, this implies that $\Frob_v = \Frob_{v_0} = 1$ in $\Gal(F/F^+)$, so $v$ is split in $F/F^+$, and $q_v = \Frob_v = \Frob_{v_0} = q_{v_0}\in \Gal(F^+(\zeta_\ell)/F^+) \subseteq (\Z/\ell\Z)^\times$, so $v$ is also banal.

Thus there are infinitely many banal primes $v\in\Sc$, and so we may fix a banal prime $v_1\in\Sc\sm(\Delta\cup\Sigma(K))$, so that $K_{v_1} = G(\Oc_{F^+,v_1})\cong \GL_n(\Oc_{F^+,v_1})$ (in fact, for our argument it would be enough to simply pick $v_1\in \Sc\sm\Delta$, so that $K_{v_1}\subseteq \GL_n(\Oc_{F^+,v_1})$ up to conjugation). Let $r=\Char(\Oc_{F^+}/v_1)\ne \ell$.

Now take any $t = (t_v)_v\in G(\A_{F^+,f})$. As the map $G(F^+)\to G(F^+_{v_1})$, $g\mapsto t_{v_1}^{-1}gt_{v_1}$ is injective, $t^{-1}G(F^+)t\cap K$ is isomorphic to a subgroup of $K_{v_1}\cong \GL_n(\Oc_{F^+,v_1})$. But if $U(v_1) = \ker(\GL_n(\Oc_{F^+,v_1})\onto \GL_n(\Oc_{F^+}/v_1))$ then $U(v_1)$ is a pro-$r$ group, and
\[
| \GL_n(\Oc_{F^+,v_1})/U(v_1)| \cong |\GL_n(\Oc_{F^+}/v_1)| = \prod_{i=0}^{n-1}(q_{v_1}^n-q_{v_1}^i) = q^{n(n-1)/2}_{v_1}\prod_{i=1}^n(q_{v_1}^i-1)
\]
and so $\ell\nmid |\GL_n(\Oc_{F^+,v_1})/U(v_1)|$, as $v_1$ is banal. Hence $ \GL_n(\Oc_{F^+,v_1})$ is a pro prime to $\ell$ group, and so any finite subgroup of $\GL_n(\Oc_{F^+,v_1})$ has order prime to $\ell$, giving that $\ell\nmid |t^{-1}G(F^+)t\cap K|$.
\end{proof}

\subsection{Duality}

We now recall the following consequence of the main result of \cite[Section 5.2]{CHT}:

\begin{prop}\label{prop:self dual}
In addition to our running assumptions, assume further that either $B\cong M_n(F)$ or $B\cong M_{n/2}(D)$, for a quaternion algebra $D$ over $F$. Then there is a perfect pairing
\[
\langle\ ,\ \rangle:M(K)_{\mf}\times M(K)_{\mf}\to \Oc
\]
which is equivariant for the action of all Hecke operators at all primes where $B$ does not ramify.
\end{prop}
\begin{proof}
This is essentially proved in \cite[Section 5.2]{CHT}. For completeness, we shall summarize the results of that paper in our notation.

Consider the opposite algebra $B^{\op}$ of $B$. The involution $\ddagger$ on $B$ induces a corresponding involution on $B^{\op}$, which we will also denote $\ddagger$. Let $G'/F^+$ be the unitary group associated to $B^{\op}$ and $\ddagger$, which will also be compact at all infinite places of $F^+$. For any $K'\subseteq G'(\A_{F^+,f})$ let $X'_{K'}$ be the corresponding Shimura set, and let $M'(K') = H^0(X'_{K'},\Oc)$.

Note that one has $G' \cong G^{\op}$ as groups, and so in particular there is a group isomorphism $I:G\isomto G'$ given by $g\mapsto g^{-1}$. For any $K\subseteq G(\A_{F^+,f})$ the Jacquet--Langlands correspondence gives a natural identification between the Hecke algebras $\T(K)$ and $\T(I(K))$ acting on $M(K)$ and $M'(I(K))$ respectively, and so we may treat both $M(K)$ and $M'(I(K))$ as $\T(K)$-modules. Now $I$ induces a bijection $X_K\isomto X'_{I(K)}$ and so induces an isomorphism $M(K)\cong M'(I(K))$ of $\Oc$-modules. However, under the identification above this map is not a $\T(K)$-module homomorphism. Instead it interchanges the operator $T_w^{(j)}$ on $M(K)$ with the operator $T_{w^c}^{(j)}$ on $M'(I(K))$.

The work of \cite[Section 5.2]{CHT} (as well as the proof of Proposition 5.3.5 of \emph{loc. cit.}) constructs a perfect pairing
\[
\langle\ ,\ \rangle:M(K)_{\mf}\times M'(I(K))_{\mf}\to \Oc
\]
which is equivariant for the actions of all Hecke operators at all places where $B$ does not ramify.

But now the condition that either $B\cong M_n(F)$ or $B\cong M_n(D)$ implies that $B\cong B^{\op}$ as $F$-algebras. From this it is easy to see that there is an isomorphism $M_(K)\cong M'(I(K))$ preserving the action of all Hecke operators. Combining this with the above perfect pairing gives the desired result.
\end{proof}

\section{Patching}\label{sec:patch}

In this section we combine the results of Sections \ref{sec:global Galois} and \ref{sec:hecke} to produce the patched module needed for our main argument.

First recall the definition of an \emph{adequate subgroup} from \cite{ThorneAdequate} (and recall that we have assumed that $\ell>n$):

\begin{defn}
A subgroup $H\subseteq \GL_n(\Fbar_\ell)$ is \emph{adequate} if:
\begin{itemize}
	\item $H^1(H,\Fbar_\ell) = 0$;
	\item $H^0(H,\ad^0(\Fbar_\ell^n)) = H^1(H,\ad^0(\Fbar_\ell^n)) = 0$, where $\ad^0(\Fbar_\ell^n)\subseteq M_{n\times n}(\Fbar_\ell)$ is the set of trace $0$ matrices, and $H$ acts on $\ad^0(\Fbar_\ell^n)$ by conjugation;
	\item The elements of $H$ with order prime to $\ell$ span $M_{n\times n}(\Fbar_\ell)$ as a $\Fbar_\ell$-vector space.
\end{itemize}
\end{defn}

As shown in \cite[Appendix A]{ThorneAdequate}, if $\ell\ge 2(n+1)$, then any subgroup $H\subseteq \GL_n(\Fbar_\ell)$ which acts absolutely irreducibly on $\Fbar_\ell^n$ is adequate.

Fix a central simple $F$-algebra $B$ satisfying the conditions from Section \ref{sec:hecke}, and let $\Delta$ and $G$ be as in that section.

Fix a level $K = \prod_v K_v\subseteq G(\A_{F^+,f})$ where:
\begin{itemize}
	\item For all $v\in \Sc$, $K_v = G(\Oc_{F^+,v})$;
	\item For all $v\not\in \Sc$, $K_v$ is a hyperspecial maximal compact subgroup of $G(F_v^+)$.
\end{itemize}

Fix a non-Eisentein maximal ideal $\mf\subseteq \T(K)$ with $\T(K)/\mf = \F$ (which can always be achieved by enlarging $\Oc$ if necessary) and consider the representation $\rhobar_{\mf}:G_{F^+}\to \Gc_n(\Fbar)$ defined after Proposition \ref{prop:local global compatibility}.

We shall apply the results of Section \ref{sec:global Galois} with $\rhobar = \rhobar_{\mf}$.

Take $\Sigma = \Delta\cup\Sigma_\ell$. Also pick $\mu = \varepsilon^{1-n}\delta_{F/F^+}^u$, so that $w = n-1$, and define $\lambda = (\lambda_{\tau,i})_{\tau,i}\in (Z^n_+)^{\Hom(F,\C)}$ by $\lambda_{\tau,i} = n-i$ for all $\tau$ and all $i=1,\ldots,n$. As we have assumed that $\ell>n$, this $\lambda$ is in the Fontaine--Laffaille range, and so each of the rings $R_v^\lambda$ for $v\in\Sigma_\ell$ is a power series ring by Proposition \ref{prop:Fontaine--Laffaille}. 

As before take
\[
R^{\loc}_{\Sigma} = \left(\widehat{\bigotimes_{v\in \Delta}} R_v^\square\right) \widehat{\otimes}_{\Oc} \left(\widehat{\bigotimes_{v\in\Sigma_\ell}} R_v^\lambda\right)
\]
where $R_v^\square$ is again the unrestricted framed deformation ring of $\rhobar_{\mf,v} = \breve{\rhobar}_{\mf}|_{G_{F_{\vt}}}$.

Now for each $v\in\Delta$, write $v=ww^c$ in $F$, and let $B_{w} \cong M_{n/d_v}(D_w)$ for some $d_v|n$ and some division algebra $D_w$ over $F_w$ (and note that $d_v$ depends only on $v$ and not on $w$, as $B_{w^c} = (B_w)^{\op} \cong M_{n/d_v}(D_w^{\op})$).

Define the quotient
\[
R_\Sigma^{\loc,B} := \left(\widehat{\bigotimes_{v\in \Delta}} R_v^\square(\rhobar_{\mf,v},\st_{d_v,\ldots,d_v})\right) \widehat{\otimes}_{\Oc} \left(\widehat{\bigotimes_{v\in\Sigma_\ell}} R_v^\lambda\right)
\]
of $R^{\loc}_\Sigma$, so that $\Spec R^{\loc,B}_{\Sigma}$ is a union of irreducible components of $\Spec F^{\loc}_\Sigma$ (in our applications, $R^{\loc,B}_{\Sigma}$ will in fact be a domain). Note that by Proposition \ref{prop:local global compatibility}, $\rhobar_{\mf}$ has a lift with inertial type $\tau_{d_v,\ldots,d_v}$ for all $v\in \Delta$, so $R_v^\square(\rhobar_{\mf,v},\st_{d_v,\ldots,d_v})\ne 0$. Hence $R^{\loc,B}_\Sigma \ne 0$.

Now for each $v\in \Delta$, $\Spec R_v^\square(\rhobar_{\mf,v},\st_{d_v,\ldots,d_v})$ is a union of irreducible components of $\Spec R^\square_v$, by Proposition \ref{prop:R(tau)}(1), and hence defines a local deformation problem (in the sense of \cite[Definition 2.2.2]{CHT}) by \cite[Lemma 3.2]{BLGHT}. By \cite[Proposition 2.2.9]{CHT}, there is a canonical choice of quotient $R^B_{\Sigma,\lambda,\mu}$ of $R^{\univ}_{\Sigma,\lambda,\mu}$ charachterized by the property that for any $A\in\Cc_{\Oc}^{\wedge}$ and any $R^{\univ}_{\Sigma,\lambda,\mu}\to A$ inducing a representation $\rho_A:G_{F^+}\onto G_{F^+,\Sigma}\to \Gc_n(R^{\univ}_{\Sigma,\lambda,\mu})\to\Gc_n(A)$, the map $R^{\univ}_{\Sigma,\lambda,\mu}\to A$ factors through $R^{\univ}_{\Sigma,\lambda,\mu}\onto R^B_{\Sigma,\lambda,\mu}$ if and only if for each $v\in \Delta$, the map $R^\square_v\to A$ induced by $\breve{\rho}_A|_{G_{F^+_v}}:R_v^\square\to \GL_n(A)$ factors through $R_v^\square \onto R_v^\square(\rhobar_{\mf,v},\st_{d_v,\ldots,d_v})$.

Note that Proposition \ref{prop:local global compatibility} implies that there exists a surjective map $R^{B}_{\Sigma,\lambda,\mu}\onto \T(K)_{\mf}$. Hence we may regard $M(K)_{\mf}$ as an $R^{B}_{\Sigma,\lambda,\mu}$-module

A standard patching argument now produces the following:

\begin{thm}\label{thm:patching}
Assume that $\breve{\rhobar}_{\mf}(G_{F(\zeta_{\ell})})\subseteq \GL_n(\Fbar_\ell)$ is adequate and that there exists a finite prime $v_0$ of $F^+$ which splits in $F$ and is banal. Then there exist:
\begin{itemize}
	\item Rings $R_\infty^B := R^{\loc,B}_\Sigma[[x_1,\ldots,x_g]]$ and $S_\infty = \Oc[[y_1,\ldots,y_r]]$, for some integers $g,r\ge 0$;
	\item An embedding $S_\infty\to R_\infty^B$ making $R_\infty^B$ into a finite $S_\infty$-algebra;
	\item A finitely generated $R_\infty^B$-module $M_\infty$.
\end{itemize}
satisfying:
\begin{enumerate}
	\item $\dim S_\infty = \dim R_\infty^B$;
	\item $M_\infty$ is a maximal Cohen--Macaulay $R_\infty^B$-module;
	\item There is a surjective map $R_\infty^B/(y_1,\ldots,y_r)R_{\infty}^B\onto R^B_{\Sigma,\lambda,\mu}$ such that the action of $R_\infty^B/(y_1,\ldots,y_r)R_{\infty}^B$ on $M_\infty/(y_1,\ldots,y_r)M_{\infty}$, factors through $R^B_{\Sigma,\lambda,\mu}$ and we have an isomorphism 
	$M_\infty/(y_1,\ldots,y_r)M_{\infty}\cong M(K)_{\mf}$ of $R^{B}_{\Sigma,\lambda,\mu}$-modules;
	\item If $F/F^+$ is unramified at all finite places, then $M_\infty$ has generic rank either $0$ or $1$ on each irreducible component of $\Spec R_\infty^B$.
	\item Provided that either $B\cong M_n(F)$ or $B\cong M_{n/2}(D)$ for a quaternion algebra $D$ over $F$, $M_\infty$ is self-dual, in the sense that
	\[
	M_\infty \cong \RHom_{R_\infty^B}(M_\infty,\omega^\bullet_{R_\infty^B})[\dim R_\infty^B]
	\]
	as $R_\infty^B$-modules, where $\omega^\bullet_{R_\infty^B}$ is the dualizing complex of $R_\infty^B$. (In the case when $R_\infty^B$ is Cohen--Macaulay, this may be written more simply as $M_\infty \cong \Hom_{R_\infty^B}(M_\infty,\omega_{R_\infty^B})$).
\end{enumerate}
\end{thm}
\begin{proof}
If $R_\infty^B$ is replaced by $R_\infty:=R^{\loc}_\Sigma[[x_1,\ldots,x_g]]$, the construction of $M_\infty$, and parts (1),(2) and (3), follow by a standard Taylor--Wiles--Kisin patching argument (see for example \cite[Section 3.6]{BLGG} and \cite[Section 6]{ThorneAdequate}), where we use Proposition \ref{prop:banal => sufficiently small} to avoid any issues with non sufficiently small levels.

Proposition \ref{prop:local global compatibility}(5) then ensures that the action of $R_\infty$ on $M_\infty$ factors through $R_\infty^B$, giving (1), (2) and (3).

Proposition \ref{prop:gen rank 1}, and the fact that $M_\infty/(y_1,\ldots,y_r)M_{\infty}\cong M(K)_{\mf}$, then ensures that $M_\infty$ has generic rank $1$ on each component of $\Spec R_\infty^B$ on which it is supported, proving (4).

Part (5) then follows from Proposition \ref{prop:self dual} and 
\cite[Theorem 11.3]{IKM}. Note that the perfect pairings from Proposition \ref{prop:self dual} are equivariant for the Hecke operators at the ``Taylor--Wiles primes'' and so the patching construction does not interfere with the duality.
\end{proof}

\section{The structure of $\Mbar_\infty$}\label{sec:M_infty}

We now specialize to the setting of Theorem \ref{main theorem}. In particular, we now take $n=4$ and assume each prime $v\in\Delta$ is banal, and moreover that $d_v=2$ for all $v\in \Delta$. Define the integers:
\begin{align*}
a &= \#\left\{v\in \Delta\middle|d_v=2,\rhobar_{\mf,v} \text{ is unramified, and } \rhobar_{\mf,v}(\Frob_v)\sim \lambdabar\diag(q_v^2,q_v,q_v,1) \text{ for some }\lambdabar\in\Fbar_\ell^\times\right\}\\
b &= \#\left\{v\in \Delta\middle|d_v=2,\rhobar_{\mf,v} \text{ is unramified, and } \rhobar_{\mf,v}(\Frob_v)\sim \lambdabar\diag(q_v,q_v,1,1) \text{ for some  }\lambdabar\in\Fbar_\ell^\times\right\}.
\end{align*}

Now as each $R^\lambda_v$ (for $v\in\Sigma_\ell$) is a power series ring by Proposition \ref{prop:Fontaine--Laffaille}, Proposition \ref{prop:n=4 banal cases} (recalling that $R_\infty^B\ne 0$) gives that
\[
R_\infty^B \cong \left(\Rhat^{1,2,1}(\st_{2,2})\right)^{\widehat{\otimes} a} \widehat{\otimes} \left(\Rhat^{2,2}(\st_{2,2})\right)^{\widehat{\otimes} b}[[t_1,\ldots,t_s]]
\]
for some integer $s\ge 0$. In particular, $R_\infty^B$ is a Cohen--Macaulay normal domain, and (as $M_\infty\ne 0$) Theorem \ref{thm:patching}(4) gives that $M_\infty$ has generic rank $1$ over $R_\infty^B$.

Let $\Rhat_{\F}^{1,2,1}(\st_{2,2}) := \Rhat^{1,2,1}(\st_{2,2})/(\varpi)$, $\Rhat_{\F}^{2,2}(\st_{2,2}) := \Rhat^{2,2}(\st_{2,2})/(\varpi)$ and
\[
\Rbar^B_\infty := R^B_\infty/(\varpi) = \left(\Rhat_{\F}^{1,2,1}(\st_{2,2})\right)^{\widehat{\otimes} a} \widehat{\otimes} \left(\Rhat_{\F}^{2,2}(\st_{2,2})\right)^{\widehat{\otimes} b}[[t_1,\ldots,t_s]].
\]
Then $\Rbar^B_\infty$ is also a Cohen--Macaulay normal domain. Also let $\Mbar_\infty := M_\infty/\varpi M_\infty$. Then $\Mbar_\infty$ is also a maximal Cohen--Macaulay $\Rbar^B_\infty$-module of generic rank $1$.

Let $\Mc$ be the $\Rct^{1,2,1}(\st_{2,2})$-module from Theorem \ref{thm:R121 main results}(\ref{R121:M}), and let $\Mbar := \Mc\otimes_{\Rct^{1,2,1}(\st_{2,2})}\Rhat_{\F}^{1,2,1}(\st_{2,2})$.

The main result of this section is the following:

\begin{thm}\label{thm:M_infty structure}
There is an isomorphism
\[
\Mbar_\infty \cong \Mbar^{\boxtimes a}\boxtimes \Rhat_{\F}^{2,2}(\st_{2,2})^{\boxtimes b}[[t_1,\ldots,t_s]]
\]
of $\Rbar^B_\infty$-modules.
\end{thm}

\begin{rem}
It is likely that one could in fact prove the integral version of this theorem, namely that $M_\infty \cong M^{\boxtimes a}\boxtimes \Rhat^{2,2}(\st_{2,2})^{\boxtimes b}[[t_1,\ldots,t_s]]$, where $M := \Mc\otimes_{\Rct^{1,2,1}(\st_{2,2})}\Rhat^{1,2,1}(\st_{2,2})$, by a similar method. We have refrained from doing so here to avoid having to prove a mixed characteristic version of Proposition \ref{prop:class group completion}. Fortunately knowing the structure of $\Mbar_\infty$ is sufficient to prove Theorems \ref{main theorem} and \ref{thm:balanced}, so the full structure of $M_\infty$ is not needed for our results.
\end{rem}

We will prove Theorem \ref{thm:M_infty structure} by a class group argument, similar to the one used in \cite[Theorem 3.3]{Manning}. Before proving this, we shall note that Theorem \ref{thm:M_infty structure} implies Theorems \ref{main theorem} and \ref{thm:balanced}:

\begin{proof}[Proof of Theorems \ref{main theorem} and \ref{thm:balanced}]
By Theorem \ref{thm:patching}(3) we have
\[
M(K)_{\mf}/\mf M(K)_{\mf} \cong \frac{M_\infty/(y_1,\ldots,y_r)M_\infty}{\mf M_\infty/(y_1,\ldots,y_r)M_\infty} = M_\infty/\mf_{R_\infty^B}M_{\infty} = \Mbar_\infty/\mf_{\Rbar_\infty^B}\Mbar_{\infty}.
\]
Now the description in Theorem \ref{thm:M_infty structure} gives
\begin{align*}
\Mbar_\infty/\mf_{\Rbar_\infty^B}\Mbar_{\infty} 
&\cong \Mbar_\infty\otimes_{\Rbar^B_\infty}\F
 \cong  \left(\Mbar^{\boxtimes a}\boxtimes \Rhat_{\F}^{2,2}(\st_{2,2})^{\boxtimes b}[[t_1,\ldots,t_s]]\right) \otimes_{\Rbar_\infty^B} \F\\
&\cong\left(\Mbar\otimes_{\Rhat^{1,2,1}_{\F}(\st_{2,2})} \F\right)^{\otimes a}\otimes_{\F}
 \left(\Rhat_{\F}^{2,2}(\st_{2,2})\otimes_{\Rhat^{2,2}_{\F}(\st_{2,2})} \F\right)^{\otimes b}\\
&\cong\left(\F^{\oplus 2}\right)^{\otimes a}\otimes_{\F}
\F^{\otimes b}
\cong \F^{\oplus 2^a}
\end{align*}
where we used the fact that
\[
\Mbar\otimes_{\Rhat^{1,2,1}_{\F}(\st_{2,2})} \F \cong \Mc\otimes_{\Rct^{1,2,1}(\st_{2,2})} \F \cong \Mc/(\Ic_+\Mc+\varpi\Mc) \cong \frac{\Mc/\Ic_+\Mc}{\varpi(\Mc/\Ic_+\Mc)} \cong \frac{\Oc^{\oplus 2}}{\varpi \Oc^{\oplus 2}} \cong \F^{\oplus 2}
\]
by Theorem \ref{thm:R121 main results}(\ref{R121:mult 2}).

Thus $\dim_\F M(K)_{\mf}/\mf M(K)_{\mf} = 2^a$, proving the first part of Theorem \ref{main theorem}. At this point, one could deduce the second part of Theorem \ref{main theorem} using the fact that $\Rhat^{1,2,1}(\st_{2,2})$ and $\Rhat^{2,2}(\st_{2,2})$ are Cohen--Macaulay by the strategy from \cite[Section 5]{Snowden}. However this fact will also follow from our argument below.

For any Cohen--Macaulay ring $A$ and any finitely generated $A$-module $M$, we will let $\tau_M$ denote the natural map $M\otimes_A \Hom_A(M,\omega_A) \to \omega_A$, where $\omega_A$ is the dualizing module of $A$. By the argument in \cite[Proposition 4.14]{Manning} (using \cite[Lemmas 2.4, 2.6]{EmTheta}) showing that $\tau_{M_\infty}$ is surjective will imply that the map
\[
R_\infty^B/(y_1,\ldots,y_r)R_\infty^B \to \End_{R_\infty^B}(M_\infty/(y_1,\ldots,y_r)M_\infty) = \End_{\T(K)_{\mf}}(M(K)_{\mf})
\]
is an isomorphism. As this map factors as
\[
R_\infty^B/(y_1,\ldots,y_r)R_\infty^B\onto R^B_{\Sigma,\lambda,\mu}\onto \T(K)_{\mf}\to \End_{\T(K)_{\mf}}(M(K)_{\mf})
\]
this implies that $R_\infty^B/(y_1,\ldots,y_r)R_\infty^B\onto  R^B_{\Sigma,\lambda,\mu}$, $R^B_{\Sigma,\lambda,\mu}\onto \T(K)_{\mf}$ and $\T(K)_{\mf}\onto \End_{\T(K)_{\mf}}(M(K)_{\mf})$ are all isomorphisms, which proves the second part of Theorem \ref{main theorem} and Theorem \ref{thm:balanced}.

Note that by \cite[Lemma 3.4]{Manning}, $\tau_M$ is surjective if and only if there exists an $A$-module surjection $M\otimes_A \Hom_A(M,\omega_A) \to \omega_A$.

Let $\omegabar^{1,2,1}$ denote the dualizing module of $\Rhat_{\F}^{1,2,1}(\st_{2,2})$. As $\Rhat_{\F}^{2,2}(\st_{2,2})$ is Gorenstein we have 
\[
\omega_{\Rbar_\infty^B} \cong \left(\omegabar^{1,2,1}\right)^{\boxtimes a}\boxtimes \Rhat_{\F}^{2,2}(\st_{2,2})^{\boxtimes b}[[t_1,\ldots,t_s]].
\]
By Theorem \ref{thm:R121 main results}(\ref{R121:surj}), $\tau_{\Mc}$ is surjective. Tensoring with $\Rhat_{\F}^{1,2,1}(\st_{2,2})$ this implies that $\tau_{\Mbar}:\Mbar\otimes_{\Rhat_{\F}^{1,2,1}(\st_{2,2})}\Mbar \to \omegabar^{1,2,1}$ is surjective. Thus by Theorem \ref{thm:M_infty structure} and the above description of $\omega_{\Rbar_\infty^B}$ we may construct a surjection $\Mbar_{\infty}\otimes_{\Rbar_\infty^B} \Mbar_{\infty}\onto \omega_{\Rbar_\infty^B}$, which implies that $\tau_{\Mbar_\infty}$ is surjective.

Now as $\varpi$ is a regular element on $R_\infty^B$, and hence on $M_\infty$ and $\omega_{R_\infty^B}$, it follows that $\tau_{\Mbar_\infty}$ is the mod-$\varpi$ reduction of $\tau_{M_\infty}$ (as in \cite[Section 3A]{Manning}), and so by Nakayama's Lemma, $\tau_{M_\infty}$ is also surjective, completing the proof of Theorem \ref{thm:balanced}.
\end{proof}

\subsection{Class groups}

We will now recall the key properties of class groups that will be needed in the proof of Theorem \ref{thm:M_infty structure}.

Fix a noetherian, normal domain $A$. We will say that a finitely generated $A$-module $M$ is \emph{reflexive} if the natural map $M\to \Hom_A(\Hom_A(M,A),A)$ is an isomorphism. It is well known that if $M$ is finitely generated, then $\Hom_A(M,A)$ is always reflexive, so in fact a finitely generated $A$-module $M$ is reflexive if any only if $M\cong \Hom_A(\Hom_A(M,A),A)$.

Let $K(A)$ be the fraction field of $A$. We will say that the \emph{generic rank} of an $A$-module $M$ is $\dim_{K(A)}(M\otimes_AK(A))$.

Let $\Cl(A)$ denote the set of (isomorphism classes) of finitely generated reflexive $A$-modules of generic rank $1$. Then $\Cl(A)$ has the structure of an abelian group with group operation:
\[
M + N := \Hom_A(\Hom_A(M\otimes_A N,A),A).
\]
Now further assume that $A$ is Cohen--Macaulay. Let $\omega_A$ be its dualizing module. For any finitely generated $A$-module $M$, let $M^* = \Hom_A(M,\omega_A)$. Observe that any finitely generated maximal Cohen--Macaulay $A$-module is automatically reflexive, and hence $\omega_A\in \Cl(A)$. Moreover it is well known that for any $M\in\Cl(A)$, $M^*\in\Cl(A)$ as well and $M^* = \omega - M$ in $\Cl(A)$. 

In particular, $M$ is self-dual (in the sense that $M\cong M^*$) if and only if $2M = \omega$ in $\Cl(A)$. This means that if $\Cl(A)$ is $2$-torsion free, then there is at most one solution to the equation $2x = \omega$ in $\Cl(A)$, and so there is at most one self-dual reflexive $A$-module of generic rank $1$.

The following two results will allow us to compute $\Cl(\Rbar_\infty^B)$:

\begin{prop}\label{prop:Cl(X x Y)}
Let $\Bbbk$ be a field, and take $\Bbbk$-algebras $A_1,\ldots,A_m$ which are noetherian, normal domains, and such that each $\Spec A_i$ is geometrically integral.

Assume that for each $i$, there is a closed subscheme $\Zc_i\subseteq \Spec A_i$ of codimension at least $1$ such that each component of $\Zc_i$ is geometrically integral, and $(\Spec A_i)\sm \Zc_i$ is isomorphic to an open affine subset of $\A^{d_i}_\Bbbk$ for some $d_i$, which contains at least one point of $\A^{d_i}_\Bbbk(\Bbbk)$.

Then the natural map
\[
\varphi:\Cl(A_1)\oplus \Cl(A_2)\oplus\cdots\oplus \Cl(A_m)\to \Cl(A_1\otimes_{\Bbbk} A_2\otimes_{\Bbbk} \cdots \otimes_{\Bbbk}A_m)
\]
given by $(M_1,M_2,\ldots,M_m)\mapsto M_1\boxtimes M_2\boxtimes \cdots \boxtimes M_m$ is an isomorphism.
\end{prop}
\begin{proof}
For each $i$, let $\Uc_i = (\Spec A_i)\sm \Zc_i$ and fix some point $p_i\in \Uc_i(\Bbbk)\subseteq (\Spec A_i)(\Bbbk)$, interpreted as a prime ideal $p_i\subseteq A_i$ with $A_i/p_i = \Bbbk$. Since $\Uc_i$ is isomorphic to a Zariski open subset of $\A^{d_i}_{\Bbbk}$, $\Spec A_i$ is regular at $p_i$. It follows that if $M_i$ is any reflexive, generic rank $1$ $A_i$-module, then $M_i/p_iM_i = \F$.

For any $j=1,\ldots,m$, let 
\[I_j = p_1\boxtimes \cdots \boxtimes p_{j-1} \boxtimes 0 \boxtimes p_{j+1}\boxtimes \cdots \boxtimes p_m\in \Spec (A_1\otimes_{\Bbbk} \cdots \otimes_{\Bbbk}A_m),\]
and note that $(A_1\otimes_{\Bbbk} \cdots \otimes_{\Bbbk}A_m)/I_j \cong \Bbbk\otimes_{\Bbbk}\cdots\otimes_{\Bbbk}\Bbbk  \otimes_{\Bbbk}A_j\otimes_{\Bbbk}\Bbbk\otimes_{\Bbbk}\cdots \otimes_{\Bbbk}\Bbbk \cong A_j$.

Then for any $(M_1,\ldots,M_m)\in \Cl(A_1)\oplus\cdots\oplus \Cl(A_m)$ we get that
\begin{align*}
\frac{M_1\boxtimes \cdots \boxtimes M_m}{I_j(M_1\boxtimes \cdots \boxtimes M_m)} &\cong \frac{M_1}{p_1M_1}\boxtimes \cdots \boxtimes \frac{M_{j-1}}{p_{j-1}M_{j-1}} \boxtimes \frac{M_j}{0} \boxtimes \frac{M_{j+1}}{p_{j+1}M_{j+1}}\boxtimes \cdots \boxtimes \frac{M_m}{p_1M_m}\\
&\cong \Bbbk\boxtimes \cdots \boxtimes \Bbbk\boxtimes M_j\boxtimes \Bbbk\boxtimes \cdots \boxtimes \Bbbk
\cong M_j
\end{align*}
as $A_j$-modules. It follows that $\varphi$ is indeed injective.

For surjectivity, write $\Xc = \Spec(A_1\otimes_{\Bbbk} \cdots \otimes_{\Bbbk}A_m) = \Spec(A_1)\times_{\Bbbk}\cdots\times_{\Bbbk} \Spec(A_m)$ and let
\begin{align*}
\Zc &= \bigcup_{j=1}^m\left(\Zc_j\times_{\Bbbk}\prod_{i\ne j}\Spec(A_i)\right)&
&\text{and}&
\Uc &= \Xc\sm \Zc = \prod_{i=1}^m\Uc_i.
\end{align*}
Now let $\Ic$ be the set of irreducible components of $\Zc$ which have codimension $1$ in $\Xc$. By \cite[Proposition II.6.5]{Hartshorne}, there is an exact sequence
\[\Z^\Ic\xrightarrow{\psi} \Cl(\Xc)\onto\Cl(\Uc)\to 0\] 
where $\psi:\Z^{\Ic}\to \Cl(\Xc):= \Cl(A_1\otimes_{\Bbbk}\cdots \otimes_{\Bbbk} A_m)$ sends each $D\in \Ic$  to the reflexive module associated to $D$.

But now as each $\Uc_i$ is a Zariski open subset of $\A^{d_i}_{\Bbbk}$, $\Uc$ is a Zariski open subset of $\A^{d_1+\cdots+d_m}_{\Bbbk}$. Hence (by \cite[Proposition II.6.5]{Hartshorne} again), there is a surjective map $0 = \Cl(\A^{d_1+\cdots+d_m}_{\Bbbk})\onto \Cl(\Uc)$, and so $\Cl(\Uc) = 0$, which implies that $\psi:\Z^{\Ic}\to \Cl(\Xc)$ is surjective, and so $\Cl(\Xc)$ is generated by the set $\{\psi(D)|D\in \Ic\}$.

But now by assumption, each irreducible component of each $\Zc_j$ is geometrically integral, as is each $\Spec A_i$. It follows that every irreducible component of $Y\subset\Zc$ must be in the form $Y = Y_j\times_{\Bbbk}\prod_{i\ne j}\Spec(A_i)$ for some irreducible component $Y_j\subseteq \Zc_j$. It now easily follows (as $Y$ will have codimension $1$ in $\Xc$ if and only if $Y_j$ has codimension $1$ in $\Spec A_j$) that for any $D\in\Ic$, its image $\psi(D)\in \Cl(\Xc)$ is in the image of the map 
\begin{align*}
\Cl(A_j)&\to \Cl(A_1)\oplus \cdots\oplus \Cl(A_m)\to \Cl(\Xc)=\Cl(A_1\otimes_{\Bbbk}\cdots\otimes_{\Bbbk}A_m)\\
M_j&\mapsto (A_1,\ldots,M_j,\ldots,A_j)\mapsto A_1\boxtimes \cdots \boxtimes M_j\boxtimes \cdots \boxtimes A_m
\end{align*}
for some $j$, and so is in the image of $\varphi$. It follows that $\varphi$ is indeed surjective, and hence is an isomorphism. 
\end{proof}

\begin{prop}\label{prop:class group completion}
Let $\Bbbk$ be a field and let $A_1,\ldots,A_m$ be finitely generated \emph{graded} $\Bbbk$-algebras which are noetherian normal domains.

Assume that for each $i$, $\Proj A_j$ is a (geometrically integral) projective variety of dimension at least $3$.

Let $R = (A_1\otimes_{\Bbbk} A_2\otimes_{\Bbbk} \cdots \otimes_{\Bbbk} A_m)[t_1,\ldots,t_s]$ for some integer $s\ge 0$, and let $\Rhat$ denote the completion of $R$ as a graded $\Bbbk$-algebra.

Then if each $A_i$ is Cohen--Macaulay, the natural map $\Cl(R)\to \Cl(\Rhat)$ given by $M\mapsto (M\otimes_R\Rhat)$ is an isomorphism.
\end{prop}
\begin{proof}
This is simply \cite[Proposition 3.19]{Manning}, combined with the observation that (by \cite[Exercise 18.16]{Eisenbud}) the conditions that $A_i$ is Cohen--Macaulay and $\dim \Proj A_i \ge 3$ imply all of the necessary vanishing results for sheaf cohomology.
\end{proof}

We are now ready to prove Theorem \ref{thm:M_infty structure}:

\begin{proof}[Proof of Theorem \ref{thm:M_infty structure}]
Let $\Rct^{2,2}_{\F}(\st_{2,2})$ and $\Rct^{1,2,1}_{\F}(\st_{2,2})$ be the rings from Theorems \ref{thm:R22 main results} and \ref{thm:R121 main results}, and let
\[
\Rc_\infty^B := \left(\Rct^{1,2,1}_{\F}(\st_{2,2})\right)^{\widehat{\otimes} a} \widehat{\otimes} \left(\Rct^{2,2}_{\F}(\st_{2,2})\right)^{\widehat{\otimes} b}[t_1,\ldots,t_s].
\]
Then $\Rc^B_\infty$ is a finitely generated graded $\F$-algebra, and the completion of $\Rc^B_\infty$ at the irrelevant ideal is $\Rbar^B_\infty$.

Now by Propositions \ref{prop:Cl(X x Y)} and \ref{prop:class group completion}, using Theorems \ref{thm:R22 main results} and \ref{thm:R121 main results} and in particular the fact that $\Cl(\Rct^{2,2}_{\F}(\st_{2,2})) = 0$ and $\Cl(\Rct^{1,2,1}_{\F}(\st_{2,2}))\cong \Z$, we have an isomorphism
\[
\Cl(\Rbar_\infty^B)\cong \Cl(\Rc^B_\infty)\cong \Cl(\Rct^{1,2,1}_{\F}(\st_{2,2}))^{\oplus a}\oplus \Cl(\Rct^{2,2}_{\F}(\st_{2,2}))^{\oplus b}\cong \Z^a.
\]
In particular, $\Cl(\Rbar_\infty^B)$ is $2$-torsion free, and hence there is, up to isomorphism, at most one self-dual $\Rbar_\infty^B$-module of generic rank $1$. 

Now as $\Mbar$ is a self-dual $\Rhat_{\F}^{1,2,1}(\st_{2,2})$-module of generic rank $1$, and $\Rhat_{\F}^{2,2}(\st_{2,2})$ is Gorenstein, it follows that
\[
\Mbar^{\boxtimes a}\boxtimes \Rhat_{\F}^{2,2}(\st_{2,2})^{\boxtimes b}[[t_1,\ldots,t_s]]
\]
is self-dual with generic rank $1$.

Also by Theorem \ref{thm:patching}, $M_\infty$ is a self-dual $R_\infty^B$ module of generic rank $1$, and hence $\Mbar_\infty$ is also a self-dual $\Rbar_\infty^B$ module of generic rank $1$. Thus by the above, we indeed have
\[\Mbar_\infty \cong \Mbar^{\boxtimes a}\boxtimes \Rhat_{\F}^{2,2}(\st_{2,2})^{\boxtimes b}[[t_1,\ldots,t_s]],\]
completing the proof.
\end{proof}

\begin{rem}
The proof of Theorem \ref{thm:M_infty structure} relied on a number of strong assumptions about the rings $\Rhat^{1,2,1}(\st_{2,2})$ and $\Rhat^{2,2}(\st_{2,2})$, in particular that they are Cohen--Macaulay. While this was provable in the special cases needed for our main results in this paper, proving that components of local Galois deformation are Cohen--Macaulay is quite difficult in general, which would make it difficult to directly generalize this method.

However these requirements are far less necessary than may appear to be based on the above argument, and can likely be weakened significantly with a little extra work.

In particular, the Cohen--Macaulay assumption in Proposition \ref{prop:class group completion} was needed only to ensure that certain sheaf cohomology groups vanish. The argument in \cite[Proposition 3.19]{Manning} in fact shows that without this assumption, the cokernel of the map $\Cl(R)\to \Cl(\Rhat)$ can be described explicitly as a inverse limit of certain sheaf cohomology groups, each of which is a finite dimensional $\Bbbk$-vector space. Provided that $\Char \Bbbk \ne 2$, this implies that $\Cl(\Rhat)[2]\cong \Cl(R)[2]$, so if our goal is simply to prove that $\Cl(\Rhat)$ is $2$-torsion free, the Cohen--Macaulay assumption is not needed. Note that in the case when $\Rbar_\infty^B$ is not Cohen--Macaulay one must replace the definition of $M^*$ by $M^*\cong \RHom_{\Rbar_\infty^B}(M,\omega^\bullet_{\Rbar_\infty^B})$, however as $M^*$ is automatically concentrated in a single degree whenever $M$ is maximal Cohen--Macaulay, one can still interpret $M^*$ in terms of $\Cl(\Rbar_\infty^B)$ whenever $M$ is maximal Cohen--Macaulay, so this is only a minor inconvenience.

The smoothness and normality assumptions can also likely be weakened to the condition that $\Rbar_\infty^B$ is regular in codimension $1$. Indeed, if $(\Spec \Rbar_\infty^B)^{\mathrm{sm}}$ is the smooth locus of $\Spec \Rbar_\infty^B$, then  $(\Spec \Rbar_\infty^B)\sm (\Spec \Rbar_\infty^B)^{\mathrm{sm}}$ has codimension $2$. As $\Mbar_\infty$ is maximal Cohen--Macaulay over $\Rbar_\infty^B$, $\Mbar_\infty$ is uniquely determined by the restriction $\Mbar_\infty|_{(\Spec \Rbar_\infty^B)^{\mathrm{sm}}}$, and hence it suffices to determine the class group of $(\Spec \Rbar_\infty^B)^{\mathrm{sm}}$, which is automatically normal and smooth. It is likely that a careful analysis of the situation will allow us to prove that $\Cl((\Spec \Rbar_\infty^B)^{\mathrm{sm}})$ is $2$-torsion free in much greater generality than we can prove that $\Rbar_\infty^B$ itself is normal.
\end{rem}
\part{Computations}\label{part:computations}

To prove our main results, it remains to prove Theorems \ref{thm:R22 main results} and \ref{thm:R121 main results}, which we will do in Sections \ref{sec:R22} and \ref{sec:R121}, by analyzing the spaces $\Nc^{2,2}(\st_{2,2})$ and $\Nc^{1,2,1}(\st_{2,2})$.

\section{Commutative algebra preliminaries}\label{sec:comm alg}

Here we present the main commutative algebra results that will be needed for our computations in Section \ref{sec:R22} and \ref{sec:R121}.

\subsection{Graded algebras}

Let $R$ be a commutative ring. By a \emph{graded $R$-algebra} we will always mean a commutative graded ring $A = A_0\oplus A_1 \oplus A_2\oplus \cdots$, with $A_0 = R$ (which will then make $A$ into an $R$-algebra).

We will say that the \emph{irrelevant ideal} of $A$ is the ideal $A_+ = A_1\oplus A_2\oplus \cdots$ generated by the elements of positive degree. Then $A/A_+ \cong R$, so $A_+$ is a prime ideal if $R$ is a domain, and is a maximal ideal if $R$ is a field.

Most of the rings we will consider in Sections \ref{sec:R22} and \ref{sec:R121} will be finitely generated graded $\Oc$-algebras. For any finitely generated graded $\Oc$-algebra $A$, define $A_{\F} := A/\varpi A$ and $A_E := A[1/\varpi]$, and note that $A_\F$ is a finitely generated graded $\F$-algebra and $A_E$ is a finitely generated graded $E$-algebra. As it is often easier to work with algebras over a field rather than over $\Oc$, many of our arguments will works with the rings $A_\F$ and $A_E$ rather than $A$ directly. In this section, we will introduce some results that allow us to deduce properties of $A$ from those of $A_\F$ and $A_E$.

\begin{lemma}\label{lem:graded minimal primes}
Let $R$ be a local Noetherian ring with maximal ideal $\mf_R$. Let $A$ be a finitely generated graded $R$-algebra. Consider the maximal ideal $\mf_A = \mf_R A + A_+$ of $A$. Then:
\begin{enumerate}
	\item the localization map $A\to A_{\mf_A}$ is injective;
	\item The map $\Pf\mapsto \Pf_{\mf_A}$ induces a bijection between the minimal primes of $A$ and those of $A_{\mf_A}$.
	\item For each minimal prime $\Pf$ of $A$, the rings $A/\Pf$ and $A_{\mf_A}/\Pf_{\mf_A}$ have the same dimension.
\end{enumerate} 
\end{lemma}
\begin{proof}
Take any $s\in A\sm \mf_A$, and write $s = s_0+s_1+\cdots$ with $s_i\in A_i$. By the definition of $\mf_A$, $s_0\in A_0\sm \mf_A = R\sm \mf_R = R^\times$.

We claim that $s$ is not a zero divisor in $A$. Indeed, assume that $sa = 0$ for some $a\in A\sm \{0\}$. Write $a = a_0+a_1+\cdots$ for $a_i \in A_i$. Let $j\in \Z_{\ge 0}$ be the smallest integer with $a_j\ne 0$. Then
\[
0=sa =(s_0+s_1+\cdots)(a_j+a_{j+1}+\cdots)= s_0a_j+(s_1a_j+s_0a_{j+1})+\cdots
\]
This implies that each homogeneous piece of $sa$ is equal to $0$, so $s_0a_j = 0$. But now as $A_j$ is a $R$-module and $s_0\in R^\times$, this implies that $a_j = s_0^{-1}(s_0a_j) = 0$ in $A_j$, contradicting the choice of $j$. So indeed $s$ is not a zero divisor.

Claims (1) and (2) immediately follow from this. Indeed, if $f/1 = 0$ in $A_{\mf_A}$ for some $f\in A$, then $sf = 0$ for some $s\in A\sm \mf_A$ by definition, and so $f=0$, proving (1).

For (2), the localization map $\Pf \mapsto \Pf_{\mf_A}$ induces a bijection between the prime ideals of $A$ contained in $\mf_A$ and the prime ideals of $A_{\mf_A}$. But any minimal prime $\Pf$ of $A$ consists entirely of zero divisors in $A$, and so $\mf_A$ contains all minimal primes of $A$, proving (2).

Also note that any minimal prime $\Pf$ of $A$ must be homogeneous by \cite[\href{https://stacks.math.columbia.edu/tag/00JU}{Lemma 00JU}]{stacks-project}.

For (3), write $C = A/\Pf$ and $\mf_C = \mf_A/\Pf$. As $C$ is a domain, the map $\Oc\to C$ has image either $\F = \Oc/\varpi$ or $\Oc$. In the first case, $C$ is a finite type $\F$-algebra which is an integral domain, and so $\dim A_{\mf_A}/\Pf_{\mf_A} = \dim C_{\mf_C} = \dim C = \dim A/\Pf$ by \cite[\href{https://stacks.math.columbia.edu/tag/00OS}{Lemma 00OS}]{stacks-project}. 

In the second case, $\Pf\cap A_0 = 0$, so as $\Pf$ is homogeneous, $\Pf\subseteq A_+$. Let $\pf = A_+/\Pf\subseteq B$. Then $\pf$ is a prime ideal with $\pf\subsetneq \mf_C$ and so 
\[\dim A_{\mf}/\Pf_{\mf} = \dim C_{\mf_C}= \height \mf_C \ge 1+\height \pf = 1+\dim C_{\pf}.\]
But now $C_{\pf}$ is the localization of the integral domain $C[1/\varpi]$ at the maximal ideal $\pf[1/\varpi]$ and so by \cite[\href{https://stacks.math.columbia.edu/tag/00OS}{Lemma 00OS}]{stacks-project} again $\dim C_{\pf} = \dim C[1/\varpi]\ge \dim C - 1$. Combining this gives $\dim C_{\mf_C} \ge \dim C$ and so $\dim C_{\mf_C} = \dim C$, again giving $\dim A_{\mf_A}/\Pf_{\mf_A} = \dim A/\Pf$. This proves (3).
\end{proof}

\begin{prop}\label{prop:graded O flat}
Let $A$ be a finitely generated graded $\Oc$-algebra. Assume that $A[1/\varpi]$ and $A/\varpi A$ are equidimensional of the same dimension and have the same number of minimal primes, and that $A/\varpi A$ is reduced. Then $A$ is flat over $\Oc$ and equidimensional, and the minimal primes of $A$ are in bijection with those of $A[1/\varpi]$. Furthermore if $A[1/\varpi]$ is reduced then $A$ is reduced as well.
\end{prop}
\begin{proof}
Note that all of the remaining claims follow from the claim that $A$ is $\Oc$-flat. Certainly this implies that $A\subseteq A[1/\varpi]$, so if $A[1/\varpi]$ is reduced then $A$ is reduced. Furthermore if  $\Pf\subseteq A$ is any minimal prime, then $\varpi\not\in \Pf$, and so $\Pf[1/\varpi]$ is a minimal prime of $A[1/\varpi]$ and we have $\dim A[1/\varpi]/\Pf[1/\varpi] = \dim (A/\Pf)[1/\varpi] = \dim A/\Pf - 1$. Moreover every minimal prime of $A[1/\varpi]$ can be written as $\Pf[1/\varpi]$ for a minimal prime $\Pf$ of $A$. So indeed $A$ is equidimensional of dimension $\dim A[1/\varpi]+1$, and its minimal primes are in bijection with those of $A[1/\varpi]$.

So it remains to show that $A$ is $\Oc$-flat, or equivalently that $\varpi$ is not a zero divisor on $A$.

Let $\mf_A = \varpi A + A_+$ be as in Lemma \ref{lem:graded minimal primes} and let $B = A_{\mf_A}$. By Lemma \ref{lem:graded minimal primes}(1), there is an injection $A\into B$, so it will suffice to prove that $\varpi$ is not a zero divisor on $B$. We shall prove this via \cite[Proposition 2.2.1]{Snowden}. Note that $B$ is Catenary, as it is the localization of a finitely generated $\Oc$-algebra.

Consider the finitely generated graded $\F$- and $E$-algebras $A_{\F} = A/\varpi A$ and $A_E = A[1/\varpi]$ from before. Let $A_{\F,+}$ and $A_{E,+}$ be their irrelevant ideals.

We have $B/\varpi B = A_{\mf_A}/\varpi A_{\mf_A} = (A/\varpi A)_{\mf_A} = (A_\F)_{A_{\F,+}}$. Thus by Lemma \ref{lem:graded minimal primes} (with $R = \F$) we get that $B/\varpi B$ is equidimensional of dimension $\dim A_\F$ and has the same number of minimal primes as $A_\F$. As as $A_\F$ is reduced by assumption, its localization $B/\varpi B$ is reduced as well.

Also we have a sequence of localizations $A_E \to B[1/\varpi] \to (A_E)_{A_{E,+}}$, as every element of $A\sm \mf_A$ is certainly a unit in $(A_E)_{A_{E,+}}$. By Lemma \ref{lem:graded minimal primes} again, $(A_E)_{A_{E,+}}$ is equidimensional of dimension $\dim A_E$ with the same number of minimal primes as $A_E$. But this is only possible if $B[1/\varpi]$ is also equidimensional of dimension $\dim A_E$ with the same number of minimal primes as $A_E$ and $(A_E)_{A_{E,+}}$.

Thus as, by assumption, $\dim A_\F = \dim A_E$ and $A_\F$ and $A_E$ have the same number of minimal primes, \cite[Proposition 2.2.1]{Snowden} indeed applies to $B$, giving that $\varpi$ is not a zero divisor on $\varpi$, completing the proof.
\end{proof}

\begin{cor}\label{cor:graded domain}
Let $A$ be a finitely generated graded $\Oc$-algebra. Assume that $A[1/\varpi]$ and $A/\varpi A$ are both domains of dimension $d$. Then $A$ is domain of dimension $d+1$, which is flat over $\Oc$.
\end{cor}
\begin{proof}
This is simply the special case of Proposition \ref{prop:graded O flat} in which $A/\varpi A$ and $A[1/\varpi]$ are each reduced with a single minimal prime.
\end{proof}

\begin{lemma}\label{lem:I cap J = IJ}
Let $R$ be a ring and let $I,J\subseteq R$ be two ideals for which $R/I$ and $R/J$ are both reduced. Then $I\cap J = IJ$ in $R$ if and only if $R/IJ$ is reduced.
\end{lemma}
\begin{proof}
Since $R/I$ and $R/J$ are reduced, $I$ and $J$ are radical ideals, i.e. $I = \sqrt{I}$ and $J = \sqrt{J}$. It follows that $\sqrt{I\cap J} = \sqrt{I}\cap \sqrt{J} = I\cap J$. Now as $I\cap J\subseteq I,J$ we get $(I\cap J)^2 \subseteq IJ \subseteq I\cap J$ and so $\sqrt{IJ} = \sqrt{I\cap J} = I\cap J$. Hence $I\cap J = IJ$ if and only if $IJ = \sqrt{IJ}$, which indeed happens if and only if $R/IJ$ is reduced.
\end{proof}

\begin{prop}\label{prop:graded p cap q = pq}
Let $A$ be a finitely generated graded $\Oc$-algebra, and let $A_\F = A/\varpi A$ and $A_E = A[1/\varpi]$. Assume that $\dim A_\F = \dim A_E$. Let $\pf,\qf\subseteq A$ be two homogeneous ideals. Let $\pf_{\F},\qf_{\F}\subseteq A_\F$ be the images of $\pf$ and $\qf$ in $A_\F$, and let $\pf_E = \pf[1/\varpi], \qf_E = \qf[1/\varpi] \subseteq A_E$. Assume that there is some integer $d\ge 0$ such that for each $k = \F,E$ the following conditions are satisfied:
\begin{itemize}
	\item $\pf_k$ and $\qf_k$ are distinct prime ideals of $A_k$ of height $d$;
	\item $\pf_k\cap \qf_k = \pf_k\qf_k$ in $A_k$.
\end{itemize}
Then $\pf\cap \qf = \pf\qf$ in $A$.
\end{prop}
\begin{proof}
By Lemma \ref{lem:I cap J = IJ}, it is enough to prove that the rings $C_1:=A/\pf$, $C_2:=A/\qf$ and $C_3:=A/\pf\qf$ are reduced. Let $I_1 = \pf$, $I_2 = \qf$ and $I_3 = \pf\qf$, so that $C_i = A/I_i$. Also write $I_{1,k} = \pf_k$, $I_{2,k} = \qf_k$ and $I_{3,k} = \pf_k\qf_k$ for each $k =\F,E$ and set $C_{i,k} = A_k/I_{i,k}$. We claim that for each $i=1,2,3$, $C_{i,\F} = C_i/\varpi C_i$ and $C_{i,E} = C_i[1/\varpi]$. 

For $C_{i,\F}$, note that note that for each $i$, $I_{i,\F} = (I_i+\varpi A)/\varpi A$, that is, $I_{i,\F}$ is the image of $I_i$ under the quotient map $\pi:A\onto A_{\F}$. For $i=1,2$ this is by the definition of $\pf_\F$ and $\qf_\F$. For $i=3$ it follows from $\pi(I_3) = \pi(\pf\qf) = \pi(\pf)\pi(\qf) = \pf_\F\qf_\F = I_{3,\F}$. Thus 
\[C_i/\varpi C_i = (A/I_i)/\varpi (A/ I_i) = (A/\varpi A)/((I_i+\varpi A)/\varpi A) = A_\F/I_{i,\F}=C_{i,\F}.\]
For $C_{i,E}$, we have
\[
C_i[1/\varpi] = (A/I_i)[1/\varpi] = A[1/\varpi]/I_i[1/\varpi] = A_E/I_i[1/\varpi].
\]
For $i=1,2$ we have $I_i[1/\varpi] = I_{i,E}$ by definition. For $i=3$, inverting $\varpi$ in the map $\pf\otimes_A \qf\to A$ gives $I_{3}[1/\varpi] = \pf[1/\varpi]\qf[1/\varpi] = I_{3,E}$. So indeed, $C_i[1/\varpi] = C_{i,E}$ for each $i$.

Now take any $k= \F,E$. By assumption $C_{1,k} = A_k/\pf_k$ and $C_{2,k} = A_k/\pf_k$ are both domains of dimension $d$. By Corollary \ref{cor:graded domain}, $C_1$ and $C_2$ are domains, and in particular are reduced. Also by Lemma \ref{lem:I cap J = IJ}, the assumption that $\pf_k\cap\qf_k=\pf_k\qf_k$ implies that $C_{3,k} = A_k/\pf_k\qf_k$ is reduced. Also by assumption $C_{3,k}$ has two distinct minimal primes: $\pf_k$ and $\qf_k$, and $A_k/\pf_k$ and $A_k/\qf_k$ both have dimension $\dim A_\F - d = \dim A_E-d$. So for each $k=\F,E$, $C_{3,k}$ is reduced and equidimensional of dimension $\dim A_{\F}-d = \dim A_E-d$, with exactly two minimal primes. Hence by Proposition \ref{prop:graded O flat}, $C_3$ is reduced, completing the proof.
\end{proof}

\subsection{Class groups and dualizing modules}

Let $A$ be a \emph{normal} Noetherian domain. In this section, we will present some techniques for computing the class group $\Cl(A)$ and (in the case when $A$ is Cohen--Macaulay) the dualizing module $\omega_A$ of $A$.

Let $K(A)$ denote the fraction field of $A$. For any divisor $D$ of $A$, define the $A$-module
\[
\Oc(D) = \{f\in K(A)|\div(f)+D\ge 0\}.
\]
The assumption that $A$ is normal implies that $\Oc(0) = A$. Now note that we have:

\begin{lemma}\label{lem:O(-D) = P}
Let $\pf_1,\pf_2,\ldots,\pf_n\subseteq A$ be \emph{distinct} height $1$ primes of $A$. Let $D_i = \Spec A/\pf_i$. Then
\[
\Oc(-D_1-D_2-\cdots-D_n) = \pf_1\cap\pf_2\cap \cdots \cap \pf_n.
\]
as $A$-modules.
\end{lemma}
\begin{proof}
As $\Dc:=D_1+D_2+\cdots+D_n$ is an effective divisor, the condition $\div(f)\ge \Dc$ implies that $f\in \Oc(0) =  A$. Thus, using the fact that $D_1,D_2,\ldots,D_n$ are \emph{distinct} prime divisors, we indeed get
\begin{align*}
\Oc(-\Dc) 
&= \{f\in K(A)|\div(f)-\Dc\ge 0\}
 = \{f\in K(A)|\div(f)\ge \Dc\}\\
&= \{f\in A|\div(f)\ge \Dc\}= \{f\in A|\div(f)\ge D_1+D_2+\cdots+D_n\}\\
&= \bigcap_{i=1}^n \{f\in A|\div(f)\ge D_i\}= \bigcap_{i=1}^n \{f\in A|\ord_{\pf_i}(f)>0\}
 = \bigcap_{i=1}^n \pf_i.
\end{align*}
\end{proof}

Now write $\Div(A)$ for the group of divisors on $A$. This is the free abelian group generated by the prime divisors on $A$. Say that two divisors $D_1,D_2\in \Div(A)$ are \emph{linearly equivalent}, and write $D_1\sim D_2$ if $D_1-D_2 = \div(f)$ for some $f\in K(A)^\times$. It is well known that $\Oc(D_1) \cong \Oc(D_2)$ as $A$-modules if and only if $D_1\sim D_2$, and moreover that the map $D\mapsto \Oc(D)$ defines an isomorphism $\Div(A)/{\sim} \cong \Cl(A)$, where as before, $\Cl(A)$ is group of reflexive $A$-modules of generic rank $1$.

From now on, we will identify $\Cl(A)$ with $\Div(f)/{\sim}$. For $D\in \Div(A)$, write $[D]$ for the corresponding element $\Oc(D)$ of $\Cl(A)$.

Now we recall a few standard results that will be needed for our computations of class groups in Sections \ref{sec:R22} and \ref{sec:R121}:

\begin{lemma}\label{lem:Cl(UFD)}
If $A$ is a noetherian UFD, then $\Cl(A) = 0$.
\end{lemma}
\begin{proof}
This follows from \cite[\href{https://stacks.math.columbia.edu/tag/0AFT}{Lemma 0AFT}]{stacks-project}.
\end{proof}

\begin{lemma}\label{lem:Cl(X-Z)}
Let $A$ be a noetherian normal domain, and let $x\in A\sm\{0\}$ be any nonzero element. Say that
\[
\div(x) = \sum_{i=1}^n a_iD_i
\]
for distinct prime divisors $D_1,\ldots,D_n$ and $a_1,\ldots,a_n\ge 1$. Then there is an exact sequence
\[
\Z^n/\langle(a_1,\ldots,a_n)\rangle\to \Cl(A)\to \Cl(A[1/x]) \to 0
\]
where the first map is given by $(j_1,\ldots,j_n)\mapsto \sum_{i=1}^n j_iD_i$ and the second map is given by $M\mapsto M[1/x] := M\otimes_AA[1/x]$.
\end{lemma}
\begin{proof}
This follows from a repeated application of \cite[Proposition II.6.5]{Hartshorne}, and the fact that $\sum_{i=1}^na_i[D_i] = [\div(x)] = 0$ in $\Cl(A)$.
\end{proof}

\begin{cor}\label{cor:Cl(A)=Cl(A[1/x])}
If $A$ is a normal noetherian domain of dimension $d$, and $x\in A$ is any element for which $A/(x)$ is a domain of dimension $d-1$, then the map $\Cl(A)\to \Cl(A[1/x])$ defined by $M\mapsto M\otimes_AA[1/x]$ for any $M\in \Cl(A)$ is an isomorphism.
\end{cor}
\begin{proof}
This is simply the special case of Lemma \ref{lem:Cl(X-Z)} with $n = 1$ and $a_1 = 1$.
\end{proof}

We now turn our attention to computing dualizing modules. 

\begin{lemma}\label{lem:A[P]}
Let $A$ be a reduced Noetherian ring, and let $P_0,P_1,\ldots,P_n$ be the minimal primes of $A$. Then 
\[
A[P_0] = P_1\cap \cdots \cap P_n
\]
as ideals of $A$.
\end{lemma}
\begin{proof}
As $A$ is reduced we have $P_0\cap P_1\cap \cdots \cap P_n = 0$. It follows that $P_0(P_1\cap \cdots \cap P_n) = 0$ and so $P_1\cap \cdots \cap P_n\subseteq A[P_0]$. On the other hand, if $x\in A[P_0]$ then for any $i=1,\ldots,n$ we have $xP_0 = 0 \subseteq P_i$. Since $P_i$ is prime and $P_0\not\subseteq P_i$, this is only possible if $x \in P_i$ for each $i$. Thus indeed $A[P_0] = P_1\cap \cdots \cap P_n$.
\end{proof}

\begin{lemma}\label{lem:Gorenstein component}
Let $A$ be a reduced complete local Noetherian ring, and assume that $A$ is Gorenstein. Let $P_0,P_1,\ldots,P_n$ be the minimal primes of $A$ and let $B = A/P_0$. Assume that $B$ is Cohen--Macaulay and normal.

For each $i=1,\ldots,n$ let $D_i = \Spec A/(P_0+P_i)\subseteq \Spec B$. Assume (after reordering the $P_i$'s if necessary) that for some $m\le n$, $D_1,\ldots,D_m$ are \emph{distinct} prime divisors of $\Spec B$ and for each $i> m$, $\dim D_i \le \dim B - 2$.

Then if $\omega_B$ is the dualizing module of $B$ we have
\[
\omega_B \cong \Oc(-D_1-D_2-\cdots-D_m)
\]
and so $\omega_B = -[D_1]-\cdots-[D_m]$ in $\Cl(B)$.
\end{lemma}
\begin{proof}
The case $n=0$ is trivial (as then $B=A$ is Gorenstein and so $\omega_B = B$) so from now on we will assume $n>0$.	
	
First, as $A$ is Gorenstein, and hence Cohen--Macaulay, it is equidimensional and so $\dim A = \dim B$. Thus by \cite[Lemma 4.12]{Manning} and Lemma \ref{lem:A[P]} we have
\[
\omega_B \cong \Hom_B(B,\omega_B) \cong \Hom_A(B,\omega_A) \cong \Hom_A(A/P_0,A) = A[P_0] = P_1\cap \cdots \cap P_n 
\]

Now let $I = (P_0+A[P_0])/P_0 = (P_0 + ( P_1\cap \cdots \cap P_n))/P_0\subseteq B$ be the image of $A[P_0] = P_1\cap \cdots \cap P_n$ under the quotient map $A\onto A/P_0 = B$. Since $A[P_0]\cap P_0  = 0$, we have $I \cong A[P_0]\cong \omega_B$ as $B$-modules.

For $i = 1,\ldots,m$ let $Q_i = (P_0+P_i)/P_0\subseteq B$. By assumption $Q_1,\ldots,Q_m$ are distinct height one primes of $B$. Hence we have
\begin{align*}
\Oc(-D_1-\cdots - D_m) &= Q_1\cap \cdots \cap Q_m
\end{align*}
by Lemma \ref{lem:O(-D) = P}. So it will suffice to show that $I = Q_1\cap\cdots\cap Q_m$ as ideals of $B$. Clearly $P_0+(P_1\cap \cdots \cap P_n) \subseteq P_0+P_i$ for each $i$, and so $I\subseteq Q_1\cap\cdots\cap Q_m$.

But now note that as $B$ is Cohen--Macaulay, $I\cong \omega_B$ is a maximal Cohen--Macaulay $B$-module, so $\depth(B) = \depth(I) = \dim(B)$. The exact sequence $0\to I \to B \to B/I\to 0$ now implies that $\depth(B/I)\ge \dim(B)-1$. As $\dim B/I < \dim B$ (since $B$ is a domain and $I\ne 0$) this implies that $B/I$ is a Cohen--Macaulay ring of dimension $\dim(B)-1$. In particular, it is equidimensional of dimension $\dim(B)-1$, and so every minimal prime of $B/I$ is of the form $Q/I$ for a height $1$ prime $Q$ of $B$ containing $I$.

Now let $Q$ be any height $1$ prime of $B$, and consider the localization $(B/I)_Q = B_Q/I_Q$. Write $Q = Q'/P_0$, for some prime $Q'$ of $A$. As $P_0\subseteq Q'$ we have $(A[P_0])_{Q'} = A_{Q'}[(P_0)_{Q'}]$. Now as $P_0$ is a minimal prime of $A$ contained in $Q'$, $(P_0)_{Q'}$ will be a minimal prime of $A_{Q'}$. Moreover all minimal primes of $A_{Q'}$ are in the form $(P_i)_{Q'}$ for some $i$ satisfying $P_i\subseteq Q'$. Also as $A$ is reduced, the ring $A_{Q'}$ will be reduced as well, so Lemma \ref{lem:A[P]} gives
\[
I_Q = I_{Q'} = \frac{(P_0)_{Q'}+A_{Q'}[(P_0)_{Q'}]}{(P_0)_{Q'}} = 
\frac{(P_0)_{Q'}+\bigcap_{i>0,P_i\subseteq Q'} (P_i)_{Q'}}{(P_0)_{Q'}}.
\]
Now if there is no $i>0$ with $P_i\subseteq Q'$ this gives $I_Q = B_Q$ and so $(B/I)_Q = 0$. Now assume that $P_i\subseteq Q'$ for some $i>0$. This implies $P_0+P_i\subseteq Q'$ and so there is a quotient map
\[
B/Q_i = A/(P_0+P_i) \onto A/Q' = B/Q.
\] 
As $Q$ is a height $1$ prime of $B$ this implies $\dim(B/Q_i) \ge \dim(B/Q) = \dim(B)-1$. By assumption this is only possible if $i\le m$ and $Q = Q_i$ (since for $i\le m$, $B/Q_i$ is a domain of dimension $\dim(B)-1$). In this case, $P_0$ and $P_i$ are the only minimal primes of $A$ contained in $Q'$ (as if $P_i,P_j\subseteq Q'$ we would have $Q_i = Q = Q_j$ by the above, contradicting the fact that the $Q_i$'s are assumed to be distinct), and so we get
\[
I_{Q_i} = ((P_0)_{Q_i}+(P_i)_{Q_i})/((P_0)_{Q_i}) = ((P_0+P_i)_{Q_i})/((P_0)_{Q_i}) = (Q_i)_{Q_i}\subseteq B_{Q_i}.
\]
So for any height $1$ prime $Q$ of $B$, we have $(B/I)_Q = 0$, unless $Q = Q_i$ for some $i = 1,\ldots,m$, in which case we have $(B/I)_{Q_i} = B_{Q_i}/(Q_i)_{Q_i}$.

From this we get that the minimal primes of $B/I$ are precisely $Q_1/I,\ldots,Q_m/I$, which implies $\sqrt{I} = Q_1\cap\cdots\cap Q_m$. Also as $(B/I)_{Q_i} = B_{Q_i}/(Q_i)_{Q_i}$ is a field, and hence reduced, for each $i = 1,\ldots,m$, $B/I$ is generically reduced. Since $B/I$ is Cohen--Macaulay, this implies that it is reduced, and so $I = \sqrt{I} = Q_1\cap\cdots\cap Q_m$, completing the proof.
\end{proof}

\subsection{\Grobner bases}\label{ssec:Grobner}

In this subsection, we will review the properties of \Grobner bases we will need for our computations. Fix a field $\Bbbk$ and a polynomial ring $\Bbbk[x_1,\ldots,x_n]$. For any $\aun = (a_1,\ldots,a_n)\in \Z_{\ge 0}^n$, write $x^{\aun} = x_1^{a_1}\cdots x_n^{a_n}$. Let $\Mf=\{x^{\aun}|\aun\in \Z_{\ge 0}^n\}$ be the set of all monomials in $\Bbbk[x_1,\ldots,x_n]$. For any $x^{\aun}\in \Mf$ write $\deg x^{\aun} = a_1+\cdots+a_n$ for the total degree of $x^a$.

We will say that a \emph{monomial ordering} on $\Bbbk[x_1,\ldots,x_n]$ is a total order $\succ$ on $\Mf$ satisfying:
\begin{itemize}
	\item For any $m,p\in\Mf$, $m\preceq mp$.
	\item For any $m_1,m_2,p\in \Mf$, $m_1\preceq m_2$ if and only if $m_1p\preceq m_2p$.
\end{itemize}
For convenience, we will also write $0\prec m$ for all monomials $m\in\Mf$ 

We will mainly be concerned with the following monomial ordering, called the \emph{reverse lexicographic order with variable ordering $(x_1,\ldots,x_n)$} defined by:
\begin{center}
$x^{\aun}\prec x^{\bun}$ iff $\deg x^{\aun}<\deg x^{\bun}$ or $\deg x^{\aun} = \deg x^{\bun}$ and $a_i>b_i$, where $i$ is the largest index such that $a_i\ne b_i$.
\end{center}
(So for example $x_1x_3\prec x_2^2\prec x_1x_2$).

For any $\ds f = \sum_{\aun\in \Z_{\ge 0}^n}c_{\aun}x^{\aun}\in \Bbbk[x_1,\ldots,x_n]$ we will say the the \emph{leading term} of $f$ is $\LT(f) = c_{\aun}x^{\aun}$ where $x^{\aun}$ is the largest monomial (with respect to $\succ$) for which $c_{\aun}\ne 0$ (and we say $\LT(0) = 0$). While this depends on the choice of $\succ$, we will suppress this from our notation.

We will also let the \emph{leading coefficient} and \emph{leading monomial} of $f$ be $\LC(f) = c_{\aun}$ and $\LM(f) = x^{\aun}$, respectively (and again, $\LC(0) = \LM(0) = 0$). 

For any ideal $I\subseteq \Bbbk[x_1,\ldots,x_n]$, we will let $\init(I)$ be the ideal:
\[
\init(I) = \left(\LT(f) \middle|f\in I\right)\subseteq \Bbbk[x_1,\ldots,x_n]
\]
(which again depends on the choice of $\succ$). We will say that a finite set $\Gs = \{g_1,\ldots,g_t\}\subseteq \Bbbk[x_1,\ldots,x_n]$ is a \emph{\Grobner basis} (with respect to $\succ$) for $I$ if $I = (g_1,\ldots,g_t)$ and $\init(I) = (\LT(g_1),\ldots,\LT(g_t))$.

We will need the following useful criterion for identifying \Grobner bases. First, for any nonzero $g,h\in \Bbbk[x_1,\ldots,x_n]$ define
\[
\Ssf(g,h) = \frac{\LT(h)}{\gcd(\LT(g),\LT(h))}g - \frac{\LT(g)}{\gcd(\LT(g),\LT(h))}h.
\]
Given a set $\Gs = \{g_1,\ldots,g_t\}\subseteq \Bbbk[x_1,\ldots,x_n]$ and a pair of elements $g_i,g_j\in \Gs$, we say that $(g_i,g_j)$ satisfies (\ref{buchberger}) (with respect to $\Gs$ and $\prec$) if:
\begin{equation*}
\label{buchberger}
\begin{split}
\text{There are polynomials $f^{(ij)}_1,\ldots,f^{(ij)}_t\in \Bbbk[x_1,\ldots,x_n]$ for which}\\
\Ssf(g_i,g_j) = f^{(ij)}_1g_1+f^{(ij)}_2g_2+\cdots+f^{(ij)}_tg_t\\
\shoveleft{\text{and $\LM(f^{(ij)}_s)\LM(g_s) \preceq \LM(\Ssf(g_i,g_j))$ for each $s=1,\ldots,t$}.}
\end{split}
\tag{B}
\end{equation*}
Clearly $(g_i,g_i)$ always satisfies (\ref{buchberger}) and $(g_i,g_j)$ satisfies (\ref{buchberger}) if and only if $(g_j,g_i)$ does.

Then we have (see \cite[Theorem 15.8]{Eisenbud}, as well as exercises 15.19 and 15.20):

\begin{prop}\label{prop:Buchberger}[Buchberger's Criterion]
If $I=(g_1,\ldots,g_t)\subseteq \Bbbk[x_1,\ldots,x_n]$ is an ideal, then $\Gs=\{g_1,\ldots,g_t\}$ is a \Grobner basis for $I$ if and only if $(g_i,g_j)$ satisfies (\ref{buchberger}) for each $1\le i<j\le t$.

Moreover, (\ref{buchberger}) automatically holds whenever $\gcd(\LT(g_i),\LT(g_j)) = 1$.
\end{prop}

Our computations will depend heavily on the following property of \Grobner bases (see \cite[Theorem 15.13]{Eisenbud}):
\begin{prop}\label{prop:Grob reg seq}
If $\succ$ is the reverse lexicographical order with the variable ordering $(x_1,\ldots,x_n)$,\footnote{While most of the results mentioned in this section still hold with $\succ$ replaced by an arbitrary monomial ordering, Proposition \ref{prop:Grob reg seq} specifically requires the reverse lexicographical order.} then for any $r\le n$, $(x_n,x_{n-1},\ldots,x_r)$ is a regular sequence for $\Bbbk[x_1,\ldots,x_n]/I$ if and only if it is a regular sequence for $\Bbbk[x_1,\ldots,x_n]/\init(I)$.

Equivalently, if $\Gs = (g_1,\ldots,g_t)$ is a \Grobner basis for $I$, then $(x_n,x_{n-1},\ldots,x_r)$ is a regular sequence for $\Bbbk[x_1,\ldots,x_n]/I$ if and only $\LT(g_1),\ldots,\LT(g_t)\in \Bbbk[x_1,\ldots,x_{r-1}]$ (i.e. none the leading terms of $g_1,\ldots,g_t$ contain any of the variables $x_r,\ldots,x_n$).
\end{prop}

The proof of Theorem \ref{thm:R121 main results} will also require us to compute the intersection of two ideals. Fortunately \Grobner bases also give us a simple way of calculating such intersections.

First, for any monomial ordering $\succ$ on $\Bbbk[x_1,\ldots,x_n]$, define a new ordering $\succ_t$ on $\Bbbk[x_1,\ldots,x_n,t]$ by:
\begin{center}
	$t^ax^{\aun}\succ_t t^bx^{\bun}$ iff $a>b$ or $a=b$ and $x^{\aun}\succ x^{\bun}$
\end{center}
(so in particular, any monomial containing $t$ will be greater than any monomial not containing $t$). Then \cite[Proposition 15.29, Excercise 15.43]{Eisenbud} gives:

\begin{prop}\label{prop:Grob intersection}
Let $I = (f_1,\ldots,f_p)$ and $J = (g_1,\ldots,g_q)$ be two ideals of $k[x_1,\ldots,x_n]$ and let $\succ$ be any monomial ordering on $\Bbbk[x_1,\ldots,x_n]$ ($\{f_1,\ldots,f_s\}$ and $\{g_1,\ldots,g_t\}$ are not assumed to be \Grobner bases). Define the ideal
\[
\Kc = (tf_1,tf_2,\ldots,tf_p,(t-1)g_1,(t-1)g_2,\ldots,(t-1)g_q) = \left(ta+(1-t)b\middle|a\in I,b\in J\right)\subseteq \Bbbk[x_1,\ldots,x_n,t].
\]
If $\Gs$ is a \Grobner basis for $\Kc$ with respect to $\succ_t$ then the elements of $\Gs$ which do not contain the variable $t$ (i.e. the set $\Gs\cap \Bbbk[x_1,\ldots,x_n]$) form a \Grobner basis for $I\cap J$ with respect to $\succ$. In particular, $I\cap J = (\Gs\cap \Bbbk[x_1,\ldots,x_n])$.
\end{prop}

\section{The space $\Nc^{2,2}(\st_{2,2})$}\label{sec:R22}
The goal of this section is to prove Theorem \ref{thm:R22 main results}. Write $\Nc^{2,2} = \Spec \Rc^{2,2}$, so that $\Rc^{2,2}$ is the quotient of $\Oc[A,B,C,D,S,T,U,V,X,Y,Z,W]$ cut out by the matrix equation
\begin{align*}
	\begin{pmatrix}S&T\\U&V\end{pmatrix}\begin{pmatrix}A&B\\C&D\end{pmatrix} &= \begin{pmatrix}A&B\\C&D\end{pmatrix}\begin{pmatrix}X&Y\\Z&W\end{pmatrix}
\end{align*}

By Proposition \ref{prop:st(2,2) cases}, $\Nc^{2,2}(\st_{2,2})$ is an irreducible component of $\Nc^{2,2}$. Let 
\[\Rc^{2,2}(\st_{2,2}) = \Rc^{2,2}/(S+V-X-W, SV-TU-XW+YZ)\]

\begin{prop}
	$\Nc^{2,2}(\st_{2,2}) \subseteq \Spec \Rc^{2,2}(\st_{2,2})$ as subschemes of $\Nc^{2,2}$.
\end{prop}
\begin{proof}
By definition, $(\Spec\Rc^{2,2}(\st_{2,2}))(\Ebar)\subseteq \Nc^{2,2}(\Ebar)$ is the set of points $(X_1,N_1,X_2)\in \Nc^{2,2}(\Ebar)$ with $\tr X_1 = \tr X_2$ and $\det X_1 = \det X_2$. Thus the equation $X_1N_1 = N_1X_2$ implies that $(X_1,N_1,X_2)\in (\Spec\Rc^{2,2}(\st_{2,2}))$ whenever $N_1\in \GL_2(\Ebar)\subseteq \Wc^{2,2}(\Ebar)$. So $(\Pi^{2,2})^{-1}(\GL_2(\Ebar)) \subseteq (\Spec\Rc^{2,2}(\st_{2,2}))(\Ebar)$.

By Proposition \ref{prop:st(2,2) cases}, 
\[\Nc^{2,2}(\st_{2,2}) = \Nc^{1,2,1}(2) = \overline{(\Pi^{2,2})^{-1}(\Wc^{2,2}(2))} = \overline{(\Pi^{2,2})^{-1}(\GL_2(\Ebar))},
\]
so the claim follows as $\Spec \Rc^{2,2}(\st_{2,2})$ is Zariski closed inside $\Nc^{2,2}$.
\end{proof}

From this proposition, it follows that $\Nc^{2,2}(\st_{2,2})$ is an irreducible component of $\Spec \Rc^{2,2}(\st_{2,2})$ of relative dimension $8$ over $\Oc$.

Let $\Rct^{2,2} = \Rc^{2,2}(\st_{2,2})/(S+V)$, so that $\Rct^{2,2}(\st_{2,2})=\Oc[A,B,C,D,S,T,U,X,Y,Z]/\Ict$ where $\Ict$ is the ideal cut out by the equations:
\begin{align*}
	\begin{pmatrix}S&T\\U&-S\end{pmatrix}\begin{pmatrix}A&B\\C&D\end{pmatrix} &= \begin{pmatrix}A&B\\C&D\end{pmatrix}\begin{pmatrix}X&Y\\Z&-X\end{pmatrix},&
	S^2+TU&=X^2+YZ.
\end{align*}
One can easily verify that there is an isomorphism $\Rc^{2,2}(\st_{2,2})\isomto\Rct^{2,2}[t]$ given by
\[
(A,B,C,D,S,T,U,V,X,Y,Z,W)\mapsto (A,B,C,D,t+S,T,U,t-S,t+X,Y,Z,t-X).
\]
We shall take $\Rct^{2,2}$ to be the ring $\Rct^{2,2}(\st_{2,2})$ from Theorem \ref{thm:R22 main results} (we are omitting the $\st_{2,2}$ to simplify notation). By definition, $\Rct^{2,2}$ is a graded $\Oc$-algebra with irrelevant ideal $\Ic_+ = (A,B,C,D,S,T,U,X,Y,Z)$, and the isomorphism $\Rc^{2,2}(\st_{2,2})\isomto\Rct^{2,2}[t]$ identifies $\Spec (\Rct^{2,2}(\st_{2,2})/\Ic_+)\times \A^1_{\Oc}$ with the subscheme 
\[\left\{\left(\begin{pmatrix}t&0\\0&t\end{pmatrix}, \begin{pmatrix}0&0\\0&0\end{pmatrix},\begin{pmatrix}t&0\\0&t\end{pmatrix} \right)\right\}.\]
of $\Nc^{2,2}(\st_{2,2})\subseteq \Spec \Rc^{2,2}(\st_{2,2})$.

Now for any $\Oc$-algebra $\Ac$, let $\Rct^{2,2}_\Ac = \Rct^{2,2}\otimes_{\Oc} \Ac$ and $0_{\Ac} = (\Spec \Rct^{2,2}/\Ic_+)\times_{\Oc}\Ac$. Let $\Ict_A = \Ict \otimes_{\Oc} \Ac$, so that $\Rct^{2,2}_A = \Ac[A,B,C,D,S,T,U,X,Y,Z]/\Ict_{\Ac}$. We shall prove the following:

\begin{prop}\label{prop:R22tilde}
For $k = \F,E$,	$\Rct^{2,2}_k$ is a normal Cohen--Macaulay domain of dimension $7$, which is not a complete intersection. $\Spec \Rct^{2,2}_k$ is smooth away from the point $0_{k}$.
\end{prop}

Note that this implies parts (\ref{R22:domain}) and (\ref{R22:smooth}) of Theorem \ref{thm:R22 main results}. Indeed, Corollary \ref{cor:graded domain} implies that $\Rct^{2,2}$ is a domain of relative dimension $7$ over $\Oc$, which in particular implies that $\Rc^{2,2}(\st_{2,2}) \cong \Rct^{2,2}[t]$ is a domain, giving $\Nc^{2,2}(\st_{2,2}) = \Spec \Rc^{2,2}(\st_{2,2})$. (\ref{R22:domain}) follows. (\ref{R22:smooth}) follows from the statement that $\Spec\Rct^{2,2}_{\F}$ is smooth away from $0_{\F}$ by noting that if $x\in \Nc^{2,2}(\st_{2,2})(\Fbar)$ and $R^{2,2}_x(\st_{2,2})$ is the complete local ring at $x$, then $R^{2,2}_x(\st_{2,2})$ will be regular if any only if $R^{2,2}_x(\st_{2,2})/(\varpi)$ is (as $\varpi\in \mf_{R^{2,2}_x(\st_{2,2})}\sm \mf_{R^{2,2}_x(\st_{2,2})}^2$), and this latter ring is the complete local ring of $\Nc^{2,2}(\st_{2,2})_{\F}$ at the point $x$.

Also this implies that $\Rct^{2,2}$ is Cohen--Macaulay, as $\varpi$ is a regular element on $\Rct^{2,2}$ and $\Rct^{2,2}/(\varpi) = \Rct^{2,2}_{\F}$ is Cohen--Macaulay. As $\Rct^{2,2}$ is Cohen--Macaulay (and hence $S_2$) and smooth away $0_{\Oc}$ (and hence $R_1$) Serre's criterion implies that $\Rct^{2,2}$ is normal.

From now on let $k$ denote either $\F$ or $E$. Let $\succ$ be the reverse lexicographical ordering on $k[A,B,C,D,S,T,U,X,Y,Z]$ with variable ordering $(X,Y,Z,A,D,B,C,S,T,U)$. For ease of reading, we will underline the leading term of every polynomial under a monomial ordering.

Then we have:

\begin{prop}\label{prop:R22 Grob}
	The polynomials
	\begin{align*}
		g_1 &= \ul{XA}+ZB-AS-CT,&
		g_2 &= \ul{YA}-XB-BS-DT,&
		g_3 &= \ul{ZD}+XC+CS-AU,\\
		g_4 &= \ul{XD}-YC-DS+BU,&
		g_5 &= \ul{X^2}+YZ-S^2-TU
	\end{align*}
	form a \Grobner basis for $\Ict_k$ with respect to $\succ$. In particular, $(U,T,S,C,B)$ is a regular sequence for $\Rct^{2,2}_k$.
\end{prop}
\begin{proof}
	Note that $\{g_1,\ldots,g_5\}$ is just the generating set for $\Ict_k$ given above. Hence it generates $\Ict_k$, and so it remains to verify Buchberger's Criterion (Proposition \ref{prop:Buchberger}).
	
	By the last part of Proposition \ref{prop:Buchberger}, we do not need to check (\ref{buchberger}) for $(g_1,g_3), (g_2,g_3),(g_2,g_4),(g_2,g_5)$ or $(g_3,g_5)$. For the remaining $5$ pairs of polynomials we have the expressions:
	\begin{align*}
		\Ssf(g_1,g_2) &= Yg_1-Xg_2 = \ul{X^2B}+YZB-YAS+XBS+XDT-YCT = -Sg_2+Tg_4+Bg_5\\
		\Ssf(g_1,g_4) &= Dg_1-Ag_4 = \ul{ZDB}+YAC-DCT-ABU = Cg_2+Bg_3\\
		\Ssf(g_1,g_5) &= Xg_1-Ag_5 = -\ul{YZA}+XZB-XAS+AS^2-XCT+ATU = -Sg_1-Zg_2+Tg_3\\
		\Ssf(g_3,g_4) &= Xg_3-Zg_4 = \ul{X^2C}+YZC+ZDS+XCS-XAU-ZBU = -Ug_1+Sg_3+Cg_5\\
		\Ssf(g_4,g_5) &= Xg_4-Dg_5 = -\ul{YZD}-XYC-XDS+DS^2+XBU+DTU = Ug_2-Yg_3-Sg_4
	\end{align*}
	all of which satisfy the criterion $\LM(f^{(ij)}_s)\LM(g_s)\preceq \LM(\Ssf(g_i,g_j))$ from (\ref{buchberger}). So indeed, $\{g_1,g_2,g_3,g_4,g_5\}$ is a \Grobner basis for $\Ict_k$. The last claim follows Proposition \ref{prop:Grob reg seq}.
\end{proof}

Now let 
\[\Sc_k = \Rct^{2,2}_k/(U,T,S,C,B) = \frac{k[A,D,X,Y,Z]}{(AX,AY,DX,DZ,X^2+YZ)}\]
Letting $\alpha = Y-A$ and $\beta = D-Z$ we may rewrite this as $\Sc_k = k[X,Y,Z,\alpha,\beta]/\Jc$ where
\[
\Jc = \bigg(X(Y+\alpha),Y(Y+\alpha),X(Z+\beta),Z(Z+\beta),X^2+YZ\bigg)\subseteq  k[X,Y,Z,\alpha,\beta]
\]
Now let $\succ_{\Sc}$ be the reverse lexicographical ordering on $k[X,Y,Z,\alpha,\beta]$ with variable ordering $(X,Y,Z,\alpha,\beta)$.

\begin{prop}\label{prop:S22 Grob}
	The polynomials
	\begin{align*}
		h_1 &= \ul{X}(\ul{Y}+\alpha),&
		h_2 &= \ul{Y}(\ul{Y}+\alpha),&
		h_3 &= \ul{X}(\ul{Z}+\beta),&
		h_4 &= \ul{Z}(\ul{Z}+\beta),&
		h_5 &= \ul{X^2}+YZ
	\end{align*}
	form a \Grobner basis for $\Jc$ with respect to $\succ_{\Sc}$. In particular, $(\beta,\alpha) = (D-Z,Y-A)$ is a regular sequence for $\Sc_k$.
\end{prop}
\begin{proof}
	By definition, $\Jc = (h_1,h_2,h_3,h_4,h_5)$, so we must verify Proposition \ref{prop:Buchberger}. By the last part of Proposition \ref{prop:Buchberger}, we do not need to check (\ref{buchberger}) for $(h_1,h_4),(h_2,h_3),(h_2,h_4),(h_2,h_5)$ or $(h_4,h_5)$. For the remaining $5$ pairs of polynomials we have:
	\begin{align*}
		\Ssf(h_1,h_2) &= Yh_1-Xh_2 = 0\\
		\Ssf(h_1,h_3) &= Zh_1-Yh_3 = \ul{XZ\alpha}-XY\beta = -\beta h_1+\alpha h_3\\
		\Ssf(h_1,h_5) &= Xh_1-Yh_5 = -\ul{Y^2Z}+X^2\alpha = -Zh_2+\alpha h_5\\
		\Ssf(h_3,h_4) &= Zh_3-Xh_4 = 0\\
		\Ssf(h_3,h_5) &= Xh_3-Zh_5 = -\ul{YZ^2}+X^2\beta = -Yh_4+\beta h_5
	\end{align*}
	so indeed $\{h_1,h_2,h_3,h_4,h_5\}$ is a \Grobner basis for $\Jc$. Again, the last claim follows by Proposition \ref{prop:Grob reg seq}.
\end{proof}

Combining these propositions, we get that $(U,T,S,C,B,D-Z,Y-A)$ is a regular sequence for $\Rc^{2,2}_k$. But now
\[\frac{\Rct^{2,2}_k}{(U,T,S,C,B,D-Z,Y-A)} = \frac{\Sc_k}{(\alpha,\beta)} = \frac{k[X,Y,Z]}{(XY,Y^2,XZ,Z^2,X^2+YZ)},\]
which clearly has dimension $0$ (observe that $X^3 = -XYZ = 0$ in this ring, so $X,Y$ and $Z$ are all nilpotent). It follows that $\Rct^{2,2}_k$ is Cohen--Macaulay of dimension $7$. 

Moreover, as the ideal $(XY,Y^2,XZ,Z^2,X^2+XY)$ has $5$ generators in its minimal generating set, $\Sc_k/(\alpha,\beta)$ is not a complete intersection, and so $\Rc^{2,2}_k$ isn't either.

To complete the proof of Proposition \ref{prop:R22tilde} we need to show that $\Rct^{2,2}_k$ is a domain, smooth away from $0_{\Rct,k}$. 

Let $\Hct_k = \Hc^{2,2}_k\times \G_{m,k} = (\GL_{2,k})^2\times \G_{m,k}$ treated as an algebraic group over $k$, and note that $\Hct_k$ acts on $\Spec \Rct^{2,2}_k$ via
\[
(g,h,t)\cdot\left(\begin{pmatrix}S&T\\U&-S\end{pmatrix};\begin{pmatrix}A&B\\C&D\end{pmatrix};\begin{pmatrix}X&Y\\Z&-X\end{pmatrix}\right)
=
\left(tg\begin{pmatrix}S&T\\U&-S\end{pmatrix}g^{-1};g\begin{pmatrix}A&B\\C&D\end{pmatrix}h^{-1};th\begin{pmatrix}X&Y\\Z&-X\end{pmatrix}h^{-1}\right)
\]
(for $g,h\in \GL_{2,k}$ and $t\in \G_{m,k}$). One can check that the action of $\Hct_k(\kbar)$ on $(\Spec \Rct^{2,2}_k)(\kbar)$ has $13$ orbits, generated by the points:
\begin{align*}
	x &= \left(\begin{pmatrix}1&0\\0&-1\end{pmatrix};\begin{pmatrix}1&0\\0&1\end{pmatrix};\begin{pmatrix}1&0\\0&-1\end{pmatrix}\right) &
	x_{\e} &= \left(\begin{pmatrix}0&\e\\0&0\end{pmatrix};\begin{pmatrix}1&0\\0&1\end{pmatrix};\begin{pmatrix}0&\e\\0&0\end{pmatrix}\right)\\
	y &= \left(\begin{pmatrix}1&0\\0&-1\end{pmatrix};\begin{pmatrix}1&0\\0&0\end{pmatrix};\begin{pmatrix}1&0\\0&-1\end{pmatrix}\right) &
	y_{\e_1\e_2} &= \left(\begin{pmatrix}0&\e_1\\0&0\end{pmatrix};\begin{pmatrix}0&1\\0&0\end{pmatrix};\begin{pmatrix}0&\e_2\\0&0\end{pmatrix}\right)\\
	z &= \left(\begin{pmatrix}1&0\\0&-1\end{pmatrix};\begin{pmatrix}0&0\\0&0\end{pmatrix};\begin{pmatrix}1&0\\0&-1\end{pmatrix}\right) &
	z_{\e_1\e_2} &= \left(\begin{pmatrix}0&\e_1\\0&0\end{pmatrix};\begin{pmatrix}0&0\\0&0\end{pmatrix};\begin{pmatrix}0&\e_2\\0&0\end{pmatrix}\right)
\end{align*}
for $\e,\e_1,\e_2\in \{0,1\}$. Now as $\Hct_k(\kbar)$ is a connected group variety of dimension $9$, for any $p\in (\Spec \Rct^{1,2,1}_k)(\kbar)$ the orbit $\Hct_k(\kbar)(p)$ is irreducible of dimension $9-\dim\stab_{\Hct_k(\kbar)}(p)$. By computing stabilizers it is now easy to compute the dimension of each of these $13$ orbits (we omit the details):

\begin{lemma}\label{lem:H22 orbits dimensions}
	The dimensions of the $13$ orbits of $\Hct_k(\kbar)$ on $(\Spec \Rct^{2,2}_k)(\kbar)$ are as follows:
	
	For $p\in \{x,y,z,x_0,x_1\}$ the dimension of $\Hct_k(\kbar)(p)$ is:
	\begin{align*}
	\dim \Hct_k(\kbar)(x) &= 7, &
	\dim \Hct_k(\kbar)(y) &= 6, &
	\dim \Hct_k(\kbar)(z) &= 5,\\
	\dim \Hct_k(\kbar)(x_0) &= 4, &
	\dim \Hct_k(\kbar)(x_1) &= 6,
	\end{align*}
	and the dimensions of the remaining $8$ orbits are given in the following table:
	\[
	\begin{array}{c|c|c|c|c|}
		(\e_1,\e_2) & (0,0) & (0,1) & (1,0) & (1,1)\\
		\hline
		\dim \Hct_k(\kbar)(y_{\e_1\e_2}) &3&4&4&5\\
		\hline
		\dim \Hct_k(\kbar)(z_{\e_1\e_2}) &0&2&2&4\\
		\hline
	\end{array}
	\]
\end{lemma}

In particular, $\dim \Hct_k(\kbar)(x) = 7$ and all other orbits have dimension strictly smaller than $7$. Thus $(\Spec \Rct^{2,2})(\kbar)$ has a single component of dimension $7$: $\overline{\Hct_k(\kbar)(x)}$. As $\Rct^{2,2}_k$ is Cohen--Macaulay of dimension $7$, $\Spec \Rct^{2,2}_k$ must be equidimensional of dimension $7$, and so it follows that $\Spec \Rct^{2,2}_k$ must be irreducible and $(\Spec \Rct^{2,2}_k)(\kbar)= \overline{\Hc_k(\kbar)(x)}$.

Also note that the only fixed point of the action is $z_{00} = 0_{k}$.

To show that $\Rct^{2,2}_k$ is a domain, it remains to show that $\Spec \Rct^{2,2}_k$ is reduced. Again as it is Cohen--Macaulay, it suffices to show it is generically reduced. Now we can compute:

\begin{lemma}
	$\Rct^{2,2}_k$ is regular at all of the points $x,x_{\e},y,y_{\e_1\e_2},z,z_{\e_1\e_2}$, except for $z_{00}$.
\end{lemma}
\begin{proof}
	Let $g_1,\ldots,g_5$ be as in Proposition \ref{prop:R22 Grob}. Write $\vec{g} = \begin{pmatrix}g_1&g_2&g_3&g_4&g_5\end{pmatrix}$. Then we can compute that 
	\[
	\begin{pmatrix}
		\del \vec{g}/\del A\\
		\del \vec{g}/\del B\\
		\del \vec{g}/\del C\\
		\del \vec{g}/\del D\\
		\del \vec{g}/\del S\\
		\del \vec{g}/\del T\\
		\del \vec{g}/\del U\\
		\del \vec{g}/\del X\\
		\del \vec{g}/\del Y\\
		\del \vec{g}/\del Z
	\end{pmatrix}
	=
	\begin{pmatrix}
		X-S&Y&-U&0&0\\
		Z&-X-S&0&U&0\\
		-T&0&X+S&-Y&0\\
		0&-T&Z&X-S&0\\
		-A&-B&C&-D&-2S\\
		-C&-D&0&0&-U\\
		0&0&-A&B&-T\\
		A&-B&C&D&2X\\
		0&A&0&-C&Z\\
		B&0&D&0&Y
	\end{pmatrix}
	\]
	One can easily verify that this matrix has rank $10-7=3$ at each of the $12$ points specified (and is the zero matrix at at $z_{00}$) so as $\Spec \Rct^{2,2}$ has codimension $10-7=3$ in $\A^{10}_k$, it follows that $\Rct^{2,2}_k$ is regular at the requested points.
\end{proof}

Now the action of $\Hc_k(\kbar)$ on $(\Spec \Rct^{2,2}_k)$ implies that if $\Rct^{2,2}_k$ is regular at some point $p$, then it is regular at all points in $\Hc_k(\kbar)p$. Thus the lemma implies that $\Rct^{2,2}_k$ is regular at all points in $(\Spec \Rct^{2,2}_k)(\kbar)\sm \Hc_k(\kbar)z_{00} = (\Spec \Rct^{2,2}_k)(\kbar)\sm\{0_{k}\}$. Thus $(\Spec \Rct^{2,2}_k)\sm \{0_{k}\}$ is smooth.

In particular, this implies that $\Spec \Rct^{2,2}_k$ is generically reduced. As $\Rct^{2,2}_k$ is Cohen--Macaulay and irreducible, this implies that it is reduced, and so it is a domain.

This completes the proof of Proposition \ref{prop:R22tilde}, and so of parts (\ref{R22:domain}) and (\ref{R22:smooth}) of Theorem \ref{thm:R22 main results}, and also shows that $\Rct^{2,2}$ is Cohen--Macaulay.

At this point, one could show that $\Rct^{2,2}$ is Gorenstein by directly showing that the Artinian ring $\Sc_{\F}$ is Gorenstein. We will omit this computation, as this result also follows from part (\ref{R22:class group}) of Theorem \ref{thm:R22 main results}. Indeed, let $\omega$ be the dualizing module of $\Rct^{2,2}$. Since $\omega$ is a generic rank $1$ reflexive module, it corresponds to an element $[\omega]\in \Cl(\Rct^{2,2})$. If $\Cl(\Rct^{2,2}) = 0$, it will then follow that $\omega$ is free over $\Rct^{2,2}$, and hence $\Rct^{2,2}$ is Gorenstein, proving Theorem \ref{thm:R22 main results}(\ref{R22:Gorenstein}).

So it remains to prove Theorem \ref{thm:R22 main results}(\ref{R22:Rational},\ref{R22:class group}). By Corollary \ref{cor:Cl(A)=Cl(A[1/x])}, to prove (\ref{R22:class group}) it is enough to prove that $\Cl(\Rct^{2,2}_k) = 0$ for $k=\F,E$.

Now consider the closed subscheme $\Dc_k = \Spec \Rct^{2,2}_k/(AD-BC)$ of $\Spec \Rct^{2,2}_k$.

\begin{lemma}\label{lem:R/(AD-BC)}
$\Dc_k$ is geometrically integral of dimension $6$.
\end{lemma}
\begin{proof}
	As $\Rct^{2,2}_k$ is Cohen--Macaulay and $AD-BC$ is a regular element on $\Rct^{2,2}_k$ (as $\Rct^{2,2}_k$ is a domain), $\Dc_k$ is Cohen--Macaulay of  dimension $6$. This implies that $\Dc_k$ is equidimensional of dimension $6$, and geometrically reduced provided that it is generically geometrically reduced.
	
	Now the action of $\Hct_k = (\GL_{2,k})^2\times \G_{m,k}$ on $\Spec \Rct^{2,2}_k$ clearly preserves $\Dc_k$, so $\Dc_k(\kbar)$ is a union of orbits of $\Hct_k(\kbar)$. In fact, by the description of the orbits of $\Hct_k(\kbar)$ above we see that $\Dc_k(\kbar)$ is a union of the $10$ orbits $\Hct_k(\kbar)(y)$, $\Hct_k(\kbar)(z)$, $\Hct_k(\kbar)(y_{\e_1\e_2})$ and $\Hct_k(\kbar)(z_{\e_1\e_2})$ for $\e_1,\e_2\in\{0,1\}$. Of these orbits, $\Hct_k(\kbar)(y)$ has dimension $6$, and the other $9$ orbits all have dimension strictly smaller than $6$ by Lemma \ref{lem:H22 orbits dimensions}. Thus $\Dc_k$ has a unique irreducible component of dimension $6$, namely $\overline{\Hct_k(\kbar)(y)}$. Since $\Dc_k$ is equidimensional of dimension $6$, this implies $\Dc_k=\overline{\Hct_k(\kbar)(y)}$, and so is geometrically irreducible.
	
	Now consider the ring $\Rc^{2,2}_k/(AD-BC)[A^{-1}]$, and note that this is the quotient of $k[A,B,C,D,S,T,U,X,Y,Z][A^{-1}]$, given by the equations:
	\begin{align*}
		\begin{pmatrix}S&T\\U&-S\end{pmatrix}\begin{pmatrix}A&B\\C&D\end{pmatrix} &= \begin{pmatrix}A&B\\C&D\end{pmatrix}\begin{pmatrix}X&Y\\Z&-X\end{pmatrix},&
		S^2+TU&=X^2+YZ,&
		AD=BC
	\end{align*}
	These equations imply
	\begin{align*}
		D&=A^{-1}BC,\\
		S&=X+A^{-1}CZ-A^{-1}CT,\\
		U&=A^{-1}CX+A^{-1}DZ+A^{-1}SC,\\
		Y&=A^{-1}BS+A^{-1}DT+A^{-1}BX
	\end{align*}
	implying that $\Rct_k^{2,2}/(AD-BC)[A^{-1}]$ is generated as a $k$-algebra by $A^{\pm 1},B,C,T,Y,Z$, and is thus a quotient of $k[A^{\pm 1},B,C,T,Y,Z]$. Since $\dim \Rct^{2,2}_k/(AD-BC)[A^{-1}] = 6$, we get that $\Rct^{2,2}_k/(AD-BC)[A^{-1}] = k[A,B,C,T,Y,Z][A^{-1}]$. As $k[A,B,C,T,Y,Z][A^{-1}]$ is reduced of dimension $6$, it follows that $\Dc_k$ is generically geometrically reduced, and hence geometrically reduced.
\end{proof}

Now it follows that $\Rct^{2,2}_k/(AD-BC)$ is a domain of dimension $6$, so Corollary \ref{cor:Cl(A)=Cl(A[1/x])} gives that $\Cl(\Rct^{2,2}_k) \cong \Cl(\Rct^{2,2}_k[1/(AD-BC)])$.

But now in the ring $\Rct^{2,2}_k[1/(AD-BC)]$, the matrix $\begin{pmatrix}A&B\\C&D\end{pmatrix}$ is invertible. Thus the equations defining $\Rct^{2,2}_k[1/(AD-BC)]$ simplify to 
\[\begin{pmatrix}X&Y\\Z&-X\end{pmatrix} = \begin{pmatrix}A&B\\C&D\end{pmatrix}^{-1}\begin{pmatrix}S&T\\U&-S\end{pmatrix}\begin{pmatrix}A&B\\C&D\end{pmatrix} = \frac{1}{AD-BC}\begin{pmatrix}D&-B\\-C&A\end{pmatrix}\begin{pmatrix}S&T\\U&-S\end{pmatrix}\begin{pmatrix}A&B\\C&D\end{pmatrix}\]
as this matrix equation clearly implies
\[
S^2 + TU = \det \begin{pmatrix}S&T\\U&-S\end{pmatrix} = \det \begin{pmatrix}X&Y\\Z&-X\end{pmatrix} = X^2+YZ.
\]
Hence $\Rc^{2,2}_k[1/(AD-BC)] \cong k[A,B,C,D,S,T,U]\left[\frac{1}{AD-BC}\right]$, which is the localization of a UFD, and hence also a UFD. Thus $\Cl(\Rct_k^{2,2}) \cong \Cl(\Rc^{2,2}_k[1/(AD-BC)]) = 0$, proving Theorem \ref{thm:R22 main results}(\ref{R22:class group}). Also this implies that $(\Spec \Rc^{2,2}_k)\sm \Dc_k = \Spec \Rc^{2,2}_k[1/(AD-BC)]$, which is isomorphic to a Zariski open subset of $\A^7_k$ (and which obviously contains points of $\A^7_k(k)$), so by Lemma \ref{lem:R/(AD-BC)}, \ref{thm:R22 main results}(\ref{R22:Rational}) holds with $\Zc^{2,2}_k := \Dc_k$. This completes the proof of Theorem \ref{thm:R22 main results}.
	
\section{The space $\Nc^{1,2,1}(\st_{2,2})$}\label{sec:R121}
The goal of this section is to prove Theorem \ref{thm:R121 main results}. We will follow a similar approach to the one used in Section \ref{sec:R22}.

Write $\Nc^{1,2,1} = \Spec \Rc^{1,2,1}$, so that $\Rc^{1,2,1}$ is the quotient of $\Oc[A,B,C,D,R,S,X,Y,Z,W]$ cut out by the matrix equations
\begin{align*}
\begin{pmatrix}A&B\end{pmatrix}\begin{pmatrix}X&Y\\Z&W\end{pmatrix} &= R\begin{pmatrix}A&B\end{pmatrix},&
\begin{pmatrix}X&Y\\Z&W\end{pmatrix}\begin{pmatrix}C\\D\end{pmatrix} &= S\begin{pmatrix}C\\D\end{pmatrix}.
\end{align*}
By Proposition \ref{prop:st(2,2) cases}, $\Nc^{1,2,1}(\st_{2,2})$ is an irreducible component of $\Nc^{1,2,1}$. Let 
\[\Rc^{1,2,1}(\st_{2,2}) = \Rc^{1,2,1}/(AC+BD,R+S-X-W,RS-XW+YZ).\]
\begin{prop}
$\Nc^{1,2,1}(\st_{2,2}) \subseteq \Spec \Rc^{1,2,1}(\st_{2,2})$ as subschemes of $\Nc^{1,2,1}$.
\end{prop}
\begin{proof}
The equation $AC+BD = 0$, and the functions $R,S,X+W$ and $XW-YZ$ are all clearly preserved by the action of the group $\Hc^{1,2,1}= \G_m\times \GL_2\times \G_m$ on $\Nc^{1,2,1}$, and so the subscheme $\Spec \Rc^{1,2,1}(\st_{2,2})\subseteq \Nc^{1,2,1}$ is stabilized by $\Hc^{1,2,1}$.

Consider the point $x = \left(\begin{pmatrix}1&0\end{pmatrix}, \begin{pmatrix}0\\1\end{pmatrix}\right)\in \Wc^{1,2,1}(\Ebar)$. We have
\[
\Delta^{1,2,1}\left(x\right) = 
\left(\rank \begin{pmatrix}1&0\end{pmatrix}, \rank \begin{pmatrix}0\\1\end{pmatrix}, \rank \begin{pmatrix}1&0\end{pmatrix} \begin{pmatrix}0\\1\end{pmatrix}
\right)
= (1,1;0)
\]
so by Proposition \ref{prop:H orbits} we get $x \in \Wc^{1,2,1}(1,1;0)$ and so $\Wc^{1,2,1}(1,1;0) = \Hc^{1,2,1}(\Ebar)x$. By Proposition \ref{prop:st(2,2) cases}, $
\Nc^{1,2,1}(\st_{2,2}) = \Nc^{1,2,1}(1,1;0)$, so
\[
\Nc^{1,2,1}(\st_{2,2}) = \overline{(\Pi^{1,2,1})^{-1}(\Hc^{1,2,1}(\Ebar)x)} = \overline{\Hc^{1,2,1}(\Ebar)(\Pi^{1,2,1})^{-1}(x)}.
\]
Now $(\Pi^{1,2,1})^{-1}(x)$ is the subscheme of $\Nc^{1,2,1}$ cut out by the equations $A=D=1,B=C=0$, so on $(\Pi^{1,2,1})^{-1}(x)$ we have
\begin{align*}
	\begin{pmatrix}1&0\end{pmatrix}\begin{pmatrix}X&Y\\Z&W\end{pmatrix} &= R\begin{pmatrix}1&0\end{pmatrix},&
	\begin{pmatrix}X&Y\\Z&W\end{pmatrix}\begin{pmatrix}0\\1\end{pmatrix} &= S\begin{pmatrix}0\\1\end{pmatrix}.
\end{align*}
giving $X=R,Y=0$ and $W=S$. But the equations $B=C=Y=0$, $X=R$ and $W=S$ clearly imply the equations $AC+BD = 0$, $R+S-X-W=0$ and $RS-XW+YZ = 0$, so we see that $(\Pi^{1,2,1})^{-1}(x)\subseteq \Spec \Rc^{1,2,1}(\st_{2,2})(\Ebar)$.

Since $\Spec \Rc^{1,2,1}(\st_{2,2})(\Ebar)$ is stabilized by the action of $\Hc^{1,2,1}(\Ebar)$ it follows that $\Hc^{1,2,1}(\Ebar)(\Pi^{1,2,1})^{-1}(x)\subseteq\Spec \Rc^{1,2,1}(\st_{2,2})$, so taking Zariski closures indeed gives $\Nc^{1,2,1}(\st_{2,2})\subseteq \Spec \Rc^{1,2,1}(\st_{2,2})$.
\end{proof}

From this proposition, it follows that $\Nc^{1,2,1}(\st_{2,2})$ is an irreducible component of $\Spec \Rc^{1,2,1}(\st_{2,2})$ of relative dimension $6$ over $\Oc$. 

Let $\Rct^{1,2,1} = \Rc^{1,2,1}(\st_{2,2})/(R+S)$, so that $\Rct^{1,2,1}\cong\Oc[A,B,C,D,R,X,Y,Z]/\Ict$, where $\Ict$ is the ideal generated by the equations:
\begin{align*}
	\begin{pmatrix}A&B\end{pmatrix}\begin{pmatrix}X&Y\\Z&-X\end{pmatrix} &= R\begin{pmatrix}A&B\end{pmatrix},&
	\begin{pmatrix}X&Y\\Z&-X\end{pmatrix}\begin{pmatrix}C\\D\end{pmatrix} &= -R\begin{pmatrix}C\\D\end{pmatrix},\\
	AC+BD&=0, &
	R^2&=X^2+YZ.
\end{align*}
One can easily verify that there is an isomorphism $\Rc^{1,2,1}(\st_{2,2})\isomto\Rct^{1,2,1}[t]$ given by
\[
(A,B,C,D,R,S,X,Y,Z,W)\mapsto (A,B,C,D,t+R,t-R,t+X,Y,Z,t-X)
\]

We shall take $\Rct^{1,2,1}$ to be the ring $\Rct^{1,2,1}(\st_{2,2})$ from Theorem \ref{thm:R121 main results} (we are again omitting the $\st_{2,2}$ to simplify notation). By definition, $\Rct^{1,2,1}$ is a graded $\Oc$-algebra with irrelevant ideal $\Ic_+ = (A,B,C,D,R,X,Y,Z)$, and the isomorphism $\Rc^{2,2}(\st_{2,2})\isomto\Rct^{2,2}[t]$ identifies $\Spec (\Rct^{2,2}(\st_{2,2})/\Ic_+)\times \A^1_{\Oc}$ with the subscheme 
\[\left\{\left(t,\begin{pmatrix}0&0\end{pmatrix},\begin{pmatrix}t&0\\0&t\end{pmatrix}, \begin{pmatrix}0\\0\end{pmatrix},t\right)\right\}.\]
of $\Nc^{1,2,1}(\st_{2,2})\subseteq \Spec \Rc^{1,2,1}(\st_{2,2})$.

Now for any $\Oc$-algebra $\Ac$, let $\Rct^{2,2}_\Ac = \Rct^{2,2}\otimes_{\Oc} \Ac$ and $0_{\Ac} = (\Spec \Rct^{2,2}/\Ic_+)\times_{\Oc}\Ac$. Let $\Ict_A = \Ict \otimes_{\Oc} \Ac$, so that $\Rct^{2,2}_A = \Ac[A,B,C,D,R,X,Y,Z]/\Ict_{\Ac}$. We shall prove the following:

\begin{prop}\label{prop:R121tilde}
	For $k = \F,E$,	$\Rct^{1,2,1}_k$ is a normal Cohen--Macaulay domain of dimension $5$, which is not Gorenstein. $\Spec \Rct^{1,2,1}_k$ is smooth away from the point $0_{k}$.
\end{prop}

By the same argument as in Section \ref{sec:R22}, this will imply parts (\ref{R121:domain}), (\ref{R121:smooth}) and (\ref{R121:CM}) of Theorem \ref{thm:R121 main results}.

Let $\succ$ be the reverse lexicographical ordering on $k[A,B,C,D,R,X,Y,Z]$ with variable ordering $(X,Y,Z,A,C,B,D,R)$. Then we have:

\begin{prop}\label{prop:R121 Grob}
The polynomials
\begin{align*}
g_1 &= \ul{XA}+ZB-AR,&
g_2 &= \ul{YA}-XB-BR,&
g_3 &= \ul{XC}+YD+CR,\\
g_4 &= \ul{ZC}-XD+DR,&
g_5 &= \ul{X^2}+YZ-R^2,&
g_6 &= \ul{AC}+BD
\end{align*}
form a \Grobner basis for $\Ict_k$ with respect to $\succ$. In particular, $(R,D,B)$ is a regular sequence for $\Rct^{1,2,1}_k$.
\end{prop}
\begin{proof}
Note that $\{g_1,\ldots,g_6\}$ is just the generating set for $\Ict$ given above. Hence it generates $\Ict_k$, and so it remains to verify Buchberger's Criterion (Proposition \ref{prop:Buchberger}).

By the last part of Proposition \ref{prop:Buchberger}, we do not need to check (\ref{buchberger}) for $(g_1,g_4), (g_2,g_3),(g_2,g_4),(g_2,g_5),(g_4,g_5)$ or $(g_5,g_6)$. For the remaining $9$ pairs of polynomials we have the expressions:
\begin{align*}
\Ssf(g_1,g_2) &= Yg_1-Xg_2 = \ul{X^2B}+YZB-YAR+XBR = -Rg_2+Bg_5\\
\Ssf(g_1,g_3) &= Cg_1-Ag_3 = \ul{ZCB}-YAD-2ACR = -Dg_2+Bg_4-2Rg_6\\
\Ssf(g_1,g_5) &= Xg_1-Ag_5 = -\ul{YZA}+XZB-XAR+AR^2 = -Rg_1-Zg_2\\
\Ssf(g_1,g_6) &= Cg_1-Xg_6 = \ul{ZCB}-XBD-ACR = Bg_4-Rg_6\\
\Ssf(g_2,g_6) &= Cg_2-Yg_6 = -\ul{XCB}-YBD-BCR = -Bg_3\\
\Ssf(g_3,g_4) &= Zg_3-Xg_4 = \ul{X^2D}+YZD+ZCR-XDR = Rg_4+Dg_5\\
\Ssf(g_3,g_5) &= Xg_3-Cg_5 = -\ul{YZC}+XYD+XCR+CR^2 = Rg_3-Yg_4\\
\Ssf(g_3,g_6) &= Ag_3-Xg_6 = \ul{YAD}-XBD+ACR = Dg_2+Rg_6\\
\Ssf(g_4,g_6) &= Ag_4-Zg_6 = -\ul{XAD}-ZBD+ADR = -Dg_1,
\end{align*}
all of which satisfy the criterion $\LM(f^{(ij)}_s)\LM(g_s)\preceq \LM(\Ssf(g_i,g_j))$ from (\ref{buchberger}). So indeed, $\{g_1,g_2,g_3,g_4,g_5,g_6\}$ is a \Grobner basis for $\Ict_k$. The last claim follows from Proposition \ref{prop:Grob reg seq}.
\end{proof}

Now let 
\[\Sc_k = \Rct^{1,2,1}_k/(R,B,D) = \frac{k[A,C,X,Y,Z]}{(AX,AY,CX,CZ,AC,X^2+YZ)}\]
Letting $\alpha = Y-A$ and $\beta = C-Z$ we may rewrite this as $\Sc_k = k[X,Y,Z,\alpha,\beta]/\Jc$ where
\[
\Jc = \bigg(X(Y+\alpha),Y(Y+\alpha),X(Z+\beta),Z(Z+\beta),(Y+\alpha)(Z+\beta),X^2+YZ\bigg)\subseteq  k[X,Y,Z,\alpha,\beta]
\]
Now let $\succ_{\Sc}$ be the reverse lexicographical ordering on $k[X,Y,Z,\alpha,\beta]$ with variable ordering $(X,Y,Z,\alpha,\beta)$.

\begin{prop}\label{prop:S121 Grob}
The polynomials
\begin{align*}
	h_1 &= \ul{X}(\ul{Y}+\alpha),&
	h_2 &= \ul{Y}(\ul{Y}+\alpha),&
	h_3 &= \ul{X}(\ul{Z}+\beta),\\
	h_4 &= \ul{Z}(\ul{Z}+\beta),&
	h_5 &= (\ul{Y}+\alpha)(\ul{Z}+\beta),&
	h_6 &= \ul{X^2}+YZ
\end{align*}
form a \Grobner basis for $\Jc$ with respect to $\succ_{\Sc}$. In particular, $(\beta,\alpha) = (C-Z,Y-A)$ is a regular sequence for $\Sc_k$.
\end{prop}
\begin{proof}
By definition, $\Jc = (h_1,h_2,h_3,h_4,h_5,h_6)$, so we must verify Proposition \ref{prop:Buchberger}. By the last part of Proposition \ref{prop:Buchberger}, we do not need to check (\ref{buchberger}) for $(h_1,h_4),(h_2,h_3),(h_2,h_4),(h_2,h_6),(h_4,h_6)$ or $\Ssf(h_5,h_6)$. For the remaining $9$ pairs of polynomials we have:
\begin{align*}
\Ssf(h_1,h_2) &= Yh_1-Xh_2 = 0\\
\Ssf(h_1,h_3) &= Zh_1-Yh_3 = \ul{XZ\alpha}-XY\beta = -\beta h_1+\alpha h_3\\
\Ssf(h_1,h_5) &= Zh_1-Xh_5 = XZ(Y+\alpha)-X(Z+\beta)(Y+\alpha)=-\ul{X\beta}(\ul{Y}+\alpha) = -\beta h_1\\
\Ssf(h_1,h_6) &= Xh_1-Yh_6 = -\ul{Y^2Z}+X^2\alpha = -Zh_2+\alpha h_6\\
\Ssf(h_2,h_5) &= Zh_2-Yh_5 = YZ(Y+\alpha)-Y(Y+\alpha)(Z+\beta) = -\ul{Y\beta}(\ul{Y}+\alpha) = -\beta h_2\\
\Ssf(h_3,h_4) &= Zh_3-Xh_4 = 0\\
\Ssf(h_3,h_5) &= Yh_3-Xh_5 = XY(Z+\beta)-X(Z+\beta)(Y+\alpha)=-\ul{X\alpha}(\ul{Z}+\beta) = -\alpha h_3\\
\Ssf(h_3,h_6) &= Xh_3-Zh_6 = -\ul{YZ^2}+X^2\beta = -Yh_4+\beta h_6\\
\Ssf(h_4,h_5) &= Yh_4-Zh_5 = YZ(Z+\beta)-Z(Y+\alpha)(Z+\beta) = -\ul{Z\alpha}(\ul{Z}+\beta) = -\alpha h_4
\end{align*}
so indeed $\{h_1,h_2,h_3,h_4,h_5,h_6\}$ is a \Grobner basis for $\Jc$. Again, the last claim follows by Proposition \ref{prop:Grob reg seq}.
\end{proof}

Combining these propositions, we get that $(R,D,B,C-Z,Y-A)$ is a regular sequence for $\Rct^{1,2,1}_k$. But now
\[\frac{\Rct^{1,2,1}_k}{(R,D,B,C-Z,Y-A)} = \frac{\Sc}{(\alpha,\beta)} = \frac{k[X,Y,Z]}{(XY,Y^2,XZ,Z^2,YZ,X^2+YZ)} = \frac{k[X,Y,Z]}{(X^2,Y^2,Z^2,XY,XZ,YZ)},\]
which clearly has dimension $0$. It follows that $\Rct^{1,2,1}_k$ is Cohen--Macaulay of dimension $5$. Moreover, it is easy to check that ${k[X,Y,Z]}/{(X^2,Y^2,Z^2,XY,XZ,YZ)}$ is non-Gorenstein (with dualizing module $(X,Y,Z)$) and so $\Rct^{1,2,1}_k$ is non-Gorenstein (although we will give a different proof of this below).

To complete the proof of Proposition \ref{prop:R121tilde} we  need to show that $\Rct^{1,2,1}_k$ is a domain, smooth away from $0_{k}$. 

Let $\Hct_k = \Hc^{1,2,1}_k\times \G_{m,k} = \GL_{2,k}\times (\G_{m,k})^3$ treated as an algebraic group over $k$, and note that $\Hct_k$ acts on $\Spec \Rct^{1,2,1}_k$ via
\[
(g,r,s,t)\cdot\left(R; \begin{pmatrix}A&B\end{pmatrix};\begin{pmatrix} X&Y\\Z&-X\end{pmatrix}; \begin{pmatrix}C\\D\end{pmatrix}\right)
=
\left(rR; s\begin{pmatrix}A&B\end{pmatrix}g^{-1};rg\begin{pmatrix} X&Y\\Z&-X\end{pmatrix}g^{-1}; tg\begin{pmatrix}C\\D\end{pmatrix}\right)
\]
(for $g\in \GL_{2,k}$ and $r,s,t\in \G_{m,k}$). It is easy to check that the action of $\Hct_k(\kbar)$ on $(\Spec \Rct^{1,2,1}_k)(\kbar)$ has $12$ orbits, generated by the points:
\begin{align*}
x_{\e_1\e_2} &= \left(1; \begin{pmatrix}\e_1&0\end{pmatrix};\begin{pmatrix} 1&0\\0&-1\end{pmatrix}; \begin{pmatrix}0\\\e_2\end{pmatrix}\right)\\
y_{\e_1\e_2} &= \left(0; \begin{pmatrix}0&\e_1\end{pmatrix};\begin{pmatrix} 0&1\\0&0\end{pmatrix}; \begin{pmatrix}\e_2\\0\end{pmatrix}\right)\\
z_{\e_1\e_2} &= \left(0; \begin{pmatrix}\e_1&0\end{pmatrix};\begin{pmatrix} 0&0\\0&0\end{pmatrix}; \begin{pmatrix}0\\\e_2\end{pmatrix}\right)
\end{align*}
for $\e_1,\e_2\in \{0,1\}$. Now as $\Hct_k(\kbar)$ is a connected group variety of dimension $7$, for any $x\in (\Spec \Rct^{1,2,1}_k)(\kbar)$ the orbit $\Hct_k(\kbar)(x)$ is irreducible of dimension $7-\dim\stab_{\Hct_k(\kbar)}(x)$. As in Lemma \ref{lem:H22 orbits dimensions}, it is now easy to compute the dimension of each of these $12$ orbits (we again omit the details):

\begin{lemma}\label{lem:H121 orbits dimensions}
The dimensions of the $12$ orbits of $\Hct_k(\kbar)$ on $(\Spec \Rct^{1,2,1}_k)(\kbar)$ are given in the following table:
\[
\begin{array}{c|c|c|c|c|}
	(\e_1,\e_2) & (0,0) & (0,1) & (1,0) & (1,1)\\
	\hline
	\dim \Hct_k(\kbar)(x_{\e_1\e_2}) &3&4&4&5\\
	\hline
	\dim \Hct_k(\kbar)(y_{\e_1\e_2}) &2&3&3&4\\
	\hline
	\dim \Hct_k(\kbar)(z_{\e_1\e_2}) &0&2&2&3\\
	\hline
\end{array}
\]
\end{lemma}

In particular, $\dim \Hct_k(\kbar)(x_{11}) = 5$ and all other orbits have dimension strictly smaller than $5$, so by the same logic as in Section \ref{sec:R22}, $\Spec \Rct^{1,2,1}_k$ must be irreducible of dimension $5$ and 
$(\Spec \Rct^{1,2,1}_k)(\kbar) = \overline{\Hc_k(\kbar)(x_{11})}$.
	
Again, the only fixed point of the action is $z_{00} = 0_{k}$.

To show that $\Rct^{1,2,1}_k$ is a domain, it remains to show that $\Spec \Rct^{1,2,1}_k$ is reduced. Again as it is Cohen--Macaulay, it suffices to show it is generically reduced. Now we can compute:

\begin{lemma}\label{lem:R121 regular}
$\Rct^{1,2,1}_k$ is regular at all of the points $x_{\e_1\e_2}$, $y_{\e_1\e_2}$ and $z_{\e_1\e_2}$, except for $z_{00}$.
\end{lemma}
\begin{proof}
Let $g_1,\ldots,g_6$ be as in Proposition \ref{prop:R121 Grob}. Write $\vec{g} = \begin{pmatrix}g_1&g_2&g_3&g_4&g_5&g_6\end{pmatrix}$. Then we can compute that 
\[
\begin{pmatrix}
\del \vec{g}/\del A\\
\del \vec{g}/\del B\\
\del \vec{g}/\del C\\
\del \vec{g}/\del D\\
\del \vec{g}/\del R\\
\del \vec{g}/\del X\\
\del \vec{g}/\del Y\\
\del \vec{g}/\del Z
\end{pmatrix}
=
\begin{pmatrix}
 X-R&Y&0&0&0&C\\
Z&-X-R&0&0&0&D\\
0&0& X+R&Z&0&A\\
0&0&Y&-X+R&0&B\\
-A&-B&C&D&-2R&0\\
A&-B&C&-D&2X&0\\
0&A&D&0&Z&0\\
B&0&0&C&Y&0
\end{pmatrix}
\]
One can easily verify that this matrix has rank $8-5=3$ at each of the $11$ points specified (and is the zero matrix at at $z_{00}$) so as $\Spec \Rct^{1,2,1}$ has codimension $8-5=3$ in $\A^8_k$, it follows that $\Rct^{1,2,1}_k$ is regular at the requested points.
\end{proof}

As in Section \ref{sec:R22}, this implies that $(\Spec \Rct^{1,2,1}_k)\sm \{0_{k}\}$ is smooth, and hence it is generically reduced and so (as it is Cohen--Macaulay and irreducible) it is a domain. This completes the proof of Proposition \ref{prop:R121tilde}, and so of parts (\ref{R121:domain}), (\ref{R121:smooth}) and (\ref{R121:CM}) of Theorem \ref{thm:R121 main results}.

We now turn our attention to the computation of $\Cl(\Rct^{1,2,1})$ and $\omega$. Again let $k=\F$ or $E$. We shall first show that $\Cl(\Rct_k^{1,2,1}) \cong \Z$, which will again imply $\Cl(\Rct^{1,2,1})\cong \Z$ by Corollary \ref{cor:Cl(A)=Cl(A[1/x])}.

We will use Lemma \ref{lem:Cl(X-Z)}. We first consider the localization $\Rct_k^{1,2,1}[A^{-1}]$.

\begin{lemma}\label{lem:R[1/A]}
$\Rct_k^{1,2,1}[A^{-1}] = k[A,B,D,R,Z][A^{-1}]$. In particular, $\Cl(\Rct_k^{1,2,1}[A^{-1}]) = 0$.
\end{lemma}
\begin{proof}
	The equations $\begin{pmatrix}A&B\end{pmatrix}\begin{pmatrix}X&Y\\Z&-X\end{pmatrix} = R\begin{pmatrix}A&B\end{pmatrix}$ and $AC+BD=0$ imply that in $\Rc^{1,2,1}[A^{-1}]$ we have
	\begin{align*}
	X&= R-A^{-1}BZ, &
	Y&= A^{-1}B(R+X), &
	C &= -A^{-1}BD.
	\end{align*}
	Thus $\Rct_k^{1,2,1}[A^{-1}]$ can be generated by $A^{\pm 1}, B, D, R$ and $Z$. Moreover one can verify that these equations imply $\begin{pmatrix}X&Y\\Z&-X\end{pmatrix}\begin{pmatrix}C\\D\end{pmatrix} =  -R\begin{pmatrix}C\\D\end{pmatrix}$ and $R^2 = X^2+YZ$.	It follows that indeed $\Rct_k^{1,2,1}[A^{-1}] = k[A,B,D,R,Z][A^{-1}]$. As this is the localization of a UFD, it is also a UFD and so $\Cl(\Rct_k^{1,2,1}[A^{-1}]) = 0$.
\end{proof}

Now define the closed subschemes
\begin{align*}
\Dc_k &= \Spec \Rct^{1,2,1}_k/(A,B),&
\Dc_k' &= \Spec \Rct^{1,2,1}_k/(C,D),&
\Zc_k &= \Spec \Rct^{1,2,1}_k/(R),&
\end{align*}
of $\Spec\Rct^{1,2,1}_k$. Then
\begin{lemma}\label{lem:Dk and Zk prime}
$\Dc_k$, $\Dc_k$ and $\Zc_k$ are all geometrically integral of dimension $4$.
\end{lemma}
\begin{proof}
Note that by definition, the action of $\Hct_k$ on $\Spec \Rct^{1,2,1}_k$ preserves $\Dc_k$, $\Dc_k'$ and $\Zc_k$. Hence $\Dc_k(\kbar)$, $\Dc_k'(\kbar)$ and $\Zc_k(\kbar)$ are all unions of orbits of $\Hct_k(\kbar)$. From the description of the orbits given above we see that in fact:
\begin{itemize}
	\item $\Dc_k(\kbar)$ is a union of the orbits $\Hct_k(\kbar)(x_{0\e_2}),\Hct_k(\kbar)(y_{0\e_2})$ and $\Hct_k(\kbar)(z_{0\e_2})$ for $\e_2\in\{0,1\}$;
	\item $\Dc_k'(\kbar)$ is a union of the orbits $\Hct_k(\kbar)(x_{\e_10}),\Hct_k(\kbar)(y_{\e_10})$ and $\Hct_k(\kbar)(z_{\e_10})$ for $\e_1\in\{0,1\}$;
	\item $\Zc_k(\kbar)$ is a union of the orbits $\Hct_k(\kbar)(x_{\e_1\e_2})$ and $\Hct_k(\kbar)(y_{\e_1\e_2})$ for $\e_1,\e_2\in\{0,1\}$.
\end{itemize}
By Lemma \ref{lem:H121 orbits dimensions}, each of these contains exactly one orbit of dimension $4$, and all other orbits have dimension strictly less than $4$. Thus by the same logic as in Lemma \ref{lem:R/(AD-BC)}, $\Dc_k$, $\Dc_k'$ and $\Zc_k$ all have dimension $4$ with a single $4$-dimensional geometrically irreducible component. Call these $4$-dimensional irreducible components $\Dc_k^\circ$, $\Dc_k^{'\circ}$ and $\Zc_k^\circ$.

Now let $\Uc = \Spec \Rct_k^{1,2,1}[A^{-1}]$, which is a Zariski open subset of $\Spec \Rct_k^{1,2,1}$. By the computations in Lemma \ref{lem:R[1/A]} we get
\begin{align*}
\Dc'_k\cap \Uc &= \Spec \Rct_k^{1,2,1}[A^{-1}]/(C,D) = \Spec k[A,B,D,R,Z][A^{-1}]/(-A^{-1}BD,D)\\
&= \Spec k[A,B,D,R,Z][A^{-1}]/(D) = \Spec k[A,B,R,Z][A^{-1}]\\
\Zc_k\cap \Uc &= \Spec \Rct_k^{1,2,1}[A^{-1}]/(R) = \Spec k[A,B,D,R,Z][A^{-1}]/(R) = \Spec k[A,B,D,Z][A^{-1}].
\end{align*}
In particular $\Dc'_k\cap \Uc$ and $\Zc_k\cap \Uc$ are both geometrically reduced of dimension $4$, which in particular implies that $\Uc$ intersects  $\Dc_k^{'\circ}$ and $\Zc_k^\circ$ nontrivially and $\Dc_k^{'\circ}\cap \Uc$ and $\Zc_k^\circ\cap \Uc$ are reduced. In other words $\Dc_k^{'\circ}$ and $\Zc_k^\circ$ are generically geometrically reduced.

An analogous computation (using $\Rct^{1,2,1}_k[C^{-1}]$) gives that $\Dc_k^\circ$ is also generically geometrically reduced.

It now suffices to show that $\Dc_k$, $\Dc'_k$ and $\Zc_k$ are all Cohen--Macaulay. Indeed, this will imply that they are equidimensional of dimension $4$, and so they equal $\Dc_k^\circ$, $\Dc_k^{'\circ}$ and $\Zc_k^\circ$, respectively, all of which are irreducible. As these are also all generically geometrically reduced and Cohen--Macaulay, they are geometrically reduced, which finishes the argument.
 
As $\Rct^{1,2,1}_k$ is a domain, $R$ is a regular element, and so $\Zc_k = \Spec \Rct^{1,2,1}_k/(R)$ is automatically Cohen--Macaulay. 

Now $\Dc_k$ is cut out in $\Spec k[C,D,R,S,X,Y,Z,W]$ by the equations $\begin{pmatrix}X&Y\\Z&W\end{pmatrix}\begin{pmatrix}C\\D\end{pmatrix} = S\begin{pmatrix}C\\D\end{pmatrix}$, $RS = XW-YZ$ and $R+S=X+W$.

Making the variable change $x = X-S$, $w=W-S$ and substituting in $R = X+W-S = x+w+S$, these equations become $\begin{pmatrix}x&Y\\Z&w\end{pmatrix}\begin{pmatrix}C\\D\end{pmatrix} = \begin{pmatrix}0\\0\end{pmatrix}$ and $xw-YZ=0$ in $\Spec k[x,w,C,D,S,Y,Z]$.

It is easy to check that $\ul{xw}-YZ,\ul{xC}+YD$ and $\ul{wD}+ZC$ forms a \Grobner basis for the ideal defining $\Dc_k$ with respect to the reverse lexicographical monomial order with variable ordering $(x,w,C,D,Y,Z,S)$ and so Proposition \ref{prop:Grob reg seq} implies that $(Z,Y,S)$ is regular for $\Rct^{1,2,1}_k/(A,B)$. Now $\Rct^{1,2,1}_k/(A,B,S,Y,Z)\cong k[x,w,C,D]/(xw,xC,wD)$, and it's easy to check that $(C-D,x+w+C)$ is a regular sequence for this ring. Thus $\Dc_k$ is indeed Cohen--Macaulay. By symmetry, $\Dc_k'$ is Cohen--Macaulay as well, completing the proof.
\end{proof}

So $\Dc_k$, $\Dc_k'$ and $\Zc_k$ are all prime divisors in $\Spec \Rct^{1,2,1}_k$. The definition of $\Zc_k$ gives $\div(R) = \Zc_k$, so $\Zc_k\sim 0$. Also we have

\begin{lemma}\label{lem:R/A}
$\Spec \Rct_k^{1,2,1}/(A)$ is reduced and equidimensional of dimension $4$. It has $2$ irreducible components: $\Dc_k$, and $\Ec_k := \Spec \Rct_k^{1,2,1}/(A,D,Z,R+X)$. Hence $\div(A) = \Dc_k+\Ec_k$.
\end{lemma}
\begin{proof}
Lemma \ref{lem:Dk and Zk prime} implies that $\Dc_k$ is reduced and irreducible of dimension $4$. Also, one can show that ${\Rct_k^{1,2,1}}/{(A,D,Z,R+X)} = k[B,C,X,Y]$ and so $\Ec\cong \A^4_k$, which is obviously reduced and irreducible.

Write $\Fc_k = \Spec \Rct_k^{1,2,1}/(A)$. As $\Rct_k^{1,2,1}$ is Cohen--Macaulay and $A$ is a regular element on $\Rct_k^{1,2,1}$ (as $\Rct^{1,2,1}_k$ is a domain), $\Fc_k$ is Cohen--Macaulay of dimension $4$. This implies that $\Fc_k$ is equidimensional of dimension $4$, and reduced provided that it is generically reduced. 

By definition, $\Dc_k\cup\Ec_k\subseteq \Fc_k$, and so $\Dc_k$ and $\Ec_k$ are both irreducible components of $\Fc_k$. Now take any $\qf\in \Spec \Rct^{1,2,1}_k/(A)$, that is, a prime ideal $\qf\subseteq \Rct^{1,2,1}_k$ containing $A$. 

Now as $\Rct_k^{1,2,1}/(A)$ is the quotient of $k[B,C,D,R,X,Y,Z]$ defined by the equations
\begin{align*}
	\begin{pmatrix}0&B\end{pmatrix}\begin{pmatrix}X&Y\\Z&-X\end{pmatrix} &= R\begin{pmatrix}0&B\end{pmatrix},&
	\begin{pmatrix}X&Y\\Z&-X\end{pmatrix}\begin{pmatrix}C\\D\end{pmatrix} &= -R\begin{pmatrix}C\\D\end{pmatrix},\\
	BD&=0, &
	R^2&=X^2+YZ.
\end{align*}
we get that $BD,BZ,BR+BX\in\qf$. As $\qf$ is prime, this implies that either $B\in \qf$ or $D,Z,R+X\in \qf$. In the first case, $\qf\in\Dc_k$ and in the second case, $\qf\in\Ec_k$. This implies that $\Dc_k$ and $\Ec_k$ are the only irreducible components $\Fc_k$.

Now let $\Uc_B = \Spec \Rct^{1,2,1}_k[B^{-1}]$ and $\Uc_D = \Spec \Rct^{1,2,1}_k[D^{-1}]$. By the same logic as above, inverting $B$ in the above equations implies that $D=Z=R+X=0$, while inverting $D$ implies $B=0$. Hence $\Fc_k\cap \Uc_B$ is a closed subscheme of $\Ec_k\cap \Uc_B$ and $\Fc_k\cap \Uc_D$ is a closed subscheme of $\Dc_k\cap \Uc_D$. As we already have $\Dc_k,\Ec_k\subseteq \Fc_k$ as schemes this implies that $\Fc_k\cap \Uc_B=\Ec_k\cap \Uc_B$ and $\Fc_k\cap \Uc_D=\Dc_k\cap \Uc_D$ as schemes. Since $\Ec_k\cap \Uc_B$ and $\Dc_k\cap \Uc_D$ are reduced and clearly nonempty, this implies that $\Fc_k$ is generically reduced on each component, and so is reduced.

The statement that $\div(A) = \Dc_k+\Ec_k$ is now immediate.
\end{proof}

As $\Dc_k$ and $\Ec_k$ are geometrically integral, Theorem \ref{thm:R121 main results}(\ref{R121:Rational}) now follows with $\Zc_k^{1,2,1} := \Dc_k\cup\Ec_k$.

Now by Lemmas \ref{lem:Cl(X-Z)}, \ref{lem:R[1/A]} and \ref{lem:R/A}, there is an exact sequence $\Z^2\to \Cl(\Rct^{1,2,1}_k)\to \Cl(\Rct_k^{1,2,1}[1/A]) = 0$ and so $\Cl(\Rct_k^{1,2,1})$ is generated by $[\Dc_k]$ and $[\Ec_k]$. Moreover, as $\div(A) = \Dc_k+\Ec_k$, $[\Dc_k] = - [\Ec_k]$ in $\Cl(\Rct_k^{1,2,1})$, and so $\Cl(\Rct_k^{1,2,1})$ is generated by $[\Dc_k]$.

We claim that $[\Dc_k]$ has infinite order in $\Cl(\Rct_k^{1,2,1})$, and hence $\Cl(\Rct_k^{1,2,1})\cong \Z$. Assume that $m[\Dc_k] = 0$ for some $m\in\Z$. Then by definition, there is some $f\in K(\Rct_k^{1,2,1})$ with $\div(f) = m\Dc$ (which in particular implies $\div(f)\ge 0$, so $f\in \Rct_k^{1,2,1}$). But this implies that $\supp f \cap \Spec \Rct_k^{1,2,1}[1/A] = \es$, and so $f$ is a unit in $\Rct_k^{1,2,1}[1/A]$. But now by Lemma \ref{lem:R[1/A]}, 
\[\Rct_k^{1,2,1}[1/A]^\times = k[A,B,D,R,Z,W][A^{-1}]^\times = \{uA^x|u\in k^\times, x\in\Z\}.\]
But if $f=uA^\times$ then $m\Dc_k+0\Ec_k = \div(f) = x\div(A)= x\Dc_k+x\Ec_k$ and so $m=x=0$. So indeed $\Cl(\Rc^{1,2,1})\cong \Z$. Now define an isomorphism $i_k:\Cl(\Rct_k^{1,2,1})\isomto\Z$ by $\Dc_k\mapsto -1$.

Also define the isomorphism $i:\Cl(\Rct^{1,2,1})\isomto \Z$ to be the composition $\Cl(\Rct^{1,2,1})\isomto \Cl(\Rct^{1,2,1}_E)\xrightarrow{i_E}\Z$ where the first map is the isomorphism $[M]\mapsto [M\otimes_\Oc E]$ from Corollary \ref{cor:Cl(A)=Cl(A[1/x])}.

Let $\Dc = \Spec \Rct^{1,2,1}/(A,B)$. From Lemma \ref{lem:R/A} and Corollary \ref{cor:graded domain} it follows that $\Dc$ is a prime divisor of $\Spec \Rct^{1,2,1}$. Note that $i(\Dc) = i_E(\Dc\cap \Spec \Rct^{1,2,1}_E) = i_E(\Dc_E)=-1$, so $\Dc$ is a generator of $\Cl(\Rct^{1,2,1})$.

Let $\omegat$ be the dualizing module of $\Rct^{1,2,1}$ and for $k =\F,E$ let $\omegat_k = \omegat \otimes_\Oc k$, and note that this is the dualizing module of $\Rct^{1,2,1}_k$. To prove Theorem \ref{thm:R121 main results}(\ref{R121:class group}) it remains to show that $i_k([\omegat_k]) = 2$ (as then $i([\omegat]) = i_E([\omegat_E]) = 2$ by definition).

As $(A,B)$ and $(C,D)$ are prime ideals of $\Rct^{1,2,1}_k$ by Lemma \ref{lem:Dk and Zk prime}, Lemma \ref{lem:O(-D) = P} gives $\Oc(-\Dc_k)=(A,B)$ and $\Oc(-\Dc'_k)=(C,D)$.

Now as $\Rct_k^{1,2,1}$ is a domain,
\[
\Oc(-\Dc_k) = (A,B)\cong C(A,B) = (AC, BC) = (-BD,BC) = B(C,D)\cong (C,D) = \Oc(-\Dc_k')
\]
as $\Rct_k^{1,2,1}$-modules, and so $[\Dc_k] = [\Dc'_k]$ in $\Cl(\Rct_k^{1,2,1})$.

\begin{prop}\label{prop:omega}
$\omegat_k \cong\Oc(-\Dc_k-\Dc'_k)\cong \Oc(-2\Dc_k)$ and $\omegat \cong \Oc(-2\Dc)$. If $\Mc_k$ is a self-dual $\Rct_k^{1,2,1}$-module of generic rank $1$, then $\Mc_k\cong \Oc(-\Dc_k)$. If $\Mc$ is a self-dual $\Rct^{1,2,1}$-module of generic rank $1$, then $\Mc\cong \Oc(-\Dc)$
\end{prop}
\begin{proof}
Since the isomorphism $\Cl(\Rct^{1,2,1})\to \Cl(\Rct^{1,2,1}_E)$ maps $[\Dc]$ to $[\Dc_E]$, $[\Dc']$ to $[\Dc'_E]$ and $\omega$ to $\omega_E$, the claims for $\Rct^{1,2,1}$ follow immediately from those for $\Rct^{1,2,1}_E$, so it is enough to prove the claims for $\Rct^{1,2,1}_k$ with $k = \F,E$.

The last claim will follow from the computation of $\omega_k$, as $\Mc_k$ is self-dual if and only if $2 \Mc_k  =  \omegat_k $ in $\Cl(\Rct_k^{1,2,1})$, and there can only be one solution to this equation, as $\Cl(\Rct_k^{1,2,1})\cong\Z$ is $2$-torsion free.

To determine $\omegat_k$, recall the isomorphism $\Rc^{1,2,1}_k(\st_{2,2}) \cong \Rct^{1,2,1}_k[t]$ from above, and let $\omega_k$ be the dualizing module of $\Rc^{1,2,1}_k$. Then we have $\omega_k/t\omega_k \cong \omegat_k$, so it will suffice to determine $\omega_k$.

Now recall that the irreducible components of $\Nc^{1,2,1}_k$ are in bijection with the set
\[
\Is^{1,2,1} = \{(0,0;0),(1,0;0),(0,1;0),(1,1;0),(1,1;0)\}.
\]
For each $P\in \Is^{1,2,1}$, let $\Nc^{1,2,1}_k(P)$ be the irreducible component of $\Nc^{1,2,1}_k$ corresponding to $P$. Then $\Nc^{1,2,1}(1,1;0) = \Spec \Rc^{1,2,1}(\st_{2,2})$. Now by Corollary \ref{cor:Nc^nun is CI}, $\Nc^{1,2,1}_k$ is a reduced complete intersection, and in particular is Gorenstein, hence we can compute $\omega_k$ using Lemma \ref{lem:Gorenstein component}.

For $P\in \Is^{1,2,1}\sm\{(1,1;0)\}$, let $\Yc_P = \Nc^{1,2,1}_k(P)\cap \Nc^{1,2,1}_k(1,1;0) \subseteq \Spec \Rc^{1,2,1}_k(\st_{2,2})$. To use Lemma \ref{lem:Gorenstein component}, we must first determine $\dim \Yc_P$ for each $P$. Recall that $\dim \Nc^{1,2,1}_k = 1^2+2^2+1^2 = 6$

First, on $\Nc^{1,2,1}_k(0,0;0)$ we have $A=B=C=D=0$, and so $\Yc_{(0,0;0)}$ is a subscheme of
\[
\Spec \frac{\Rc^{1,2,1}(\st_{2,2})}{(A,B,C,D)} = \Spec\frac{k[R,S,W,X,Y,Z]}{(R+S-X-W,RS-XW+YZ)} = \Spec\frac{k[R,W,X,Y,Z]}{(R(X+W-R)-XW+YZ)}
\]
and so $\dim \Yc_{(0,0;0)}\le 4 < 5$.

Next, on $\Nc^{1,2,1}_k(0,1;0)$ we have $A=B=0$, and so $\Yc_{(0,1;0)}$ is a closed subscheme of 
\[\Spec \Rc^{1,2,1}_k(\st_{2,2})/(A,B)\cong \Spec \Rct^{1,2,1}_k[t]/(A,B)\cong \Dc_k\times_k \A^1_k.\]
Moreover, for any $x = (R,N_1,X_2,N_2,S)\in (\Spec \Rc^{1,2,1}(\st_{2,2})/(A,B))(\kbar)$ we have $N_1 = 0$ by definition, so 
\[\Pi^{1,2,1}(x) = (N_1,N_2) = (0,N_2) \in \Wc^{1,2,1}_k(0,1;0)\sqcup \Wc^{1,2,1}_k(0,0;0).\]
Hence 
\begin{align*}
\Spec \left(\frac{\Rc^{1,2,1}(\st_{2,2})}{(A,B)}\right)(\kbar)&\subseteq (\Spec \Rc^{1,2,1}(\st_{2,2}))(\kbar)\cap \left((\Pi^{1,2,1})^{-1}\left(\Wc^{1,2,1}_k(0,1;0)(\kbar)\sqcup \Wc^{1,2,1}_k(0,0;0)(\kbar)\right)\right)\\
&\subseteq \Yc_{(0,1;0)}(\kbar)\cup \Yc_{(0,0;0)}(\kbar)
\end{align*}
and so $\Spec \Rc^{1,2,1}_k(\st_{2,2})/(A,B)$ is a subscheme of $\Yc_{(0,1;0)}\cup \Yc_{(0,0;0)}$. As $\Spec \Rc^{1,2,1}_k(\st_{2,2})/(A,B) = \Dc_k\times_k\A^1_k$ is reduced and irreducible of dimension $5$, while   $\dim \Yc_{(0,0;0)} \le 4$, this implies that $\dim \Yc_{(0,1;0)} = 5$ and $\Yc_{(0,1;0)} = \Spec \Rc^{1,2,1}_k(\st_{2,2})/(A,B) = \Dc_k\times \A^1_k$ (where we have identified $\Rc^{1,2,1}_k(\st_{2,2})= \Rct^{1,2,1}_k[t]$).

A similar argument gives $\Yc_{(1,0;0)} = \Spec \Rc^{1,2,1}_k(\st_{2,2})/(C,D) =  \Dc_k'\times\A^1_k$ and $\dim \Yc_{(1,0;0)} = 5$.

Lastly consider $\Yc_{(1,1;1)}$. Take any $(R,N_1,X_2,N_2,S)\in (\Pi^{1,2,1})^{-1}(\Wc^{1,2,1}_k(1,1;1))$. Then by definition we have $R(N_1N_2) = N_1X_2N_2 = N_1N_2S = S(N_1N_2)$ and $\rank(N_1N_2) = 1$, giving $N_1N_2 \in k^\times$, so $R = S$. It follows that $R=S$ on $\Nc^{1,2,1}(1,1;1)$, and so $R=S$ on $\Yc_{(1,1;1)}$. Thus $\Yc_{(1,1;1)}$ is a closed subscheme of $\Spec \Rc^{1,2,1}_k(\st_{2,2})/(R-S)$. But now the isomorphism $\Rc^{1,2,1}_k(\st_{2,2})\isomto \Rct^{1,2,1}_k[t]$ identifies $\Rc^{1,2,1}_k(\st_{2,2})/(R-S)$ with $(\Rct^{1,2,1}_k/(R))[t]$ (recall that $\Char k \ne 2$) and so we get $\Yc_{(1,1;1)}\subseteq \Zc_k\times\A^1_k$. Lemma \ref{lem:Dk and Zk prime} now implies that either $\Yc_{(1,1;1)}= \Zc_k\times \A^1_k$ or $\dim \Yc_{(1,1;1)}< 5$.\footnote{In fact, one can show that $\Yc_{(1,1;1)}= \Zc_k\times\A^1_k$. We will omit this argument, as it is not needed for any of our results.}

If $\Yc_{(1,1;1)}= \Zc_k\times \A^1_k$, Lemma \ref{lem:Gorenstein component} gives 
\[\omega_k \cong \Oc(-\Yc_{(0,1;0)}-\Yc_{(1,0;0)}-\Yc_{(1,1;1)}) = \Oc\left(-(\Dc_k\times\A^1_k)-(\Dc_k'\times\A^1_k)-(\Zc_k\times\A^1_k)\right)\]
and so the isomorphism $\omegat_k\cong \omega_k/t\omega_k$ gives $\omegat_k\cong \Oc(-\Dc_k-\Dc'_k-\Zc_k)$.

Similarly if $\dim \Yc_{(1,1;1)}< 5$ then we get $\omega_k \cong\Oc\left(-(\Dc_k\times\A^1_k)-(\Dc_k'\times\A^1_k)\right)$ and so $\omegat_k\cong \Oc(-\Dc_k-\Dc'_k)$.

But now as $\Zc_k\sim 0$, $\Oc(-\Dc_k-\Dc'_k-\Zc_k)\cong \Oc(-\Dc_k-\Dc'_k)$ and so in either case we get $\omegat_k\cong \Oc(-\Dc_k-\Dc'_k)$. The rest of the claims follow immediately from this, together with the fact that $\Dc_k\sim\Dc'_k$ and $\Cl(\Rct_k^{1,2,1})\cong \Z$ is generated by $[\Dc_k]$.
\end{proof}
This completes the proof of \ref{thm:R121 main results}(\ref{R121:class group}). Also as $[\omegat_k]\ne 0$ in $\Cl(\Rct_k^{1,2,1})$, $\omegat_k$ is indeed not free as $\Rct_k^{1,2,1}$-module, and so $\Rct_k^{1,2,1}$ is not Gorenstein.

Thus, $\Mc := \Oc(-\Dc)=(A,B)$ is the unique self-dual $\Rct^{1,2,1}$-module of generic rank $1$, and for $k = \F,E$, $\Mc_k := \Oc(-\Dc_k) = (A,B) = \Mc\otimes_{\Oc}k$ is the unique self-dual $\Rct_k^{1,2,1}$-module of generic rank $1$. This proves Theorem \ref{thm:R121 main results}(\ref{R121:M}). 

As $\Rct^{1,2,1}$ is a graded $\Oc$-algebra and $\Mc = (A,B)$ is a homogeneous ideal it is now clear that $\Mc/\Ic^+\Mc \cong \Oc A \oplus \Oc B\cong \Oc^2$, proving Theorem \ref{thm:R121 main results}(\ref{R121:mult 2}).

We now turn our attention to Theorem \ref{thm:R121 main results}(\ref{R121:surj}). First, Lemma \ref{lem:O(-D) = P} gives 
\begin{align*}
\omegat &\cong \Oc(-\Dc-\Dc') = (A,B)\cap (C,D).
\end{align*}
To simplify this description we will show the following:

\begin{lemma}\label{lem:(A B)cap(C D)}
$(A,B)\cap (C,D) = (A,B)(C,D)$ in $\Rct^{1,2,1}$.
\end{lemma}
\begin{proof}
By Proposition \ref{prop:graded p cap q = pq}  it suffices to show that $(A,B)\cap (C,D) = (A,B)(C,D)$ in $\Rct^{1,2,1}_\F$ and in $\Rct^{1,2,1}_E$ (since $(A,B)$ and $(C,D)$ are distinct height $1$ primes in $\Rct^{1,2,1}_k$ for each $k=\F,E$). As usual let $k$ denote either $\F$ or $E$.

Consider the polynomial ring $\Pc_k = k[A,B,C,D,R,X,Y,Z]$ and again let $\succ$ be the reverse lexicographical ordering with variable ordering $(X,Y,Z,A,C,B,D,R)$. Then $\Rct_k^{1,2,1} = \Pc_k/(g_1,g_2,g_3,g_4,g_5,g_6)$ where
\begin{align*}
g_1 &= \ul{XA}+ZB-AR,&
g_2 &= \ul{YA}-XB-BR,&
g_3 &= \ul{XC}+YD+CR,\\
g_4 &= \ul{ZC}-XD+DR,&
g_5 &= \ul{X^2}+YZ-R^2,&
g_6 &= \ul{AC}+BD
\end{align*}
as in Proposition \ref{prop:R121 Grob}. Let 
\begin{align*}
 I_k &= (g_1,g_2,g_3,g_4,g_5,g_6,A,B) = (g_3,g_4,g_5,A,B) \subseteq \Pc_k\\
 J_k &= (g_1,g_2,g_3,g_4,g_5,g_6,C,D) = (g_1,g_2,g_5,C,D) \subseteq \Pc_k
\end{align*}
so that $I_k/(g_1,\ldots,g_6) = (A,B)\subseteq \Rct^{1,2,1}_k$ and $J_k/(g_1,\ldots,g_6) = (C,D)\subseteq \Rct^{1,2,1}_k$. It will suffice to compute $I_k\cap J_k$ in $\Pc_k$, which we will do via \Grobner bases, using Proposition \ref{prop:Grob intersection}.

Let $\succ_t$ be the monomial ordering on $\Pc_k[t] = k[A,B,C,D,R,X,Y,Z,t]$ defined in Section \ref{ssec:Grobner} and let 
\[\Kc_k 
= \left(ta+(1-t)b\middle|a\in I_k,b\in J_k\right).\] Using the above generating sets, and the fact that $I_k\cap J_k\subseteq \Kc_k$ we get
\begin{align*}
\Kc_k &= (g_1,g_2,g_3,g_4,g_5,\ul{tA},\ul{tB},\ul{tC}-C,\ul{tD}-D).
\end{align*}
Now we claim the set
\[
\Gs = \{g_1,g_2,g_3,g_4,g_5,\ul{tA},\ul{tB},\ul{tC}-C,\ul{tD}-D,\ul{AC},\ul{AD},\ul{CB},\ul{BD}\}
\]
is a \Grobner basis for $\Kc_k$ with respect to $\succ_t$. Clearly $\ul{AC},\ul{AD},\ul{CB},\ul{BD}\subseteq I_kJ_k\subseteq I_k\cap J_k\subseteq \Kc_k$, so this is a generating set for $\Kc_k$. It remains to check (\ref{buchberger}) for each unordered pair of distinct elements $f,g\in \Gs$.

The proof of Proposition \ref{prop:R121 Grob} showed that each pair $(g_i,g_j)$ with $i,j\in\{1,2,3,4,5\}$ satisfies (\ref{buchberger}) (since none of these polynomials contain the variable $t$, these computations are unaffected by the change in monomial ordering).

Next we consider the $28$ pairs of distinct elements $f,g \in \{\ul{tA},\ul{tB},\ul{tC}-C,\ul{tD}-D,\ul{AC},\ul{AD},\ul{CB},\ul{BD}\}$. It is easy to check that $\Ssf(f,g) = 0$ for all $f,g\in \{\ul{tA},\ul{tB},\ul{AC},\ul{AD},\ul{CB},\ul{BD}\}$ and that $\Ssf(\ul{tC}-C,\ul{tD}-D) = 0$, which accounts for $\binom{6}{2}+1 = 16$ of these pairs. Of the remaining $12$ pairs, if 
\[(f,g) = (\ul{tC}-C,\ul{AD}), (\ul{tC}-C,\ul{BD}),  (\ul{tD}-D,\ul{AC}), (\ul{tD}-D,\ul{CB})\]
then $\gcd(\LT(f),\LT(g)) = 1$. The remaining $8$ pairs are in the form $(tx,ty-y)$ or $(xy,ty-y)$ for $x\in\{A,B\}$ and $y\in \{C,D\}$, which all satisfy (\ref{buchberger}) as
\begin{align*}
\Ssf(tx,ty-y) = \Ssf(xy,ty-y) &= txy-x(ty-y) = xy \in \{\ul{AC},\ul{AD},\ul{CB},\ul{BD}\} \subseteq \Gs.
\end{align*}

It remains to show that $(f,g_i)$ satisfies (\ref{buchberger}) for each $f\in \{\ul{tA},\ul{tB},\ul{tC}-C,\ul{tD}-D,\ul{AC},\ul{AD},\ul{CB},\ul{BD}\}$ and each $i=1,2,3,4,5$. If $i=5$ then $\LT(g_5) = X^2$, which is relatively prime to $\LT(f)$. If $i=1,2$ then $\LT(g_i)$ is relatively prime to $\LT(f)$ unless $f = sA$ for $s\in \{t,C,D\}$. For any such $s$ we have
\begin{align*}
\Ssf(g_1,sA) &= s(\ul{XA}+ZB-AR)-X(sA) = sZB-sAR = Z(sB)-R(sA)\\
\Ssf(g_2,sA) &= s(\ul{YA}-XB-BR)-Y(sA) = sXB-sBR = (\ul{X}-R)(sB)
\end{align*}
which all satisfy (\ref{buchberger}) as $sA,sB\in\Gs$. Lastly if $i=3,4$ then $\LT(g_i)$ is relatively prime to $\LT(f)$ unless $f = \ul{tC}-C, \ul{AC}$ or $\ul{CB}$. Then we have
\begin{align*}
\Ssf(g_3,\ul{tC}-C) &= t(\ul{XC}+YD+CR)-X(\ul{tC}-C) = \ul{tYD}+tCR+XC = \ul{R}(\ul{tC}-C)+\ul{Y}(\ul{tD}-D)+g_3\\
\Ssf(g_4,\ul{tC}-C) &= t(\ul{ZC}-XD+DR)-Z(\ul{tC}-C) = -\ul{tXD}+tDR+ZC = (-\ul{X}+R)(\ul{tD}-D)+g_4
\end{align*}
and for $s\in\{A,B\}$ we have
\begin{align*}
\Ssf(g_3,sC) &= s(\ul{XC}+YD+CR)-X(sC) = sYD+sCR = Y(sD)-R(sC)\\
\Ssf(g_4,sC) &= s(\ul{ZC}-XD+DR)-Z(sC) = -sXD+sDR = (\ul{X}-R)(sD)
\end{align*}
which all satisfy (\ref{buchberger}) (as $sC,sD\in\Gs$).

Thus by Proposition \ref{prop:Buchberger}, $\Gs$ is indeed a \Grobner basis for $\Kc_k$.

Hence by Proposition \ref{prop:Grob intersection}, $I_k\cap J_k = (g_1,g_2,g_3,g_4,g_5,AC,AD,BC,BD)$ in $\Pc_k$, and so in $\Rc_k/(g_1,\ldots,g_6) = \Rct^{1,2,1}_k$ we have
\[
(A,B)\cap(C,D) = (AC,AD,BC,BD) = (A,B)(C,D),
\]
as desired.
\end{proof}

So now 
\[\omegat\cong (A,B)\cap (C,D) = (A,B)(C,D) = (AC,AD,BC,BD) = (AC,AD,BC)\]
(where the last equality is from $BD=-AC$ in $\Rct^{1,2,1}$). 

By the above we have $\Mc = (A,B) = \Oc(-\Dc) \cong \Oc(-\Dc') = (C,D)$ and $(A,B)(C,D) = \Oc(-\Dc-\Dc') \cong \omegat$. It follows that there exists an $\Rct^{1,2,1}$-module surjection $\Mc\otimes_{\Rct^{1,2,1}}\Mc\onto \omegat$. By \cite[Lemma 3.4]{Manning}, this implies that $\tau_\Mc:\Mc\otimes_{\Rct^{1,2,1}}\Mc\to \omegat$ is surjective.

This completes the proof of Theorem \ref{thm:R121 main results}.
\bibliographystyle{amsalpha}
{\footnotesize
\bibliography{refs}}
\end{document}